\documentclass[11pt,a4]{amsart}

\usepackage{amsfonts,amsmath,amssymb,bbm}
\usepackage{verbatim}
\usepackage{textgreek,environ}
\usepackage{color}
\usepackage{hyperref}
\usepackage{etoolbox}
\makeatletter
\let\ams@starttoc\@starttoc
\makeatother
\usepackage[parfill]{parskip}
\makeatletter
\let\@starttoc\ams@starttoc
\patchcmd{\@starttoc}{\makeatletter}{\makeatletter\parskip\z@}{}{}
\makeatother
\usepackage{soul}

\usepackage{enumerate}
\usepackage{pdfsync}
\usepackage{color}

\usepackage{mathbbol}
\usepackage{amsfonts}
\DeclareSymbolFontAlphabet{\mathbb}{AMSb}
\DeclareSymbolFontAlphabet{\mathbbl}{bbold}

\numberwithin{equation}{section}

\newtheorem{theorem}{Theorem}[section]
\newtheorem{proposition}[theorem]{Proposition}
\newtheorem{lemma}[theorem]{Lemma}
\newtheorem{corollary}[theorem]{Corollary}

\theoremstyle{definition}

\theoremstyle{remark}
\newtheorem{remark}[theorem]{Remark}

\newcommand{\lb}{\left(}
\newcommand{\rb}{\right)}
\newcommand{\lbb}{\left[}
\newcommand{\rbb}{\right]}
\newcommand{\labs}{\left|}
\newcommand{\rabs}{\right|}
\newcommand{\lbrb}[1]{\lb#1\rb}
\newcommand{\lbbrbb}[1]{\lbb#1\rbb}
\newcommand{\labsrabs}[1]{\labs#1\rabs}
\newcommand{\lbbrb}[1]{\lbb#1\rb}
\newcommand{\lbrbb}[1]{\lb#1\rbb}

\newcommand{\lbcurly}{\left\{}
\newcommand{\rbcurly}{\right\}}
\newcommand{\lbcurlyrbcurly}[1]{\lbcurly#1\rbcurly}
\newcommand{\lrfloor}[1]{\lfloor#1\rfloor}

\newcommand{\simi}{\stackrel{\infty}{\sim}}
\newcommand{\simo}{\stackrel{0}{\sim}}

\newcommand{\intervalII}{\lbrb{-\infty,\infty}}
\newcommand{\intervalOI}{\lbrb{0,\infty}}
\newcommand{\abs}[1]{\labs#1\rabs}

\newcommand{\curly}[1]{\lbcurly#1\rbcurly}
\newcommand{\complex}[2]{#1+i#2}

\newcommand{\bospace}[1]{\mathrm{O}\!\lbrb{#1}}
\renewcommand{\so}[1]{\mathrm{o}\lbrb{#1}}
\newcommand{\sospace}[1]{\mathrm{o}\!\lbrb{#1}}

\newcommand{\pls}{+}
\newcommand{\mis}{-}
\newcommand{\tightoverset}[2]{%
	\mathop{#2}\limits^{\vbox to 0.28ex{\kern -0.2ex\hbox{$#1$}\vss}}}

\newcommand{\pms}{\pm}

\newcommand{\Pbb}[1]{\Pb\lb #1\rb}
\newcommand{\Pbbs}[1]{\Pb\!\lb #1\rb}
\newcommand{\Ebb}[1]{\Eb\lbb #1\rbb}

\newcommand{\LL}{L\'{e}vy }
\newcommand{\LLP}{L\'{e}vy process }
\newcommand{\LLPs}{L\'{e}vy processes }
\newcommand{\LLK}{\LL\!\!-Khintchine }

\newcommand{\Pnn}{\Pi_{\mis}}
\newcommand{\PP}{\overline{\Pi}}

\newcommand{\PPn}{\overline{\Pi}_{\mis}}

\newcommand{\PPPn}{\overline{\overline{\Pi}}_{\mis}}

\newcommand{\mubr}{\bar{\mu}}
\newcommand{\mubar}[1]{\bar{\mu}\!\lbrb{#1}}
\newcommand{\mubrp}{\bar{\mu}_{\pls}}
\newcommand{\mubarp}[1]{\bar{\mu}_{\pls}\lbrb{#1}}
\newcommand{\mubarpspace}[1]{\bar{\mu}_{\pls}\!\lbrb{#1}}
\newcommand{\mubrn}{\bar{\mu}_{\mis}}
\newcommand{\mubarn}[1]{\bar{\mu}_{\mis}\lbrb{#1}}
\newcommand{\mubarnspace}[1]{\bar{\mu}_{\mis}\!\lbrb{#1}}
\newcommand{\mup}[1]{\mu_{\pls}\lbrb{#1}}

\newcommand{\munspace}[1]{\mu_{\mis}\!\lbrb{#1}}
\newcommand{\uun}[1]{\upsilon_{\mis}\lbrb{#1}}
\newcommand{\uunspace}[1]{\upsilon_{\mis}\!\lbrb{#1}}

\newcommand{\ttinf}[1]{_{#1\to \infty}}

\newcommand{\limi}[1]{\lim\limits_{#1\to \infty}}
\newcommand{\limsupi}[1]{\varlimsup\limits_{#1\to \infty}}
\newcommand{\liminfi}[1]{\varliminf\limits_{#1\to \infty}}
\newcommand{\limo}[1]{\lim\limits_{#1\to 0}}
\newcommand{\limsupo}[1]{\varlimsup\limits_{#1\to 0}}

\newcommand{\Cb}{\mathbb{C}}
\newcommand{\C}{\mathbb{C}}
\newcommand{\Eb}{\mathbb{E}}

\newcommand{\Kb}{\mathbb{K}}
\newcommand{\Nb}{\mathbb{N}}
\newcommand{\N}{\mathbb{N}}
\newcommand{\Rb}{\mathbb{R}}
\newcommand{\R}{\mathbb{R}}
\newcommand{\Rp}{\mathbb{R}^+}
\newcommand{\Rn}{\mathbb{R}^-}
\newcommand{\Pb}{\mathbb{P}}

\newcommand{\Zb}{\mathbb{Z}}
\renewcommand{\P}{\mathbb{P}}

\newcommand{\Ac}{\mathcal{A}}
\newcommand{\Bc}{\mathcal{B}}

\newcommand{\Hc}{\mathcal{H}}

\newcommand{\Mcc}{\mathcal{M}}
\newcommand{\M}{\mathcal{M}}
\newcommand{\Nc}{\mathcal{N}}
\newcommand{\Oc}{\mathcal{O}}
\newcommand{\Pc}{\mathcal{P}}

\newcommand{\Sc}{\mathcal{S}}

\newcommand{\Wc}{\mathcal{W}}

\newcommand{\Zc}{\mathcal{Z}}

\newcommand{\npow}{\mathcal{P}}
\newcommand{\nexp}{\mathcal{E}}

\newcommand{\Be}{\mathcal{B}}

\newcommand{\Bi}{\Be_{\npow}(\infty)}
\newcommand{\Beb}{\Be_{\npow}(\beta)}
\newcommand{\Bep}[1]{\Be_{\npow}(#1)}

\newcommand{\Bth}{\Be_{\nexp}(\theta)}

\newcommand{\Bthd}{\Be_{\nexp}(\frac{\pi}{2})}

\newcommand{\Bthaspace}{\Be_{\nexp}\!\lbrb{\frac{\pi}{2}\alpha}}

\newcommand{\Ne}{\mathcal{N}}

\newcommand{\Npi}{\overNc_{\npow}(\infty)}
\newcommand{\Npb}{\overNc_{\npow}(\beta)}
\newcommand{\NpNP}{\overNc_{\npow}(\NPs)}
\newcommand{\NepN}{\Ne_{\npow}(\NPs)}

\newcommand{\Nn}{\Ne_{\npow}(\Ntt_{\Psi})}
\newcommand{\Ni}{\Ne_{\npow}(\infty)}

\newcommand{\Nth}{\overNc_{\nexp}(\Theta)}

\newcommand{\Ntpspace}[1]{\overNc_{\nexp}\!\left(#1\right)}

\newcommand{\Nthdspace}{\overNc_{\nexp}\!\left(\frac{\pi}{2}\right)}

\newcommand{\Nthaspace}{\overNc_{\nexp}\!\left(\frac{ \pi}{2}\alpha\right)}

\newcommand{\supp}{{\rm{Supp}}}
\newcommand{\Fo}[1]{\mathcal{F}_{#1}}
\newcommand{\Ctt}{\mathtt{C}}

\newcommand{\Mtt}{\mathtt{M}}
\newcommand{\Ntt}{\mathtt{N}}

\newcommand{\ak}{\widehat{a}}
\newcommand{\bk}{\widehat{b}}

\newcommand{\rk}{\mathfrak{r}}

\newcommand{\aP}{\mathfrak{a}_{\Psi}}
\newcommand{\ap}{\mathfrak{a}_{\pls}}

\newcommand{\Ir}{I}

\newcommand{\ind}[1]{\mathbb{I}_{\{#1\}}}

\newcommand{\IntOI}{\int_{0}^{\infty}}
\newcommand{\IntII}{\int_{-\infty}^{\infty}}
\newcommand{\IntIO}{\int_{-\infty}^{0}}

\newcommand{\Lspace}[2]{\mathrm{L}^{#1}\lbrb{#2}}
\newcommand{\Lspaces}[2]{\mathrm{L}^{#1}\!\lbrb{#2}}

\newcommand{\LtwoRp}{L^2\lbrb{\Rb^+}}
\newcommand{\cco}{\mathtt{C}_0}

\newcommand{\ccoiRp}{\mathtt{C}^{\infty}_0\lbrb{\Rb^+}}

\newcommand{\overNc}{\overline{\Nc}}

\newcommand{\Aph}{A_\phi}
\newcommand{\aph}{\mathfrak{a}_\phi}
\newcommand{\aphp}{\mathfrak{a}_{{\pls}}}
\newcommand{\aphn}{\mathfrak{a}_{{\mis}}}
\newcommand{\Aphn}{A_{\phi_{\mis}}}

\newcommand{\Aphp}{A_{\phi_{\pls}}}
\newcommand{\Aphm}{A_{\phi_{\pm}}}
\newcommand{\dph}{\overline{\mathfrak{a}}_\phi}
\newcommand{\dphn}{\overline{\mathfrak{a}}_{\mis}}
\newcommand{\dphp}{\overline{\mathfrak{a}}_{\pls}}
\newcommand{\phn}{\phi_{\mis}}
\newcommand{\phnspace}{\phi_{\mis}\!}
\newcommand{\php}{\phi_{\pls}}
\newcommand{\phpspace}{\phi_{\pls}\!}
\newcommand{\phpm}{\phi_{\pms}}

\newcommand{\dr}{{\mathtt{d}}}
\newcommand{\dep}{{\dr}_{\pls}}
\newcommand{\dem}{{\dr}_{\mis}}
\newcommand{\ptt}{{\mathtt{p}}}

\newcommand{\Eph}{E_\phi}
\newcommand{\Gph}{G_\phi}
\newcommand{\Hph}{H_\phi}
\newcommand{\Rph}{R_\phi}
\newcommand{\gamph}{\gamma_\phi}

\newcommand{\Tph}{T_\phi}
\newcommand{\tph}{\mathfrak{u}_{\phi}}
\newcommand{\tphp}{\mathfrak{u}_{{\pls}}}

\newcommand{\tphn}{\mathfrak{u}_{\mis}}
\newcommand{\Wp}{W_\phi}
\newcommand{\Wpn}{W_{\phi_{{\mis}}}}
\newcommand{\Wpnspace}{W_{\phi_{{\mis}}}\!}
\newcommand{\Wpp}{W_{\phi_{\pls}}}
\newcommand{\Wppspace}{W_{\phi_{\pls}}\!}
\newcommand{\Zcph}{\Zc_0(\phi)}

\newcommand{\PcBinf}{ \Bi }

\newcommand{\PcNinf}{\Ni}

\newcommand{\dPs}{d_\Psi}
\newcommand{\pPs}{f_\Psi}

\newcommand{\PIp}{F_\Psi}

\newcommand{\PIPs}[1]{\overline{F}_{\Psi}(#1)}
\newcommand{\PIoPs}[1]{F_{\Psi}(#1)}

\newcommand{\PIoPsnspace}[1]{F^{(n)}_{\Psi}\!(#1)}

\newcommand{\thPsml}{\mathfrak{u}_{-}^{\hspace{-0.05cm}\leta}}

\newcommand{\AbOI}{\Ac_{\intervalOI}}
\newcommand{\CbOI}{\Cb_{\intervalOI}}

\newcommand{\IPsi}{I_\Psi}
\newcommand{\IPsiq}[1]{I_{\Psi_{#1}}}
\newcommand{\IPsq}[1]{I_{\Psi^{#1}}}

\newcommand{\MPsi}{\Mcc_{I_\Psi}}
\newcommand{\MPsispace}{\Mcc_{I_\Psi}\!}
\newcommand{\MPIsi}{\Mcc_{F_\Psi}}
\newcommand{\MPoIsi}{\Mcc_{\overline{F}_\Psi}}
\newcommand{\MPIssi}{H_{\Psi}}
\newcommand{\MPs}{M_{\Psi}}
\newcommand{\MPspace}{M_{\Psi}\!}
\newcommand{\MPsq}[1]{M_{\Psi^{\dagger #1}}}
\newcommand{\MPsqq}[1]{M_{\Psi^{#1}}}
\newcommand{\MPsiq}{\Mcc_{I_{\Psi^{\dagger q }}}}

\newcommand{\MPIsia}[1]{\Mcc_{V_{#1}^q}}

\newcommand{\pPsi}{f_\Psi}
\newcommand{\gPsi}{g_\Psi}
\newcommand{\hPsi}{h_\Psi}
\newcommand{\hPsih}{h_{\Psi,1}}
\newcommand{\hPsihh}{h_{\Psi,2}}
\newcommand{\VPsi}{V_\Psi}

\newcommand{\NPs}{\mathtt{N}_{\Psi}}
\newcommand{\Nps}{\mathtt{N}_{\phi}}

\newcommand{\leta}{(\eta)}

\newcommand{\Psime}{\Psi^{\leta}}

\newcommand{\seta}{(\eta)}

\newcommand{\Logm}[1]{{\underline{\mathfrak{L}}}_{#1}}
\newcommand{\Logp}[1]{{\mathfrak{L}}_{#1}}

\newcommand{\phntc}{\phi_{\mis}^{ (c)}}

\newcommand{\Untc}{U^{(c)}}

\newcommand{\phnstc}{\tilde{\phi}_{\mis}^{(c)}}
\newcommand{\ab}{a+ib}
\newcommand{\BP}{\Bc_{\dr}}
\renewcommand{\Ac}{\mathtt{A}}
\renewcommand{\Im}{\mathtt{Im}}
\renewcommand{\Re}{\mathtt{Re}}

\renewcommand{\unrhd}{\text{ent}}

\NewEnviron{seq}{
	\begin{equation}\begin{split}
	\BODY
	\end{split}\end{equation}
}
\NewEnviron{seq*}{
	\begin{equation}\begin{split}
	\BODY
	\end{split}\end{equation}
}
\author{P. Patie}
\address{School of Operations Research and Information Engineering, Cornell University, Ithaca, NY 14853.}
\email{	pp396@cornell.edu}

\author{M. Savov} \thanks{This work was partially supported by NSF Grant DMS-
	1406599 and ARC IAPAS, a fund of the Communaut´ee francaise de Belgique. The second author also acknowledges the support of the project MOCT, which has received funding from the European Union’s
	Horizon 2020 research and innovation programme under the Marie Sklodowska-Curie grant
	agreement No 657025.}
\address{Institute of Mathematics and informatics,  Bulgarian academy of sciences, Akad.  Georgi Bonchev street
	Block 8, Sofia 1113.}
\email{mladensavov@math.bas.bg}

\title[Bernstein-gamma functions and exponential functionals]{Bernstein-gamma functions and exponential functionals of L\'{e}vy Processes}

\begin{document}
	\begin{abstract}
		In this work we analyse the solution to the recurrence equation
		\[\MPs(z+1)=\frac{-z}{\Psi(-z)}\MPs(z), \quad \MPs(1)=1,\]
		defined on a subset of the imaginary line and where $-\Psi$ runs through the set of all continuous negative definite functions. Using the analytic Wiener-Hopf method we furnish the solution to this equation as a product of functions that extend the classical gamma function.  These latter functions,  being in bijection with the class of Bernstein functions, are called Bernstein-gamma functions.  Using their Weierstrass product representation we establish universal Stirling type asymptotic which is explicit in terms of the constituting Bernstein function.  This allows the thorough understanding of the decay of $\abs{\MPs(z)}$ along imaginary lines and an access to quantities important for many theoretical and applied studies in probability and analysis.
		
		This functional equation appears as a central object in several recent studies ranging from analysis and spectral theory to probability theory. In this paper, as an application of the results above, we investigate from a global perspective the exponential functionals of \LLPs whose Mellin transform satisfies the equation above. Although these variables have been intensively studied, our new approach based on a combination of probabilistic and analytical techniques enables us to derive comprehensive properties and strengthen several results on the law of these random variables for some classes of L\'evy processes that could be found in the literature. These encompass smoothness for its density, regularity and  analytical properties,  large and small asymptotic behaviour, including asymptotic expansions,  bounds, and Mellin-Barnes representations of its successive derivatives.  We also furnish a thorough study of the weak convergence of exponential functionals on a finite time horizon when the latter expands to infinity.  As a result of new Wiener-Hopf and  infinite product factorizations of the law of the exponential functional  we deliver important intertwining relation between members of the class of positive self-similar semigroups.
		Some of the results presented in this paper have been announced in the note \cite{Patie-Savov-13}.
	\end{abstract}

		\maketitle
		\newpage
	\tableofcontents



\section{Introduction}\label{sec:intro}
The main aim of this work is to develop an in-depth analysis of the solution  in the space of analytic functions and in the space of Mellin transforms of random variables to the functional equation defined,  for any function $\Psi $ such that $e^{\Psi} \in \mathcal{P}$, where $\mathcal{P}$ is the set of positive definite functions or equivalently $-\Psi$ is  a continuous negative definite function, by
\begin{equation}\label{eq:fe}
\MPs(z+1)=\frac{-z}{\Psi(-z)}\MPs(z), \quad \MPs(1)=1,
\end{equation}
and valid (at least) on the domain $i\R\setminus\lbrb{\Zc_0\!\lbrb{\Psi}\cup\curly{0}}$, where we set \[\Zc_0\!\lbrb{\Psi}=\curly{z\in i\R:\,\Psi(-z)=0}.\label{sym:P}\]
The foremost motivation underlying this study is the methodology underpinning an approach developed by the authors for understanding the spectral decomposition of at least some non-self-adjoint Markov semigroups. This program has been carried out for a class of generalized Laguerre semigroups and a class of positive self-similar semigroups, see the recent papers \cite{Patie-Savov-16}, \cite{Patie-Savov-GL}, \cite{Patie-Savov-Zhao}  and \cite{Patie-Zhao}. The latter study has revealed that solutions to equations of the type \eqref{eq:fe}  play a central role in obtaining and quantitatively characterizing the spectral representation of the aforementioned semigroups. Since \eqref{eq:fe} is  defined on a subset of $ i\R$, a natural approach to derive and understand its solution stems from the classical Wiener-Hopf method. It is well-known that for any $\Psi\in\overNc$, where  $\overNc$\label{overNc} stands for the space of the negative of continuous negative definite functions, see \eqref{eq:defoNc}  below for definition, we have the analytic Wiener-Hopf factorization
\begin{eqnarray}\label{eq:WH1}
\Psi(z)=-\php(-z)\phn(z),\: z\in i\R,
\end{eqnarray}
where $\php,\phn\in\Bc$\label{Bc}, that is $\phi_{\pms}$ are Bernstein functions, see \eqref{eq:phi0}  for definition. {Note that  $\phpm$ are uniquely defined up to a positive multiplicative constant. The latter as will be seen is of no significance.
	Exploiting \eqref{eq:WH1} the characterization of the solution of \eqref{eq:fe}   in the space of analytic functions and in  the space of Mellin transforms of random variables is,  up to an uniqueness argument,  reduced to the solution  in the space of Mellin transforms of random variables of equations of the type
	\begin{eqnarray}\label{eq:feb}
	\qquad\qquad\,\,\,\qquad W_{\phi}(z+1)&=&\phi(z)W_{\phi}(z),\quad\Wp(1)=1, \quad\phi\in\Bc, \label{eq:Wp}
	\end{eqnarray}
	on the region  $z\in\CbOI=\curly{z\in\Cb:\,\Re(z)>0}$.
	In turn the solution to \eqref{eq:feb} can be  represented on $\CbOI$\label{CbI} as an infinite Weierstrass product involving $\phi\in\Bc$, see \cite[Chapter 6]{Patie-Savov-16}. In Theorem \ref{thm:Wp} we obtain the main complex-analytical properties of $\Wp$ via a couple of parameters pertaining to all $\phi\in\Bc$. Also, new asymptotic representations of $\Wp$ are contained in Theorem \ref{thm:Stirling}. From them, the asymptotic of $\Wp$ along imaginary lines of the form $a+i\R, a>0,$ can be related to the geometry of the set $\phi\lbrb{\CbOI}$.  In many instances  this asymptotic can be computed or well-estimated as it is amply illustrated in Proposition \ref{thm:imagineryStirling}.
	
	All results are reminiscent of the Stirling asymptotic for the gamma function which solves \eqref{eq:feb} with $\phi(z)=z$. For this purpose we call the functions $\Wp$ Bernstein-gamma functions. Due to their ubiquitous presence in many theoretical studies they are an important class of special functions. If $\Rp=\intervalOI$  the restriction of \eqref{eq:feb}  on $\Rp$ has been considered in a larger generality by Webster in \cite{Webster-97} and for the class of Bernstein functions by Hirsch and Yor \cite{Hirsch-Yor-13}. More information on the literature can be found in Section \ref{sec:mainResults}.
	
	The results on the solution $\Wp$  of \eqref{eq:feb}  lead to asymptotic representation and characterization  of particular, generic solution  $\MPs$ of \eqref{eq:fe}.  The complex-analytical properties of $\MPs$ are then described fully in terms of four global parameters describing the analyticity and the roots of $\phpm$. The latter are related to the properties of  $\Psi$ as stated by Remark \ref{rem:d}. In Theorem \ref{thm:asympMPsi} we also conduct asymptotic analysis of $\abs{\MPs(z)}$. We emphasize that \eqref{eq:WH1} does not reduce the study of \eqref{eq:fe} to the decoupled investigation of \eqref{eq:feb} for $\phpm$. In fact the interplay between $\php$ and $\phn$ induced by \eqref{eq:WH1} is the key to get exhaustive results on the properties of $\MPs$ as illustrated by \eqref{eq:NPs} of Theorem \ref{thm:asympMPsi}. The latter gives explicit information regarding the rate of polynomial decay of $\abs{\MPspace(z)}$ along complex lines of the type $a+i\R$.
	
	As a major application of our results we present a general and unified study of the exponential functionals of \LL processes.  To facilitate the discussion of our main motivation, aims and achievements in light of the existing body of literature we recall that a possibly killed \LLP $\xi=(\xi_t)_{t\geq0}$ is an a.s.~right-continuous, real-valued stochastic process which possesses stationary and independent increments that is killed at an independent of the conservative version of $\xi$ exponential random variable (time) $\textbf{e}_q$ of parameter $q\geq 0$  and $\xi_t=\infty$ for any $t\geq {\textbf{e}_q}$.  Note that $\textbf{e}_0=\infty$.  The law of a possibly killed \LLP $\xi$ is  characterized via its characteristic exponent $\Psi\in\overNc$ with killing rate $-\Psi(0)=q$, i.e.  \[\Eb\left[e^{z\xi_t}\right]=e^{\Psi(z)t}, z \in i\R,\]
	 which defines a bijection between the class of possibly killed \LLPs and $\overNc.$ Denote the exponential functional of the L\'evy process $\xi$ by
	\begin{equation*}
	\IPsi(t) =\int_0^{t}e^{-\xi_s}ds, \quad t\geq 0,
	\end{equation*}
	and its associated perpetuity by
	\begin{equation}\label{eq:Ipsi}
	\IPsi =\int_0^{\infty}e^{-\xi_s}ds=\int_{0}^{\textbf{e}_q}e^{-\xi_s}ds.
	\end{equation}
	The study of $\IPsi$ has been initiated by Urbanik in \cite{Urbanik-95} and proceeded by M. Yor with various co-authors, see e.g.~ \cite{Bertoin-Yor-05,Hirsch-Yor-13,Yor-01}. There is also a number of subsequent and intermediate contributions to the study of these random variables, a small sample of which comprises of \cite{Alili-Jadidi-14, Arista-Rivero-16,Pardo2012,Patie2012,Patie-Savov-11,Patie-Savov-13,Patie-Savov-16}. The interest in these variables seems to be due to the fact that they appear and play a crucial role in various theoretical and applied contexts such as the spectral theory of some non-reversible Markov semigroups (\cite{Patie-Savov-16,Patie-Savov-GL}),  the study of random planar maps (\cite{Bertoin-Curien-15}), limit theorems of Markov chains (\cite{Bertoin-Igor-16}),  positive self-similar Markov processes  (\cite{Bertoin-Caballero-02,Bertoin-Yor-02-b,Caballero-Chaumont-06-b,Patie2009d}), financial and insurance mathematics (\cite{Hackmann-Kuznetsov-14,Patie-As}), branching processes with immigration (\cite{Patie2009}), fragmentation processes (\cite{Bertoin-Yor-05}), random affine equations, perpetuities, etc.. Starting from \cite{Maulik-Zwart-06} it has become gradually evident that studying the Mellin transform of the exponential functional is the right tool in many contexts. For particular subclasses, that include and allow the study of the supremum of the stable process, this transform has been evaluated and sometimes via inversion the law of exponential functional has been obtained, see \cite{Bernyk2008,Kuznetsov2011,Kuznetsov-Pardo,Patie2012}.
	In this paper, as a consequence of the study of $\MPs$ and $\Wp$, and the fact that whenever $\IPsi<\infty$ the Mellin transform of $\IPsi$, denoted by $\MPsispace$\label{sym:MT}, satisfies $\MPsispace(z)=\phn(0)\MPs(z)$ at least on $\Re(z)\in\lbrb{0,1}$, we obtain, refine and complement various results on the law of $\IPsi$, which we now briefly summarize.
	
	\begin{enumerate}[a)]
		\item Deriving  information on the decay of $\abs{\MPs(z)}$ along complex lines allows us to show that the law of $\IPsi$ is infinitely differentiable unless $\xi$ is a compound  Poisson process with strictly positive drift. In the latter case \eqref{eq:NPs} of Theorem \ref{thm:asympMPsi} evaluates the minimum number of smooth derivatives the law of $\IPsi$ possesses.
		
		\item Under no restriction we provide a Mellin-Barnes representation for the law of $\IPsi$ and thereby bounds for the law of $\IPsi$ and its derivatives. In Theorem \ref{cor:smoothness}\eqref{it:smallExp} and Corollary \ref{cor:smallExp} we show that polynomial small asymptotic expansion is possible if and only if the \LLP is killed, in which case we obtain explicit evaluation of the terms of this expansion.
		
		\item In Theorem \ref{thm:largeAsymp} general results on the tail of the law are offered. These include the computation of both the Pareto index for any exponential functional and under Cram\'er's condition, depending on the decay of  $\abs{\MPsispace(z)}$ and under minute additional requirements, the  asymptotic of the tail and its derivatives. The latter for example immediately recovers the asymptotic behaviour of the density of the supremum of a stable \LL process as investigated in \cite{Bernyk2008,Doney-Savov-10,Kuznetsov2011,Patie2012}.
		
		\item In Theorem \ref{thm:smallTime} comprehensive results are also presented for the behaviour of the law at zero.
		
		\item Next, when $\limi{t}\IPsi(t)=\int_{0}^{\infty}e^{-\xi_s}ds=\infty$ and under the Spitzer's condition for $\xi$ we establish the  weak convergence of the probability measures $\Pbb{\IPsi(t)\in dx}$ after proper rescaling in time and space. This result is particularly relevant  in the world of random processes in random environments, where such information strengthens significantly
		the results of \cite{Li-XU-16,Palau-Pardo-Smadi-16}. To achieve this we have developed necessary and sufficient conditions for the finiteness of the positive and negative moments of $\IPsi(t)$.

		\item We proceed by showing that the Wiener-Hopf type factorization of the law of $\IPsi$ which was proved, under some moderate technical conditions, in \cite{Pardo2012,Patie-Savov-11} holds in fact in complete generality, see Theorem \ref{thm:factorization} which also contains additional interesting factorizations.
		
		\item Finally, by means of a classical relation between the entrance law of  positive self-similar Markov processes and the law of the exponential functional of L\'evy processes, we compute explicitly the Mellin transform of the former, see Theorem \ref{thm:PSSMP}\eqref{it:entrance}. Moreover, exploiting this relation and the Wiener-Hopf decomposition of the law of $\IPsi$ mentioned earlier, we derive intertwining relations between the positive self-similar semigroups, see Theorem \ref{thm:PSSMP}\eqref{it:intertwining}.
	\end{enumerate}
	
	The Mellin transform $\MPsi$ turns out to be the key tool for understanding the exponential functional of a \LL process. The reason for the latter is the representation of $\MPsi$ as a product combination of the Bernstein-gamma functions $\Wpn$ and $\Wpp$. Thus  the complex-analytical and asymptotic properties of $\MPsi$ are accessible. However, as it can be most notably seen in the proofs of Theorem \ref{thm:asympMPsi}, Theorem \ref{thm:largeAsymp} and Theorem \ref{thm:smallTime}, the most precise results depend on mixing analytical tools with probabilistic techniques and the properties of \LL processes.
	
	The paper is structured as follows. Section \ref{sec:MR} is dedicated to the main results and their statements. Section \ref{sec:BernFunc} outlines the main properties of the Bernstein functions and proves some new results on them. Section \ref{sec:mainResults} introduces and  studies in detail the Bernstein-gamma functions. Section \ref{sec:ProofsBE} furnishes the proof regarding the statements of Sections \ref{sec:BernFunc} and \ref{sec:mainResults}. Section \ref{sec:FE} considers the proofs of the results related to the functional equation \eqref{eq:fe}. Section \ref{sec:EFLP} furnishes the proofs for the results regarding the exponential functionals of \LL processes. Section \ref{sec:addons} deals with the factorizations of the law of the exponential functional and the intertwinings between  positive self-similar semigroups. The Appendix provides some additional information on \LLPs and results on them that cannot be easily detected in the literature, e.g.~the version of  \textit{\'{e}quation amicale invers\'{e}e} for killed \LL processes. Also we have included  Section \ref{sec:symbols} which contains tables with the main symbols introduced and used throughout the paper.
	\section{Main Results}\label{sec:MR}
	\subsection{Wiener-Hopf factorization, Bernstein-Weierstrass representation and the asymptotic analysis of the solution of \eqref{eq:fe} }
	We start by introducing some notation. We use $\N$ for the set of non-negative integers and the standard notation $\Ctt^k(\Kb)$\label{Ck} for the $k$ times differentiable functions on  some complex or real domain $\Kb$. The space $\cco^k\lbrb{\Rp}$\label{Cok} stands for the $k$ times differentiable functions which together with their $k$ derivatives vanish at infinity, whereas $\Ctt_b^k\lbrb{\Rp}$\label{Cob} requires only boundedness. When $k=0$ we drop the superscript. For any $z\in\Cb$ set $z=|z|e^{i \arg z}$ with the branch of the argument function defined via the convention $\arg:\Cb\mapsto\lbrbb{-\pi,\pi}$. For any $-\infty\leq a<b\leq \infty$, we denote by $\Cb_{\lbrb{a,b}}=\lbcurlyrbcurly{z\in\Cb:\,a<\Re(z)<b}$ and for any $a\in\intervalII$ we set $\Cb_a=\lbcurlyrbcurly{z\in\Cb:\,\Re(z)=a}$\label{Ca}. We use $\Ac_{\lbrb{a,b}}$ for the set of holomorphic functions on $\Cb_{\lbrb{a,b}}$, whereas if $-\infty<a$ then $\Ac_{\lbbrb{a,b}}$ stands for the holomorphic functions on $\Cb_{\lbrb{a,b}}$ that can  be extended continuously to $\Cb_a$. Similarly, we have the spaces $\Ac_{\lbbrbb{a,b}}$ and $\Ac_{\lbrbb{a,b}}$. Finally, we use  $\Mtt_{\lbrb{a,b}}$  for the set of meromorphic functions on $\Cb_{\lbrb{a,b}}$.
	
	We proceed by recalling the definition of the set of  positive definite functions
	\[ \mathcal{P} =\{M : i\R \mapsto \C;\:  \forall n \in \N, \sum_{j,k=1}^n M(s_j-s_k)z_j\bar{z}_k \geq 0 \textrm{ for } s_1,\ldots, s_n \in i\R, z_1,\ldots,z_n \in \C \}, \]
	that is $M \in \mathcal{P}$ is the characteristic function of a real-valued random variable, or, equivalently, the Mellin transform of a positive random variable.
	
	It is well-known, see \cite{Jacob-01},  that  a function $\Psi : i\R \rightarrow \mathbb{C}$  is  such that $e^{\Psi} \in \mathcal{P}$, i.e.~$-\Psi$ is a continuous negative definite function, if and only if  it admits the following L\'evy-Khintchine representation
	\begin{eqnarray} \label{eq:lk0}
	\Psi(z) =  \frac{\sigma^2}{2} z^2 + \gamma z +\IntII \left(e^{zr} -1
	-zr \ind{|r|<1}\right)\Pi(dr)+\Psi(0),
	\end{eqnarray}
	where $\Psi(0)\leq0$, $\sigma^2 \geq 0$, $\gamma\in \R$, and, the sigma-finite measure
	$\Pi$ satisfies the integrability condition
	\[\IntII(1
	\wedge r^2 )\:\Pi(dr) <\infty.\]
	We then write
	\begin{equation}\label{eq:defoNc}
	\overNc = \{\Psi : i\R \rightarrow \mathbb{C}; \: \Psi \textrm{ is of the form } \eqref{eq:lk0}\}.
	\end{equation}
	
	As a subset of continuous negative definite functions we have the set of Bernstein functions $\Bc$ which consists of all functions $\phi\not\equiv0$ represented as follows
	\begin{equation}\label{eq:phi0}
	\phi(z)=\phi(0)+\dr z+\IntOI \lbrb{1-e^{-zy}}\mu(dy)=\phi(0)+\dr z+z\IntOI e^{-zy}\mubar{y}dy,
	\end{equation}
	at least for $z\in\Cb_{\lbbrb{0,\infty}}$, where $\phi(0)\geq 0,\,\dr\geq 0$, $\mu$ is a sigma-finite measure satisfying
	\[\IntOI (1\wedge y) \mu(dy)<\infty\quad\text{ and }\quad\mubar{y}=\int_{y}^{\infty}\mu(dr),\,y\geq0.\]
	We then set
	\begin{equation}\label{eq:defbe}
	\Be = \left\{ \phi : \Cb_{\lbbrb{0,\infty}} \rightarrow \mathbb{C}; \: \phi \textrm{ is of the form } \eqref{eq:phi0}\right\}.
	\end{equation}

    With any function $\phi\in\Be$ since $\phi\in\Ac_{\lbbrb{0,\infty}}$, see \eqref{eq:phi0}, we associate the quantities
    \begin{align}
    \aph&=\inf\,\,\,\{u<0:\phi\in  \Ac_{(u,\infty)}\}\in [-\infty,0],\label{eq:aphi}\\
    \tph&=\sup\lbcurlyrbcurly{u\in\lbbrbb{\aph,0}:\phi(u)=0}\in\lbbrbb{-\infty,0},\label{eq:thetaphi}\\
    \dph&=\max\curly{\aph,\tph}\in\lbbrbb{-\infty,0},\label{eq:dphi}
    \end{align}
    which are well defined thanks to the form of $\phi$, see \eqref{eq:phi0}, and the convention $\sup{\emptyset}=-\infty$ and $\inf{\emptyset}=0$. Note that $\phi$ is a non-zero constant if and only if (iff) $\dph=-\infty$. Indeed, otherwise, if $\aph=-\infty$ then necessarily from \eqref{eq:phi0}, $\limi{a}\phi(-a)=-\infty$. Hence, $\tph>-\infty$ and $\dph=\tph\in\lbrbb{-\infty,0}$. For these three quantities associated to $\phpspace,\phn$ in \eqref{eq:WH1} for the sake of clarity we drop the subscript $\phi$ and use $\tphp,\tphn,\aphp,\aphn,\dphp,\dphn$.
    For $\phi\in\Bc$, we write the generalized Weierstrass product
    	\begin{equation*}
    W_\phi(z)=\frac{e^{-\gamma_{\phi}z}}{\phi(z)}\prod_{k=1}^{\infty}\frac{\phi(k)}{\phi(k+z)}e^{\frac{\phi'\!(k)}{\phi(k)}z},\,z\in\CbOI,
    \end{equation*}
    where
    \begin{equation}\label{sym:gphi}
    \gamma_\phi=\lim_{n \to \infty} \lb \sum_{k=1}^n\frac{\phi'\!(k)}{\phi(k)} - \log \phi(n) \rb\in\lbbrbb{-\ln\phi(1),\frac{\phi'\!(1)}{\phi(1)}-\ln\phi(1)}.
    \end{equation}
    For more information on the existence and properties of $\Wp$ and $\gamph$ we refer to Section \ref{sec:mainResults}.
    Observe that if $\phi(z)=z$, then $W_\phi$ corresponds to the Weierstrass product representation of the gamma function $\Gamma$, valid on $\C \setminus \{{0,-1,-2,\ldots}\}$, and $\gamma_\phi$ is the Euler-Mascheroni constant, see e.g.~\cite{Lebedev-72}, justifying both the terminology and notation. Note also that when $z=n \in \N$ then \[W_\phi(n+1)=\prod_{k=1}^n\phi(k).\]
    We are ready to state the first of our central results which, for any $\Psi\in\overNc$, provides an explicit representation of $\MPs$ in terms of generalized Bernstein-gamma functions.
    
    \begin{theorem}\label{thm:FormMellin} \label{lemma:FormMellin}
    	Let $\Psi\in \overNc$ and recall that $\Psi(z)=-\php(-z)\phn(z),\: z\in i\R$. Then, the mapping $\MPs$ defined by
    	\begin{equation}\label{eqM:MIPsi}
    	\MPs(z)=\frac{\Gamma(z)}{W_{\phi_{\pls}}\!(z)}W_{\phi_{\mis}}\!(1-z) 
    	\end{equation}	
    	satisfies the recurrence relation  \eqref{eq:fe} at least on $i\R\setminus \lbrb{\Zc_0\!\lbrb{\Psi}\cup\curly{0}}$, where we recall that $\Zc_0(\Psi)=\curly{z\in i\R:\,\Psi(-z)=0}$. Setting  $\aP = \ap\mathbb{I}_{\{\dphp=0\}}\leq 0$, we have that
    	\begin{equation}\label{eq:domainAnalMPs}
    	\MPs \in  \Ac_{(\aP,\,1-\dphn)}.
    	\end{equation}
    	If $\aP=0$, then $\MPs$ extends continuously to $i\R\setminus\{0\}$ provided additionally $\phi_{\pls}'\!(0^+)=\infty$ or $\dphp<0$, and otherwise $\MPs\in\Ac_{\lbbrb{0,1-\dphn}}$.
    	In any case
    	\begin{equation}\label{eq:domainMerMPs}
    	\MPs \in   \Mtt_{\lbrb{\aphp,\,1-\aphn}}.
    	\end{equation}
    	Let $\aphp\leq\dphp<0$. If $\tphp=-\infty$ or $-\tphp\notin\N$ then on $\Cb_{\lbrb{\aphp,1-\dphn}}$, $\MPs$ has simple poles at all points $-n$ such that $-n>\aphp,\,n\in\N$. Otherwise, if $-\tphp\in\N$, on $\Cb_{\lbrb{\aphp,1-\dphn}}$ the function $\MPs$ has simple poles at all points $-n$ such that $n \in \N \setminus \{|\tphp|,|\tphp|+1,\ldots\}$. In both cases the residue at each simple pole $-n$ is of value $\phpspace(0)\frac{\prod_{k=1}^{n} \Psi(k)}{n!}$ where we apply the convention $\prod_{k=1}^{0}=1$. Finally \eqref{eqM:MIPsi} is invariant of the choice of the Wiener-Hopf factors up to a multiplicative constant, that is $\php\mapsto c\php,\phn\mapsto c^{-1}\phn,c>0$.
    \end{theorem}
{\emph{This theorem is proved in Section \ref{sec:proof_thm1}.}}
\begin{remark}\label{rem:d}
	We note, from \eqref{eq:WH1} and \eqref{eq:thetaphi}-\eqref{eq:dphi}, that the quantities $\aphp,\aphn,\tphp,\tphn,\dphp,\dphn$ can be computed from the analytical properties of $\Psi$. For example, if $\Psi\notin \Ac_{\lbrb{u,0}}$ for any $u<0$ then $\aphn=0$. Similarly, if $\limsupi{t}\xi_t=-\liminfi{t}\xi_t=\infty$ a.s.~then clearly $\dphn=\dphp=0$ as $\php(0)=\phn(0)=0$. The latter indicates that the ladder height processes of $\xi$, which are defined in Appendix \ref{subsec:LP} are not killed, see e.g.~\cite[Chapter VI]{Bertoin-96}.
\end{remark}
In view of the fact that, for any $\phi \in \Be$, $\Wp$ has a Stirling type asymptotic representation, see Theorem \ref{thm:Stirling} below, and \eqref{eqM:MIPsi} holds, we proceed with a definition of two classes that will encapsulate different modes of decay of $\abs{\MPs(z)}$ along complex lines.
For any  $\beta\in [0,\infty]$, we write
\begin{equation}\label{eq:polyclass}
\Npb\! =\!\left\{ \Psi \in \overNc:\limi{|b|}\frac{\abs{\MPspace\lbrb{a+ib}}}{|b|^{-\beta+\varepsilon}}\!=\!\limi{|b|}\frac{|b|^{-\beta -\varepsilon}}{\abs{\MPspace\lbrb{a+ib}}}=0,\forall a\in \lbrb{0,1-\dphn},\forall \varepsilon>0\right\},
\end{equation}
where if $\beta=\infty$ we  understand
\begin{equation}\label{eq:Ninf}
\Npi = \left\{ \Psi \in \overNc:\:\limi{|b|}|b|^{\beta}\abs{\MPspace\lbrb{a+ib}}=0,\,\,\forall a\in \lbrb{0,1-\dphn},\forall \beta \geq0 \right\}.
\end{equation}
Moreover, for any $\Theta>0$ we set
\begin{equation}\label{eq:polyclass1}
\Nth = \left\{ \Psi \in \overNc:\:\limsupi{|b|} \frac{\ln \abs{\MPspace\lbrb{a+ib}}}{|b|}\leq -\Theta,\,\,\forall a\in \lbrb{0,1-\dphn} \right\}.
\end{equation}
Finally, we shall also need the set of regularly varying functions at $0$. For this purpose we introduce some more notation. We use in the standard manner $ f\stackrel{a}{\sim} g$  (resp.~$f\stackrel{a}{=}\bospace{g}$) for any $a\in [-\infty,\infty]$, to denote that  $\lim\limits_{x\to a}\frac{f(x)}{g(x)}=1$ (resp.~$\varlimsup\limits_{x \to a}\labsrabs{\frac{f(x)}{g(x)}}=C<\infty$). The notation $\sospace{.}$ specifies that in the previous relations the constants are zero. We shall drop the superscripts if it is explicitly stated or clear that $x\to a$. Let us introduce the regularly varying functions. We say that, for some $\alpha\in\lbbrb{0,1}$, \[ f\in RV_\alpha\iff f(y)\simo y^\alpha \ell(y),\] where $\ell\in RV_0$ is a slowly varying function, that is, $\ell(cy)\simo \ell(y)$ for any $c>0$. Furthermore, $\ell\in RV_0$ is said to be quasi-monotone, if $\ell$ is of bounded variation in a neighbourhood of zero and  for any $ \gamma>0$
\begin{equation}\label{eq:quasi-monotone}
\int_{0}^xy^\gamma |d \ell(y)|\stackrel{0}{=}\bospace{x^{\gamma}\ell(x)}.
\end{equation}
Define, after  recalling that $\bar{\mu}(y)=\int_{y}^{\infty}\mu(dr)$, see \eqref{eq:phi0},
\begin{equation}\label{eq:classesPhi1}
\Bc_{\alpha}=\lbcurlyrbcurly{\phi\in\Bc:\,\dr=0,  \bar{\mu}\in RV_\alpha \textrm{ and }y\mapsto \frac{\bar{\mu}(y)}{y^\alpha} \text{ is quasi-monotone}}.		
\end{equation}  
Next, we define the class of Bernstein functions with a positive drift that is
\begin{equation}\label{eq:BP}
\BP=\curly{\phi\in\Bc:\,\dr>0}.
\end{equation}
Finally, we denote by $\mu_{\pls},\mu_{\mis}$  the measures associated to $\phpspace,\phn\in\Bc$ stemming from \eqref{eq:WH1} and \[\Pi_{\pls}(dy)=\Pi(dy)\ind{y>0},\,\,\Pi_{\mis}(dy)=\Pi(-dy)\ind{y>0}\] for the measure in \eqref{eq:lk0}. Finally, $f(x^\pm)$ will stand throughout for the right, respectively left, limit at $x$. We are now ready  to provide an exhaustive claim concerning the decay of $\abs{\MPspace}$ along complex lines.
\begin{theorem}\label{thm:asympMPsi}
	Let $\Psi\in\overNc$.
	\begin{enumerate}
		\item  \label{eq:subexp} \label{it:decayA}    $\Psi \in \NpNP$ with
		\begin{align}
		\NPs\!=\!\begin{cases} \frac{\upsilon_{\mis}\!\lbrb{0^+}}{\phn(0)+\mubarnspace{0}}+\frac{\php(0)+\mubarpspace{0}}{\dep}<\infty & \textrm{if } \dep>0, \dem\!=0 \textrm{ and } \PP(0)\!=\!\int_{-\infty}^{\infty}\Pi(dy)<\infty, \\
		\infty &\textrm{otherwise,} \label{eq:NPs}
		\end{cases}
		\end{align}
		where  we have used implicitly  the fact that if $\dep>0$ and $\PP(0)<\infty$ hold, then $\munspace{dy}=\upsilon_{\mis}(y)dy\ind{y>0}$ with $\upsilon_{\mis}$ right-continuous on $\lbbrb{0,\infty}$ and $\upsilon_{\mis}\!\lbrb{0^+}\in\lbbrb{0,\infty}$. Moreover,  we have that \[\NPs=0\iff \PP(0)=0, \phn(z)=\phn(0)>0 \text{ and } \php(z)=\dep z.\]
		\item\label{it:classes} Moreover, if $\phn\in\BP$, that is $\dem >0$, or $\php\equiv  c\phn$, $c>0,$ then $\Psi \in \Nthdspace$. If $\phn\in\Bc_{\alpha}$ or $\php\in\Bc_{1-\alpha}$, with $\alpha \in (0,1)$, see \eqref{eq:classesPhi1} for the definition of these sets involving regularly varying functions, then $\Psi\in\Nthaspace$. Finally, if \[\Theta_{\pms}=\frac{\pi}{2}+\underline{\Theta}_{\phi_{\mis}}-\overline{\Theta}_{\phi_{\pls}}>0,\] where  \[\underline{\Theta}_{\phi}=\varliminf\limits_{b\to\infty}\frac{\int_{0}^{|b|}\arg \phi\!\lbrb{1+iu}du}{|b|} \quad\text{ and }\quad \overline{\Theta}_{\phi}=\varlimsup\limits_{b\to\infty}\frac{\int_{0}^{|b|}\arg \phi\!\lbrb{1+iu}du}{|b|},\] then $\Psi\in\Ntpspace{\Theta_{\pms}}$.
	\end{enumerate}		
\end{theorem}           
               {\emph{This theorem is proved in Section \ref{subsec:decay}.}}
               
  \subsection{Exponential functionals of L\'evy processes}
  We now introduce the following subclasses of $\overNc$
  \begin{equation}\label{eq:Nc}
  \qquad\quad\,\Ne=\lbcurlyrbcurly{\Psi \in \overNc: \: \Psi(z)=-\php(-z)\phn(z), z \in i\R, \textrm{ with } \phn(0)>0} 
  \end{equation}
  and
  \begin{equation}\label{eq:Ncdag}
  \!\!\!\!\!\!\!\!\!\Nc_\dagger=\curly{\Psi\in\overNc:\:\Psi(0)=-\php(0)\phn(0)<0}\subseteq \Nc.
  \end{equation}    
  We note that, with $-\Psi(0)=q\geq 0$,
  \begin{equation}\label{eq:Ipsi1}
  \IPsi =\int_{0}^{\textbf{e}_q}e^{-\xi_s}ds<\infty\text{ a.s. }\iff \,\,\, \Psi \in \Ne \iff \,\,\,\phn(0)>0,
  \end{equation}
  which is evident when $q>0$ and is due to the strong law of large numbers  when $q=0$. The latter includes the case \[\Ebb{\xi_1\ind{\xi_1>0}}=\Ebb{-\xi_1\ind{\xi_1<0}}=\infty\] but yet a.s.~$\limi{t}\xi_t=\infty$ since in this case a.s. $\limi{t}\xi_t-ct=\infty$ for any $c>0$, see \cite[Chapter IV]{Doney-07-book}. The latter is known as the Erickson's test, for more information and discussion, see also \cite[Section 6.7]{Doney-07-book}.
  Let us write for any $x\geq0$,
  \[ \PIoPs{x}=\P(\IPsi \leq x).\]
  From \cite{Bertoin-Lindner-07}, we know that the law of $\IPsi$ is absolute continuous with density denoted by $\pPs$, i.e.~$\PIp'(x)=\pPs(x)$ a.e..
  Introduce the Mellin transform of the positive random variable $\IPsi$ defined formally, for some $z\in \C$, as follows
  \begin{eqnarray}\label{eq:MTIPsi}
  \MPsispace(z) = \Ebb{\IPsi^{z-1}}=\int_0^{\infty}x^{z-1}\pPs(x)dx.
  \end{eqnarray}
  We also use the ceiling function $\lceil.\rceil:\lbbrb{0,\infty}\mapsto \N$, that is $\lceil x\rceil=\min\curly{n\in\N:\,n\geq x}$.
  \subsubsection{Regularity, analyticity and representations  of the density and its successive derivatives}
  We start our results on the exponential functional of L\'evy processes by providing a result that can be regarded as a  corollary to Theorem \ref{thm:asympMPsi}.
  \begin{theorem}\label{cor:smoothness}
  	Let $\Psi\in\Nc$.
  	\begin{enumerate}
  		\item\label{it:MTfact} We have
  		\begin{equation}\label{eq:MIPsi}
  		\MPsispace(z)= \phn(0)\MPs(z) =\phn(0)\frac{\Gamma(z)}{\Wppspace(z)}\Wpnspace(1-z)\in\Ac_{\lbrb{\ap\mathbb{I}_{\{\dphp=0\}},\,1-\dphn}}\cap\,\Mtt_{\lbrb{\aphp,\,1-\aphn}}
  		\end{equation}
  		and $\MPsi$ satisfies the recurrence relation \eqref{eq:fe} at least on $i\R\setminus\lbrb{\Zc_0\!\lbrb{\Psi}\cup\curly{0}}$.
  		\item\label{it:supp} If $\phn$ is a constant, that is $\phn(z)=\phn(\infty)\in\lbrb{0,\infty}$,  then  $\supp\,\IPsi=\lbbrbb{0,\frac{1}{\phn(\infty)\dep }}$, unless $\php(z)=\dep z,\,\dep\in\lbrb{0,\infty}$, in which case $\supp\,\IPsi=\curly{\frac{1}{\phn(\infty)\dep }}$. If $\phn$ is not identical to a constant and $\php(z)=\dep z$ then $\supp\,\IPsi=\lbbrbb{\frac{1}{\phn(\infty)\dep},\infty}$, where we use the convention that $\frac{1}{\infty}=0$.  In all other cases $\supp\,\IPsi=\lbbrbb{0,\infty}$.
  		\item \label{it:smooth} $\PIp \in\mathrm{C}^{\lceil\NPs\rceil-1}_0\lbrb{\Rp}$ and if $\NPs>1$ (resp.~$\NPs>\frac12$)
  		for any $n=0,\ldots, \lceil\NPs\rceil-2$
  		and any $a\in\lbrb{\ap\mathbb{I}_{\{\dphp=0\}},1-\dphn}$,
  		\begin{equation}\label{eq:MITd}
  		\pPs^{(n)}\!(x)=(-1)^n\frac{\phi_{\mis}(0)}{2\pi i}\int^{a+i\infty}_{a-i\infty}x^{-z-n}\Gamma\!\lbrb{z+n}\frac{\Wpnspace(1-z)}{\Wppspace(z)}dz,
  		\end{equation}
  		where the integral is absolutely convergent for any $x>0$ (resp.~is defined in the $L^2$-sense, as in the book of Titchmarsh \cite{Titchmarsh-58}).
  		\item\label{it:smallExp}	Let $\Psi\in\Nc_{\dagger}$, i.e.~$\Psi(0)<0$, and let  \[\Ntt_{+}=\abs{\tphp}\mathbb{I}_{\{\abs{\tphp} \in \N\}}+\left(\lceil |\aphp|+1\rceil\right)\mathbb{I}_{\{\abs{\tphp}\notin \N\}}.\] Then, we have, for any  $0\leq n<\NPs$, any positive integer $M$ such that $ \Ntt_+>M$, any $a\in\lbrb{\lbrb{-M-1}\vee\lbrb{\aphp-1},-M}$ and $x>0$,
  		\begin{equation}\label{eq:smallExpansion}
  		\begin{split}
  		\PIoPsnspace{x}&=-\Psi(0)\sum_{k=1\vee n}^{M}\frac{\prod_{j=1}^{k-1}\Psi\!\lbrb{j}}{\lbrb{k-n}!}x^{k-n}\\
  		&+\lbrb{-1}^{n+1}\frac{\phn(0)}{2\pi i}\int_{a-i\infty}^{a+i\infty}x^{-z-n}\Gamma\!\lbrb{z+n}\frac{\Wpnspace(-z)}{\Wppspace(1+z)}dz,
  		\end{split}
  		\end{equation}
  		where by convention $\prod_{1}^{0}=1$ and the sum vanishes if $1\vee n>M$.
  		\item\label{it:holomorphic}   If $\Psi\in\Nth\,\cap\,\Nc, \: \Theta \in\lbrbb{0,\pi}$, then $\pPs$ is holomorphic in the complex sector\\ $\curly{z\in\Cb:\,\abs{\arg z}<\Theta}$.
  		
  	\end{enumerate}
  \end{theorem}
  {\emph{This theorem is proved in Section \ref{sec:EFLP}.}}
  \begin{remark}\label{rem:asympMPsi}
  	Item \eqref{it:smooth}  confirms the conjecture that $\pPs\in\ccoiRp$ when either $\sigma^2>0$ and/or $\int_{-\infty}^{\infty}\Pi(dy)=\infty$ hold in \eqref{eq:lk0} or equivalently $\xi$ has infinite activity of the jumps. Indeed, from Theorem \ref{thm:asympMPsi}\eqref{it:decayA}, under each of these conditions $\Psi\in\Ni$, where
  	\begin{equation}\label{eq:Ni}
  	\Ni=\Npi\cap\Nc.
  	\end{equation}
  	The surprising fact is that $\Psi\in\Ni$ and hence $\pPs\in\ccoiRp$ even when $\xi$ is a pure compound Poisson process with a non-positive drift, whereas if the drift is positive then $\Psi\in\NepN,\,\NPs<\infty$. Finally, if $\phn\equiv 1$, $\php(z)=q+z,\,q>0,$ then $\xi_t=t$ is killed at rate $q$ and $\PIp(x)=1-\lbrb{1-x}^q,\,x\in\lbrb{0,1}$, which since $\NPs=q$, see \eqref{eq:NPs}, shows that the claim $\PIp \in\mathrm{C}^{\lceil\NPs\rceil-1}_0\lbrb{\Rp}$ is sharp unless $q\in\N$. In the latter case evidently $\pPs\in\Ac_{\lbrb{-\infty,\infty}}$.
  \end{remark}	
  \begin{remark}\label{rem:asympMPsi2}
  	Let $\php(z)=z\in\BP$ and $\phn(0)>0$ so that $\Psi(z)=z\phn(z)\in\Nc$. Then
  	$\IPsi=\IntOI e^{-\xi_t}dt$ is a self-decomposable random variable, see \cite[Chapter 5]{Patie-Savov-16}. The rate of decay of the Fourier transform of $\IPsi$ is $\lambda$ in the notation of \cite{Sato-Yam-78}. One can check that $\lambda=\NPs$.
  \end{remark}
  \begin{remark}\label{rem:genRemark}
  	If $\dphp=0$, \eqref{eq:MIPsi}  evaluates the moments $\Ebb{\IPsi^{-a}}$ for $a\in\lbrb{0,1-\aphp}$. This extends \cite[Proposition 2]{Bertoin-Yor-02}, which gives the entire negative moments of $\IPsi$ when $\Ebb{\xi_1}=\Psi'\!\lbrb{0^+}\in\lbrb{0,\infty}$ and $\aphp=-\infty$.
  \end{remark}
  \begin{remark}\label{rem:Maulik}
  	The proof of \eqref{eq:MIPsi} shows that, provided $\phn(0)>0$,  $\MPsi$ is the unique solution  to \eqref{eq:fe} in the space of Mellin transforms of random variables. When $\dphn<0$ this is  shown for $\Psi\in\Nc\setminus\Nc_\dagger$ in \cite{Maulik-Zwart-06} and for $\Psi\in\Nc$ in \cite{Arista-Rivero-16}. In the latter cases the recurrent equation \eqref{eq:fe} holds at least on the strip  $\Cb_{\lbrb{0,-\dphn}}$.
  \end{remark}
Item \eqref{it:smallExp} of Theorem \ref{cor:smoothness} can be refined as follows.
\begin{corollary}\label{cor:smallExp}
	Let $\Psi\in\Nc_\dagger$, $\abs{\aphp}=\infty$ and $-\tphp\notin \N$. Then $-\Psi(0)>0$ and
	\[ \PIp(x)  \stackrel{0}{\approx}  -\Psi(0)\sum_{k=1}^{\infty}\frac{\prod_{j=1}^{k-1}\Psi(j)}{k!}x^{k}\]
	is the asymptotic expansion of $\PIp$ at zero in the sense  that, for any $N \in \N$,
	\[\PIp(x)+\Psi(0)\sum_{k=1}^{N}\frac{\prod_{j=1}^{k-1}\Psi(j)}{k!}x^{k}\stackrel{0}{=}\sospace{x^{N}}.\]
	The asymptotic expansion cannot be a convergent series for any $x>0$ unless $\php(z)=\php(\infty)\in\lbrb{0,\infty}$ or $\php(z)=\php(0)+\dep,\dep>0,$ and $\phn(\infty)<\infty$. Then in the first case it converges for $x<\frac{1}{\php(\infty)\dem}$ and diverges for $x>\frac{1}{\php(\infty)\dem}$, and in the second it converges for $x<\frac{1}{\phn(\infty)\dep }$ and diverges for $x>\frac{1}{\phn(\infty)\dep }$.
\end{corollary}
\begin{remark}\label{rem:smallExp}
	When $\php(z)=\php(\infty)\in\lbrb{0,\infty}$ it is known from \cite[Corollary 1.3]{Patie-Savov-11} and implicitly from \cite{Pardo-Rivero-Scheik-13} that the asymptotic expansion is convergent if and only if $x<\frac{1}{\php(\infty)\dem}$. Therefore, the series is convergent on $\Rp$ if and only if $\dem=0$.
\end{remark}
\subsubsection{Large asymptotic behaviour of the distribution of $\IPsi$  and its successive derivatives.}\label{subsec:asymp}
We introduce the well-known \textit{non-lattice subclass} of $\Bc$ and $\Nc$. Recall that
\begin{equation}\label{eq:zerosPsi}
\Zc_0\!\lbrb{\Psi}=\curly{z\in i\R:\,\Psi(-z)=0}=\curly{z\in i\R:\,\Psi(z)=0},
\end{equation}
where the latter identity follows from $\overline{\Psi(z)}=\Psi(\overline{z})$, see \eqref{eq:lk0}.
For any $\Psi\in\Nc,\,\phi\in\Bc$ set
\[\Psi^\sharp(z)=\Psi(z)-\Psi(0)\, \text{ and }\, \phi^\sharp(z)=\phi(z)-\phi(0).\]
Then the \textit{non-lattice subclass} is defined as follows
\begin{equation}\label{eq:nonlatticeclass}
\begin{split}
\Psi\in \Nc_{\Zc}&\iff \Zc_{0}\!\lbrb{\Psi^\sharp}=\lbcurlyrbcurly{0}.
\end{split}
\end{equation}
We note that we use the terminology \textit{non-lattice class} since the underlying \LLPs does not live on a sublattice of $\Rb$, e.g. if $\Psi\in\Nc_\Zc$ then the support of the distribution of the underlying \LLP $\xi$ is either $\Rb$ or $\Rp$. If $\tphn\in\lbrb{-\infty,0}$, see \eqref{eq:thetaphi}, we introduce the \textit{weak non-lattice} class as follows
\begin{equation}\label{eq:weaknonlattice}
\begin{split}
\Psi\in\Nc_\Wc&\iff \tphn\in\lbrb{-\infty,0}\text{ and }\exists \: k\,\in\N \textrm{ such that~} \,\liminfi{|b|}|b|^k\abs{\Psi\!\lbrb{\tphn+ib}}>0\\
&\iff \tphn\in\lbrb{-\infty,0}\text{ and }\exists \: k\,\in\N \textrm{ such that~}\,\liminfi{|b|}|b|^k\abs{\phnspace\lbrb{\tphn+ib}}>0.
\end{split}	
\end{equation}
We note that the \textit{weak non-lattice} class seems not to have been introduced in the literature yet.
Clearly, $\Nc\setminus\Nc_\Zc\subseteq \Nc\setminus\Nc_\Wc$ since $\Psi^\sharp\in\Nc\setminus\Nc_\Zc$ vanishes on $\curly{2\pi k i/h:\,k\in\Zb}$, where $h>0$ is the span of the lattice that supports the distribution of the conservative \LL process $\xi^{\sharp}$. Indeed the latter ensures that for any $l\in\Zb$
\begin{equation*}
\begin{split}
e^{\Psi^\sharp\lbrb{\tphn+\frac{2\pi l i}{h}}}=\Ebb{e^{\lbrb{\tphn+\frac{2\pi l i}{h}}\xi_1^{\sharp}}}=\Ebb{e^{\tphn\xi_1^{\sharp}}}=e^{\Psi^\sharp(\tphn)}=e^{-\Psi(0)},
\end{split}
\end{equation*}
where the last identity follows since from \eqref{eq:WH1} and \eqref{eq:thetaphi} we get that \[\Psi(\tphn)=-\phn(\tphn)\php(-\tphn)=0.\]
Therefore, $\Psi^\sharp\!\lbrb{\tphn+\frac{2\pi l i}{h}}=-\Psi(0)$ and hence from the definition of $\Psi^\sharp$ we arrive at the identity $\Psi\!\lbrb{\tphn+\frac{2\pi l i}{h}}=0$. Since the latter is valid for any $l\in\N$, \eqref{eq:weaknonlattice} cannot be satisfied for any $k\in\N$ and hence $\Nc\setminus\Nc_\Zc\subseteq \Nc\setminus\Nc_\Wc$.

 We phrase our first main result which  encompasses virtually all exponential functionals. We write throughout, for $x\geq 0$,
\[\overline{F}_{\Psi}(x)=1-\PIoPs{x}=\Pbb{\IPsi>x}\]
for the tail of $\IPsi$.
\begin{theorem}\label{thm:largeAsymp}
	Let $\Psi\in \Nc$.
	\begin{enumerate}
		\item \label{it:dPsi}
		In all cases
		\begin{align}\label{eq:logLargedPs}
		\limi{x}\frac{\ln\overline{F}_{\Psi}(x)}{\ln x}=\dphn.
		\end{align}	
		\item \label{it:Cramer} \label{it:Cramer1}\label{it:Cramer2}
		If in addition $\Psi\in\Nc_\Zc$, $\dphn=\tphn\in \lbrb{-\infty,0}$ and $\abs{\Psi'\!\lbrb{\tphn^+}}<\infty$  then
		\begin{equation}\label{eq:tailCramer}
		\overline{F}_{\Psi}(x) \simi \frac{\phn(0)\Gamma\!\lbrb{-\tphn}\Wpnspace\lbrb{1+\tphn}}{\phn'\!\lbrb{\tphn^+}\Wppspace\lbrb{1-\tphn}} x^{\tphn}.
		\end{equation}
		Moreover, if $\Psi\in \PcNinf\cap\Nc_\Wc$ (resp.~$\Psi\in \Nn,\,\NPs<\infty$) and $\abs{\Psi''\!\lbrb{\tphn^+}}<\infty$ then for every $n\in\Nb$ (resp.~$n\leq \lceil\NPs\rceil-2$)
		\begin{equation}\label{eq:densityCramer}
		\pPs^{(n)}\!(x) \simi \lbrb{-1}^{n}\frac{\phn(0)\Gamma\!\lbrb{n+1-\tphn}\Wpnspace\lbrb{1+\tphn}}{\phn'\!\lbrb{\tphn^+}\Wppspace\lbrb{1-\tphn}}x^{-n-1+\tphn}.
		\end{equation}
	\end{enumerate}
\end{theorem}
{\emph{This theorem is proved in Section \ref{subsec:Cramer}.}}
\begin{remark}\label{rem:Cramer}
	The requirement $\Psi\in \Nc_{\Zc}$, $\dphn=\tphn<0$ and $\abs{\Psi'\!\lbrb{\tphn^+}}<\infty$ is known as the Cram\'er's condition for the underlying \LL process $\xi$. In this case \[\limi{x}x^{-\tphn}\PIPs{x}=C\in\lbrb{0,\infty}\] or \eqref{eq:tailCramer} holds is known from \cite[Lemma 4]{Rivero-05}. Here, $C$ is explicit. We emphasize that the refinement of \eqref{eq:tailCramer} to derivatives of the tail, see \eqref{eq:densityCramer}, is due to the decay of $\abs{\MPsi\!(z)}$ along  $\Cb_{1-\tphn}$. The latter is an extension of \eqref{eq:NPs} in Theorem \ref{thm:asympMPsi} which is valid only when  $\Psi\in \Nc_\Wc$ and $\abs{\Psi''\lbrb{\tphn^+}}<\infty$. The latter is always satisfied if $\tphn>\aphn$. The requirements  $\Psi\in \Nc_\Wc$ and $\abs{\Psi''\!\lbrb{\tphn^+}}<\infty$ are not needed for the non-increasing function $\PIPs{x}$ since the Wiener-Ikehara theorem is applicable in this instance.
\end{remark}
\begin{remark}\label{rem:generalAsymp}
	Relation \eqref{eq:logLargedPs} strengthens \cite[Lemma 2]{Arista-Rivero-16} by quantifying the rate of the power decay of $\PIPs{x}$ as $x\to\infty$. Since $\IPsi$ is also a perpetuity with thin tails, see \cite{Goldie-Grubel-96}, we provide precise estimates for the tail behaviour of this class of perpetuities.
\end{remark}
\begin{remark}\label{rem:Alexey}
	Since the supremum of a stable \LL process is related to an exponential functional for which the asymptotic \eqref{eq:densityCramer} is valid, then our result recovers the main statements about the asymptotic of the density of the supremum of a stable \LL process and its derivatives that appear in \cite{Doney-Savov-10,Kuznetsov2011}.
\end{remark}
{Under specific conditions,  see \cite{Patie2012,Patie-Savov-11} for the class of (possibly killed) spectrally negative L\'evy processes and \cite{Kuznetsov2011, Kuznetsov13} for some special  instances, the density $\pPs$ can be expanded into a converging series of negative powers of $x$. This is achieved by a subtle pushing to infinity of the contour of the Mellin inversion when the meromorphic extension $\MPsi\in\Mtt_{\intervalOI}$ is available. Our Theorem \ref{thm:asympMPsi}, which ensures a priori knowledge for the decay of $\abs{\MPsi\!(z)}$, allows for various asymptotic expansions or evaluation of the speed of convergence of $\pPs$ at infinity as long as $\Psi\in \Nc_\Wc$ and $\MPsi$ extends to the right of $\Cb_{1-\tphn}$. However, for sake of generality, we leave this line of research aside.
	
	\subsubsection{Small asymptotic behaviour of the distribution of $\IPsi$ and its successive derivatives.}
	We proceed with  the small asymptotic behaviour.
	\begin{theorem}\label{thm:smallTime}
		Let $\Psi\in\Nc$.  Then
		\begin{equation}\label{eq:smallTime}
		\limo{x}\frac{\PIoPs{x}}x= -\Psi(0)
		\end{equation}
		with $\pPs(0^+)=-\Psi(0)$.
		Hence, if  $\pPs$ is continuous at zero or $\Psi\in\NepN$ with $\NPs>1$, we have that
		\begin{equation}\label{eq:smallTime1}
		\limo{x}\pPs(x)=\pPs(0)=-\Psi(0).
		\end{equation}
	\end{theorem}
	{\emph{This theorem is proved in Section \ref{subsec:smallTime}.}}
	\begin{remark}\label{rem:smallTime}
		When $\Psi(0)<0$, the limit \eqref{eq:smallTime}  has been proved in \cite[Theorem 7(i)]{Arista-Rivero-16} using the non-trivial \cite[Theorem 2.5]{Pardo-Rivero-Scheik-13}.
		We note that \eqref{eq:smallTime1} is a new result for general exponential functionals $\IPsi$. For exponential functionals based on an increasing \LLP \eqref{eq:smallTime1} has already appeared in \cite[Theorem 2.5]{Pardo-Rivero-Scheik-13} without any assumptions. This can be deducted from \eqref{eq:smallTime} if $\pPs$ is continuous at zero. When $\NPs\leq 1$ then from \eqref{eq:NPs} in Theorem \ref{thm:asympMPsi} it is valid that $\mubarpspace{0}\leq\dep<\infty$.  Then  from its functional equation in \cite[Theorem 2.4]{Pardo-Rivero-Scheik-13} $\pPs$ is continuous at zero  and  $\pPs(0)=-\Psi(0)$ from \eqref{eq:smallTime}.
	\end{remark}
\begin{remark}\label{rem:SN}	
	Denote by $\Nc_-=\curly{\Psi\in\Nc:\,\php(z)=\php(0)+\dep z}$ the class of \LLK exponents of  the so-called spectrally negative \LL processes, that is \LLPs that do not jump upwards. Assume that $\php(0)=0$. Then $\Psi\lbrb{z}=z\phn(z)$ and if $\phnspace\lbrb{\infty}=\infty$ the small-time asymptotic of $\pPs$ and its derivatives, according to \cite{Patie-Savov-16} reads off, with $\varphi_{\mis}$ such that $\phn(\varphi_{\mis}(u))=u$, as follows
	\begin{equation}\label{eq:Tauberian}
	\pPs^{(n)}\!\!\lbrb{x}\simo \frac{C_{\phn}\phn(0)}{\sqrt{2\pi}}\frac{\varphi^n_{\mis}\!\lbrb{\frac{1}{x}}}{x^{n+1}}\sqrt{\varphi'_{\mis}\!\lbrb{\frac{1}{x}}}e^{-\int_{\phn(0)}^{\frac{1}{x}} \varphi_{\mis}(r)\frac{dr}{r}}.
	\end{equation}
	If $\phnspace\lbrb{\infty}<\infty$ then according to Theorem \ref{cor:smoothness}\eqref{it:supp}  we have that $\supp \,\pPs\in\lbbrb{\frac{1}{\phn\lbrb{\infty}},\infty}$. Comparing the small asymptotic behaviour of $f_{\Psi}$ in \eqref{eq:Tauberian}  with the one in \eqref{eq:smallTime1}  reveals that  a simple killing of the underlying \LLP leads to a  drastic change in the small time asymptotic.
\end{remark}
\subsubsection{Finiteness of negative and positive moments, and asymptotic behaviour of the exponential functionals on a finite time horizon.}
For any $\Psi\in\overNc$ and $t\geq0$ let
\[ \IPsi(t)=\int_{0}^{t}e^{-\xi_s}ds.\]
We have the following claim which furnishes necessary and sufficient conditions for the finiteness of the negative and positive moments of $\IPsi(t)$.
\begin{theorem}\label{thm:momentsIt} 	Let $\Psi\in\overNc\setminus\Nc_\dagger$ that is $\Psi(0)=0$. 
	\begin{enumerate}
		\item\label{it:negative} Then, for any $t>0$, we have the following relations for the negative moments of $\IPsi(t)$
		\begin{eqnarray}
		\Ebb{\IPsi^{-a}(t)}<\infty &\impliedby& a\in\lbrb{0,1-\aphp},\label{eq:momentsIt}\\
		\Ebb{\IPsi^{-1+\aphp}\!(t)}<\infty&\iff&\abs{\Psi\!\lbrb{-\aphp}}<\infty\iff\abs{\php(\aphp)}<\infty,\label{eq:momentsIt1}\\
		\Ebb{\IPsi^{-1}(t)}<\infty &\iff& \abs{\Psi'(0^+)}<\infty,\label{eq:momentsIt2}\\
		\Ebb{\IPsi^{-a}(t)}=\infty &\impliedby& a>1-\aphp.\label{eq:momentsIt3}
		\end{eqnarray}
		Finally, we have that for any  $a\in\lbrb{0,1-\aphp}$
		\begin{equation}\label{eq:limitIt}
		\limo{t}t^{a}\Ebb{\IPsi^{-a}(t)}=1.
		\end{equation}
		\item \label{it:positive} Then, for any $t>0$, we have the following relations for the positive moments of $\IPsi(t)$
		\begin{eqnarray}
		\Ebb{\IPsi^{a}(t)}<\infty &\impliedby& a\in\lbrb{0,-\aphn},\label{eq:pmomentsIt}\\
		\Ebb{\IPsi^{-\aphn}\!(t)}<\infty&\iff&\abs{\Psi\!\lbrb{\aphn}}<\infty\iff\abs{\phn(\aphn)}<\infty,\label{eq:pmomentsIt1}\\
		\Ebb{\IPsi^{a}(t)}=\infty &\impliedby& a>-\aphn.\label{eq:pmomentsIt3}
		\end{eqnarray}
	\end{enumerate}

\end{theorem}
{\emph{ This theorem is proved in Section \ref{subsec:momentsIt}}}
\begin{remark}\label{rem:momentsIt}
	Results on the finiteness of $\Ebb{\IPsi^{-a}(t)}$ appear in \cite{Palau-Pardo-Smadi-16} where the authors limit their attention on the range $a\in\lbrb{0,-\aphp}$ which is substantially easier to prove via the relation \eqref{eq:finiteMoment} in the proof below.
\end{remark}
Next, we consider the case of conservative \LLPs such that $\liminfi{t}\xi_t=-\infty$ a.s.~ or equivalently $\phn(0)=0$.
In the setting of the next claim we use superscript $\dagger q$ for
\[\Psi^{\dagger q}(z)=\Psi(z)-q=-\php^{\dagger q}(-z)\phn^{\dagger q}(z),\,q\geq0,\]
and all related quantities. From Lemma \ref{lem:WH} and \eqref{eq:Fristed} we have that $\phn^{\dagger q}(z)=\kappa_{\mis}\lbrb{q,z}$, where $\kappa_{\mis}$ is associated to the descending ladder height process.
Recall that $RV_\alpha$ stands for the class of regularly varying functions of index $\alpha\in\R$ at zero. Then the following result elucidates the behaviour of the measures  $\Pbb{\IPsi(t)\in dx}$ as $t\to\infty$ in quite a general framework.
 \begin{theorem}\label{thm:Nc}
	Let $\Psi\in\overNc$ and $\phn(0)=0$.
	\begin{enumerate}
		\item\label{it:NC1} For any $x>0$, $1-\aphp>n\in \N$ and $a\in\lbrb{n,\max\curly{n+1, 1-\aphp}}, -a\notin\N,$ we get
		\begin{equation}\label{eq:It}
		\begin{split}
		\frac{e-1}{e}\limsupi{t}\frac{\Pbb{\IPsi(t)\leq x}}{\kappa_{\mis}\lbrb{\frac{1}{t},0}}
		&\leq \php^{\dagger \frac{1}{t}}(0)\sum_{k=1}^{n \wedge \Ntt_+}\frac{\abs{\prod_{j=1}^{k-1}\Psi\lbrb{j}}}{k!}x^{k}\\
		&\,\,+\frac{x^{a}}{2\pi }\int_{-\infty}^{\infty}\frac{\abs{\MPs(-a+1+ib)}}{\sqrt{a^2+b^2}}db,
		\end{split}	
		\end{equation}
		where we recall that $\kappa_{\mis}(q,0)=\phn^{\dagger q}(0)$, \[\Ntt_{+}=\abs{\tphp}\mathbb{I}_{\{\abs{\tphp} \in \N\}}+\left(\lceil |\aphp|+1\rceil\right)\mathbb{I}_{\{\abs{\tphp}\notin \N\}}\] and we use the convention $\sum_{1}^{0}=0$.
		\item\label{it:NC2} Let now $-\liminfi{t}\xi_t=\limsupi{t}\xi_t=\infty$ a.s.~or equivalently $\php(0)=\phn(0)=0$. Assume that \[\limi{t}\Pbb{\xi_t<0}=\rho\in\lbbrb{0,1},\] that is the Spitzer's condition holds. Then $\kappa_{\mis}(q,0)=\phn^{\dagger q}(0)\in RV_\rho$. Also,
		for any $a\in\lbrb{0,1-\aphp}$ and any $f\in\Ctt_b(\Rp)$
		\begin{equation}\label{eq:weak}
		\limi{t}\frac{\Ebb{\IPsi^{-a}(t)f\!\lbrb{\IPsi(t)}}}{\kappa_{\mis}\lbrb{\frac{1}{t},0}}=\IntOI f(x)\vartheta_a(dx),
		\end{equation}
		where $\vartheta_a$ is a finite positive measure on $\intervalOI$ such that,  for any $c\in\lbrb{a-1+\aphp,0}$ and any $x>0$,
		\begin{equation}\label{eq:V_a}
		\vartheta_a((0,x))=-\frac{x^{-c}}{2\pi \Gamma\lbrb{1-\rho}}\IntII x^{-ib}\frac{\MPs\!\lbrb{c+1-a+ib}}{c+ib}db.
		\end{equation}
		Finally
		\begin{equation}\label{eq:RVIPsi}
		\vartheta_a\lbrb{\Rp}= 
		\frac{1}{\Gamma\!\lbrb{1-\rho}}\frac{\Gamma(1-a)}{\Wppspace(1-a)}\Wpnspace\lbrb{a}.
		\end{equation}
	\end{enumerate}
\end{theorem}
{\emph{This theorem is proved in Section \ref{subsec:Li}.}}
 \begin{remark}\label{rem:Nc}
	Note that if $\Ebb{\xi_1}=0,\, \Ebb{\xi^2_1}<\infty$ then $\kappa_{\mis}(q,0)\simo Cq^{\frac12}$  with $C$ extracted from the Fristedt's formula, \cite[Chapter VI]{Bertoin-96}. Therefore, if $f(x)=x^{-a}f_1(x)$ with $a\in\lbrb{0,1-\aphp}$ and $f_1\in\Ctt_b\lbrb{\Rp}$ then \[\limi{t} C\sqrt{t}\Ebb{f(I_t)}=\IntOI f_1(x)\vartheta_a(dx).\] This result is at the core of the studies \cite{Li-XU-16} and \cite{Palau-Pardo-Smadi-16} on the large temporal asymptotic behaviour of extinction and explosion probabilities of continuous state branching processes in \LL random environment. However, in \cite{Li-XU-16,Palau-Pardo-Smadi-16}  some  very stringent restrictions are imposed on $f$.
	Also $\xi$ is required to have some two-sided exponential moments. These requirements are due to reduction to more general results on specific random sums/random walks, which can be found in \cite{Af-Gei-Ker-Vat-05,Kozlov-76}. Thus, Theorem \ref{thm:Nc} can significantly extend the aforementioned results of \cite{Li-XU-16,Palau-Pardo-Smadi-16} and furnish new ones when $\Ebb{\xi^2_1}=\infty$ and $\limi{t}\Pbb{\xi_t<0}=\rho\in\lbbrb{0,1}$.
\end{remark}
\subsection{Intertwining relations of self-similar semigroups and factorization of laws}

By Mellin identification, we obtain as a straightforward consequence of the representation \eqref{eq:MIPsi},  the following probabilistic factorizations of the distribution of the exponential functional.
\begin{theorem} \label{thm:factorization}
	For any $\Psi \in \Ne,$ the following multiplicative Wiener-Hopf factorizations of the  exponential functional hold
	\begin{eqnarray}\label{eq:GammaType2}\label{eq:InfW}
	\IPsi&\stackrel{d}{=}& \Ir_{\php}\times X_{\phn} \stackrel{d}{=} \bigotimes_{k=0}^{\infty}C_\Psi(k) \lbrb{\mathfrak{B}_k X_{\Psi}\times\mathfrak{B}_{-k} Y_{\Psi}},
	\end{eqnarray}
	where $\times$ stands for the product of independent random variables and $X_{\phn}$ is with distribution
	\[\Pbb{X_{\phn}\in dx}=\phn(0)x\Pbb{Y^{-1}_{\phn}\in dx},\:x>0,\]
	where $Y_{\phn}$ is the  random variable whose Mellin transform is $\Wpn$, see the exact definition below \eqref{eq:defWb}.
	The laws of the positive variables $X_{\Psi}, Y_\Psi$ are given by
	\begin{equation}\label{eq:rvs}
	\begin{split}
	\Pbb{X_{\Psi} \in dx}&=\frac{1}{\php(1)}\left(\bar{\mu}_{\pls}(-\ln x)dx+\php(0)dx+\dep \delta_{1}(dx)\right),\,x\in\lbrb{0,1},\\
	\Pbb{Y_{\Psi} \in dx}&=\phn(0)\Upsilon_{\mis}(dx),\,x>1,
	\end{split}
	\end{equation}
	where $\Upsilon_{\mis}(dx)=U_{\mis}\lbrb{d\ln(x)},\,x>1,$ is the image of the potential measure $U_{\mis}$, associated to $\phn$, see Proposition \ref{propAsymp1} \eqref{it:bernstein_cmi}, by the mapping $y\mapsto \ln y$,	
	\[C_\Psi(0)=e^{\gamma_{\php}+\gamma_{\phn}-\gamma},\: C_\Psi(k)=e^{\frac1{k}-\frac{\php'\!(k)}{\phpspace(k)}-\frac{\phn'\!(k)}{\phnspace(k)}}, \: k=1,2,\ldots,\]
	where we recall that $\gamma_{\php}$ and $\gamma_{\phn}$ were defined in \eqref{sym:gphi} and $\gamma$ is the Euler-Mascheroni constant,
	and for integer $k, \mathfrak{B}_k X$ is the variable defined by
	\[ \Ebb{f(\mathfrak{B}_k X)} = \frac{\Ebb{X^kf(X)}}{\Eb[X^k]}.\]
\end{theorem}
\begin{remark}\label{rem:factorization}
	The first factorization in \eqref{eq:GammaType2} is proved in \cite{Pardo2012, Patie-Savov-11} under the assumption that $\Pi(dx)\ind{x>0}=\pi_{\pls}(x)dx$ with $\pi_{\pls}$ being non-decreasing on $\Rp$. Then $X_{\phn}=I_\psi$ with $\psi(z)=z\phn(z)\in\Nc$. The factorization \eqref{eq:GammaType2} has been announced in generality in \cite{Patie-Savov-13}. The second factorization of \eqref{eq:InfW} is new. When $\Psi(z)=-\php(-z)\in\Bc$ that is $\xi$ is a subordinator then \eqref{eq:InfW} is contained in \cite[Theorem 3]{Alili-Jadidi-14}. For the class of meromorphic \LL processes,  $\IPsi$ has been factorized in an infinite product of independent Beta random variables, see \cite{Hackmann-Kuznetsov-14}.
\end{remark}
Next, we discuss some immediate results for the positive self-similar Markov processes whose  semigroups we call for brevity the positive self-similar semigroups and are denoted by $K^\Psi=\lbrb{K_t^\Psi}_{t\geq0}$. The dependence on $\Psi\in\overNc$ is due to the Lamperti transformation which establishes a bijection between the positive self-similar semigroups and $\overNc$, see \cite{Lamperti-62}. We recall that for some $\alpha>0$, any $f\in\cco\!\lbrb{\lbbrb{0,\infty}}$ and any $x,c>0$, $K^\Psi_t\!f(cx)=K^\Psi_{c^{-\alpha}t}d_cf(x)$, where $d_cf(x)=f(cx)$ is the dilation operator. Without loss of generality we set $\alpha=1$ and write \[\Nc_{\unrhd}=\curly{\Psi\in\Nc:\,\Psi(0)=0,\,\php'(0^+)<\infty,\,\Zc_0(\Psi)=\curly{0}}.\label{Nm}\]
From \eqref{eq:WH1}, \eqref{eq:Nc} and Theorem \ref{thm:Wp}\eqref{it:D}, it is clear that the class $\Nc_{\unrhd}$ stands for the characteristic exponents of conservative \LLPs that do not live on a lattice, which either drift to infinity and possess a finite positive mean, or oscillate, but the ascending ladder height process has a finite mean, that is $\php'\!\lbrb{0^+}<\infty$. It is well-known from \cite{Bertoin-Savov} that $\Psi\in\Nc_{\unrhd} $ if and only if $K^\Psi$ possesses an entrance law from zero. More specifically, for each $\Psi \in \Nc_{\unrhd}$, there exists a family of probability measures $\nu^\Psi=\lbrb{\nu^\Psi_t}_{t>0}$ such that $\limo{t}\nu^\Psi_t=\delta_0$ and for any  $f\in\cco\!\lbrb{\lbbrb{0,\infty}}$ and any $t,s>0$,
\[\nu^\Psi_{t+s}f=\IntOI f(x)\nu^\Psi_{t+s}(dx)=\IntOI K_s^\Psi\!f(x) \nu^\Psi_t(dx)=\nu^\Psi_t K^\Psi_s\!f.\]
We denote by $\VPsi$ the random variable whose law is $\nu_1^\Psi$.
\begin{theorem}\label{thm:PSSMP}  Let $\Psi\in\Nc_{\unrhd}$.
	\begin{enumerate}
		\item\label{it:entrance} Then the Mellin transform $\mathcal{M}_{\VPsi}$ of $\VPsi$ is the unique solution in  the space of Mellin transforms of non-negative random variables
		to the following functional equation
		\begin{equation}\label{eq:feVPsi}
		\mathcal{M}_{\VPsi}\!(z+1)=\frac{\Psi(z)}{z}\mathcal{M}_{\VPsi}\!(z), \: z\in i\R,
		\end{equation}
		with initial condition $\Mcc_{\VPsi}(1)=1$, and admits the representation
		\begin{equation}\label{eqn:GammaType}
		\mathcal{M}_{\VPsi}(z)=\frac{1}{\php'(0^+)}\frac{\Gamma(1-z)}{W_{\phi_{\pls}}\!(1-z)}W_{\phi_{\mis}}\!(z),\quad  z\in \C_{\lbrb{\dphn,1}}.
		\end{equation}
		\item \label{it:intertwining} Let in addition $\Pi(dx)\mathbb{I}_{\{x>0\}}=\pi_{\pls}(x)dx,$ with $\pi_{\pls}$ non-increasing on $\Rp$, see \eqref{eq:lk0}. Then $\Lambda_{\php}f(x)=\Ebb{f\left(xV_{\php}\right)}$ is a continuous linear operator from  $\cco\!\lbrb{\lbbrb{0,\infty}}$, endowed with the uniform topology,  into itself  and we have the following intertwining identity on $\cco\!\lbrb{\lbbrb{0,\infty}}$
		\begin{equation}\label{MainInter}
		K^{\psi}_t\!\Lambda_{\php}f=\Lambda_{\php}\!K^{\Psi}_t\!f, \quad t\geq 0,
		\end{equation}
		where  $\psi(z)=z\phn(z)\in \Nc_-$, see Remark \ref{rem:SN}.
	\end{enumerate}
\end{theorem}
This theorem is proved in Section \ref{sec:prof_int}.
\begin{remark}
	The literature on intertwining of Markov semigroups is  very rich and reveals that it is useful in a variety of contexts, see e.g.~ Diaconis and Fill \cite{Diaconis1990} in relation
	with strong uniform times, by Carmona, Petit and Yor \cite{Carmona-Petit-Yor-98} in relation to the so-called selfsimilar
	saw tooth-processes, by Borodin and Corwin \cite{Borodin-Corwin} in
	the context of Macdonald processes, by Pal and Shkolnikov \cite{Pal-Shkolnikov} for linking diffusion operators,
	and, by Patie and Simon \cite{Patie2012b} to relate classical fractional operators. In this direction, it seems that the family of intertwining relations \eqref{MainInter}  are the first instances involving Markov processes with two-sided jumps.
\end{remark}
\begin{remark}
	Recently, this type of commutation relations proved to be a natural concept in some new  developments of spectral theory. We refer to the work of Miclo \cite{Miclo-16} where it is shown that the notions  of isospectrality and  intertwining  of some self-adjoint Markov operators are equivalent leading to an alternative view of the work of B\'erard \cite{Berard-89} on isospectral compact Riemanian manifolds, see also Arendt et al.~\cite{Arendt} for similar developments that enable them to give counterexamples to the famous Kac's problem. Intertwining  is also the central idea in the recent works of the authors \cite{Patie-Savov-GL,Patie-Savov-16} on the spectral analysis of  classes of non-self-adjoint, non-local Markov semigroups and also in the extension of Krein's theory offered by Patie et al.~\cite{Patie-Savov-Zhao}. We stress that the intertwining relation \eqref{MainInter} and more generally the analytical properties of the solution of the recurrence equation \eqref{eq:fe}, studied in this paper, are critical in the spectral theory of the full class of positive self-similar Markov semigroups developed in \cite{Patie-Savov-Zhao}.
\end{remark}

\section{Known and new results on Bernstein functions and some  related functions}\label{sec:BernFunc}
Recall from \eqref{eq:phi0} that $\phi\in\Bc$ if and only if $\phi\not\equiv 0$ and, for $z\in\Cb_{\lbbrb{0,\infty}}$,
\begin{equation}\label{eq:phi}
\begin{split}
\phi(z)&=\phi(0)+\dr z+\IntOI \lbrb{1-e^{-zy}}\mu(dy)\\
&=\phi(0)+\dr z+z\IntOI e^{-zy}\bar{\mu}(y)dy,
\end{split}
\end{equation}
where $\phi(0)\geq 0,\,\dr\geq 0$ and $\mu$ is a sigma-finite measure satisfying $\IntOI (1\wedge y) \mu(dy)<\infty$. Also here and hereafter we denote the tail of a measure $\lambda $ by $\bar{\lambda}(x)=\int_{|y|\geq x}\lambda(dy)$, provided it exists. Due to its importance, the class $\Bc$ has been studied extensively in several monographs and papers, see e.g.~\cite{Patie-Savov-16,Schilling2010}. Recall the definitions of $\aph,\tph,\dph$ respectively in \eqref{eq:aphi},\eqref{eq:thetaphi} and \eqref{eq:dphi}.

We start by providing some basic functional properties for $\phi\in\Bc$ most of which can be found in \cite{Schilling2010} and \cite[Section 4]{Patie-Savov-16}.
\begin{proposition}\label{propAsymp1}
	Let $\phi\in\Bc$.
	\begin{enumerate}
		\item  For any $z\in\Cb_{\intervalOI}$,
		\begin{equation}\label{eq:phi'}
		\phi'(z)=\dr +\IntOI ye^{-zy}\mu(dy)=\dr+\IntOI e^{-zy}\bar{\mu}(y)dy-z\IntOI e^{-zy}y\bar{\mu}(y)dy.
		\end{equation}
		\item \label{it:bernstein_cm} For any $u\in \R^+$,
		\begin{equation}\label{specialEstimates11}
		0\leq u \phi'(u)\leq \phi(u)\quad \text{ and }\quad |\phi''(u)|\leq 2\frac{\phi(u)}{u^2}.
		\end{equation}
		\item\label{it:shiftBern} If $\tph\in\lbrb{-\infty,0}$ then for any $u\geq \tph$  we have that $\phi(u+\cdot)\in\Bc$.
		\item \label{it:asyphid}  $\phi(u)\stackrel{\infty}{=} \dr u +\sospace{u}$ and $\phi'(u)\stackrel{\infty}{=}\dr+\sospace{1}$. Fix  $a>\aph$, then $\labsrabs{\phi\lbrb{a+ib}}=\labsrabs{a+ib}\lbrb{ \dr+\sospace{1}}$ as $|b|\to\infty$. The latter extends to $a=\aph$ provided $\abs{\phi(\aph)}<\infty$. Moreover, $\phi(z)=\dr z\lbrb{1+\sospace{1}}$ uniformly on $\Cb_{\lbbrb{\tph,\infty}}$.
		\item\label{it:finPhi} If $\phi\lbrb{\infty}<\infty$ and $\mu$ is absolutely continuous then for any fixed $a>\aph$, we have that $\limi{|b|}\phi\lbrb{\ab}=\phi\lbrb{\infty}$.
		\item \label{it:bernstein_cmi} The mapping $u\mapsto \frac{1}{\phi(u)},\, u\in \R^+,$ is completely monotone, i.e.~there exists a positive measure $U$, whose support is contained in $\lbbrb{0,\infty}$, called the potential measure, such that the Laplace transform of $U$ is given via the identity
		\[ \frac{1}{\phi(u)} = \int_0^{\infty} e^{-uy}U(dy).\]
		Moreover, if $\phi\in\BP$, that is $\dr>0,$ then $U$ has a density $u\in \Ctt\!\lbrb{\lbbrb{0,\infty}}$ with $u(0)=\frac{1}{\dr}$ and $u>0$ on $\Rp$.
		\item \label{it:zero-free} $\phi$ is zero-free on $\CbOI$.
		\item In any case, \label{it:flatphi}
		\begin{equation}\label{lemmaAsymp1-1}
		\lim_{u\to\infty}\frac{\phi(u\pm a)}{\phi(u)}=1\,\,\text{ uniformly for $a$-compact intervals  on $\R^+$.}
		\end{equation}
		\item \label{it:RePhi} For any $\phi\in\Bc$ and $a>\dph$
		\begin{equation}\label{eq:rephi}
		\Re\!\lbrb{\phi(a+ib)}=\phi(0)+\dr a+\IntOI\lbrb{1-e^{-ay}\cos\lbrb{by}}\mu(dy)\geq\phi(a)>0.
		\end{equation}
		\item For any $\phi\in\Bc$ and $z\in\CbOI$
		\begin{equation}\label{eq:ineqArg}
		\abs{\arg\phi(z)}\leq \abs{\arg(z)}.
		\end{equation}
		\item\label{it:argPhi} Let $\phi\in\BP^c$\label{BPc}, that is $\dr=0$. Fix any $a>0$. Then, for any $v>0$, there is $u(v)>0$ such that, for any $u>u(v)$,        \begin{equation}\label{eq:argPhi}
		\abs{\arg\lbrb{\phi(a+iu)}}\leq \frac{\pi}{2}-\arctan\!\lbrb{\frac{v\phi(a)}{u}}.
		\end{equation}
	\end{enumerate}
\end{proposition}
Next, we define important functions based on $\phi$ which play a key role in the understanding of the asymptotic behaviour of the Bernstein-gamma functions. Let $\phi\in\Bc$. Define formally, for any $z=a+ib\in\Cb_{\intervalOI}$,
\begin{align}\label{eq:Aphi}
\Aph\!\lbrb{z}&=\int_{0}^{b}\arg \phi\lbrb{a+iu}du.
\end{align}
The following functions are defined for $a>0$
\begin{equation}\label{eq:Gphi}
\begin{split}
\Gph\lbrb{a}&=\int_{1}^{1+a}\ln\phi(u)du,\\
\Hph(a)&=\int_{1}^{1+a}\frac{u\phi'(u)}{\phi(u)}du \quad\textrm{ and }\quad \Hph^*(a)=a\lbrb{\frac{\phi(a+1)-\phi(a)}{\phi(a)}}.
\end{split}
\end{equation}
Next, recall the floor and ceiling functions defined as follows $\lfloor u\rfloor=\max\lbcurlyrbcurly{n\in\Nb:\,n\leq u}$ and $\lceil u\rceil=\min\lbcurlyrbcurly{n\in\Nb:\, u\geq n}$.   For $z=a+ib\in\Cb_{\intervalOI}$, writing \[\mathrm{P}(u)=\lbrb{u-\lrfloor{u}}\lbrb{1-\lbrb{u-\lrfloor{u}}},\] we set
\begin{equation}\label{eq:Ephi}
\,\,\Eph\!\lbrb{z}=\frac12\int_{0}^{\infty}\mathrm{P}(u)\lbrb{\ln\frac{\abs{\phi(u+z)}}{\phi(u+a)}}''du,
\end{equation}
\begin{equation}\label{eq:Rphi}
\Rph(a)=\frac12\int_{1}^{\infty}\mathrm{P}(u)\lbrb{\ln\frac{\phi(u+a)}{\phi(u)}}''du,
\end{equation}
where the derivatives under the integrals are with respect to $u$ and
\begin{align}\label{eq:realInfty}	T_\phi=-\frac12\int_{1}^{\infty}\mathrm{P}(u)\lbrb{\ln\phi(u)}''du.
\end{align}
We start by proving some general properties of these functions.
\begin{theorem}\label{thm:genFuncs}
	Let $\phi\in\Bc$.
	\begin{enumerate}
		\item \label{it:Aphi} For $z=\ab\in\CbOI$
		\begin{eqnarray}\label{eq:Tphi} \label{eq:A=Theta}
		\Aph\!\lbrb{\ab}=\int_{a}^{\infty}\ln\lbrb{\frac{\abs{\phi\lbrb{u+ib}}}{\phi(u)}}du \in \left[0,\frac\pi2|b|\right].
		\end{eqnarray}
		As a result, for any $b\in\R$, $a\mapsto \Aph\!\lbrb{\ab}$ is non-increasing on $\R^+$ and if
		\[\phi^{\dag q }(z)=q+\phi(0)+\dr z+\IntOI \lbrb{1-e^{-zy}}\mu(dy),\,q\geq 0,\]
		then $q\mapsto A_{\phi^{\dag q }}(z)$ is non-increasing for any $z\in\CbOI$.
		\item\label{it:RE}  For any $a>0$, we have that
		\begin{equation}\label{eq:uniBoundRE1}
		\sup_{\phi\in\Bc}\sup_{z\in\Cb_{\lbrb{a,\infty}}}\abs{\Eph\!\lbrb{z}}<\infty \quad\textrm{ and }\quad \sup_{\phi\in\Bc}\sup_{c>a} \abs{\Rph\!\lbrb{c}}<\infty.
		\end{equation}
		\item \label{it:H}	We have the asymptotic relations
		\begin{equation}\label{eq:AsympHphi}
		\begin{split}
		&0\leq \liminfi{a}\frac{\Hph(a)}a\leq \limsupi{a}\frac{\Hph(a)}a\leq 1\\
		&0\leq \liminfi{a}\Hph^*(a)\leq \limsupi{a}\Hph^*(a)\leq 1,
		\end{split}
		\end{equation}
		and, for large $a$,
		\begin{equation}\label{eq:Gph}
		\Gph(a)= a\ln\phi(a)-\Hph(a)+\Hph^*(a)+\ln\phi(a)-\ln\phi(1)+\bospace{\frac1a}.
		\end{equation}
		\item\label{it:Tphi} For any $b\in\R$ we have  	$T_\phi=\limi{a}\lbrb{\Eph\!\lbrb{a+ib}+\Rph(a)}$.
	\end{enumerate}
\end{theorem}
We note that the last result of Theorem \ref{thm:genFuncs} \eqref{it:Aphi} concerning the monotonicity of $\Aph$ upon killing can be significantly extended when $\phi$ is a Wiener-Hopf factor in the sense of \eqref{eq:WH1}, see Lemma \ref{lem:Aphi}. The next result gives explicit estimates on $\Aph$ in some special but important cases.\newpage
\begin{theorem}\label{thm:Aph}
	Let $\phi\in\Bc$.
	\begin{enumerate}
		\item \label{it:awpid} If  $\phi\in\BP$, that is $\dr>0$, then, for any $a>0$ fixed and as $|b|\to\infty$,
		\begin{align}\label{eq:AphiAsymp}	 			
		\Aph\!\lbrb{a+ib} &= \frac{\pi}{2}|b|+\sospace{|b|}.
		\end{align}
		Even more precisely, as $|b|\to\infty,$
		\begin{align}\label{eq:AphiAsymp4}	 			
		\Aph\!\lbrb{a+ib} &\leq \frac{\pi}{2}|b|-\lbrb{a+\frac{\phi(0)}{\dr}+\frac{\bar{\mu}\!\lbrb{\frac{1}{|b|}}\lbrb{1+\sospace{1}}}{\dr}}\ln |b|.
		\end{align}		
		\item \label{it:awpia} Next, let $\phi\in\Bc_{\alpha}$, see \eqref{eq:classesPhi1}, with $\alpha\in\lbrb{0,1}$. Then, for any fixed $a>0$,
		\begin{align}\label{eq:AphiAsymp1}
		\Aph\!\lbrb{a+ib}\stackrel{\infty}{=}\frac{\pi}{2}\alpha|b|\lbrb{1+\sospace{1}}.
		\end{align}
		\item \label{eq:subexpp}  Let $\phi\in\BP^c$, that is $\dr=0,$ with $\mu(dy)=\upsilon(y)dy$. If $\upsilon(0^+)<\infty$ exists and $||\upsilon||_\infty=\sup_{y\geq 0}\upsilon(y)<\infty$ then, for any $a>0$,
		\[\limi{b}\frac{\Aph\!\lbrb{a+ib}}{\ln(b)}=\frac{\upsilon(0^+)}{\phi(\infty)}.\]
		Otherwise, if  $\upsilon(0^+)=\infty$, $\upsilon(y)=\upsilon_1(y)+\upsilon_2(y),\,\upsilon_1,\upsilon_2\in\Lspace{1}{\Rp}$, $\upsilon_1\geq0$ is non-increasing in $\Rp$, \[\IntOI \upsilon_2(y)dy\geq0 \text{ and } \abs{\upsilon_2(y)}\leq \lbrb{\int_{y}^{\infty}\upsilon_1(r)dr}\vee C\] for some $C>0$, then 	\[\limi{b}\frac{\Aph\!\lbrb{a+ib}}{\ln(b)}=\infty.\]
	\end{enumerate}
	\textit{This result is proved  mainly in Section \ref{sec:ProofsBE} with item \eqref{eq:subexpp} established at the end of Section \ref{sec:FE}, see page \pageref{page:subexpp}.}
\end{theorem}
\begin{remark}\label{rem:precision}
	Proposition \ref{thm:imagineryStirling} below  gives more detailed information for $\Aph$ when $\phi\in\BP$ which we have omitted here for the sake of clarity.
\end{remark}
\begin{remark}\label{rem:3a}
	The requirements for item \eqref{eq:subexpp} are not stringent. They impose that in a neighbourhood of $0$, the density can be decomposed as a non-increasing function and an oscillating function $\upsilon_2$ of order  $\bospace{\int_{y}^{\infty}\upsilon_1(r)dr}$. This is the case when $\upsilon$ itself is non-increasing and therefore $\upsilon_2\equiv0$.
\end{remark}
\section{The class of Bernstein-Gamma functions} \label{sec:mainResults}
Perhaps one of the most used special functions is the gamma function $\Gamma$ introduced by Euler in \cite{Euler-1728}. Amongst  various properties it satisfies the recurrence equation
\begin{equation}\label{eq:GammaRec}
\Gamma\!\lbrb{z+1}=z\Gamma(z),\quad\Gamma(1)=1,
\end{equation}
valid on $\Cb\setminus\Nb^-$.
For example, relation \eqref{eq:GammaRec} allows for the derivation of both the Weierstrass product representation of $\Gamma(z)$ and the precise Stirling asymptotic expression for the behaviour of the gamma function as $|z|\to\infty$, see e.g.~\cite{Lebedev-72} for more information. Here, we use it to introduce the class of Bernstein-gamma functions denoted by $\Wc_\Bc$, which appear in most of the results above.
\begin{eqnarray}\label{eq:defWb} \Wc_\Bc\!&=&\!\left\{W \in \mathcal{P}\!:W\!\lbrb{1}=1, W\!\lbrb{z+1}=\phi(z)W\!\lbrb{z}\!,\,\Re(z)>0,  \textrm{ for some } \phi\!\in\!\Be\right\}\!.\label{def:WB}
\end{eqnarray}
It means that $\Wp \in \Wc_\Bc$,  for some  $\phi \in \Be$  if there exists a positive random variable $Y_{\phi}$ such that $\Wp\!\lbrb{z+1}=\Ebb{Y_{\phi}^z},\,\Re(z)\geq 0$.

Note that, when in \eqref{eq:GammaRec}, $\phi(z)=z \in \Be$, then $\Wp$ boils down to the gamma function with $Y_{\phi}$  the standard exponential random variable. This yields to the well known integral  representation of the gamma function, i.e.~ $\Gamma(z)=\int_{0}^{\infty}x^{z-1}e^{-x}dx$ valid on $\Re(z)> 0$. 
The functions $\Wp \in\Wc_\Bc$ have already appeared explicitly, see \cite{Alili-Jadidi-14,Hirsch-Yor-13,Patie-Savov-13,Patie-Savov-16}, or implicitly, see \cite{Bertoin-Yor-05,Maulik-Zwart-06}, in the literature. However, with the exception of \cite[Chapter 6]{Patie-Savov-16} we are not aware of other studies that focus on the understanding of $\Wp$ as a holomorphic function on the complex half-plane $\CbOI$. The latter is of importance at least for the following reasons. First, the class $\Wc_\Bc$ arises in the spectral study of some Markov semigroups and the quantification of its analytic properties in terms of $\phi\in\Bc$ opens the door to obtaining explicit information about most of the spectral objects and quantities of interest, see \cite{Patie-Savov-16}. Then, the class $\Wc_\Bc$ appears in the description of $\MPs$ and hence of $\MPsi$, see \eqref{eqM:MIPsi} and \eqref{eq:MIPsi}. Thus, the understanding of the analytic properties of $\Wp$ yields detailed information about the law of the exponential functionals.  Also, the class $\Wc_\Bc$ contains some well-known special functions, e.g.~the Barnes-gamma function and the q-gamma function related to the q-calculus, see \cite{Bertoin-Biane-Yor-04,Chhaibi-2016}, \cite[Remark 6.4]{Patie-Savov-16}, and the derivation of the analytic properties of $\Wc_\Bc$ renders many special computations to a direct application of the results concerning the functions $\Wp$. Equations of the type \eqref{eq:Wp} have been considered on $\Rp$ in greater generality. For example when $\phi$ is merely a $\log$-concave function on $\Rp$, Webster \cite{Webster-97} has provided comprehensive results on the solution to \eqref{eq:Wp}, which we use  readily throughout this work since $\phi\in\Bc$ is a $\log$-concave function on $\Rp$ itself.

In this Section, we start by stating the main results of our work concerning the class $\Wc_\Bc$ and postpone their proofs to the subsections \ref{sec:proof_wb}--\ref{sec:proof_wbe}. In particular, we derive and state representations, asymptotic and analytical properties of $\Wp$. To do so we introduce some notation.
For any $\phi\in\Bc$ and any $a>\aph$ we set
\begin{equation}\label{eq:zerosphi}
\Zc_{a}(\phi)=\lbcurlyrbcurly{z\in\Cb_{a}:\,\phi(z)=0}=\lbcurlyrbcurly{z\in\Cb_{a}:\,\phi(\overline{z})=0}.
\end{equation}
Also, we recall \eqref{eq:aphi}, \eqref{eq:thetaphi} and \eqref{eq:dphi}, that is
\begin{align}
\aph&=\inf\{u<0:\phi\in  \Ac_{(u,\infty)}\}\in [-\infty,0]\label{eq:aphi1},\\
\tph&=\sup\lbcurlyrbcurly{u\in\lbbrbb{\aph,0}:\,\phi(u)=0}\in\lbbrbb{-\infty,0}\label{eq:thetaphi1},\\
\dph&=\max\curly{\aph,\tph}\in\lbbrbb{-\infty,0}\label{eq:dphi1}.
\end{align}
The next theorem contains some easy but useful results which stem from the existing literature. Before we state them we recall from \cite[Chapter 6]{Patie-Savov-16} that
the class $\Bc$ is in bijection with $\Wc_\Bc$ via the absolutely convergent product on (at least) $\CbOI$
\begin{equation}\label{eq:BernWeier}
\Wp\!\lbrb{z}=\frac{e^{-\gamma_\phi z}}{\phi(z)}\prod_{k=1}^{\infty}\frac{\phi(k)}{\phi\!\lbrb{k+z}}e^{\frac{\phi'(k)}{\phi(k)}z},
\end{equation}
where
\begin{equation}\label{eq:EulerConst}
\gamma_\phi=\lim\ttinf{n}\lbrb{\sum_{k=1}^{n}\frac{\phi'(k)}{\phi(k)}-\ln\phi(n)}\in\lbbrbb{-\ln\phi(1),\frac{\phi'(1)}{\phi(1)}-\ln\phi(1)}.
\end{equation}

\begin{theorem}\label{thm:Wp}
	Let $\phi \in \Be$.
	\begin{enumerate}
		\item  \label{it:A} \label{it:B} \label{it:C}
		$\Wp \in\Ac_{\lbrb{\dph,\infty}}\cap \Mtt_{\lbrb{\aph,\infty}}$ and $\Wp$ is zero-free on $\Cb_{\lbrb{\aph,\infty}}$. If $\phi(0)>0$ then  $\Wp\in\Ac_{\lbbrb{0,\infty}}$ and $\Wp$ is zero-free on $\Cb_{\lbbrb{0,\infty}}.$
		If $\phi(0)=0$ (resp.~$\phi(0)=0$ and $\phi'(0^+)<\infty$)  then $\Wp$ (resp.~$z \mapsto z\Wp(z)$) extends continuously to $i\R\setminus\Zc_0(\phi)$ (resp.~$\lbrb{i\R\setminus\Zc_0(\phi)}\cup\curly{0}$)  and if $\mathfrak{z} \in\Zc_0(\phi)$ then \[\lim\limits_{\stackrel{\Re(z)\geq 0}{z \to \mathfrak{z}}} \phi(z)\Wp\lbrb{z}=\Wp\lbrb{\mathfrak{z}+1}.\] Finally, if $\phi(0)=0,\phi''\!(0^+)<\infty$ then
		$\Wp\lbrb{z}-\frac{1}{\phi'\!\lbrb{0^+}z}$ extends  continuously to the region $\lbrb{i\R\setminus\Zc_0(\phi)}\cup\curly{0}$.
		
		\item\label{it:D} There exists $ \mathfrak{z} \in \Zcph $ with $\mathfrak{z}\neq 0$  if and only if $\phi(0)=\dr=0$ and $\mu=\sum_{n=0}^{\infty}c_n\delta_{ h k_n}$ with $\sum_{n=1}^{\infty}c_n<\infty$,  $h>0$, $k_n\in\N$ and $c_n\geq0$ for all $n\in \N$. In this case, the mappings \[z\mapsto e^{-z\ln\phi(\infty)}\Wp(z) \text{ and } z\mapsto\abs{\Wp(z)}\] are periodic with period $\frac{2\pi i}{h}$ on $\CbOI$.
		\item \label{it:E}  Assume that $\tph\in\lbrb{-\infty,0}$, $\Zc_{\tph}(\phi)=\lbcurlyrbcurly{\tph}$ and $\abs{\phi''(\tph^+)}<\infty$, the latter being always true if $\tph>\aph$. Then \[z \mapsto \Wp(z)-\frac{\Wp(1+\tph)}{\phi'(\tph^+)\lbrb{z-\tph}}\in\Ac_{\lbbrb{\tph,\infty}}.\] In this setting $\phi'(\tph^+)=\dr+\IntOI ye^{-\tph y}\mu(dy)\in\lbrb{0,\infty}$.
		\item\label{it:F} Assume that $\aph<\tph\leq 0$ and put \[N_{\aph}=\max\curly{n\in\N:\,\tph-n>\aph}\in\N\cup\curly{\infty}.\] Then there exists an open set $\Oc\subset\Cb$ such that $[\tph-N_{\aph},1]\subset\Oc,$ if $N_{\aph}<\infty$, and $\lbrbb{-\infty,1}\subset\Oc$, if $N_{\aph}=\infty$, and $\Wp$ is meromorphic on $\Oc$ with simple poles at $\curly{\tph-k}_{0\leq k< N_{\aph}+1}$ and residues ${\curly{\mathfrak{R}_k=\frac{\Wp\lbrb{1+\tph}}{\phi'(\tph)\prod_{j=1}^{k}\phi(\tph-j)}}}_{0\leq k< N_{\aph}+1}$ with the convention $\prod_{1}^0=1$.
		\item\label{it:G} Let $\phi\in\Bc$ and $c>0$. Then $W_{c\phi}(z)=c^{z-1}\Wp(z)$.		
	\end{enumerate}
	
\end{theorem}
	\textit{This result is proved  mainly in Section \ref{sec:proof_wb}.}
	
Theorem \ref{thm:Wp} and especially representation \eqref{eq:BernWeier} allow for the understanding of the asymptotic behaviour of $\abs{\Wp\lbrb{z}}$, as $|z|\to \infty$. The main functions that control the asymptotic are $\Aph,\Gph$ introduced in \eqref{eq:Aphi} and \eqref{eq:Gphi}, but before stating the main results,
we introduce subclasses of $\Bc$ equivalent to $\Npb$ and $\Nth$, see \eqref{eq:polyclass} and \eqref{eq:polyclass1}. For any $\beta\in\lbbrb{0,\infty}$
\begin{equation}\label{eq:polyclassB}
\Beb=\,\,\curly{\phi\in\Bc:\,\limi{|b|}\frac{\abs{\Wp\!\lbrb{a+ib}}}{|b|^{-\beta+\varepsilon}}=\limi{|b|}\frac{|b|^{-\beta-\varepsilon}}{\abs{\Wp\!\lbrb{a+ib}}}=0,\,\,\forall a>\dph,\forall\varepsilon>0}
\end{equation}
and
\begin{equation*}
\Bep{\infty}=\,\,\curly{\phi\in\Bc:\,\limi{|b|}|b|^\beta\abs{\Wp\!\lbrb{a+ib}}=0,\,\,\forall a>\dph,\forall\beta\geq 0}.
\end{equation*}
 Also, for any $\theta\in\lbrbb{0,\frac{\pi}{2}}$,
\begin{equation}\label{eq:polyclassB1}
\Bth=\curly{ \phi \in \Bc:\:\limsupi{|b|} \frac{\ln \abs{\Wp\!\lbrb{a+ib}}}{|b|}\leq -\theta,\,\,\forall a>\dph}.
\end{equation}
Recalling that the functions $\Aph,\Gph,\Eph,\Rph,\Hph,\Hph^*,\Tph$ were  introduced in \eqref{eq:Aphi},\eqref{eq:Gphi}, \eqref{eq:Ephi},\eqref{eq:Rphi} and \eqref{eq:realInfty}, we now state our second main result which can be thought of as the Stirling asymptotic for the Bernstein-gamma functions.
\begin{theorem}\label{thm:Stirling} \label{thm:imagineryStirling1}
	\begin{enumerate}
		\item \label{it:gbWp}
		For any  $\phi\in\Bc$ and any $z=\ab\in\CbOI$, we have that
		\begin{align}\label{eq:Stirling}
		\abs{\Wp\!\lbrb{z}}&	=\frac{\sqrt{\phi(1)}}{\sqrt{\phi(a)\phi(1+a)|\phi(z)|}}e^{\Gph(a)-\Aph(z)}e^{-\Eph(z)-\Rph(a)}.
		\end{align}
		For any $a>0$
		\begin{equation}\label{eq:err}
		\sup_{\phi\in\Bc}\sup_{z\in\Cb_{\lbrb{a,\infty}}}\abs{e^{-\Eph(z)-\Rph(\Re(z))}}\asymp 1.
		\end{equation}
		\item \label{it:realStirling}  \label{thm:realStirling} 	For any fixed $b\in \R$ and  large $a$, we have  that
		\begin{align}\label{eq:realStirling}	  \abs{\Wp\!\lbrb{a+ib}}=\frac{e^{-T_\phi}}{\sqrt{\abs{\phi(\ab)}\phi(1)}}e^{a\ln\phi(a)-\Hph(a)+\Hph^*(a)-\Aph(\ab)}\lbrb{1+\bospace{\frac{1}{a}}},
		\end{align}
		where $\lim\limits_{a \to \infty}\Aph(a+ib)=0$.
		
		\item\label{it:cases} Let $\phi\in\Bc$
		\begin{enumerate}
			\item\label{it:case1} if $\phi\in\BP$ then $\phi\in\Bthd$;
			\item\label{it:case2} if $\phi\in\Bc_\alpha,$ for $\alpha\in\lbrb{0,1}$, then $\phi\in  \Bthaspace $;
			\item\label{it:case3} under the conditions of Theorem \ref{thm:Aph}\eqref{eq:subexpp} with $\mu(dy)=\upsilon(y)dy$ then $\phi\in \Bep{\Nps} $ with $\Nps=\frac{\upsilon\lbrb{0^+}}{\phi\lbrb{\infty}}$ provided $\upsilon(0^+)<\infty$, whereas $\phi\in \Bi $ otherwise.
		\end{enumerate}
	\end{enumerate}
\end{theorem}
\textit{The proof of this result starts with Section \ref{subsec:proof}.}
\begin{remark}\label{rem:Aphi}
	Note the dependence of \eqref{eq:Stirling} on the geometry of  $\phi\!\lbrb{\CbOI}\subseteq \CbOI$. The more $\phi$ shrinks $\CbOI$ the smaller the contribution of $\Aph$. In fact $\frac{1}{b}\Aph\!\lbrb{a+ib}$, as $b\to\infty$, measures the fluctuations of the average angle along the contour $\phi\!\lbrb{\Cb_a}$, which due to averaging are necessarily of lesser order than those of $\arg\phi(\ab)$ along $\Cb_a$. Overall,  $\Aph$  codes the asymptotic associated to changes of $\Im(z)$,  $\Gph$ reflects the asymptotic purely due to change of $\Re(z)$ and the functions $\Rph, \Eph$ control uniformly the error of approximation on $\Cb_{\lbrb{a,\infty}},\forall a>0,$ and on the space of Bernstein functions.
\end{remark}
\begin{remark}
	When $z=a$, modulo to the specifications of the constants, the central result \eqref{eq:Stirling} has appeared for $\log$-concave functions in \cite[Theorem 6.3]{Webster-97} and for Bernstein functions in \cite[Theorem 5.1]{Patie-Savov-16}. Here, we provide an explicit representation of the terms of the asymptotics of $\abs{\Wp(a+ib)}$, as simultaneously $a\to\infty$ and $|b|\to\infty$.
\end{remark}

\begin{remark}\label{rem:shiftedWp}
	Since $\Wp\in\Ac_{\lbrb{\dph,\infty}}\cap \Mtt_{\lbrb{\aph,\infty}}$ one can extend, via the recurrent equation \eqref{eq:Wp}, the estimate \eqref{eq:Stirling}, away from the poles residing in $\Cb_{\lbrbb{\aph,\,\dph}}$, to $z\in\Cb_{\lbrb{\aph,\,\infty}}$. Indeed, let $a>\aph$, $z=\ab$ not a pole and take $n\in\N$ such that $n+a>0$, then one has
	\begin{equation}\label{eq:Stirling1}
	\Wp\!\lbrb{a+ib}=\Wp\!\lbrb{a+n+ib}\prod_{j=0}^{n-1}\frac{1}{\phi\!\lbrb{a+j+ib}}
	\end{equation}
	with the convention that $\prod_{0}^{-1}=1$. Note that under these assumption $\prod_{j=0}^{n-1}\phi\lbrb{a+j+ib}\neq 0$ because otherwise if $\phi\lbrb{a+k+ib}=0,1\leq k\leq n-1,$ then via $\Wp(z+1)=\phi(z)\Wp(z)$ the points		
	$a+j+ib, 0\leq j\leq k$ would all be poles and hence a contradiction.
\end{remark}
The next result shows that often it suffices to check the decay of $\abs{\Wp}$ along a single $\Cb_c, c>\dph,$ to ensure that it decays with the same speed along any complex line $\Cb_a, a>\dph$.
\begin{lemma}\label{lem:Lindelof}
	Let $\phi \in \Be$ and  assume that there exists $\ak>\dph$ such that for all $n\in \Nb$, $\limsupi{|b|}|b|^n\abs{\Wp\!\lbrb{\ak+ib}}=0$ (resp.~there exists $\theta \in\lbrbb{0,\frac{\pi}{2}}$ such that $\limsupi{|b|}\frac{\ln \abs{\Wp\!\lbrb{\ak+ib}}}{\abs{b}}\leq -\theta$) then $\phi \in \Bi$ (resp.~$\phi \in \Bth$).
\end{lemma}

The next theorem contains alternative representation of $\Wp$, which modulo to an easy extension to $\CbOI$ is due to \cite{Hirsch-Yor-13}, as well as a number of mappings that can be useful in a variety of contexts.
\begin{theorem}\label{thm:HY}
	Let $\phi,\underline{\phi} \in \Be$.
	\begin{enumerate}
		\item \label{it:hy} $z\mapsto \log W_{\phi}(z+1) \in \Ne$ with
		\begin{equation}\label{eq:WpMid}
		\log \Wp(z+1)=\lbrb{\ln \phi(1)}z + \int_0^{\infty} \lbrb{e^{-zy}-1-z(e^{-y}-1)}\frac{\kappa(dy)}{y(e^y-1)},
		\end{equation}
		where $\kappa(dy)=\int_{0}^{y}U\lbrb{dy-r}\left(r\mu(dr)+\delta_{\dr}(dr)\right)$, $U$ is the  potential measure associated to $\phi$, see Proposition \ref{propAsymp1}\eqref{it:bernstein_cmi}, and $\dr$ is the drift of $\phi$.
		\item \label{it:ts} $z\mapsto \log \left(W_{\phi}(z+1) W_{\underline{\phi}}(1-z)\right) \in \overNc$.
		\item \label{it:sm}If  $u\mapsto \left(\phi/\underline{\phi}\right)(u)$ is a non-zero completely monotone function then there exists  a positive variable $I$ which is moment determinate and such that, for all $n\geq0$, $\Ebb{I^n}=W_{\underline{\phi}}(n+1)/W_\phi(n+1)$. Next, assume that  $\phi \in \Bep{\Ntt_{\phi}}$ and $\underline{\phi} \in \Bep{\Ntt_{\underline{\phi}}}$.  If  $\Ntt=\Ntt_{\underline{\phi}}-\Ntt_{\phi}>\frac12$, then  the law of $I$ is absolutely continuous with density $f_I \in \LtwoRp$ and if $\Ntt>1$ then $f_I\in \mathrm{C}^{\lceil\Ntt\rceil-2}_0\lbrb{\Rp}$.
		\item \label{it:mid} $u\mapsto \left(\ln\left(\underline{\phi}/\phi\right)\right)'(u)=\left(\underline{\phi}'/\underline{\phi}-\phi'/\phi\right)(u)$ is completely monotone if and only if $z\mapsto \log \frac{W_{\phi}}{W_{\underline{\phi}}}(z+1) \in \Ne$. Note that since  $\left(\ln\left(\underline{\phi}/\phi\right)\right)'$ completely monotone implies the same property for $\phi/\underline{\phi}$, then, with the notation of the previous item, $I$ is well-defined and $\log I$ is infinitely divisible on $\R$. An equivalent condition is that with the measure $\kappa$ as defined in  item \eqref{it:hy} we have that $\kappa-\underline{\kappa}\geq 0$.
		
	\end{enumerate}
\end{theorem}
\textit{The proof of this result is in Section \ref{subsec:prof}.}
\begin{remark}\label{rem:H}
	Note that in the trivial case $\phi(z)=z$ then $\kappa(dy)=dy$  and the representation  \eqref{eq:WpMid} yields to the classical Malmst\'en formula for the gamma function, see \cite{Erdelyi-55}. With the recurrence equation \eqref{eq:Wp}, the Weierstrass product \eqref{eq:BernWeier} and the Mellin transform of a positive random variable, see Definition \ref{def:WB}, this integral  provides a fourth representation  $W_{\phi}$ share with the classical gamma function, justifying our choice to name them the Bernstein-gamma functions.
\end{remark}

\begin{remark}\label{rem:HY}
	We mention that when $\phi(0)=0$ the representation \eqref{eq:WpMid} for $z\in\lbrb{0,\infty}$ appears in 	\cite[Theorem 3.1]{Hirsch-Yor-13} and for any $\phi\in\Bc$ in \cite[Theorem 2.2]{Berg-07}.  We emphasize that to get detailed information regarding bounds and asymptotic behaviour of $|W_{\phi}(z)|$, see Theorem \ref{thm:Stirling}, we found the  Weierstrass product representation more informative to work with. However, as it is illustrated by the authors of \cite{Hirsch-Yor-13}, the integral representation  is useful for other purposes such as, for instance, for proving the multiplicative infinite divisibility property of some random variables.
\end{remark}

The final claim shows  that the mapping $\phi\in\Bc\mapsto W_{\phi}\in \Wc_\Bc$ is continuous with respect to the pointwise topology in $\Bc$. This handy result is widely used throughout.
\begin{lemma}\label{lem:continuityW}
	Let $(\phi_n)_{n\geq0},\,\underline\phi\in\Bc$  and $\limi{n}\phi_n(a)=\underline{\phi}(a)$ for all $a>0$. Then pointwise $\limi{n}W_{\phi_n}(z)=W_{\underline\phi}(z),\,z\in\CbOI$.
\end{lemma}
We proceed now with the proofs of the statements of the last two sections.

\section{Proofs for Sections \ref{sec:BernFunc} and \ref{sec:mainResults}}\label{sec:ProofsBE}
\subsection{Proof of Proposition \ref{propAsymp1}}
We furnish some comments for the claims that cannot be directly found in \cite{Schilling2010} and \cite[Section 4]{Patie-Savov-16}. First we look at the case $\phi\in\BP$ of item \eqref{it:bernstein_cmi}. If $\phi(0)=0$ then the fact that $U$ has a density $u \in \Ctt\!\lbrb{\lbbrb{0,\infty}}$ with $u(0)=\frac{1}{\dr}$ and $u>0$ on $\Rp$ is known from \cite[Chapter III]{Bertoin-96}. If $\phi(0)=q>0$, then from \cite[p.10]{Doering-Savov-11} we get that
\[\dr u(x)=\Pbb{\xi^\sharp_{T^\sharp_{\lbbrb{x,\infty}}}=x,\textbf{e}_q>T^\sharp_{\lbbrb{x,\infty}}},\]
where $\xi^\sharp$ is the conservative \LLP(subordinator) associated to $\phi^\sharp(\cdot)=\phi(\cdot)-\phi(0)$ and $T^\sharp_{\lbbrb{x,\infty}}=\inf_{t>0}\curly{\xi^\sharp_t\geq x}$. Hence, we get immediately that $u>0$ on $\Rp$ since otherwise
\[\Pbb{\xi^\sharp_{T^\sharp_{\lbbrb{x,\infty}}}=x,\textbf{e}_q>T^\sharp_{\lbbrb{x,\infty}}}=q\IntOI e^{-qt}\Pbb{\xi^\sharp_{T^\sharp_{\lbbrb{x,\infty}}}=x,{t>T^\sharp_{\lbbrb{x,\infty}}}}dt=0\]
for some $x>0$, would imply, for any $t\geq 0$,
\[\Pbb{\xi^\sharp_{T^\sharp_{\lbbrb{x,\infty}}}=x,{t>T^\sharp_{\lbbrb{x,\infty}}}}=0\]
which contradicts $\dr u^\sharp(x)=\Pbb{\xi^\sharp_{T^\sharp_{\lbbrb{x,\infty}}}=x}>0$ with $u^\sharp$ the density of $U^\sharp$. Also $u\lbrb{0}=u^\sharp\lbrb{0}=\frac{1}{\dr}$ follows from
\[\limo{x}\dr u(x)=\limo{x}\Pbb{\xi^\sharp_{T^\sharp_{\lbbrb{x,\infty}}}=x,\textbf{e}_q>T^\sharp_{\lbbrb{x,\infty}}}=\limo{x}\Pbb{\xi^\sharp_{T^\sharp_{\lbbrb{x,\infty}}}=x}=\dr u^\sharp\lbrb{0}=1.\]
The continuity of $u$ follows from its series representation in \cite[Proposition 1]{Doering-Savov-11}. Next, we consider item \eqref{it:argPhi}. 	Clearly, $\Re\lbrb{\phi(a+ib)}\geq \phi(a)>0$, see \eqref{eq:rephi}. Next, for any $a>0$,
\begin{equation}\label{eq:limIm}
\limi{b}\frac{\abs{\Im\lbrb{\phi(a+ib)}}}b=0,
\end{equation}
follows from  Proposition \ref{propAsymp1}\eqref{it:asyphid}. These facts allow us to deduct that, for any $v>0$ and any $u>u\lbrb{v}>0$,
\begin{equation*}
\begin{split}
\abs{\arg\lbrb{\phi(a+iu)}}&= \abs{\arctan\!\lbrb{\frac{\Im\lbrb{\phi(a+iu)}}{\Re\lbrb{\phi(a+iu)}}}}\\	
&\leq \arctan\!\lbrb{\frac{\abs{\Im\lbrb{\phi(a+iu)}}}{\phi(a)}}
=\frac{\pi}{2}-\arctan\!\lbrb{\frac{\phi(a)u}{u\abs{\Im\lbrb{\phi(a+iu)}}}}\\
&\leq \frac{\pi}{2}-\arctan\!\lbrb{\frac{v\phi(a)}{u}},
\end{split}	
\end{equation*}
where the last inequality follows from \eqref{eq:limIm}. This proves \eqref{eq:argPhi} and thus item \eqref{it:argPhi} is settled. The rest is in the existing literature.
\subsection{Proof of Theorem \ref{thm:genFuncs}}
We start with item \eqref{it:Aphi}. From \eqref{eq:rephi} of Proposition \ref{propAsymp1}, which is valid at least on $\CbOI$, we conclude that  \begin{equation}\label{eq:phiC+}
\phi:\,\CbOI\to \CbOI.
\end{equation}
Therefore, $\log\phi\in\AbOI$. We proceed to establish the first identity of \eqref{eq:A=Theta}.  By means of the integral expression in \eqref{eq:Tphi} and an application, in the third equality below, of the Cauchy integral theorem to $\log\phi$ on the closed rectangular contour with vertices  $a+ib, u+ib, u$ and   $a$, for any $z=a+ib\in\CbOI,\,b>0,\,u>a$, we deduct that
\begin{eqnarray*}\label{eq:argarg}
	\nonumber\int_{a}^{\infty}\ln\!\lbrb{ \frac{\labsrabs{\phi\lbrb{y+ib}}}{\phi(y)}}dy&=&\lim_{u\to\infty} \int_{a}^{u}\ln\!\lbrb{\frac{\labsrabs{\phi\lbrb{y+ib}}}{\phi(y)}}dy\\
	&=& \lim_{u\to\infty}\Re\!\lbrb{\int_{a}^{u}\log\frac{\phi\lbrb{y+ib}}{\phi(y)}dy}\\
	\nonumber	&=&\lim_{u\to\infty}\lbrb{\Re\!\lbrb{\int_{u\to u+ib}\log\phi(z)dz}-\Re\!\lbrb{\int_{a\to a+ib}\log\phi(z)dz}}\\
	&=&\int_{0}^{b}\arg\phi(a+iy)dy-\lim_{u\to\infty}\int_{0}^b \arg\phi(u+iy)dy\\
	&=&\int_{0}^{b}\arg\phi(a+iy)dy-\lim_{u\to\infty}\Aph\!\lbrb{u+ib}.
\end{eqnarray*}
We investigate the last limit. From \eqref{eq:ineqArg} of Proposition \ref{propAsymp1} we have that
\[\limi{a}\abs{\arg\phi\lbrb{a+iy}}\leq \limi{a}\arctan\!\lbrb{\frac{|y|}{a}}=0,\,\text{ uniformly on $y$-compact sets}.\]
Henceforth,
\begin{equation}\label{eq:Aphto0}
\limi{u}\Aph\lbrb{u+ib}=\lim_{u\to\infty}\int_{0}^b \arg\phi(u+iy)dy=0
\end{equation}
and we conclude the first identity of \eqref{eq:A=Theta} for $b>0$. For $b=0$ there is nothing to prove whereas if $b<0$ the same follows noting that $\arg\phi\lbrb{a+ib}=-\arg\phi(a-ib)$ implies that $\Aph\lbrb{\ab}=\Aph\lbrb{a-ib}$. Next, for $b>0$,  \cite[Proposition 6.10(1) and (6.27)]{Patie-Savov-16} give immediately that
\[\frac{1}{b}\int_{a}^{\infty}\ln\!\lbrb{ \frac{\labsrabs{\phi\lbrb{y+ib}}}{\phi(y)}}dy\in\lbbrbb{0,\frac{\pi}{2}}\]
and \eqref{eq:A=Theta} is thus proved in its entirety. To conclude item \eqref{it:Aphi} it remains to prove the two monotonicity properties. The monotonicity, for fixed $b\in\R$, of the mapping $a\mapsto\Aph\!\lbrb{a+ib}$ on $\R^+$ follows from \eqref{eq:Tphi} right away since for all $u>0,$ $\ln\!\lbrb{\frac{\abs{\phi\lbrb{u+ib}}}{\phi(u)}}\geq 0$, which follows from \eqref{eq:rephi} and has been also derived in \cite[Proposition 6.10, (6.32)]{Patie-Savov-16}. Next, recalling the notation $\phi(z)=\phi(0)+\phi^\sharp(z)$, we get that the mapping}
\[q \mapsto \ln\!\lbrb{\frac{\abs{q+\phi\lbrb{u+ib}}}{q+\phi(u)}}=\ln\abs{1+\frac{\phi^\sharp(u+ib)-\phi^\sharp(u)}{q+\phi(0)+\phi^\sharp(u)}}\]
is non-increasing on $\R^+$
and \eqref{eq:Tphi} closes the proof of item \eqref{it:Aphi}. Next, we consider item \eqref{it:RE}. First, the proof and claim of \cite[Proposition 6.10(2)]{Patie-Savov-16} show that, for any $a>0$ and any $\phi\in\Bc$,
$\sup_{z\in\Cb_{\lbrb{a,\infty}}}\abs{\Eph\lbrb{z}}\leq \frac{19}{8a}$,  from where  we get the first global bound in \eqref{eq:uniBoundRE1}. It follows from the immediate $\sup_{u>0}\abs{\mathrm{P}(u)}\leq\frac{1}{4}$, where $\mathrm{P}(u)=\lbrb{u-\lrfloor{u}}\lbrb{1-\lbrb{u-\lrfloor{u}}}$, that
\begin{equation}\label{eq:RhiBound}
\begin{split}
\sup_{\phi\in\Bc}\sup_{c>a}\abs{\Rph(c)}&\leq \frac{1}{8}\sup_{\phi\in\Bc}\sup_{c>a}\int_{1}^{\infty}\lbrb{\lbrb{\frac{\phi'(u)}{\phi(u)}}^2+\lbrb{\frac{\phi'(u+c)}{\phi(u+c)}}^2+\abs{\frac{\phi''(u)}{\phi(u)}}+\abs{\frac{\phi''(u+c)}{\phi(u+a)}}}du\\
&\leq \frac{3}{4}\int_{1}^{\infty}\frac{du}{u^2}=\frac{3}{4},
\end{split}
\end{equation}
where for the second inequality we have used \eqref{specialEstimates11} of Proposition \ref{propAsymp1}. This shows the second bound in \eqref{eq:uniBoundRE1} and concludes the proof of item \eqref{it:RE}. Next, consider item \eqref{it:H}. For any $\phi\in\Bc$, we deduce by integration by parts in the expression of $\Gph$ in \eqref{eq:Gphi},   that, for any $a>0$,
\begin{align} \label{eq:gpa}
\Gph(a)&=\lbrb{a+1}\ln\phi(a+1)-\ln\phi(1)-\Hph(a).
\end{align}	
Next,  note that, for large $a$,
\[a\ln\frac{\phi(a+1)}{\phi(a)}=a\ln\!\lbrb{1+\frac{\phi(a+1)-\phi(a)}{\phi(a)}}=\Hph^*(a)+a\: \bospace{\lbrb{\frac{\phi(a+1)-\phi(a)}{\phi(a)}}^2},\]
where the last asymptotic relation follows from $\lim\limits_{a\to \infty}\frac{\phi(a+1)}{\phi(a)}=1$, see \eqref{lemmaAsymp1-1}. This shows also the relation \eqref{eq:AsympHphi} for $\Hph^*$. However, from \eqref{specialEstimates11} we get that $a\frac{\phi'(a)}{\phi(a)}\leq1,\,a>0,$ and, since $\phi'$ is  non-increasing,  we conclude further that
\[a\ln\frac{\phi(1+a)}{\phi(a)}=\Hph^*(a)+\bospace{\frac1a}.\]
These considerations also establish \eqref{eq:AsympHphi} for $\Hph$ because
\[\frac{\Hph(a)}a=\frac{\int_{1}^{1+a}\frac{u\phi'(u)}{\phi(u)}du}a\leq \frac{1}{a}\int_{1}^{1+a}du=1.\]
Thus, putting the pieces together, we deduce from \eqref{eq:gpa} that
\begin{align}\label{eq:GphiAsymp}
\Gph(a)&= a\ln\phi(a)-\Hph(a)+\Hph^*(a)+\ln\phi(a)-\ln\phi(1)+\bospace{\frac1a}
\end{align}
which is \eqref{eq:Gph}. We continue with item \eqref{it:Tphi}. Note that, for $z=a+ib\in\CbOI$, some elementary algebra yields that
\begin{align*}
\Eph(z)+\Rph(a)&=\frac12\int_0^1\mathrm{P}(u)\lbrb{\frac{\ln\abs{\phi(u+z)}}{\ln\phi(u+a)}}''du +\frac12\int_1^\infty\mathrm{P}(u)\lbrb{\frac{\ln\abs{\phi(u+z)}}{\ln\phi(u)}}''du\\
:&=\overline{E}_{\phi}(z)+\overline{R}_{\phi}(z).
\end{align*}
Then, noting that $\lbrb{\ln\abs{\phi(u+z)}}''\leq \abs{\lbrb{\log\phi(u+z)}''}$, we easily obtain the estimate
\[\abs{\overline{E}_{\phi}(z)}\leq\int_{0}^{1}\lbrb{\abs{\frac{\phi'(u+z)}{\phi(u+z)}}^2+\lbrb{\frac{\phi'(u+a)}{\phi(u+a)}}^2+\abs{\frac{{\phi''(u+z)}}{\phi(u+z)}}+\abs{\frac{\phi''(u+a)}{\phi(u+a)}}}du.\]
Clearly, from the proof of \cite[Lemma 4.5(2)]{Patie-Savov-16} combined with \eqref{specialEstimates11}, we get that
\begin{equation}\label{eq:phi'phi}
\abs{\frac{\phi'(a+ib)}{\phi(a+ib)}}\leq \sqrt{10}\frac{\phi'(a)}{\phi(a)}\leq \frac{\sqrt{10}}a
\end{equation}
and
\begin{equation}\label{eq:phi''phi}
\,\quad\abs{\frac{\phi''(a+ib)}{\phi(a+ib)}}\leq \sqrt{10}\frac{\phi''(a)}{\phi(a)}\leq \frac{2\sqrt{10}}{a^2},
\end{equation}
from where we deduce that $\lim_{\,a\to\infty}\abs{\overline{E}_{\phi}(z)}=0$. Next, we look into $\overline{R}_{\phi}(z)$. Note that
\begin{equation*}
\begin{split}
\overline{R}_{\phi}(z)&=\frac12\int_1^\infty\mathrm{P}(u)\lbrb{\ln\abs{\phi(u+z)}}''du 	-\frac12\int_1^\infty\mathrm{P}(u)\lbrb{\ln\phi(u)}''du\\
&=\overline{R}_1(z)+\overline{R}_2.
\end{split}
\end{equation*}
Clearly, $\overline{R}_2=T_\phi$, see \eqref{eq:realInfty}, and we proceed to show that $\lim_{a\to\infty}\abs{\overline{R}_1(z)}=0$. To do so we simply repeat the work done for $\overline{E}_{\phi}(z)$ above to conclude that
\begin{equation*}
\abs{\overline{R}_1(z)}\leq \frac{1}{8}\int_{1}^{\infty}\lbrb{\abs{\frac{\phi'(u+z)}{\phi(u+z)}}^2+\abs{\frac{{\phi''(u+z)}}{\phi(u+z)}}}du.
\end{equation*}
Since each integrand converges to zero as $a\to\infty$, see \eqref{eq:phi'phi} and \eqref{eq:phi''phi}, we invoke again \eqref{eq:phi'phi} and \eqref{eq:phi''phi} to employ the dominated convergence theorem.	Therefore, for any fixed $b\in\R$,
\[\limi{a}\Eph(\ab)+\Rph(a)=\Tph\]
and item \eqref{it:Tphi} is established. This concludes the proof of the theorem.

\subsection{Proof of Theorem \ref{thm:Aph}\eqref{it:awpid}}
We start the proof by introducing some notation and quantities.  For a measure  (or a function) $\mu$ on $\R$, \[\Fo{\mu}(-ib)=\int_{-\infty}^{\infty}e^{-ib y}\mu(dy),\,b\in\Rb,\] 
stands for its Fourier transform. Also, we use $\lambda*\gamma$ (resp.~$\lambda^{*n}, n\in \N$) to denote the convolution (resp.~the $n^{th}$ convolution) of measures and/or functions. Next, for any measure $\lambda$ on $\R$ such that $||\lambda||_{TV}:=\IntII \abs{\lambda(dy)}<1$ we define
\begin{equation}\label{eq:convo}
\Logp{\lambda}(dy)=\sum_{n=1}^{\infty}\frac{\lambda^{*n}(dy)}{n}  \textrm{\,\,and\,\,\,}  \Logm{\lambda} (dy)=\sum_{n=1}^{\infty}\lbrb{-1}^{n-1}\frac{\lambda^{*n}(dy)}{n}.
\end{equation}
Clearly, $||\Logp{\lambda}||_{TV}<\infty$ and $||\Logm{\lambda}||_{TV}<\infty$. Finally, we use for a measure (resp.~function) $\lambda_a(dy)=e^{-ay}\lambda(dy)$ (resp.~$\lambda_a(y)=e^{-ay}\lambda(y)$) with $a\in\R$. Let from now on $\lambda$ be a measure on $\Rp$. If $||\lambda_a||_{TV}<1$ for some $a>0$, then by virtue of the fact that $\lambda^{*n}_a(dy)=e^{-ay}\lambda^{*n}(dy),\,y\in\intervalOI$, \eqref{eq:convo} holds vaguely on $\lbrb{0,\infty}$ for $\lambda_0=\lambda$. Let next $\lambda(dy)=\lambda(y)dy,\,y\in\intervalOI$, and $\lambda\in\Lspace{1}{\Rp}$, i.e.~the $C^*$ algebra of the integrable  functions on $\Rp$, that is $\Lspace{1}{\Rp}$ considered as a subalgebra of $\Lspace{1}{\R}$ which is endowed with the convolution operation as a multiplication. Note that formally
\begin{equation}\label{eq:formalF}
-\log_0\lbrb{1- \Fo{\lambda}}=\Fo{\Logp{\lambda}}=\sum_{n=1}^{\infty}\frac{\Fo{\lambda}^n}n \quad \textrm{ and } \quad\log_0\lbrb{1+ \Fo{\lambda}}=\Fo{\Logm{\lambda}}=\sum_{n=1}^{\infty}\lbrb{- 1}^{n-1}\frac{\Fo{\lambda}^n}n.
\end{equation}
We then have the following claim which is a simple restatement of several results from the existing literature. A good references on the topic are for example \cite{Nikolski-98, Paley-Wiener-34}.
\begin{proposition}\label{prop:harmConvo}
	Let $\lambda$ be a real-valued measure on $\intervalOI$. If $||\lambda||_{TV}<1$ then both $||\Logp{\lambda}||_{TV}<\infty$ and $||\Logm{\lambda}||_{TV}<\infty$ and \eqref{eq:logFT}, \eqref{eq:logFT1} hold true. Furthermore, let $\lambda\in\Lspaces{1}{\Rp}$. If \,$\log_0(1+\Fo{\lambda}(-z))\in\Ac_{\lbbrb{0,\infty}}$ (resp.~$\log_0(1-\Fo{\lambda}(-z))\in \Ac_{\lbbrb{0,\infty}}$ then $\Logm{\lambda}\in\Lspaces{1}{\Rp}$ (resp.~$\Logp{\lambda}\in\Lspaces{1}{\Rp}$) and
	\begin{align}
	\log_0(1+\Fo{\lambda}(-ib))&=\Fo{\Logm{\lambda}}(-ib)\text{ if $\Logm{\lambda}\in\Lspaces{1}{\Rp}$}\label{eq:logFT}\\
	-\log_0(1-\Fo{\lambda}(-ib))&=\Fo{\Logp{\lambda}}(-ib)\text{ if $\Logp{\lambda}\in\Lspaces{1}{\Rp}$}.\label{eq:logFT1}
	\end{align}
\end{proposition}	
\begin{proof}
	Let $||\lambda||_{TV}<1$. Since, for all $n\in\N$, $||\lambda^{*n}||_{TV}\leq ||\lambda||^n_{TV}$ and $\sup\limits_{b\in\Rb}\abs{\mathcal{F}^{n}_{\lambda}(ib)}\leq ||\lambda||^n_{TV}$, then  \eqref{eq:convo}, \eqref{eq:formalF} and the Taylor expansion about zero of $\log_0(1\pm z)$ yield all results in the case $||\lambda||_{TV}<1$. Next, for general $\lambda\in \Lspace{1}{\Rp}$ assuming that \,$\log_0(1+\Fo{\lambda}(-z))\in\Ac_{\lbbrb{0,\infty}}$ (resp.~$\log_0(1-\Fo{\lambda}(-z))\in\Ac_{\lbbrb{0,\infty}}$) we can apply \cite[Theorem 2]{Shea-75}, which refers to \cite{Paley-Wiener-34},  to the  semisimple Banach algebra $\Lspace{1}{\R}$ and thus conclude the claims contained in relations \eqref{eq:logFT} and \eqref{eq:logFT1}.
\end{proof}	
Let us assume that $\dr>0$. Then, for each $z=a+ib, a>0,$ we have from the second expression in \eqref{eq:phi} that
\begin{equation}\label{eq:phiP}
\phi(z)=\dr z\lbrb{1+\frac{\phi(0)}{\dr z}+\mathcal{F}_{\bar{\mu}_{a,\dr}}(-ib)},
\end{equation}
where we have set $\bar{\mu}_{a,\dr}(y)=\frac{1}{\dr}e^{-ay}\bar{\mu}(y),\,y\in\intervalOI$.  With the preceding notation we can state the  next result which with the help of Lemma \ref{lem:Lindelof}
concludes the proof of Theorem \ref{thm:Aph}\eqref{it:awpid}.

\begin{proposition}\label{thm:imagineryStirling}
	Assume that $\phi\in\BP$ and $z=\ab\in\Cb_a$, with $a>0$ fixed. We have
	\begin{equation}\label{eq:AphiAsymp2}
	\begin{split}	 			\Aph\lbrb{z}&=|b|\arctan\lbrb{\frac{|b|}a}-\frac{\lbrb{a+\frac{\phi(0)}{\dr}}}2\ln\lbrb{1+\frac{b^2}{a^2}}-\int_{0}^{\infty}\frac{1-\cos(by)}{y}\Logm{\bar{\mu}_{a,\dr}}(dy)+\bar{A}_\phi(a,b)\\
	&\simi\frac{\pi}{2}|b|+\sospace{|b|},
	\end{split}
	\end{equation}
	where $\Logm{\bar{\mu}_{a,\dr}}(dy)$ is related to $ \bar{\mu}_{a,\dr}(y)dy=\dr^{-1}e^{-ay}\bar{\mu}(y)dy$ via \eqref{eq:logFT} and for any $a>0$ fixed \[\abs{\bar{A}_\phi(a,b)}=\sospace{\ln|b|}.\] For all $a>0$ big enough, $||\bar{\mu}_{a,\dr}||_{1}<1$ and \eqref{eq:convo} relates $\Logm{\bar{\mu}_{a,\dr}}(dy)$ to $\bar{\mu}_{a,\dr}\lbrb{y}dy,\,y\in\intervalOI$. Also, for those $a>0$ such that $||\bar{\mu}_{a,\dr}||_{1}<1$, we have that
	\begin{eqnarray}\label{eq:arctan}
	\int_{0}^{\infty}\frac{1-\cos(by)}{y}\Logm{\bar{\mu}_{a,\dr}}(dy)&=&\int_{0}^{|b|}\arctan\!\lbrb{\frac{\Im\!\lbrb{\Fo{\bar{\mu}_{a,\dr}}\!\lbrb{iu}}}{1+\Re\!\lbrb{\Fo{\bar{\mu}_{a,\dr}}\!\lbrb{iu}}}}du
	\end{eqnarray}
	with, as $b\to\infty$,
	\begin{equation}\label{eq:arctan1}
	\arctan\!\lbrb{\frac{\Im\!\lbrb{\Fo{\bar{\mu}_{a,\dr}}\!\lbrb{ib}}}{1+\Re\!\lbrb{\Fo{\bar{\mu}_{a,\dr}}\!\lbrb{ib}}}}=\Im\!\lbrb{\Fo{\bar{\mu}_{a,\dr}}\lbrb{ib}}\lbrb{1+\bospace{\lbrb{\Im\!\lbrb{\Fo{\bar{\mu}_{a,\dr}}\!\lbrb{ib}}}^2}}
	\end{equation}
	and
	\begin{eqnarray}\label{eq:arctan2}
	\liminfi{b}\frac{\dr}{\ln (b)\bar{\mu}\lbrb{\frac{1}{b}}}\int_{0}^{\infty}\frac{1-\cos(by)}{y}\Logm{\bar{\mu}_{a,\dr}}(dy)\geq 1.
	\end{eqnarray}
\end{proposition}

\begin{remark}\label{rem:imagineryStirling}
	We note that \eqref{eq:arctan} coupled with \eqref{eq:arctan1} give a more tractable way to compute $\Aph$ in \eqref{eq:AphiAsymp2}. When $||\bar{\mu}_{a,\dr}||_{1}<1$ thanks to \eqref{eq:convo} we have that
	\[\int_{0}^{\infty}\frac{1-\cos(by)}{y}\Logm{\bar{\mu}_{a,\dr}}(dy)=\int_{0}^{\infty}\frac{1-\cos(by)}{y}\sum_{n=1}^{\infty}\lbrb{-1}^{n-1}\frac{\bar{\mu}_{a,\dr}^{*n}(dy)}{n}\]
	and thus via \eqref{eq:AphiAsymp2} we have precise information for the asymptotic expansion of $\Aph$ when the convolutions of $\bar{\mu}_{a,\dr}(y)=\dr^{-1}e^{-ay}\mubar{y}$ or equivalently of $\mubr$ are accessible. For example, if $\bar{\mu}_{a,\dr}(y)\simo \dr^{-1} y^{-\alpha},\,\alpha\in\lbrb{0,1}$, then with $B(a,b),\,a,b>0$, standing for the classical Beta function, with the relations \[C_n=\dr^{-n}\prod_{j=1}^{n-1}B\lbrb{j-j\alpha,1-\alpha}, \quad\bar{\mu}^{*n}_{a,\dr}(y)\simo C_ny^{n-1-n\alpha}\] and, with the obvious notation for asymptotic behaviour of densities of measures, we arrive at the following asymptotc expansion
	\[\Logm{\bar{\mu}_{a,\dr}}(dy)\simo\lbrb{ \sum_{\frac{1}{1-\alpha}\geq n\geq 1}(-1)^{n-1}C_ny^{n-1-n\alpha}+\bospace{y^{n^*-1-n^*\alpha}}}dy,\]
	where $n^*$ is the minimum integer strictly larger then $\frac{1}{1-\alpha}.$
	A substitution in \eqref{eq:AphiAsymp} and elementary calculations yield that as $b\to \infty$
	\[\int_{0}^{\infty}\frac{1-\cos(by)}{y}\Logm{\bar{\mu}_{a,\dr}}(dy)\simi \sum_{\frac{1}{1-\alpha}> n\geq 1}(-1)^{n-1}\tilde{C}_nb^{1-n+n\alpha}+\ind{\frac{1}{1-\alpha}\in\Nb}\bospace{\ln b}+\bospace{1}\]
	with $\tilde{C}_n=C_n\IntOI \frac{1-\cos(v)}{v^{2-n(1-\alpha)}}dv,\,n\in\lbbrb{1,\frac{1}{1-\alpha}}$, and the asymptotic expansion of $\Aph$ follows.
\end{remark}
\begin{proof}[Proof of Proposition \ref{thm:imagineryStirling}]
	Since $\Aph\!\lbrb{a+ib}=\Aph\!\lbrb{a-ib}$ without loss of generality we assume throughout that $b>0$. From \eqref{eq:ineqArg} of Proposition \ref{propAsymp1} we have that $\arg\phi(z)-\arg z \in\lbrb{-\pi,\pi}$ and then from \eqref{eq:phiP} we have that
	\begin{equation}\label{eq:BP1}
	\arg \phi\lbrb{z}=\arg z+\arg\lbrb{1+\frac{\phi(0)}{\dr z}+\mathcal{F}_{\bar{\mu}_{a,\dr}}(-ib)}.
	\end{equation}
	However,  an application of the Riemann-Lebesgue lemma to the function $\bar{\mu}_{a,\dr}\in\Lspaces{1}{\Rp}$ yields, as $b\to\infty$,
	that	
	\begin{equation}\label{eq:toarctan1}
	\abs{\Fo{\bar{\mu}_{a,\dr}}(-ib)}=\sospace{1}.
	\end{equation}
	Therefore, for all $b$ big enough, the Taylor expansion of the $\arg$ function leads to
	\begin{equation}\label{eq:BP2}
	\begin{split}
	\arg\lbrb{1+\frac{\phi(0)}{\dr\lbrb{ \ab}}+\mathcal{F}_{\bar{\mu}_{a,\dr}}(-ib)}&=\arg\lbrb{1+\mathcal{F}_{\bar{\mu}_{a,\dr}}(-ib)}+\Im\!\lbrb{\frac{\phi(0)}{\dr\lbrb{ \ab}}}\lbrb{1+\sospace{1}}\\
	&=\arg\lbrb{1+\mathcal{F}_{\bar{\mu}_{a,\dr}}(-ib)}-\frac{\phi(0)b}{\dr \lbrb{a^2+b^2}}+\sospace{\frac{1}{b}}.
	\end{split}
	\end{equation}
	Also for any $a'\geq 0$ and $b>0$
	\begin{equation}\label{eq:toarctan2}
	\begin{split}
	\Im\lbrb{\Fo{\bar{\mu}_{a,\dr}}\lbrb{-a'-ib}}&=-\IntOI \sin\lbrb{by}e^{-a'y}\bar{\mu}_{a,\dr}(y)dy\\
	&=-\frac{1}{ b}\IntOI \lbrb{1-\cos(by)}\mu_{a+a',\dr}(dy)<0,
	\end{split}
	\end{equation}
	since for $a+a'>0$, 
	\[y\mapsto \bar{\mu}_{a+a',\dr}(y)=e^{-a'y}\bar{\mu}_{a,\dr}(y)=\dr^{-1}e^{-\lbrb{a+a'}y}\bar{\mu}(y)\]
	is strictly decreasing on $\Rp$ and hence $\mu_{a+a',\dr}(dy)=d\lbrb{e^{-a'y}\bar{\mu}_{a,\dr}(y)},\,y\in\intervalOI$, is not supported on a lattice. Also $\Im\!\lbrb{\Fo{\bar{\mu}_{a,\dr}}\lbrb{-a'-ib}}>0$ for $b<0$ and $\Fo{\bar{\mu}_{a,\dr}}(-a')\geq 0$.
	Therefore, from \eqref{eq:toarctan1}, \eqref{eq:toarctan2} and the fact that $b\mapsto \mathcal{F}_{\bar{\mu}_{a,\dr}}(-ib)$ is continuous we deduct that  $\Fo{\bar{\mu}_{a,\dr}}\lbrb{i\R}\,\cap\lbrbb{-\infty,- 1}=\emptyset$ and $\log_0\!\lbrb{1+\Fo{\bar{\mu}_{a,\dr}}\lbrb{-z}}\in \Ac_{\lbbrb{0,\infty}}$. 
	Henceforth, Proposition \ref{prop:harmConvo} gives that  $\Logm{\bar{\mu}_{a,\dr}}\in\Lspaces{1}{\Rp}$. From \eqref{eq:logFT}, for all $b>0$,
	\begin{equation}\label{eq:toarctan}
	\begin{split}
	\arg\lbrb{1+\mathcal{F}_{\bar{\mu}_{a,\dr}}(-ib)}&=\Im\!\lbrb{\log_0\!\lbrb{1+\Fo{\bar{\mu}_{a,\dr}}(-ib)}}\\
	&=\Im\!\lbrb{\mathcal{F}_{\Logm{\bar{\mu}_{a,\dr}}}(-ib)}=-\IntOI \sin(by)\Logm{\bar{\mu}_{a,\dr}}(dy).
	\end{split}	  	
	\end{equation}
	Then, from \eqref{eq:Aphi} due to \eqref{eq:BP1},\eqref{eq:BP2} and \eqref{eq:toarctan} we deduce that,  for any $a,b>0$,
	\begin{equation*}
	\begin{split}
	\Aph(a+ib)&=\int_{0}^{b}\arg\lbrb{a+iu}du+ \int_{0}^{b}\arg\lbrb{1+\frac{\phi(0)}{\dr\lbrb{ a+iu}}+\mathcal{F}_{\bar{\mu}_{a,\dr}}(-iu)}du\\
	&=\int_{0}^{b}\arctan\lbrb{\frac{u}{a}}du-\frac{\phi(0)}{2\dr}\ln\lbrb{1+\frac{b^2}{a^2}}-\IntOI \frac{1-\cos(by)}{y}\Logm{\bar{\mu}_{a,\dr}}(dy)+\bar{A}_\phi(a,b),
	\end{split}	
	\end{equation*}
	and $\bar{A}_\phi(a,b)=\sospace{\ln b}$,  as $b\to\infty$.
	Then the first relation in \eqref{eq:AphiAsymp2} follows by a simple integration by parts of $\int_{0}^{b}\arctan\lbrb{\frac{u}{a}}du$. The asymptotic relation in \eqref{eq:AphiAsymp2} comes from the fact that $\Logm{\bar{\mu}_{a,\dr}}\in\Lspaces{1}{\Rp}$ and the auxiliary claim that for any $h\in\Lspaces{1}{\Rp}$
	\begin{equation}\label{eq:o(b)}
	\abs{\IntOI \frac{1-\cos\lbrb{by}}{y}h(y)dy}=\sospace{|b|},
	\end{equation}
	which follows from the Riemann-Lebesgue lemma invoked in the middle term of
	\[	\abs{\IntOI \frac{1-\cos\lbrb{by}}{y}h(y)dy}=\abs{\int_{0}^{b}\IntOI\sin\lbrb{uy}h(y)dydu}\leq \sospace{1}\abs{\int_{0}^{b}du}.\]
	Finally, since $\limi{a}\IntOI e^{-ay}\bar{\mu}(y)dy=0$ then  $\limi{a}||\bar{\mu}_{a,\dr}||_{1}=0$ and thus eventually, for some $a$ large enough, $||\bar{\mu}_{a,\dr}||_{1}<1$ and $\sup_{b\in\R}\abs{\Fo{\bar{\mu}_{a,\dr}}(ib)}<1$.  Choose such $a>0$. Then, from \eqref{eq:toarctan1} with $z\in\Cb_a,\,b>0,$	
	\begin{equation}\label{eq:argArc}
	\begin{split}
	\arg\lbrb{1+\Fo{\bar{\mu}_{a,\dr}}\!\lbrb{-ib}}&=\arctan\!\lbrb{\frac{\Im\!\lbrb{\Fo{\bar{\mu}_{a,\dr}}\!\lbrb{-ib}}}{1+\Re\!\lbrb{\Fo{\bar{\mu}_{a,\dr}}\!\lbrb{-ib}}}}\\
	&=-\arctan\!\lbrb{\frac{\Im\!\lbrb{\Fo{\bar{\mu}_{a,\dr}}\!\lbrb{ib}}}{1+\Re\!\lbrb{\Fo{\bar{\mu}_{a,\dr}}\!\lbrb{ib}}}},
	\end{split}
	\end{equation}
	and using the latter in \eqref{eq:toarctan} then \eqref{eq:arctan} follows upon simple integration of \eqref{eq:toarctan}.
	The asymptotic relation \eqref{eq:arctan1} follows from \eqref{eq:toarctan1}, i.e.~$\abs{\Fo{\bar{\mu}_{a,\dr}}(-ib)}=\sospace{1}$, combined with the Taylor expansion of $\arctan x$. Next, consider the bound \eqref{eq:arctan2}. Since from \eqref{eq:toarctan2} with $a'=0$, we have that $\Im\!\lbrb{\Fo{\bar{\mu}_{a,\dr}}\!\lbrb{ib}}> 0$ and, as $b\to\infty$, $\abs{\Im\!\lbrb{\Fo{\bar{\mu}_{a,\dr}}\!\lbrb{ib}}}=\sospace{1}$ we get from \eqref{eq:arctan}, \eqref{eq:arctan1},  \eqref{eq:argArc} and \eqref{eq:toarctan} that
	\begin{equation*}
	\begin{split}
	\int_{0}^{\infty}\frac{1-\cos(by)}{y}\Logm{\bar{\mu}_{a,\dr}}(dy)&\,\,=\int_{0}^{b}\arctan\!\lbrb{\frac{\Im\!\lbrb{\Fo{\bar{\mu}_{a,\dr}}\!\lbrb{iu}}}{1+\Re\!\lbrb{\Fo{\bar{\mu}_{a,\dr}}\!\lbrb{iu}}}}du\\
	&\,\,\simi \lbrb{1+\sospace{1}}\int_{0}^{b}\Im\!\lbrb{\Fo{\bar{\mu}_{a,\dr}}\!\lbrb{iu}}du\\
	&=\lbrb{1+\sospace{1}}\int_{0}^{\infty}\frac{1-\cos(by)}{y}\bar{\mu}_{a,\dr}(y)dy\\
	&\,\,\geq\lbrb{1+\sospace{1}} \int_{0}^{1}\frac{1-\cos(by)}{y}\bar{\mu}_{a,\dr}(y)dy\\
	&\,\,\geq\lbrb{1+\sospace{1}} \bar{\mu}_{a,\dr}\!\lbrb{\frac{1}{b}}\int_{1}^{b}\frac{1-\cos(y)}{y}dy\simi \bar{\mu}_{a,\dr}\!\lbrb{\frac{1}{b}}\ln(b),
	\end{split}
	\end{equation*}
	where in the third line we have used \eqref{eq:toarctan2}.
	This proves \eqref{eq:arctan2} since \[\bar{\mu}_{a,\dr}\!\lbrb{\frac{1}{b}}=e^{-\frac{a}b}\frac{1}{\dr}\bar{\mu}\!\lbrb{\frac{1}{b}}\simi\frac{1}{\dr}\bar{\mu}\!\lbrb{\frac{1}{b}}.\]
\end{proof}

\subsection{Proof of Theorem \ref{thm:Aph}\eqref{it:awpia}}
Let $\phi\in\Bc_{\alpha}$ with $\alpha \in (0,1)$  and let $z=a+ib\in\Cb_a,\,a>0$. Then, $\bar{\mu}(y)=y^{-\alpha}\ell(y)$ and $\ell$ is quasi-monotone, see \eqref{eq:quasi-monotone} and \eqref{eq:classesPhi1}. Recall the second relation of \eqref{eq:phi} which, since in this setting $\dr=0$, takes the form
\begin{equation}\label{eq:phi1}
\begin{split}
\phi(z)&=\phi(0)+z\IntOI e^{-iby}e^{-ay}\bar{\mu}(y)dy\\
&=\phi(0)+z\IntOI e^{-iby}y^{-\alpha}e^{-ay}\ell(y)dy.
\end{split}
\end{equation}
Since the mapping $y\mapsto\ell(y)e^{-ay}$ is clearly quasi-monotone we conclude from \cite[Theorem 1.39]{Soulier-09} that, for fixed $a>0$ and $b\to\infty$,
\begin{equation}
\begin{split}
\IntOI e^{-iby}y^{-\alpha}e^{-ay}\ell(y)dy&\simi \Gamma\lbrb{1-\alpha}\left(be^{\frac{i\pi}{2}}\right)^{\alpha-1}e^{-\frac{a}{b}}\ell\!\lbrb{\frac{1}{b}}\\
&\simi \Gamma\lbrb{1-\alpha}\left(be^{\frac{i\pi}{2}}\right)^{\alpha-1}\ell\!\lbrb{\frac{1}{b}}.
\end{split}
\end{equation}
Therefore, from \eqref{eq:phi1} and the last relation we obtain, as $b\to\infty$, that
\begin{equation*}
\begin{split}
\arg \phi(z)&=\arg z+\arg\lbrb{\IntOI e^{-iby}y^{-\alpha}e^{-ay}\ell(y)dy+\frac{\phi(0)}{z}}\\
&\simi \arg z+\frac{\pi\lbrb{\alpha-1}}{2}\simi \frac{\pi}{2}\alpha,
\end{split}
\end{equation*}
which proves \eqref{eq:AphiAsymp1} by using the definition of $\Aph$ in \eqref{eq:Aphi}. This together with Lemma \ref{lem:Lindelof} establishes the claim.

\subsection{Proof of Theorem \ref{thm:Wp}}\label{sec:proof_wb}
We start with item \eqref{it:A}. The fact that $\Wp\in\Ac_{\lbrb{\dph,\infty}}$ is a consequence of \cite[Theorem 6.1]{Patie-Savov-16} and is essentially due to the recurrence equation \eqref{eq:Wp} and the fact that $\phi$ is zero-free on $\Cb_{\lbrb{\dph,\infty}}$. Also, $\Wp\in\Mtt_{\lbrb{\aph,\infty}}$ comes from the observation that $0\not\equiv \phi\in\Ac_{\lbrb{\aph,\infty}}$, that is $\phi$ can only have zeros of finite order, and \eqref{eq:Wp} which allows a recurrent meromorphic extension to $\Cb_{\lbrb{\aph,\infty}}$.  $\Wp$ is zero-free on $\Cb_{\lbrb{\dph,\infty}}$  follows from \cite[Theorem 6.1 and Corollary 7.8]{Patie-Savov-16}. The fact that $\Wp$ is zero-free on $\Cb_{\lbrb{\aph,\infty}}$ is thanks to $\phi\in\Ac_{\lbrb{\aph,\infty}}$ and
\begin{equation}\label{eq:Wp11}
\Wp(z)=\frac{\Wp(z+1)}{\phi(z)}
\end{equation}
which comes from \eqref{eq:Wp}. If $\phi(0)>0$ then $\Zcph=\emptyset$ and hence the facts that $\Wp$ is zero-free on $\Cb_{\lbbrb{0,\infty}}$  and $\Wp\in\Ac_{\lbbrb{0,\infty}}$ are immediate from $\Wp$ being zero-free on $\intervalOI$ and \eqref{eq:Wp11}. However, when $\phi(0)=0$ relation \eqref{eq:Wp11} ensures that $\Wp$ extends continuously to $i\R\setminus\Zc_{0}\!\lbrb{\phi}$ and clearly if $\mathfrak{z}\in\Zc_0\!\lbrb{\phi}$ then
\[\lim\limits_{\Re(z)\geq 0, z \to \mathfrak{z}} \phi(z)\Wp\lbrb{z}=\Wp\lbrb{\mathfrak{z}+1}.\] Finally, let us assume that $\phi'\!\lbrb{0^+}=\dr+\IntOI y\mu(dy)<\infty$ and $\lbcurlyrbcurly{0}\in\Zc_0\!\lbrb{\phi}$, that is $\phi(0)=0$.
From the assumption $\phi'\!\lbrb{0^+}<\infty$ and the dominated convergence theorem, we get that $\phi'$ extends to $i\R$, see \eqref{eq:phi'}. Therefore, from \eqref{eq:Wp} and the assumption $\phi(0)=0$ we get, for any $z\in\Cb_{\intervalOI}$, that
\begin{equation*}
\begin{split}
\Wp\!\lbrb{z+1}&=\phi\!\lbrb{z}\Wp\!\lbrb{z}=\lbrb{\phi\!\lbrb{z}-\phi(0)}\Wp\!\lbrb{z}\\
&=\lbrb{\phi'(0^+)z+\sospace{|z|}}\Wp\!\lbrb{z}.
\end{split}
\end{equation*}
Thus $ z\mapsto z\Wp(z) $ extends continuously to $i\R\setminus\lbrb{\Zc_0\!\lbrb{\phi}\setminus\curly{0}}$ provided $\phi'(0^+)=\dr+\IntOI y\mu(dy)>0$, which is apparently true.
If in addition $\phi''\!(0^+)<\infty$ then
\begin{equation*}
\begin{split}
\Wp\!\lbrb{z+1}&=\phi\!\lbrb{z}\Wp\!\lbrb{z}=\lbrb{\phi\!\lbrb{z}-\phi(0)}\Wp\!\lbrb{z}\\
&=\lbrb{\phi'\!(0^+)z+\phi''\!(0^+)z^2+\sospace{|z|^2}}\Wp\!\lbrb{z}.
\end{split}
\end{equation*}
Clearly, then the mapping  $z \mapsto \Wp\!\lbrb{z}-\frac{1}{\phi'\lbrb{0^+}z}$ extends continuously to $i\R\setminus\lbrb{\Zc_0\lbrb{\phi}\setminus\curly{0}}$. Let us deal with item \eqref{it:D}. We note that
\[e^{-\phi\lbrb{z}}=\Ebb{e^{-z\xi_1}},\,\,z\in\Cb_{\lbbrb{0,\infty}},\]
where $\xi=\lbrb{\xi_t}_{t\geq 0}$ is a non-decreasing \LLP (subordinator) as $-\phi(-z)=\Psi(z)\in\overNc$, see \eqref{eq:lk0}. Thus, if $\phi(z_0)=0$ then $\Ebb{e^{-z_0\xi_1}}=1$.  If in addition, $z_0\in \Cb_{\lbbrb{0,\infty}}\setminus\curly{0}$ then $\phi(0)=0$ and $z_0\in i\R$. Next, $\phi(z_0)=0$ also triggers that $\xi $ lives on a lattice of size, say $h>0$, which immediately gives that $\dr=0$ and $\mu=\sum_{n=1}^{\infty}c_n\delta_{x_n}$ with $\sum_{n=1}^{\infty}c_n<\infty$ and $\forall n\in\N$ we have that $x_n=hk_n,\,k_n\in\Nb$, $c_n\geq0$. Finally, $h$ can be chosen to be the largest such that $\xi$ lives on $\lbrb{h n}_{n\in\N}$. Thus,
\[\phi\lbrb{z}=\sum_{n=1}^{\infty}c_n\lbrb{1-e^{-z h k_n}}\]
and we conclude that $\phi$ is periodic with period $\frac{2\pi i}{h}$ on $\CbOI$. Next, note that
\begin{equation}\label{eq:phiInf}	\phi(\infty)=\lim\ttinf{u}\phi(u)=\lim\ttinf{u}\sum_{n=1}^{\infty}c_n\lbrb{1-e^{-uhk_n}}=\sum_{n=1}^{\infty}c_n<\infty.
\end{equation}
Then \eqref{eq:EulerConst} implies that
\[\lim\ttinf{n}\sum_{k=1}^n\frac{\phi'(k)}{\phi(k)}=\sum_{k=1}^{\infty}\frac{\phi'(k)}{\phi(k)}=\gamph+\ln\phi(\infty).\]
Thus, from \eqref{eq:BernWeier} we get that
\[\Wp(z)=\frac{e^{-\gamma_\phi z}}{\phi(z)}\prod_{k=1}^{\infty}\frac{\phi(k)}{\phi\!\lbrb{k+z}}e^{\frac{\phi'\!(k)}{\phi(k)}z}=\frac{e^{z\ln \phi\lbrb{\infty}}}{\phi(z)}\prod_{k=1}^{\infty}\frac{\phi(k)}{\phi\!\lbrb{k+z}}.\]
Hence, the claim for the $\frac{2\pi i}{ h}$ periodicity of the mappings
\[z \mapsto e^{-z\ln \phi\lbrb{\infty}}\Wp(z)\text{ and } z\mapsto \abs{\Wp(z)}\]
  follows immediately from the periodicity of $\phi$. Thus, item \eqref{it:D} is proved. Item \eqref{it:E} follows in the same manner as item \eqref{it:C} noting that when $\tph<0$ then $\dph=\tph$, see \eqref{eq:thetaphi1} and \eqref{eq:dphi1}. The last item \eqref{it:F} is an immediate result from \eqref{eq:Wp} and the fact that $\phi'>0$ on $\lbrb{\aph,\infty}$, see \eqref{eq:phi'}, that is $\tph$ is the unique zero of order one of $\phi$ on $\lbrb{\aph,\infty}$. Item \eqref{it:G} follows by simply changing $\phi\mapsto c\phi$ in \eqref{eq:BernWeier} and \eqref{eq:EulerConst}, and elementary algebra.  This ends the proof of Theorem \ref{thm:Wp}.

\subsection{Proof of Theorem \ref{thm:Stirling}\eqref{it:gbWp}}\label{subsec:proof}
Next, we know from the proof of \cite[Proposition 6.10]{Patie-Savov-16} and see in particular the expressions obtained for the quantities in \cite[(6.33) and (6.34)]{Patie-Savov-16}, that, for any $z=a+ib\in\CbOI,\,b>0$,
\begin{equation}\label{eq:absWp}
\begin{split}
\abs{\Wp(z)}&=\Wp\!\lbrb{a}\frac{\phi(a)}{\abs{\phi(z)}}\sqrt{\abs{\frac{\phi(z)}{\phi(a)}}}e^{-\int_{0}^{\infty}\ln\abs{\frac{\phi(u+a+ib)}{\phi(a+u)}}du}e^{-\Eph(z)}\\
&=\Wp\!\lbrb{a}\sqrt{\abs{\frac{\phi(a)}{\phi(z)}}}e^{-\int_{a}^{\infty}\ln\abs{\frac{\phi(u+ib)}{\phi(u)}}du}e^{-\Eph(z)}.
\end{split}
\end{equation}
We note that the term $-\Eph(z)$ is the limit  in $n$ of the error terms $E^B_\phi(n,a)-E^B_\phi(n,a+ib)$ in the notation of the proof of \cite[Proposition 6.10]{Patie-Savov-16} and is formulated in \eqref{eq:Ephi}. Thus, the last three terms of the second expression in \eqref{eq:absWp} are in fact the quantity $\frac{\phi(a)}{\abs{\phi(z)}}Z_\phi(z)$ in the notation of \cite[Proposition 6.10, (6.33)]{Patie-Savov-16}.
Thanks to Theorem \ref{thm:genFuncs}\eqref{it:Aphi}, for $z=a+ib\in\CbOI, b>0$, we deduce substituting \eqref{eq:A=Theta} in \eqref{eq:absWp} that
\begin{equation}\label{eq:absWp1}
\abs{\Wp(z)}=\Wp\!\lbrb{a}\sqrt{\frac{\phi(a)}{\abs{\phi(z)}}}e^{-\Aph\lbrb{z}}e^{-\Eph(z)}.
\end{equation}
Since $\abs{\overline{\Wp\lbrb{z}}}=\abs{\Wp\lbrb{\bar{z}}}$ we conclude \eqref{eq:absWp1}, for any $z=\ab\in\CbOI, b\neq 0$. Since from \eqref{eq:Aphi} and \eqref{eq:Ephi}, $\Aph\!\lbrb{\Re(z)}=\Eph\!\lbrb{\Re(z)}=0$, we deduct that \eqref{eq:absWp1} holds for $z=a\in\Rb^+$ too. Next, let us investigate $\Wp(a)$ in \eqref{eq:absWp1}. Recall that,  from \eqref{eq:BernWeier} and \eqref{eq:EulerConst}, we get, for $a>0$, that
\begin{equation}\label{eq:realWp}
\begin{split}
\Wp\!\lbrb{a}&=\frac{e^{-\gamma_\phi a}}{\phi(a)}\prod_{k=1}^{\infty}\frac{\phi(k)}{\phi\!\lbrb{k+a}}e^{\frac{\phi'(k)}{\phi(k)}a}\\
&=\frac{1}{\phi(a)}\limi{n}\prod_{k=1}^{n}\frac{\phi(k)}{\phi\!\lbrb{k+a}}e^{a\ln\phi(n)}e^{a\lbrb{\sum_{k=1}^{n}\frac{\phi'(k)}{\phi(k)}-\ln\phi(n)-\gamph}}\\
&=\frac{1}{\phi(a)}\limi{n}\prod_{k=1}^{n}\frac{\phi(k)}{\phi\!\lbrb{k+a}}e^{a\ln\phi(n)}=\frac{1}{\phi(a)}\limi{n}e^{-S_{n}(a)+a\ln\phi(n)},
\end{split}
\end{equation}
where $S_n(a)=\sum_{k=1}^{n}\ln \frac{\phi\lbrb{a+k}}{\phi(k)}$. Then, we get, from \cite[Section 8.2, (2.01), (2.03)]{Olver-74} applied to the function $f(u):=\ln\frac{\phi\lbrb{a+u}}{\phi(u)},\,u>0,$ with $m=1$ in their notation, that
\begin{equation}\label{eq:Sn}
\begin{split}
S_n(a)&=\int_{1}^{n}\ln\frac{\phi\lbrb{a+u}}{\phi(u)}du+\frac{1}{2}\ln\frac{\phi(1+a)}{\phi(1)}+\frac{1}{2}\ln\frac{\phi(n+a)}{\phi(n)}+R_2(n,a)\\
&=\frac{1}{2}\ln\frac{\phi(1+a)}{\phi(1)}-\int_{1}^{a+1}\ln\phi(u)du+\int_{0}^{a}\ln\phi(n+u)du+\frac{1}{2}\ln\frac{\phi(n+a)}{\phi(n)}+R_2(n,a),
\end{split}		
\end{equation}
where, recalling that $\mathrm{P}(u)=\lbrb{u-\lrfloor{u}}\lbrb{1-\lbrb{u-\lrfloor{u}}}$, for any $a>0$,
\begin{equation}\label{eq:Rn}
R_2(n,a)=\frac{1}{2}\int_{1}^{n}\mathrm{P}(u)\lbrb{\ln
	\frac{\phi\lbrb{a+u}}{\phi(u)}}''du.
\end{equation}
Using \eqref{lemmaAsymp1-1} in Proposition \ref{propAsymp1}  and \eqref{eq:Sn},  we get that
\begin{align*}			
\limi{n}\lbrb{S_n(a)-a\ln\phi(n)}&=\frac{1}{2}\ln \frac{\phi(1+a)}{\phi(1)}-\int_{1}^{a+1}\ln\phi(u)du\\
&\,\,\,\,\,\,+\limi{n}\lbrb{\int_{0}^{a}\ln \frac{\phi(n+u)}{\phi(n)}du
	+R_2(n,a)}\\	
\nonumber\qquad&=\frac{1}{2}\ln \frac{\phi(1+a)}{\phi(1)}-\Gph(a)+\limi{n}R_2(n,a),
\end{align*}
where $\Gph$ is defined in \eqref{eq:Gph}.
Let us show that $R_\phi(a)=\limi{n}R_2(n,a)$ exists with $\Rph$ defined in \eqref{eq:Rphi}. It follows from \eqref{eq:Rn}, the inequality $\sup_{u>0}\abs{\mathrm{P}(u)}\leq\frac{1}{4}$ and the dominated convergence theorem since
\begin{equation}\label{eq:R2Bound}
\sup_{n\geq 1}\sup_{\phi\in\Bc}\abs{R_2(n,a)}\leq \frac{1}{4}\sup_{\phi\in\Bc}\int_{1}^{\infty}\lbrb{\lbrb{\frac{\phi'(u)}{\phi(u)}}^2+\abs{\frac{\phi''(u)}{\phi(u)}}}du<2,
\end{equation}
where the finiteness follows from \eqref{specialEstimates11}. Therefore, from \eqref{eq:realWp}, \eqref{eq:Rn} and the existence of $R_\phi(a)$, we get that
\[\Wp(a)=\frac{1}{\phi(a)}\sqrt{\frac{\phi(1)}{\phi(1+a)}}e^{\Gph(a)-R_\phi(a)}.\]
Substituting this in \eqref{eq:absWp1} we prove \eqref{eq:Stirling}. Finally, the asymptotic relation \eqref{eq:err} follows from the application of \eqref{eq:uniBoundRE1}.
\subsection{Proof of Theorem \ref{thm:Stirling}\eqref{it:realStirling}}
The proof of this claim follows from representation of $\abs{\Wp}$ in \eqref{eq:Stirling}, the form of $\Gph$ as expressed in \eqref{eq:Gph} and Theorem \ref{thm:genFuncs}\eqref{it:Tphi}. Finally, the fact that $\limi{a}\Aph\!\lbrb{a+ib}=0$ has already been proved in \eqref{eq:Aphto0}. This concludes the proof of this item.

\subsection{Proof of Theorem \ref{thm:Stirling}\eqref{it:cases}} The proof is straightforward. Since $\phi\in\BP$ item \eqref{it:case1} follows immediately from \eqref{eq:AphiAsymp}, the definition of $\Bc_{\nexp}$, see \eqref{eq:polyclassB1}, and Lemma \ref{lem:Lindelof}, which allows to extend the exponential decay to all complex lines $\Cb_a,\,a>\dph$. In the same fashion since $\phi\in\Bc_\alpha$ item \eqref{it:case2} is deducted from \eqref{eq:AphiAsymp1} and Lemma \ref{lem:Lindelof}. Finally, item \eqref{it:case3} is merely a rephrasing of Theorem \ref{thm:Aph}\eqref{eq:subexpp} for $a>0$ which is then augmented by Lemma \ref{lem:Lindelof} to $a>\dph$.
\subsection{Proof of Lemma \ref{lem:Lindelof}}
To prove the claims we rely on the  Phragm\'{e}n-Lindel\"{o}f's  principle combined with the functional equation \eqref{eq:Wp}, which also appears in Definition \ref{def:WB}. First, assume that  there exists $\ak>\dph$ such that \[\limsupi{|b|}|b|^n\abs{\Wp\!\lbrb{\ak+ib}}=0,\,\forall n\in\Nb,\]
and since $\abs{\Wp\!\lbrb{\ak+ib}}=\abs{\Wp\!\lbrb{\ak-ib}}$ consider $b>0$ only. The recurrent equation $\eqref{eq:Wp}$ and $\abs{\phi\lbrb{\ak+ib}}=\dr b+\sospace{|b|},$ as $b\to\infty$ and $\ak>\dph$ fixed, see Proposition \ref{propAsymp1}\eqref{it:asyphid}, yield for all $n\in\Nb$ that
\[\limsupi{|b|}|b|^n\abs{\Wp\!\lbrb{1+\ak+ib}}=\limsupi{|b|}|b|^n\abs{\phi\!\lbrb{\ak+ib}}\abs{\Wp\!\lbrb{\ak+ib}}=0.\] Then, we apply the Phragm\'{e}n-Lindel\"{o}f's principle to the strip $\Cb^+_{\lbbrbb{\ak,\,\ak+1}}=\Cb_{\lbbrbb{\ak,\,\ak+1}}\cap\lbcurlyrbcurly{b\geq 0}$ and to the functions $f_n(z)=z^n\Wp\!\lbrb{z},\,n\in\Nb,$ which are holomorphic on $\Cb^+_{\lbbrbb{\ak,\,\ak+1}}$. Indeed, from our assumptions and the observation above, we have that, for every $n\in\Nb$ and some finite positive constants $C_n$,
\[\sup_{z\in\partial \Cb^+_{\lbbrbb{\ak,\,\ak+1}}}\abs{f_n(z)}\leq C_n\]
and clearly, since $\Wp \in
\Wc_\Bc$, see \eqref{eq:defWb}, is a Mellin transform of a random variable,
\[\sup_{z\in \Cb^+_{\lbbrbb{\ak,\,\ak+1}}}\abs{\frac{f_n(z)}{z^n}}=\sup_{z\in \Cb^+_{\lbbrbb{\ak,\,\ak+1}}}\abs{\Wp\!\lbrb{z}}=\sup_{v\in\lbbrbb{\ak,\,\ak+1}}\Wp(v)<\infty.\]
Thus, we conclude from the Phragm\'{e}n-Lindel\"{o}f's  principle that
\[\sup_{z\in \Cb^+_{\lbbrbb{\ak,\,\ak+1}}}\abs{f_n(z)}=\sup_{z\in \Cb^+_{\lbbrbb{\ak,\,\ak+1}}}\abs{z^n\Wp(z)}\leq C_n,\]
see \cite{Titchmarsh1939} or \cite[Theorem 1.0.1]{Garret-07}, which is a discussion of the paper by Phragm\'{e}n and Lindel\"{o}f, that is  \cite{Phragmen-Lindelof-1908}.
Finally, \eqref{eq:Wp} and $\abs{\phi\!\lbrb{a+ib}}=\dr b+\sospace{b},$ as $b\to\infty$, allow us to deduce that
\[\limsupi{|b|}|b|^n\abs{\Wp\!\lbrb{a+ib}}=0,\,\forall n\in\Nb,\,\forall a\geq\ak.\]
To conclude the claim for $a\in\lbrb{\dph,\ak}$ we use \eqref{eq:Wp} in the opposite direction and \eqref{eq:rephi}, that is $\Re\!\lbrb{\phi\lbrb{\ab}}\geq\phi(a)>0$, for any $a>\dph$, to get that
\begin{equation}\label{eq:opposite} \frac{1}{\phi\lbrb{a}}\abs{\Wp\!\lbrb{1+a+ib}}\geq\frac{1}{\abs{\phi\!\lbrb{\ab}}}\abs{\Wp\!\lbrb{1+a+ib}}=\abs{\Wp\!\lbrb{\ab}}.
\end{equation}
Thus, $\phi\in\PcBinf$. Next, assume that there exists $ \ak>\dph,\,\theta\in\lbrbb{0,\frac{\pi}{2}}$ such that \[\limsupi{|b|}\frac{\ln\abs{\Wp\!\lbrb{\ak+ib}}}{\abs{b}}\leq -\theta.\] Then, arguing as above, we conclude that this relation holds for $\ak+1$ too. Then, on $\Cb^+_{\lbbrbb{\ak,\,\ak+1}}$, for the functions \[f_\varepsilon(z)=\Wp(z)e^{-i\lbrb{\theta-\varepsilon} z},\,\varepsilon\in\lbrb{0,\theta},\] we have again from  the Definition \eqref{eq:defWb} with some $C>0,\,D=D(\ak)>0$, that
\[\sup_{z\in\partial \Cb^+_{\lbbrbb{\ak,\,\ak+1}}}\abs{f_\varepsilon(z)}\leq C\quad \textrm{ and } \quad \sup_{b\geq 0}\sup_{v\in\lbbrbb{\ak,\,\ak+1}}|f_\varepsilon(v+ib)|\leq De^{\lbrb{\theta-\varepsilon} b}\leq De^{\frac{\pi}{2}b}.\]
This suffices for the application of the Phragm\'{e}n-Lindel\"{o}f's principle with $\tilde{f}_\varepsilon(z)=f_\varepsilon\lbrb{iz}$. Therefore, we conclude that $\abs{f_\varepsilon(z)}\leq C$ on $\Cb^+_{\lbbrbb{\ak,\,\ak+1}}$ and thus, for all $v\in\lbbrbb{\ak,\ak+1}$,
\[\limsupi{b}\frac{\ln\abs{\Wp\!\lbrb{v+ib}}}{b}\leq -\theta+\varepsilon.\]
Sending $\varepsilon\to 0$ we conclude that, for all $v\in\lbbrbb{\ak,\ak+1}$,
\begin{equation}\label{eq:hata}
\limsupi{b}\frac{\ln\abs{\Wp\!\lbrb{v+ib}}}{b}\leq -\theta.
\end{equation}
Therefore, from the identities $\abs{\Wp\!\lbrb{a+ib}}=\abs{\Wp\!\lbrb{a-ib}}$ and \eqref{eq:Wp}, the widely used relation $\abs{\phi\!\lbrb{a+ib}}=\dr b+\sospace{b},$ as $b\to\infty$ and $a>\dph$ fixed, see Proposition \ref{propAsymp1}\eqref{it:asyphid}, and \eqref{eq:opposite} we deduct that  \eqref{eq:hata} holds for $a>\dph$. Thus, we deduce that
$\phi\in\Bth$ and conclude the entire proof of the lemma.

\subsection{Proof of Theorem \ref{thm:HY}}\label{subsec:prof}
The first item is proved, for any $\phi\in\Bc$, by Berg in \cite[Theorem 2.2]{Berg-07}.  The second item follows readily from the first one after recalling that if $\Psi \in \Ne$ then $\Psi(-\cdot)\in \overNc$.  For the item \eqref{it:sm}, since $\phi/\underline{\phi}$ is completely monotone, and  we recall that for any $n\in \N$, $W_{\phi}(n+1)=\prod_{k=1}^n\phi(k)$, we first get from \cite[Theorem 1.3]{Berg}  that the sequence $\left(f_n=W_{\underline{\phi}}(n+1)/W_\phi(n+1)\right)_{n\geq0}$ is the moment sequence of a positive variable $I$.  Next, from the recurrence equation \eqref{eq:Wp} combined with the estimates stated in Proposition \ref{propAsymp1}\eqref{it:asyphid}, we deduce that $f_{n+1}/f_{n}=\underline{\phi}(n)/\phi(n)\stackrel{\infty}{=}O(n)$.  Thus, there exists $A>0$ such that for any $a<A$,
\[ \mathbb{E}\left[e^{aI}\right] =\sum_{n=0}^{\infty}\frac{f_n}{n!}a^n <\infty,\]
implying that $I$ is moment determinate. Next, with the notation of the statement, if $\Ntt=\Ntt_{\underline{\phi}}-\Ntt_{\phi}>\frac12$   then, as \[\left|\mathcal{M}_{I}\!\left(ib-\frac12\right)\right|=\left|\frac{W_{\underline{\phi}}\!\lbrb{ib+\frac12}}{W_\phi\!\lbrb{ib+\frac12}}\right|\stackrel{\infty}{=}\bospace{|b|^{-\Ntt}}\]
and, from Theorem \ref{thm:Wp}, for any $\phi\in \Be$, $\Wp \in\Ac_{\lbrb{0,\infty}}$ and is zero-free on $\Cb_{\lbrb{0,\infty}}$, we obtain that  $b\mapsto \mathcal{M}_{I}\!\left(ib-\frac12\right) \in \Lspaces{2}{\R}$ and hence by the Parseval identity for Mellin transform we conclude that  $f_{I}\in \Lspaces{2}{\Rp}$.  Finally if $\Ntt>1$ then the result follows from a similar estimate for the Mellin transform which allows to use a  Mellin inversion technique to prove the  claim in this case. For the last item, we first  observe, from \eqref{eq:phi'} and Proposition \ref{propAsymp1}\eqref{it:bernstein_cmi}, that for any $\phi \in \Be$ and $u>0$,
\begin{equation}\label{eq:defk}
\frac{\phi'(u)}{\phi(u)}=\int_0^{\infty}e^{-uy}\kappa(dy)
\end{equation}
where we recall that $\kappa(dy)=\int_{0}^{y}U\lbrb{dy-r}\left(r\mu(dr)+\delta_{\dr}(dr)\right)$. Thus,
\begin{equation}
\underline{\Psi}(z)=\ln \frac{\Wp(z+1)}{W_{\underline{\phi}}(z+1)}=\ln \frac{\phi(1)}{\underline{\phi}(1)}z + \int_0^{\infty} \lbrb{e^{-zy}-1-z(e^{-y}-1)}\frac{\kappa(dy)-\underline{\kappa}(dy)}{y(e^y-1)}
\end{equation}
where, as in \eqref{eq:defk}, we have set $\underline{\phi}'(u)/\underline{\phi}(u)=\int_0^{\infty}e^{-uy}\underline{\kappa}(dy)$. Next, since plainly $\underline{\phi}'(1)/\underline{\phi}(1)-\phi'(1)/\phi(1)<\infty$, we have that the measure $e^{-y}K(dy)=e^{-y}\left(\kappa(dy)-\underline{\kappa}(dy)\right)$ is finite on $\Rp$.  Thus, by the L\'evy-Khintchine formula, see \cite{Bertoin-96},  $\underline{\Psi} \in \Ne$ if and only if $K$ is a non-negative measure, which by Bernstein theorem, see e.g.~\cite{Feller-71}, is equivalent to the mapping   $u\mapsto \ln \left(\underline{\phi}/\phi\right)'(u)=\left(\underline{\phi}'/\underline{\phi}-\phi'/\phi\right)(u)$ to be completely monotone. Note that this latter condition implies that $\phi/\underline{\phi}$ is completely monotone and hence with $I$ as in the previous item, we deduce easily that
$\Eb\left[e^{z \log I}\right]=e^{\underline{\Psi}(z)}$ which shows that $\log I$ is infinitely divisible on $\R$.  The last equivalent condition being immediate from the definition of $K$, the proof of the theorem is completed.

\subsection{Proof of Lemma \ref{lem:continuityW}}\label{sec:proof_wbe}
Let $Y_{\phi_n},\,Y_{\underline\phi},\,n\in\N,$ be the random variables associated to $W_{\phi_n},W_{\underline{\phi}},\,n\in\N,$ see the Definition \eqref{eq:defWb}. Clearly, since for any $\phi\in\Bc$,
\[\Ebb{e^{tY_\phi}}=\sum_{k=0}^{\infty}t^k\frac{\Ebb{Y^k_\phi}}{k!}=\sum_{k=0}^{\infty}t^k\frac{\Wp(k+1)}{k!}=\sum_{k=0}^{\infty}t^k\frac{\prod_{j=1}^{k}\phi(j)}{k!}\]
and Proposition \ref{propAsymp1}\eqref{it:asyphid} holds, we conclude that $\Ebb{e^{tY_\phi}}$ is well defined for $t<\frac{1}{\dr}\in\lbrbb{0,\infty}$. However, $\limi{n}\phi_n(a)=\underline\phi(a)$ implies that $\dr^*=\sup_{n\geq 0}\dr_n<\infty$, where $\dr_n$ are the linear terms of $\phi_n$  in \eqref{eq:phi}. Therefore, $\Ebb{e^{zY_{\underline\phi}}},\,\Ebb{e^{zY_{\phi_n}}},\,n\in\N,$ are analytic in $\Cb_{\lbrb{-\infty, \min\curly{\frac{1}{\underline{\dr}},\,\frac{1}{\dr^*}}}}\varsupsetneq\Cb_{\lbrbb{-\infty,0}}$. Moreover, for any $k\in\N$,
\begin{equation*}
\begin{split}
\limi{n}\Ebb{Y^k_{\phi_n}}&=\limi{n}W_{\phi_n}\!(k+1)\\
&=\limi{n}\prod_{j=1}^{k}\phi_n(j)=W_{\underline{\phi}}(k+1)=\Ebb{Y^k_{\underline{\phi}}}.
\end{split}
\end{equation*}
The last two observations trigger that $\limi{n}Y_{\phi_n}\stackrel{d}{=}Y_{\underline{\phi}}$, see \cite[p.269, Example (b)]{Feller-71}. Therefore, since on a suitable probability space one can choose random variables $Y_n\stackrel{d}{=} Y_{\phi_n}, n\geq 0,$ and $Y\stackrel{d}{=}Y_{\underline{\phi}}$ such that $\limi{n} Y_n=Y$ a.s., see \cite[Theorem 3.30]{Kallenberg2001}, one gets by the dominated convergence theorem and the fact that the family of measures $y^a\Pbb{Y_n\in dy}, a>0,$ are uniformly integrable
that   $\limi{n}W_{\phi_n}\!(z)=W_{\underline{\phi}}(z),\,z\in\CbOI$, which concludes the proof. To check the uniform integrability for $a\geq 1$ note that for $N>2a$
\[\limi{K}\sup_{n\geq 1}\Ebb{Y_n^a\ind{Y_n\geq K}}\leq \limi{K}\sup_{n\geq 1}\sqrt{\Ebb{Y_n^{2N}}}\sqrt{\frac{\Ebb{Y_n}}{K}}=0\]
since $\Ebb{Y_n}=W_{\phi_n}\!(2)$ and $\Ebb{Y_n^{2N}}=W_{\phi_n}\!(2N+1)$ converge as $n\to \infty$. For $a<1$ the convergence follows from the recurrence equation $W_{\phi_n}\!(z+1)=\phi_n(z)W_{\phi_n}\!(z)$.

\section{The functional equation \eqref{eq:fe}}\label{sec:FE}

\subsection{Proof of  Theorem \ref{lemma:FormMellin}} \label{sec:proof_thm1}
Recall that by definition $\MPs(z)=\frac{\Gamma(z)}{\Wpp\!(z)}\Wpn\!\lbrb{1-z}$, see \eqref{eqM:MIPsi}. From \eqref{eq:WH1} and \eqref{eq:feb} it is clear that formally for $z\in i\R$
\begin{equation}\label{eq:fe2}
\begin{split}
\MPs\!\lbrb{z+1}&=\frac{\Gamma\!\lbrb{z+1}}{\Wpp\!\lbrb{z+1}}\Wpn\!(-z)\\
&=\frac{z\Gamma\!\lbrb{z}}{\php(z)\Wpp\!\lbrb{z}}\frac{\Wpn\!\lbrb{1-z}}{\phn(-z)}=\frac{-z}{\Psi\!\lbrb{-z}}\MPs(z).
\end{split}
\end{equation}
However, from Theorem \ref{thm:Wp}\eqref{it:A} it is clear that $\Wpp$ (resp.~$\Wpn$) extend continuously to $i\R\setminus \Zc_0\lbrb{\php}$ (resp.~$i\R\setminus\Zc_0\lbrb{\phn}$). Since from \eqref{eq:WH1} we have that  $\Zc_0\!\lbrb{\Psi}=\Zc_0\!\lbrb{\php}\cup\Zc_0\!\lbrb{\phn}$, see \eqref{eq:zerosPsi} and \eqref{eq:zerosphi} for the definition of the sets of zeros, we conclude that $\MPs$ satisfies \eqref{eq:fe2} on $i\R\setminus\Zc_0\!\lbrb{\Psi}$.
The claim that $\MPs\in\Ac_{(0,\,1-\dphn)}\cap\,\Mtt_{\lbrb{\aphp,\,1-\aphn}}$ then follows from the facts that $\Wp\in\Ac_{\lbrb{\dph,\infty}}\cap\, \Mtt_{\lbrb{\aph,\infty}}$ and $\Wp$ is zero-free on $\Cb_{\lbrb{\aph,\infty}}$ for any $\phi\in\Bc$, see Theorem \ref{thm:Wp}\eqref{it:A}, lead to 
\begin{equation}\label{eq:merom}
\begin{split}
&\frac{1}{\Wpp(\cdot)}\in\Ac_{\lbrb{\aphp,\infty}},\quad \Gamma\in \Ac_{\lbrb{0,\infty}}\cap \Mtt_{\lbrb{-\infty,\infty}}  \text{ and }\\ &\Wpn\!\!\lbrb{1-\cdot}\in \Ac_{\lbrb{-\infty,1-\dphn}}\cap \Mtt_{\lbrb{-\infty,1-\aphn}}.
\end{split}
\end{equation}
Thus, \eqref{eq:domainMerMPs}, that is $\MPs\in\Mtt_{\lbrb{\aphp,\,1-\aphn}}$, in general, and \eqref{eq:domainAnalMPs}, that is $\MPs\in\Ac_{(\aP,\,1-\dphn)}$, when $\aP=\aphp\ind{\dphp=0}=0$ follow. Note that when $0=\dphp>\aphp$, i.e.~$\aP=\aphp<0$, then necessarily 
\[\php'(a^+)=\dr+\IntOI ye^{ay}\mup{dy}\in\lbrb{0,\infty}, \text{ for any $a>\aphp$},\] 
see \eqref{eq:phi'}, and  $\dphp=0$, see \eqref{eq:dphi1}, is the only zero of $\php$ on $\lbrb{\aphp,\infty}$. Therefore, Theorem \ref{thm:Wp}\eqref{it:F} applies and yields that at $z=0$, $\Wpp$ has a simple pole. Via the recurrence relation \eqref{eq:Wp} combined with the fact that $\php<0$ on $\lbrb{\aphp,0}$ this simple pole is propagated to a simple pole at all negative integers $n$ such that $n>\aphp$. These simple poles however are simple zeros for $\frac{1}{\Wpp}\in\Ac_{\lbrb{\aphp,\infty}}$ which cancel the poles of $\Gamma$. Thus, $\MPs\in\Ac_{\lbrb{\aP,\,1-\dphn}}$ and \eqref{eq:domainAnalMPs} is established. For $z\in\CbOI$ we have that
\begin{equation}\label{eq:recur} \MPs(z)=\frac{\Gamma(z)}{\Wpp\!(z)}\Wpn\!(1-z)=\frac{\php(z)}{z}\frac{\Gamma\!\lbrb{z+1}}{\Wpp\!\lbrb{z+1}}\Wpn\!(1-z).
\end{equation}
From Theorem \ref{thm:Wp}\eqref{it:A} if $\php(0)>0$ then $\Wpp\in\Ac_{\lbbrb{0,\infty}}$ and $\Wpp$ is zero-free on $\Cb_{\lbbrb{0,\infty}}$, and the pole of $\Gamma$ at zero is uncontested, see \eqref{eq:recur}. Therefore, $\MPs$ extends continuously to $i\R\setminus\curly{0}$ in this case. The same follows from \eqref{eq:recur}  when $\php'(0^+)=\infty$. Let next $\php(0)=0=\aphp$ and $\php'(0^+)<\infty$. Then \eqref{eq:recur}
shows that the claim $\MPs\in\Ac_{\lbbrb{0,\,1-\aphn}}$ clearly follows. The fact that $\MPs \in   \Mtt_{\lbrb{\aphp,\,1-\aphn}}$, that is \eqref{eq:domainMerMPs} is apparent from \eqref{eq:merom}. We proceed with the final assertions.
Let $\aphp\leq\dphp<0$. If $-\tphp\notin\N$ then $\tphp$ is the only zero of $\php$ on $\lbrb{\aphp,\infty}$ and if $\tphp=-\infty$ since $\dphp<0$ then $\php$ has no zeros on $\lbrb{\aphp,\infty}$ at all. Henceforth, from \eqref{eq:Wp} we see that $\Wpp$ does not possess poles at the negative integers. Thus, the poles of the function $\Gamma$ are uncontested. If $-\tphp\in \N$ and $\tphp=\aphp$ there is nothing to prove, whereas if $-\tphp\in \N$ and $\tphp>\aphp$ then Theorem \ref{thm:Wp}\eqref{it:E} shows that $\Wpp$ has a simple pole at $\tphp$. Thus $\frac{1}{\Wpp}$ has a simple zero at $\tphp$. Then, \eqref{eq:Wp} propagates the zeros to all $\aphp<-n \leq\tphp$ cancelling the poles of $\Gamma$ at those locations. The values of the residues are easily computed via the recurrent equation \eqref{eq:feb}  for $\Wpp,\Wpn$, the Wiener-Hopf factorization \eqref{eq:WH1}, the form of $\MPs$, see \eqref{eqM:MIPsi}, and the residues of the gamma function which are of value $\frac{\lbrb{-1}^n}{n!}$ at $-n$. Indeed let $-n$ be a pole and choose $0<\epsilon<1$. Then
\begin{equation*}
\begin{split}
\MPs\!\lbrb{-n+\epsilon}&=\frac{\Gamma\!\lbrb{-n+\epsilon}}{\Wpp\!\lbrb{-n+\epsilon}}\Wpn\!\lbrb{1+n-\epsilon}\\
&=\lbrb{\prod_{k=1}^n\frac{\php\!\lbrb{-k+\epsilon}}{-k+\epsilon}\prod_{k=1}^n\phn(k-\epsilon)}\frac{\php(\epsilon)\Gamma\!\lbrb{1+\epsilon}}{\epsilon\Wpp\!\lbrb{1+\epsilon}}\Wpn\!\lbrb{1-\epsilon}\\
&=\frac{1}{\epsilon}\lbrb{\frac{\prod_{k=1}^n\Psi(k-\epsilon)}{\prod_{k=1}^n\lbrb{k-\epsilon}}}\frac{\php(\epsilon)\Gamma\!\lbrb{1+\epsilon}}{\Wpp\!\lbrb{1+\epsilon}}\Wpn\!\lbrb{1-\epsilon}.
\end{split}
\end{equation*}
Then 
\[ \limo{\epsilon}\epsilon\MPs\!\lbrb{-n+\epsilon}= \php(0)\frac{\prod_{k=1}^{n} \Psi(k)}{n!}\]
and the term on the right-hand side is the residue at $-n$. Finally, if we alter the Wiener-Hopf factors as $\php\mapsto c\php,\phn\mapsto c^{-1}\phn,c>0$ then Theorem \ref{thm:Wp}\eqref{it:G} gives that
\[W_{c\php}(z)=c^{z-1}\Wpp(z) \text{ and } W_{c^{-1}\phn}\lbrb{1-z}=c^{z}\Wpn(1-z).\]
Plugging these identities in relation \eqref{eqM:MIPsi} furnishes the invariance. This concludes the proof of this theorem.

\subsection{Proof of Theorem \ref{thm:asympMPsi}}\label{subsec:decay}
Before we commence the proof we introduce some more notation. We use $f\asymp g$ to denote the existence of two positive constants $0<C_1<C_2<\infty$ such that \[C_1\leq \varliminf_{x\to a}\abs{\frac{f(x)}{g(x)}}\leq \varlimsup_{x\to a}\abs{\frac{f(x)}{g(x)}}\leq C_2,\] where $a$ is usually $0$ or $\infty$. The relation $f\lesssim g$, that will be employed from now on, requires only that $\varlimsup_{x\to a}\abs{\frac{f(x)}{g(x)}}\leq C_2<\infty$. Note that the relation  $f\lesssim g$ is equivalent to the relation $f=\bospace{g}$.

We recall from \eqref{eq:BP} that
$\BP=\curly{\phi\in\Bc:\,\dr>0}$ and $\BP^c$ is its complement in $\Bc$.
Appealing to various auxiliary results below we consider Theorem \ref{thm:asympMPsi}\eqref{it:decayA} first. Throughout the proof we use \eqref{eqM:MIPsi}, that is
\begin{equation}\label{eq:MPsi1}
\MPs(z)=  \frac{\Gamma(z)}{\Wpp\!(z)}\Wpn\!(1-z)\in \Ac_{(0,1-\dphn)}.
\end{equation}
We note from  Definition \ref{def:WB} and \eqref{eq:GammaType2} of Theorem  \ref{thm:factorization} that $\Wpn\!(z)$ and $\frac{\Gamma(z)}{\Wpp\!(z)}$ are Mellin transforms of positive random variables. Therefore the bounds
\begin{align}
\abs{\MPs(z)}&\leq \frac{\Gamma(a)}{\Wpp\!(a)}\,\,\abs{\Wpn\!\lbrb{1-z}}\label{eq:Ipp}\\
\abs{\MPs(z)}&\leq \frac{\abs{\Gamma(z)}}{\abs{\Wpp\!(z)}}\abs{\Wpn\!\lbrb{1-a}}\label{eq:Ypn}
\end{align}
hold for $z\in\Cb_a,\,a\in\lbrb{0,1-\dphn}$.
From \eqref{eq:Ipp} and \eqref{eq:Ypn} we have that
$\Psi\in\Npi$, see \eqref{eq:Ninf}, if and only if  $\phn\in\PcBinf$ and/or $\labsrabs{\frac{\Gamma(z)}{\Wpp(z)}}$ decays faster than any polynomial along $a+i\R,\,\forall a\in\lbrb{0,1-\dphn}$.  The former certainly holds if $\phn\in\BP$, that is $\dem>0$, since from Theorem \ref{thm:Stirling} \eqref{it:case1} or   Proposition \ref{prop:condClass2}\eqref{it:class2_A}, we have the even stronger $\phn\in\Bthd\subseteq\PcBinf$, and the latter if $\php\in\BP^c$, see  \eqref{eq:condClass1} in Proposition \ref{prop:condClass1}. Also, if $\php\in\BP$ then, for any $a>0$, \eqref{eq:condClass1} in Proposition \ref{prop:condClass1} holds for all $u>0$ iff $\mubarpspace{0}=\infty$ and hence from \eqref{eq:Ypn} we deduce that $\Psi\in\Npi$ provided $\mubarpspace{0}=\infty$. Next, from \eqref{eq:Ipp},\eqref{eq:Ypn}, Proposition \ref{prop:condClass1} and Proposition \ref{prop:condClass2}\eqref{it:class2_B} if  $\php\in\BP$, $\mubarpspace{0}<\infty$ and $\phn\in \BP^c$ then \[\phn\in\PcBinf\iff\Psi\in\Npi\iff \PPn(0)=\infty.\] Therefore to confirm the second line of \eqref{eq:NPs}
we ought to check only that if $\Psi\in\overNc$, $\php\in\BP$ and $\phn\in\BP^c$ then
\begin{equation}\label{eq:proveA}
\PP(0)=\infty\iff \PPn(0)=\infty \text{ or }\bar{\mu}_{\pls}(0)=\infty.
\end{equation}
However, if $\PPn(0)<\infty$ and $\mubarpspace{0}<\infty$ then since $\php\in\BP$ the \LLP is a positive linear drift plus compound Poisson process which proves the forward direction of \eqref{eq:proveA}. The backward part is immediate as it precludes right away the possibility of compound Poisson process with drift. In fact the expression for $\NPs$ in the first line of \eqref{eq:NPs} is derived as the sum of the rate of polynomial decay of $\abs{\frac{\Gamma(z)}{\Wpp\!(z)}}$ in Proposition \ref{prop:condClass1} and of $\abs{\Wpn\!(z)}$ in Proposition \ref{prop:condClass2}\eqref{it:class2_C} coupled with \eqref{eq:MPsi1} and subsequent simplification. Indeed immediately, we have that 
\begin{equation}\label{eq:NPS}
\NPs=\frac{\IntOI u_{\pls}(y)\Pi_{\mis}(dy)}{\phn(0)+\mubrn(0)}+\frac{1}{\dep}\lbrb{\php\!(0)+\mubarpspace{0}},
\end{equation}
where since $\dep>0$, $u_+$ is the potential density associated to the potential measure of $\php$, see Proposition \ref{propAsymp1} \eqref{it:bernstein_cmi}. Since $\dep>0$  and $\PPn\!(0)<\infty$ from Proposition \ref{prop:VigonDens} we have that 
\begin{equation}\label{eq:upsilon}
\upsilon_-(0^+)=\IntOI u_+(y)\Pi_-(dy)
\end{equation}
and the first line of \eqref{eq:NPs} for $\NPs$ follows from a substitution of \eqref{eq:upsilon} in \eqref{eq:NPS}.
However, as it is known again from Proposition \ref{propAsymp1} \eqref{it:bernstein_cmi} that $u_+>0$ on $\Rp$, whenever $\dep>0$, then $\upsilon_-(0^+)=0$ if and only if $\Pi_-$ is the zero measure. Then from \eqref{eq:NPs} we see that
\[\NPs=0\iff \php(0)=\mubarpspace{0}=\upsilon_-\!\lbrb{0^+}=0.\]
The relation  $\php(0)=\mubarpspace{0}=0$ is valid if and only if $\php(z)=\dep z$. This proves the claim and ends the proof of item \eqref{it:decayA}.
The assertions
\[\phn\in\BP\implies \Psi\in\Nthdspace\quad\text{ and }\quad \phn\in\Bc_{\alpha},\php\in\Bc_{1-\alpha}\implies\Psi\in\Nthaspace\]
of item \eqref{it:classes}  follow from \eqref{eq:Ipp} and \eqref{eq:Ypn} with the help of items \eqref{it:awpid} and \eqref{it:awpia} of Theorem \ref{thm:Stirling} and the standard asymptotic for the gamma function
\begin{equation}\label{eq:GammaStirling}
\abs{\Gamma\!\lbrb{a+ib}}=\sqrt{2\pi}|b|^{a-\frac{1}{2}}e^{-\frac{\pi}{2}|b|}\lbrb{1+\sospace{1}},
\end{equation}
which holds as $b\to\infty$ and $a$ fixed, see \cite[(8.328.1)]{Gradshteyn-Ryzhik-00}. Let $\php\equiv C\phn, C>0,$ hold. Then clearly $\arg \php\equiv\arg\phn$ and hence $\Aphp\equiv \Aphn$. Choose $a=\frac12$. Then from \eqref{eq:Stirling}  and \eqref{eq:MPsi1} we see that modulo two constants, as $b\to\infty$,
\[\abs{\MPspace\lbrb{\frac12+ib}}\asymp \frac{\sqrt{\abs{\php\!\lbrb{\complex{\frac12}{b}}}}}{\sqrt{\abs{\phn\!\lbrb{\frac12-ib}}}}\abs{\Gamma\!\lbrb{\complex{\frac12}{b}}}.\]
From \eqref{eq:rephi}, that is $\Re\!\lbrb{\phn\lbrb{\frac12-ib}}>\phn\lbrb{\frac12}>0$ and \eqref{eq:GammaStirling}  we see that $\Psi\in\Nthdspace$. Finally, the last claim of  $\Psi\in\Ntpspace{\Theta_{\pms}}$ follows readily from \eqref{eq:MPsi1} and \eqref{eq:Stirling}. This ends the proof.\qed

The next sequence of results are used in the proof above. Recall that the classes $\Bi$ and $\Bth$ are defined in  \eqref{eq:polyclassB} and \eqref{eq:polyclassB1} respectively.

The proof of Theorem \ref{thm:asympMPsi} via \eqref{eq:Ipp} and \eqref{eq:Ypn} hinges upon the assertions of Proposition \ref{prop:condClass1} and Proposition \ref{prop:condClass2}. 	Before stating and proving them, we have  the following simple result, for which we need to recall that  for any $a\geq 0$ and any function $f:\Rp\mapsto\R$ we use the notation $f_a(y)=e^{-ay}f(y),\,y>0$.
\begin{proposition}\label{prop:convoEst}
	For any function $f$ and $n \in \N$, we have that $f^{*n}_a(y)=e^{-ay}f^{*n}(y),\,y\in\intervalOI$. If $f_{a'}\in \Lspaces{\infty}{\Rp}$ for some $a'\geq 0$ then for any $a\geq a'$, all $n\geq 1$ and $y>0$,
	\begin{equation}\label{eq:convoEst}
	\abs{f^{*n}_a(y)}\leq ||f_{a'}||_{\infty}^n\frac{y^{n-1}e^{-\lbrb{a-a'}y}}{\lbrb{n-1}!}.
	\end{equation}
\end{proposition}
\begin{proof}
	First $f^{*n}_a(y)=e^{-ay}f^{*n}(y),\,y\in\intervalOI,$ is a triviality.
	Then, \eqref{eq:convoEst} is proved by an elementary inductive hypothesis based on the immediate observation that, for any $y>0$,
	\[\abs{f^{*2}_a(y)}=e^{-ay}\abs{\int_{0}^{y}f(y-v)f(v)dv}\leq ||f_{a'}||^2_{\infty} ye^{-\lbrb{a-a'}y}.\]
\end{proof}
Let us examine $\abs{\frac{\Gamma(z)}{\Wpp\!(z)}}$, that is \eqref{eq:Ypn}, first.
\begin{proposition}\label{prop:condClass1}
	Let $\phi\in\BP^c$ then for any $ u\geq 0$ and $a>0$ fixed
	\begin{equation}\label{eq:condClass1}
	\limi{|b|}|b|^u\abs{\frac{\Gamma\!\lbrb{a+ib}}{\Wp\!\lbrb{a+ib}}}=0.
	\end{equation}
	If $\phi\in\BP$ then \eqref{eq:condClass1} holds for any $u< \frac{1}{\dr}\lbrb{\phi(0)+\mubr\!\lbrb{0}}\in\lbrbb{0,\infty}$. In fact, if $\mubr\!\lbrb{0}<\infty$ the limit in \eqref{eq:condClass1} is infinity for all $u>\frac{1}{\dr}\lbrb{\phi(0)+\mubr\!\lbrb{0}}$. Finally, regardless of the value of $\mubr(0)$,  we have, for any $a>0$ such that $\dr^{-1}\IntOI e^{-ay}\mubar{y}dy<1$, that as $b\to\infty$,
	\begin{equation}\label{eq:condClass11}
	\abs{\frac{\Gamma\!\lbrb{a+ib}}{\Wp\!\lbrb{a+ib}}}\lesssim b^{\so{1}}e^{-\frac{\mubr\lbrb{\frac{1}{b}}+\phi(0)}{\dr}\ln |b| }.
	\end{equation}
\end{proposition}
\begin{proof}
	Let $\phi\in\Bc$. Fix $a>0$ and without loss of generality assume that $b>0$. Applying \eqref{eq:GammaStirling} to $\abs{\Gamma\lbrb{a+ib}}$ and \eqref{eq:Stirling} to $\abs{\Wp\lbrb{a+ib}}$ we get, as $b\to\infty$, that
	\begin{equation}\label{eq:condClass1_1}
	\abs{\frac{\Gamma\!\lbrb{a+ib}}{\Wp\!\lbrb{a+ib}}}\asymp b^{a-\frac12}\sqrt{\abs{\phi(a+ib)}}e^{\Aph(a+ib)-\frac{\pi}{2}b}.
	\end{equation}
	It therefore remains to estimate $\Aph(a+ib)$ in the different scenarios stated. Let us start with $\dr=0$ or equivalently $\phi\in\BP^c$.
	Let $v>0$. From \eqref{eq:ineqArg} we have that $\labsrabs{\arg \phi\!\lbrb{a+iu}}\leq\pi/2$ and from \eqref{eq:argPhi} of Proposition \ref{propAsymp1}\eqref{it:argPhi} we conclude that there is $u(v)>0$ such that,  for  $u\geq u(v)$, 
	\[\labsrabs{\arg \phi\!\lbrb{a+iu}}\leq\frac{\pi}{2}-\arctan\!\lbrb{\frac{v\phi(a)}u}.\]
	Therefore, from the definition of $\Aph$, see \eqref{eq:Aphi}, we get that for any $b>u\lbrb{v}$,
	\[\Aph\!\lbrb{a+ib}\leq \int_{0}^{b}\labsrabs{\arg \phi\!\lbrb{a+iu}}du \leq \frac{\pi}{2}b-\int_{u\lbrb{v}}^{b}\arctan\!\lbrb{\frac{v\phi(a)}{u}}du.\]
	However, since $\arctan x\simo x,$ we see that there exists $ \,u'$ big enough such that  for any $ b>u'>u(v)$ we have the inequality
	\[\Aph\!\lbrb{a+ib}\leq\frac{\pi}{2}b-\frac{v\phi(a)}{2}\lbrb{\ln b-\ln u'}.\]
	Plugging this in \eqref{eq:condClass1_1} and using the fact that, for a fixed $a>0$, $\abs{\phi\!\lbrb{a+ib}}\stackrel{\infty}{=}\sospace{\abs{a+ib}},$ when $\dr=0$, see  Proposition \ref{propAsymp1}\eqref{it:asyphid}, we easily get that, as $b\to\infty$,
	\[\abs{\frac{\Gamma\!\lbrb{a+ib}}{\Wp\!\lbrb{a+ib}}}\lesssim b^{a}e^{-\frac{v\phi(a)}{2}\ln b}=b^{a-\frac{v\phi(a)}{2}}.\]
	Since $v$ is arbitrary we conclude \eqref{eq:condClass1} when $\phi\in\BP^c$. Assume next that $\phi\in\BP$ and without loss of generality that $b>0$. Then from \eqref{eq:AphiAsymp2} of Proposition \ref{thm:imagineryStirling} we get for the exponent of \eqref{eq:condClass1_1} that, as $b\to\infty$ and any fixed $a>0$,
	\begin{equation*}
	\begin{split}
	\Aph(a+ib)-\frac{\pi}{2}b&=-b\arctan\!\lbrb{\frac{a}{b}}-\lbrb{a+\frac{\phi(0)}{\dr}}\!\ln b\\
	&-\frac{\lbrb{a+\frac{\phi(0)}{\dr}}}2\lbrb{\ln\!\lbrb{1+\frac{b^2}{a^2}}-\ln b^2}-\int_{0}^{\infty}\frac{1-\cos(by)}{y}\Logm{\bar{\mu}_{a,\dr}}(dy)+\bar{A}_\phi(a,b)\\
	&= -\lbrb{a+\frac{\phi(0)}{\dr}+\sospace{1}}\ln b-\int_{0}^{\infty}\frac{1-\cos(by)}{y}\Logm{\bar{\mu}_{a,\dr}}(dy),
	\end{split}
	\end{equation*}
	where we have used implicitly that $\arctan \frac{a}{b}+\arctan \frac{b}{a}=\frac{\pi}{2}$ and $\bar{A}_\phi(a,b)=\sospace{\ln b}$, see Proposition \ref{thm:imagineryStirling}. Therefore, as $b\to\infty$, \eqref{eq:condClass1_1} is simplified to
	\begin{equation*}
	\abs{\frac{\Gamma\!\lbrb{a+ib}}{\Wp\!\lbrb{a+ib}}}\asymp b^{a-\frac12}\sqrt{\abs{\phi(a+ib)}}e^{-\lbrb{a+\frac{\phi(0)}{\dr}+\so{1}}\!\ln b-\int_{0}^{\infty}\frac{1-\cos(by)}{y}\Logm{\mubr_{a,\dr}}(dy)}.
	\end{equation*}
	We recall that $\Logm{\bar{\mu}_{a,\dr}}(dy)$ is the measure associated to the absolutely continuous measure $\mubr_{a,\dr}(y)dy=\dr^{-1}e^{-ay}\mubr(y)dy$ as defined via its Fourier transform in \eqref{eq:logFT}.
	When $\phi\in\BP$ we have that $\abs{\phi(a+ib)}\sim\dr b,$ as $b\to\infty$ and $a>0$ fixed,  see  Proposition \ref{propAsymp1}\eqref{it:asyphid}, and thus \eqref{eq:condClass1_1} is simplified further to
	\begin{equation}\label{eq:condClass1_2}
	\abs{\frac{\Gamma\!\lbrb{a+ib}}{\Wp\!\lbrb{a+ib}}}\asymp e^{-\lbrb{\frac{\phi(0)}{\dr}+\so{1}}\ln b-\int_{0}^{\infty}\frac{1-\cos(by)}{y}\Logm{\mubr_{a,\dr}}(dy)}.
	\end{equation}
	Next, choose $a>a_0>0$ so as to have 
	\[||\mubr_{a,\dr}||_{1}=\dr^{-1}\IntOI e^{-ay}\mubr(y)dy<1.\]
	Then \eqref{eq:arctan2} of Proposition \ref{thm:imagineryStirling} applies and yields
	\begin{equation*}
	\abs{\frac{\Gamma\!\lbrb{a+ib}}{\Wp\!\lbrb{a+ib}}}\lesssim b^{\so{1}} e^{-\frac{\phi(0)}{\dr}\ln b-\frac{\bar{\mu}\lbrb{\frac{1}{b}}}{\dr}\ln b}.
	\end{equation*}
	This is precisely \eqref{eq:condClass11} regardless of the value of $\mubr(0)$. Moreover, \eqref{eq:condClass11} also settles \eqref{eq:condClass1}  for those $a>a_0$ and any $u\geq 0$, whenever $\mubr(0)=\infty$. However, since
	\[\labsrabs{\frac{\Gamma\!\lbrb{1+a+ib}}{\Wp\!\lbrb{1+a+ib}}}=\frac{\labsrabs{a+ib}}{\labsrabs{\phi\!\lbrb{a+ib}}}\labsrabs{\frac{\Gamma\!\lbrb{a+ib}}{\Wp\!\lbrb{a+ib}}}\]
	and $\limi{b}\frac{\labsrabs{a+ib}}{\labsrabs{\phi\lbrb{a+ib}}}=\dr^{-1}$, see Proposition \ref{propAsymp1}\eqref{it:asyphid}, we trivially conclude \eqref{eq:condClass1}, when $\mubr(0)=\infty$, for any $ a>0,\,u\geq 0$. Next, let $\mubr(0)<\infty$  and choose again $a>a_0$ so that $||\mubr_{a,\dr}||_{1}<1$. Then \eqref{eq:convo} holds and  thanks to Proposition \ref{prop:convoEst} we have that $\mubr^{*n}_{a,\dr}(y)=e^{-ay}\frac{\mubr^{*n}(y)}{\dr^n}$. Therefore, on $\intervalOI$,
	\begin{equation*}
	\begin{split}
	\Logm{\mubr_{a,\dr}}(dy)&=\lbrb{\sum_{n=1}^{\infty}\lbrb{- 1}^{n-1}\frac{\mubr^{*n}_{a,\dr}(y)}{n}}dy 		\\
	&=e^{-ay}\lbrb{\sum_{n=1}^{\infty}\lbrb{- 1}^{n-1}\frac{\mubr^{*n}(y)}{\dr^n n}}dy=\frac{1}{\dr}e^{-ay}\mubr(y)dy+h(y)dy.
	\end{split}	
	\end{equation*}	
	We will use this decomposition to simplify the last term in \eqref{eq:condClass1_2}. Now since $\mubr(0)<\infty$ and therefore $\mubr\in\Lspaces{\infty}{\Rp}$ we have from \eqref{eq:convoEst} in Proposition \ref{prop:convoEst} applied with $a'=0$ that
	\[\abs{h(y)}\leq e^{-ay}\sum_{n=2}^{\infty}\frac{\mubr^{*n}(y)}{\dr^n n}\leq ye^{-ay}\sum_{n=2}^{\infty}\frac{\mubr^n(0)y^{n-2}}{\dr^n n!}.\]
	Thus
	\begin{eqnarray*}
		\abs{\int_{0}^{\infty}\frac{1-\cos(by)}{y}h(y)dy}\leq 2\sum_{n=2}^{\infty}\frac{\mubr^n\!(0)}{n(n-1)\dr^na^{n-1}}<\infty\iff a\geq \frac{\mubr(0)}{\dr}.
	\end{eqnarray*}
	Therefore, if we choose $a>a_0\vee \frac{\mubr(0)}{\dr}$ we see that the term above does not contribute to the asymptotic in \eqref{eq:condClass1_2}. We are then left with the relation
	\begin{equation}\label{eq:condClass1_3}
	\abs{\frac{\Gamma\!\lbrb{a+ib}}{\Wp\!\lbrb{a+ib}}}\asymp e^{-\lbrb{\frac{\phi(0)}{\dr}+\so{1}}\ln b-\frac{1}{\dr}\int_{0}^{\infty}\frac{1-\cos(by)}{y}e^{-ay}\mubr(y)dy}.
	\end{equation}
	To investigate the integral term in the exponent of \eqref{eq:condClass1_3} we split up the range of integration in three regions: $\lbrbb{0,\frac{1}{b}},\lbrb{\frac{1}{b},1},\lbrb{1,\infty}.$
	First, since $y\mapsto e^{-ay}\frac{\mubr(y)}{y}$ is integrable on $\lbrb{1,\infty}$, the Riemann-Lebesgue lemma yields that
	\begin{equation}\label{eq:condClass1_3_1}
	\limi{b}\int_{1}^{\infty}\frac{1-\cos(by)}{y}e^{-ay}\mubr(y)dy=\int_{1}^{\infty} e^{-ay}\mubr(y)\frac{dy}{y}.
	\end{equation}
	Next, recall that $\mubr(0)<\infty$. Henceforth, from the dominated convergence theorem
	\begin{equation}\label{eq:condClass1_3_2}  	\begin{split}  	
	\limi{b}\int_{0}^{\frac{1}{b}}\frac{1-\cos(by)}{y}e^{-ay}\mubr(y)dy&=\limi{b}\int_{0}^{1}\frac{1-\cos y}{y}e^{-a\frac{y}b}\mubr\!\lbrb{\frac{y}{b}}dy\\
	&=\mubar{0}\int_{0}^{1}\frac{1-\cos y }{y}dy.
	\end{split}
	\end{equation}
	Recall that $y\mapsto \mubr_a(y)=e^{-ay}\mubr(y)$ is decreasing on $\Rp$, thus defining a measure $\mubr_a(dy)$ on $\intervalOI$. Therefore, the remaining portion of the integral in the exponent of relation \eqref{eq:condClass1_3} can be written as
	\begin{equation}\label{eq:condClass1_3_3}
	\begin{split}
	\int_{\frac{1}{b}}^{1}\frac{1-\cos(by)}{y}e^{-ay}\mubr(y)dy&=\int_{\frac{1}{b}}^{1}\frac{1-\cos(by)}{y}\lbrb{\mubr_a(y)-\mubr_a\!\lbrb{\frac{1}{b}}+\mubr_a\!\lbrb{\frac{1}{b}}}dy\\
	&=\mubr_a\!\lbrb{\frac{1}{b}}\lbrb{\ln b-\int_{1}^{b}\frac{\cos y }{y}dy}\\
	&-\int_{\frac{1}{b}}^{1}\lbrb{\mubr_a\!\lbrb{\frac1b}-\mubr_a(y)}\frac{dy}{y}+\int_{\frac{1}{b}}^{1}\frac{\cos(by)}{y}\int_{\frac{1}{b}}^{y}\mubr_a(dv)dy.
	\end{split}
	\end{equation}
	Clearly, then
	\begin{align*}  			\limsupi{b}\abs{\int_{\frac{1}{b}}^{1}\frac{\cos(by)}{y}\int_{\frac{1}{b}}^{y}\mubr_a(dv)dy}&\leq\limsupi{b}\int_{\frac{1}{b}}^{1}\abs{\int_{bv}^{b}\frac{\cos y}{y}dy}\mubr_a(dv)=0,
	\end{align*}
	since, for all $v\in\lbrb
	{0,1}$,
	\[\limsupi{b}\abs{\int_{bv}^{b}\frac{\cos y }{y}dy}=\limsupi{b}\abs{\int_{v}^{1}\frac{\cos by }{y}dy}=0\]
	and the validity of the dominated convergence theorem which is due to 
	\[\sup_{x\geq 1}\abs{\int_{x}^{\infty}\frac{\cos y }{y}dy}<\infty\]
	 and \[\int_{0}^{1}|\mubr_a(dv)|=\mubar{0}-e^{-a}\mubar{1}<\infty.\] Henceforth,  from \eqref{eq:condClass1_3_1}, \eqref{eq:condClass1_3_2} and \eqref{eq:condClass1_3_3}, we obtain, as $b\to\infty$, that
	\begin{equation*}
	\begin{split}
	\abs{\int_{0}^{\infty}\frac{1-\cos(by)}{y}e^{-ay}\mubr(y)dy-\mubr(0)\ln b}&\leq \lbrb{\mubr(0)-\mubr_a\!\lbrb{\frac{1}{b}}}\ln b \\
	&\quad\,\,+\int_{\frac{1}{b}}^{1}\lbrb{\mubr(0)-\mubr_a(y)}\frac{dy}{y}+\bospace{1}.
	\end{split}
	\end{equation*}
	However, since $\limo{y}\mubr(y)=\mubr(0)$, it can be seen easily that the first and the second term on the right-hand side are of order $\sospace{\ln b }$ and therefore, as $b\to\infty$,
	\begin{equation*}
	\int_{0}^{\infty}\frac{1-\cos(by)}{y}e^{-ay}\mubr(y)dy=\mubr(0)\ln b +\sospace{\ln b }.
	\end{equation*}
	This fed in \eqref{eq:condClass1_3} yields  that
	\begin{equation}\label{eq:condClass1_4}
	\abs{\frac{\Gamma\!\lbrb{a+ib}}{\Wp\!\lbrb{a+ib}}}\asymp e^{-\frac{\phi(0)}{\dr}\ln b-\frac{\mubr(0)}{\dr}\ln b+\so{\ln b}},
	\end{equation}
	which proves \eqref{eq:condClass1} for $u<\frac{1}{\dr}\lbrb{\phi(0)+\mubar{0}}$ and shows that the limit in \eqref{eq:condClass1} is infinity for $u>\frac{1}{\dr}\lbrb{\phi(0)+\mubar{0}}$. This concludes the proof of Proposition \ref{prop:condClass1}.
\end{proof}
Proposition \ref{prop:condClass1} essentially deals exhaustively with the proof of Theorem \ref{thm:asympMPsi} via the term $\abs{\frac{\Gamma(z)}{\Wpp\!(z)}}$ in \eqref{eq:Ypn}. Since $\abs{\frac{\Gamma(z)}{\Wpp\!(z)}}$  decays faster than any polynomial except when $\php\in\BP$ and $\mubarpspace{0}<\infty$, it remains to discuss this scenario. Before stating and proving Proposition \ref{prop:condClass2} we introduce some more notation needed throughout below.   We recall that with any $\phi\in\Bc$ we have an associated non-decreasing (possibly killed at independent exponential rate of parameter $\phi(0)> 0$) \LLP $\xi$. Then the potential measure of $\xi$ and therefore of $\phi$ is defined as
\begin{equation}\label{eq:potentialMeasure}
U(dy)=\IntOI e^{-\phi(0)t}\Pbbs{\xi_t\in dy}dt,\,y>0.
\end{equation}
The latter combined with Fubini's theorem gives, for $z\in\CbOI$,
\begin{equation}\label{eq:LTU}
\IntOI e^{-zy}U(dy)=\frac{1}{\phi(z)},
\end{equation}
which is the expression in  Proposition \ref{propAsymp1}\eqref{it:bernstein_cmi}.
The renewal or potential function $U(y)=\int_{0}^{y}U(dx),\,y>0$, is subadditive on $\intervalOI$. Recall that if $\Psi\in\overNc$ then $\Psi(z)=-\php(-z)\phn(z)$ and we have two potential measures $U_\pms$ related to $\phi_\pms$ respectively.
If in addition $\phi\in\BP$ then  it is well known from \cite[Chapter III]{Bertoin-96} or from Proposition \ref{propAsymp1}\eqref{it:bernstein_cmi} that the potential density $u(y)=\frac{U(dy)}{dy}$ exists. Moreover, it is continuous, strictly positive and bounded on $\lbbrb{0,\infty}$, that is $||u||_\infty<\infty$. Furthermore, \cite[Proposition 1]{Doering-Savov-11} establishes that in this case
\begin{equation}\label{eq:u}
u(y)=\sum_{j=0}^{\infty}\frac{\lbrb{-1}^j}{\dr^{j+1}}\lbrb{\mathbf{1}*(\phi(0)+\mubr)^{*j}}\!(y)=\frac{1}{\dr}+\tilde{u}(y),\,y\geq0,
\end{equation}
where $\mathbf{1}(y)=\ind{y>0}$ stands for the Heavyside function and recall that $f*g(x)=\int_{0}^{x}f(x-v)g(v)dv$ represents the convolution of two functions supported on $\Rp$. We keep the last notation for convolutions of measures too. Then we have the result.
\begin{proposition}\label{prop:condClass2}
	Let $\php\in\BP$ and $\mubrp\!\lbrb{0}<\infty$. Then, we have the following scenarios.
	\begin{enumerate}
		\item \label{it:class2_A} If $\phn\in\BP$ then $\phn\in \Bthd$.
		\item \label{it:class2_B} If $\phn\in\BP^c$ then $\phn\in\PcBinf\iff\,\PPn(0)=\infty$.
		\item \label{it:class2_C} If $\phn\in\BP^c$ and $\PPn(0)<\infty$ then for any $u<\frac{\IntOI u_{\pls}(y)\Pi_{\mis}(dy)}{\phn(0)+\mubrn(0)}$ and any $a>0$ we have that \[\limi{|b|}|b|^u\abs{\Wpn\!\lbrb{a+ib}}=0.\] If $u>\frac{\IntOI u_{\pls}(y)\Pi_{\mis}(dy)}{\phn(0)+\mubrn(0)}$ then for any $a>0$ the following limit is valid \[\limi{|b|}|b|^u\abs{\Wpn\!\lbrb{a+ib}}=\infty.\] 
	\end{enumerate}
\end{proposition}
\begin{remark}\label{rem:condClass2}
	For general $\phi\in\Bc$ the relation $\phi\in\PcBinf$ can be established as long as the following three conditions are satisfied. First, the \LL measure of $\phi$ is absolutely continuous with right-continuous density, that is $\mu(dy)=\upsilon(y)dy,\,y>0,$ and $\upsilon(0^+)=\infty$. Secondly,  $\upsilon(y)=\upsilon_1(y)+\upsilon_2(y)$ such that $\upsilon_1,\upsilon_2\in\Lspaces{1}{\Rp}$ and $\upsilon_1\geq 0$ is non-increasing on $\Rp$. Finally, $\IntOI\upsilon_2(y)dy\geq 0$ and $\abs{\upsilon_2(x)}\leq \lbrb{\int_{x}^{\infty}\upsilon_1(y)dy}\vee C$ for some $C>0$ on $\Rp$. Imposing $\php\in\BP$ and $\mubrp\!\lbrb{0}<\infty$ ensures precisely those conditions for $\phn$. In fact Lemma \ref{lem:aux}, Lemma \ref{prop:tau}\eqref{it:tau} and Lemma \ref{lem:aux1} modulo to \eqref{eq:residual}, \eqref{eq:red} serve only the purpose to check the validity of those conditions. For item \eqref{it:class2_C} for general $\phi\in\Bc$ it suffices to assume that $\mu(dy)=\upsilon(y)dy,\,y>0$, $\upsilon\!\lbrb{0^+}=\limo{y}\upsilon\!\lbrb{y}<\infty$ and $\upsilon\in\Lspaces{\infty}{\Rp}$.
\end{remark}
We split the proof into several steps and start with a sequence of lemmas.
The first auxiliary result uses the expansion \eqref{eq:u} of the potential density $u_{\pls}$ and an extension to killed \LLPs of Vigon's \textit{\'{e}quation amicale invers\'{e}e}  that is stated in Proposition \ref{prop:Vigon}, to decompose and relate the \LL measure of $\phn$ to the \LL measure associated to the \LLP underlying $\Psi\in\overNc$. The essential aim of this decomposition is to obtain a Bernstein function whose \LL measure possesses a non-increasing density. This will lead us into a setting with well-understood results. 
\begin{lemma}\label{lem:aux} \label{prop:tau}
	Let $\Psi\in\overNc$ such that $\php\in\BP$ and $\mubarpspace{0}<\infty$ and recall, from  \eqref{eq:u}, the decomposition of the potential density $u_{\pls}(y)=\frac{1}{\dep}+\tilde{u}_{\pls}(y),\,y\geq 0$.
	Then, there exists $ c_0=c_0(\Psi)>0$ such that for any $c\in\lbrb{0,c_0}$ the following holds.
	\begin{enumerate}[(1)]
		\item We have the  identity of measures  on $\lbrb{0,c}$
		\begin{align}\label{eq:mu_-_1}
		\munspace{dy}&=\frac{\PPn^{(c)}\!(y)+\PPn(c)-\PPn(y+c)}{\dep}dy\\
		&+
		\int_{0}^{\infty}(\tilde{u}_{\pls}(v)\ind{v<c} +  u_{\pls}(v)\ind{v>c})\Pnn(y+dv)dy \nonumber \\
		&=\frac{\PPn^{(c)}\!(y)}{\dep}dy+\tau^{(c)}_1\!(y)dy+\tau^{(c)}_2\!(y)dy,
		\end{align}
		where $\PPn^{(c)}\!(y)=\lbrb{\PPn(y)-\PPn(c)}\ind{y\leq c}$,
		\begin{equation}\label{eq:tau1}
		\tau^{(c)}_1\!(y)=\ind{y\leq\frac{c}{2}}\int_{0}^{\frac{c}{2}}\tilde{u}_{\pls}(v)\Pnn(y+dv)
		\end{equation}
		and
		\begin{equation}\label{eq:tau2}
		\begin{split}
		\tau^{(c)}_2\!(y)&=\lbrb{\int_{0}^{\frac{c}{2}}\tilde{u}_{\pls}(v)\Pnn(y+dv)\ind{y\in\lbrb{\frac{c}{2},c}}+\int_{\frac{c}{2}}^{c}\tilde{u}_{\pls}(v)\Pnn(y+dv)\ind{y<c}}\\
		&+\frac{\lbrb{\PPn(c)-\PPn(y+c)}}{\dep}\ind{y<c}+\lbrb{\int_{c}^{\infty}u_{\pls}(v)\Pnn(y+dv)}\ind{y<c}.
		\end{split}
		\end{equation}
		\item  $\sup\limits_{y\in\lbrb{0,c}
		}\abs{\tau^{(c)}_2\!(y)}<\infty$, and, for some $C_1>0$,
		\begin{equation}\label{eq:boundTau}
		\abs{\tau^{(c)}_1\!(y)}\leq C_1\lbrb{\int_{0}^{\frac{c}{2}}\PPn^{(c)}\!\lbrb{\rho+y}-\PPn^{(c)}\!\lbrb{\frac{c}{2}+y}d\rho}\ind{y\leq\frac{c}{2}}.
		\end{equation}
		\item \label{it:tauint}$\tau^{(c)}_1$ and $\tau^{(c)}_2$ are absolutely integrable on $\lbrb{0,c}$.
		\item  \label{it:tau} Finally, for any $d\in\lbbrb{0,\mubarnspace{0}}$ there exists $c_1=c_1(d,\Psi)\in\lbrb{0,c_0}$ (with $c_0$ as above) such that $\IntOI \tau^{(c)}\!(y)dy\in\lbrb{d,\mubarnspace{0}}$ for all $c\in\lbrb{0,c_1}$, where
		\begin{equation}\label{def:tau}
		\tau^{(c)}\!(y):= \tau^{(c)}_1\!(y)+ \tau^{(c)}_2\!(y)+\uunspace{y}\ind{y\geq c}
		\end{equation}
		and $\uunspace{y}$ is the density of $\munspace{dy}$.
	\end{enumerate}
\end{lemma}
\begin{proof}
	Since  $\php\in\BP$ and the underlying \LLP is not a compound Poisson process, from \eqref{eq:mu_-2} of Proposition \ref{prop:VigonDens}  we that
	\begin{equation}\label{eq:mu_-1}
	\begin{split}
	\munspace{dy}&=\uunspace{y}dy=\lbrb{\IntOI u_{\pls}(v)\Pnn(y+dv)}dy.
	\end{split}
	\end{equation}
	Note from \eqref{eq:Vigon} that 
	\begin{equation}\label{eq:muinfPinf1}
	\begin{split}
	\mubarnspace{0}=\IntOI\uunspace{y}dy=\infty&\iff \int_{0}^{1}\PPn(y)u_\pls(y)dy=\infty\\
	&\iff \int_{0}^{1}\PPn(y)dy=\infty
	\end{split}
	\end{equation}
	because $u_{\pls}>0$ on $\lbbrbb{0,1}$.
	Next, for any $c>0$, we decompose  \eqref{eq:mu_-1} when $y\in\lbrb{0,c}$ as follows
	\begin{equation}\label{eq:mu_-}
	\begin{split}
	\munspace{dy}&=\uunspace{y}dy=\lbrb{\IntOI u_{\pls}(v)\Pnn(y+dv)}dy
	\\
	&=\lbrb{\int_{0}^{c} u_{\pls}(v)\Pnn(y+dv)}dy+\lbrb{\int_{c}^{\infty}u_{\pls}(v)\Pnn(y+dv)}dy.
	\end{split}
	\end{equation}
	Next recalling from \eqref{eq:u}  that
	\begin{equation}\label{eq:u_+}	u_{\pls}(y)=\sum_{j=0}^{\infty}\frac{\lbrb{-1}^j}{\dr^{j+1}_{\pls}}\lbrb{\mathbf{1}*\lbrb{\php(0)+\mubrp}^{*j}}\!(y)=\frac{1}{\dep}+\tilde{u}_{\pls}(y),
	\end{equation}
	one gets, by plugging the right-hand side of \eqref{eq:u_+} in the first term of the last identity of \eqref{eq:mu_-}, the first identity in \eqref{eq:mu_-_1}. The expressions for $\tau^{(c)}_1,\,\tau^{(c)}_2$, see \eqref{eq:tau1} and \eqref{eq:tau2}, are up to a mere choice. Trivially, for the second term to the right-hand side of \eqref{eq:tau2}, we get that
	\[\frac{1}{\dep}\lbrb{\PPn(c)-\PPn(y+c)}\ind{y< c}\leq \frac{1}{\dep}\PPn(c)\ind{y\leq c}\]
	and it is finite and integrable on $\lbrb{0,c}$. Also since $||u_{\pls}||_\infty<\infty$, which is thanks to $\php\in\BP$, we deduce that
	\[\lbrb{\int_{c}^{\infty}u_{\pls}(v)\Pnn(y+dv)}\ind{y<c}\leq||u_{\pls}||_\infty\PPn(y+c)\ind{y<c}\leq||u_{\pls}||_\infty\PPn(c)\ind{y<c}.\]
	Clearly, the upper bound is finite and integrable on $\lbrb{0,c}$. Thus, the last term  to the right-hand side of \eqref{eq:tau2} has been dealt with as well. Finally, using in an evident manner  \eqref{eq:u_+}, we study the first term to the right-hand side of \eqref{eq:tau2}, to get, writing $C_{\dep,\infty}=\frac{1}{\dep}+||u_{\pls}||_\infty$ with $||\tilde{u}_{\pls}||_\infty\leq C_{\dep,\infty}$, that
	\begin{equation}\label{eq:split}
	\begin{split}
	&\abs{\int_{0}^{c}\tilde{u}_{\pls}(v)\left(\ind{v<\frac{c}{2}}\ind{\frac{c}{2}<y<c}+\ind{\frac{c}{2}<v<c}\ind{y<c}\right)\Pnn(y+dv)}\\
	&\leq  C_{\dep,\infty} \lbrb{\PPn\lbrb{y}\ind{\frac{c}{2}<y<c}+\PPn\lbrb{\frac{c}{2}}\ind{y<c}}  \\
	& \leq  2C_{\dep,\infty} \: \PPn\lbrb{\frac{c}{2}}\ind{y<c}.
	\end{split}
	\end{equation}
	However, the upper bound in \eqref{eq:split} is clearly both bounded and integrable on $\lbrb{0,c}$. Thus, we have proved that  $\sup_{y\in\lbrb{0,c}
	}\abs{\tau^{(c)}_2(y)}<\infty$ and $\tau^{(c)}_2$ is absolutely integrable on $\lbrb{0,c}$. It remains to investigate $\tau^{(c)}_1$. 
	Note that the term defining \eqref{eq:u_+} has, for any $x\geq 0$, the form
	\[\mathbf{1}*\lbrb{\php(0)+\mubrp}\!(x)=\int_{0}^{x}\lbrb{\php(0)+\mubarpspace{y}}dy=\php(0)x+\int_{0}^{x}\mubarpspace{y}dy.\]
	Since $\mubarpspace{0}<\infty$ we conclude that \[\mathbf{1}*\lbrb{\php(0)+\mubrp}\!(x)\simo \lbrb{\php(0)+\mubarp{0}}\!x.\]
	Then, since $1*\lbrb{\php(0)+\mubrp}$ is non-decreasing on $\Rp$, \cite[(4.2)]{Doering-Savov-11} gives, for any $j\in \N$, that
	\[\lbrb{\mathbf{1}*\lbrb{\php(0)+\mubrp}^{*j}}\!(x)\leq\lbrb{\mathbf{1}*\lbrb{\php(0)+\mubarp{x}}}^j\]
	and we conclude that for some $h>0$ and all $x\in\lbrb{0,h}$,
	\[\lbrb{\mathbf{1}*\lbrb{\php(0)+\mubrp}^{*j}}\!(x)\leq 2^j\lbrb{\php(0)+\mubarpspace{0}}^jx^{j}.\]
	Therefore from \eqref{eq:u_+},
	for $x<\frac{1}{4\lbrb{\php(0)+\mubarpspace{0}}}\wedge h$,
	\begin{equation}\label{eq:tildeu_+}
	\abs{\tilde{u}_{\pls}(x)}\leq C_1x,
	\end{equation}
	where $C_1>0$ is some positive constant.
	Hence, from now on, we choose an arbitrary $c<c_0=\frac{1}{4\lbrb{\php(0)+\mubarpspace{0}}}\wedge h$.  Using \eqref{eq:tildeu_+} in \eqref{eq:tau1} we get that
	\begin{equation*}
	\abs{\tau^{(c)}_1\!(y)}\leq C_1\lbrb{\int_{0}^{\frac{c}{2}}v\Pnn(y+dv)}\ind{y\leq\frac{c}{2}}
	\end{equation*}
	and \eqref{eq:boundTau} follows by integration by parts. However, from \eqref{eq:boundTau} we get for $y\in\lbrb{0,\frac{c}{2}}$ that
	\begin{equation}\label{eq:barbarP}
	\begin{split}
	\frac{\tau^{(c)}_1\!(y)}{C_1}&\leq\int_{0}^{\frac{c}{2}}\lbrb{\PPn^{(c)}\!\lbrb{\rho+y}-\PPn^{(c)}\!\lbrb{\frac{c}{2}+y}}d\rho
	\\
	&\leq  \PPPn^{(c)}\!(y):=\int_{y}^{c}\PPn^{(c)}\!(v)dv
	\end{split}
	\end{equation}
	and since
	\[\int_{0}^c \PPPn^{(c)}\!(y)dy=\int_{0}^c\int_{y}^{c}\PPn^{(c)}\!(v)dvdy\leq \int_{0}^c\int_{y}^{c}\PPn(v)dvdy=\int_{0}^{c}v\PPn(v)dv<\infty,\]
	we conclude that $\abs{\tau^{(c)}_1}$ is integrable on $\lbrb{0,\frac{c}{2}}$. This completes item \ref{it:tauint}.
	%
	To show item \eqref{it:tau}, we note, from \eqref{def:tau} and Lemma \ref{lem:aux}, since $ \tau^{(c)}_1, \tau^{(c)}_2$ are supported on $\lbrb{0,c}$, that for all $c\in\lbrb{0,c_0}$, \[\int_{c}^{\infty}\tau^{(c)}\!(y)dy=\int_{c}^{\infty}\uunspace{y}dy=\mubarnspace{c}.\]
	By a simple inspection of \eqref{eq:tau1} and \eqref{eq:tau2} we note that the only potential negative contribution to $\tau^{(c)}$ comes from the terms whose integrands are $\tilde{u}_{\pls}$.  Since \eqref{eq:tildeu_+}, that is 	$\abs{\tilde{u}_{\pls}(x)}\leq C_1x$, holds for all $x\in\lbrb{0,c_0}$ clearly an upper bound of the absolute value of any of those terms is the following expression
	\begin{equation*}
	\overline{\tau}(y)=C_1\int_{0}^{c}v\Pnn(y+dv)\ind{y<c}.
	\end{equation*}
	Therefore, integrating by parts, we get that
	\begin{equation*}
	\begin{split}
	\int_{0}^{c}\overline{\tau}(y)dy&=C_1\int_{0}^{c}\int_{0}^{c}\lbrb{\PPn(y+w)-\PPn(y+c)}dwdy\\
	& \leq C_1\int_{0}^{c}\int_{0}^{c}\PPn(y+w)dwdy\leq C_1\int_{0}^{c}y\PPn(y)dy+C_1c\PPn(c).
	\end{split}
	\end{equation*}
	As the upper bound tends to $0$, as $c\to 0$, we conclude the claim as the negative contribution of $\tau^{(c)}_1,\tau^{(c)}_2$ cannot exceed in absolute value this quantity. Indeed, since $\uun{y}=\tau^{(c)}(y)+\frac{1}{\dep}\PPn^{(c)}(y)\ind{y\leq c}$ and 
	\[\limo{c}\int_{c}^{\infty}\tau^{(c)}\!(y)dy=\mubarnspace{0}=\IntOI \uunspace{y}dy\] the claim follows.
\end{proof}
Lemma \ref{lem:aux} allows us to prove the following result which transforms the decomposition of $\mu_{\mis}$ on $\lbrb{0,c}$ to a decomposition of $\phn$. We stress that although one of the terms in the aforementioned decomposition is a Bernstein function, the second term need not belong to $\Bc$.
\begin{lemma}\label{lem:aux1}
	Let $\Psi\in\overNc$ such that $\php\in\BP$ and $\mubarpspace{0}<\infty$. Let $c\in\lbrb{0,c_0}$ so that Lemma \ref{lem:aux} is valid. Then with $\ptt=\min\curly{1,\mubar{0}/4}\ind{\phn(0)=0}$, the function
	\begin{equation}\label{eq:phic}
	\phntc\!\!\lbrb{z}=\phn(0)+\ptt+\dr_{\mis}z+\int_{0}^{c}\lbrb{1-e^{-zy}}\frac{\PPn^{(c)}\lbrb{y}}{\dep}dy\in\Bc
	\end{equation}
	and with the definition of $\tau^{(c)}$, see \eqref{def:tau},
	\begin{equation}\label{eq:phi2}
	\phn(z)=\phntc\!(z)+\IntOI\lbrb{1-e^{-zy}}\tau^{(c)}(y)dy-\ptt=\phntc\!(z)+\phnstc\!(z).
	\end{equation}
	For any such choice and $a>0$ fixed
	\begin{equation}\label{eq:residual}
	\limi{|b|}\Im\!\lbrb{\phnstc\!(a+ib)}=0\quad \textrm{ and } \quad\limi{|b|}\Re\!\lbrb{\phnstc\!(a+ib)}=\phnstc\!\lbrb{\infty}=\IntOI\tau^{(c)}(y)dy-\ptt,
	\end{equation}
	whereas
	\begin{equation}\label{eq:red}
	\limi{|b|}\Re\!\lbrb{\phntc\!(a+ib)}=\phntc\!\!\lbrb{\infty}\textrm{ and if  $b>0$ then } \Im\!\lbrb{\phntc\!(a+ib)}\geq 0.
	\end{equation}
	Finally, there exists  $c_2=c_2\lbrb{d,\Psi}\in\lbrb{0,c_1}$ such that for any $c\in\lbrb{0,c_2}$ we have that 
	\begin{equation}\label{eq:red1}
	\phnstc\!(\infty)=\IntOI \tau^{(c)}(y)dy-\ptt>0.
	\end{equation}
\end{lemma}
\begin{proof} Note that \eqref{eq:phic} can be defined for any $c>0$ but we fix $c\in\lbrb{0,c_0}$ ensuring the validity of Lemma \ref{lem:aux}.
	The validity of \eqref{eq:phi2} is a simple rearrangement. Next, \eqref{eq:residual} follows from the application of the Riemann-Lebesgue lemma to the function $\tau^{(c)}\in\Lspaces{1}{\Rp}$. The latter is a consequence from the absolute integrability of its constituting components $\tau^{(c)}_1,\tau^{(c)}_2$, see Lemma \ref{lem:aux} and \eqref{def:tau}, and the fact that \[\mubarnspace{c}=\int_{c}^{\infty}\uunspace{y}dy<\infty,\] see \eqref{eq:phi}. Next, recall that \eqref{eq:muinfPinf1} states that
	\begin{equation}\label{eq:muinfPinf}
	\mubarnspace{0}=\IntOI\uunspace{y}dy=\infty\iff \int_{0}^{1}\PPn(y)dy=\infty.
	\end{equation}
	Thus, if $\mubarnspace{0}=\infty$, \eqref{eq:muinfPinf} shows that the \LL measure of $\phntc$, that is the quantity \[\PPn^{(c)}(y)dy=\frac{1}{\dep}\lbrb{\PPn(y)-\PPn(c)}\ind{y<c}dy,\] 
	assigns infinite mass on $\lbrb{0,c}$ and is absolutely continuous therein. However, the latter facts trigger the validity of \cite[Theorem 27.7]{Sato-99} and thus the distribution of the non-decreasing \LLP underlying $\phntc$ is absolutely continuous. This, in turn, thanks to the Riemann-Lebesgue lemma yields to
	\begin{equation}\label{eq:rec}
	\limi{|b|}e^{-\phntc\!\lbrb{a+ib}}=\limi{|b|}\Ebb{e^{-\lbrb{a+ib}\xi^{(c)}_1}}=\limi{|b|}\IntOI e^{-ax-ibx}h^{(c)}(x)dx=0,
	\end{equation}
	where $\xi^{(c)}_1$ is the non-decreasing \LLP associated to $\phntc$ taken at time $1$ with probability density on $\intervalOI$
	\[h^{(c)}(x)dx=\Pbb{\xi^{(c)}_1\in dx}.\]
	Thus, from \eqref{eq:rec} we have that \[\limi{|b|}\Re\!\lbrb{\phntc\!(a+ib)}=\infty\] and the first assertion of \eqref{eq:red} is valid with $\phntc(\infty)=\infty$. Relation \eqref{eq:rec}  is clearly valid with $\phntc(\infty)=\infty$ if $\dem>0$ as well.  It remains to settle the first statement of \eqref{eq:red} when $\mubarnspace{0}<\infty$ and $\dem=0$. It follows from consideration of \eqref{eq:phic}, wherein by assumption $\dem=0$ and the Riemann-Lebesgue lemma, which since \eqref{eq:muinfPinf} implies that $\int_{0}^{c}\PPn^{(c)}(y)dy<\infty$, gives that
	\[\limi{|b|}\int_{0}^{c}e^{-ay-i|b|y}\frac{\PPn^{(c)}\lbrb{y}}{\dep}dy=0.\]
	Regardless of $\mubarnspace{0}$ being finite or not the second claim of \eqref{eq:red} follows by integration by parts of $\PPn^{(c)}(y)=\int_{y}^{c}\Pnn(dr),\,y\in\lbrb{0,c},$ in \eqref{eq:phic} or the proof of \cite[Lemma 4.6]{Patie-Savov-16}  since $\PPn^{(c)}$ is non-increasing on $\Rp$.  The final claim of the lemma, that is \eqref{eq:red1}, follows easily from the assertion of Lemma \ref{prop:tau}\eqref{it:tau} and the definition of $\ptt$.
\end{proof}
The next result is the first step to the understanding of the quantity $\Aphn$ via studying its integrand $\arg\phn$, see \eqref{eq:Aphi}. We always fix $c$ such that $\phnstc(\infty)>0$ in Lemma \ref{lem:aux1} and all claims of Lemma \ref{lem:aux} hold, which from the final assertion of Lemma \ref{lem:aux1} is always possible as long as $\mubarnspace{0}>0$.  We then decompose $\arg \phn$ as a sum of $\arg\phntc$ and an error term and we simplify the latter. For this purpose we introduce some further notation. Let assume in the sequel  that \eqref{eq:phi2} holds. Then we denote by $\Untc(dy),\,y>0$, the potential measure of the \LLP associated to $\phntc\in\Bc$ and by $\Untc_a(dy)=e^{-ay}\Untc(dy)$. Recall that $\tau^{(c)}_a(y)=e^{-ay}\tau^{(c)}(y),\,y>0$, where $\tau^{(c)}$  is defined in \eqref{def:tau}. Then, the following claim holds.
\begin{lemma}\label{lem:argphi}
	Let $\Psi\in\overNc$ such that $\php\in\BP$ and $\mubarpspace{0}<\infty$. Assume furthermore that $\mubarnspace{0}=\infty$ or equivalently $\int_{0}^{1}\PPn\!(y)dy=\infty$, see \eqref{eq:muinfPinf}. Fix $a>0$ and $c\in\lbrb{0,c_2}$ so that both Lemma \ref{lem:aux} and Lemma \ref{lem:aux1} are valid. Then, modulo to $\lbrbb{-\pi,\pi}$ for all $b>0$ and directly for all $b$ large enough
	\begin{equation}\label{eq:argphic}
	\begin{split}
	\arg\phn(a+ib)&=\arg \phntc\!(a+ib)+\arg\lbrb{1+\frac{\phnstc\!(a+ib)}{\phntc\!(a+ib)}}\\
	&= \arg \phntc\!(a+ib)+\arg\lbrb{1+\phnstc\!(\infty)\Fo{\Untc_a}\!(-ib)-\Fo{\Untc_a*\tau_a^{(c)}}\!(-ib)},
	\end{split} 
	\end{equation}
	where $\phntc,\phnstc$ are as in the decomposition \eqref{eq:phi2}.
	For any $c\in\lbrb{0,c_2}$ there exists $a_c>0$ such that for any $a\geq a_c$ and as $b\to\infty$
	\begin{align}\label{eq:argphicSim}
	\arg\phn(a+ib)=(\arg\phntc\!(a+ib))\lbrb{1+\sospace{1}}+\arg\lbrb{1-\Fo{\Xi_a^{(c)}}\!(-ib)},
	\end{align}
	where $\Xi_a^{(c)}$ is an absolutely continuous finite measure on $\Rp$. Moreover, its density $\chi_a^{(c)}$ is such that $\chi_a^{(c)}\in\Lspaces{1}{\Rp}\cap\Lspaces{\infty}{\Rp}$ and $\limi{a}||\Xi_a^{(c)}||_{TV}=0$.
\end{lemma}
\begin{proof}
	Since the assumptions of Lemma \ref{lem:aux} and Lemma \ref{lem:aux1} are satisfied we conclude that $\phn(z)=\phntc\!(z)+\phnstc\!(z)$, see \eqref{eq:phi2}. Then, modulo to $\lbrbb{-\pi,\pi}$, the first identity of \eqref{eq:argphic} is immediate whereas the second one follows from the fact that
	\[\frac{\phnstc\!(a+ib)}{\phntc\!(a+ib)}=\frac{1}{\phntc\!(a+ib)}\lbrb{\IntOI \tau^{(c)}\!(y)dy-\ptt-\Fo{\tau^{(c)}_a}(-ib)}=\frac{\phnstc\!\lbrb{\infty}}{\phntc\!(a+ib)}-\frac{\Fo{\tau^{(c)}_a}\!(-ib)}{\phntc\!(a+ib)},\]
	see \eqref{eq:phi2}, and \eqref{eq:LTU}, which we have restated as
	\begin{equation}\label{eq:LTU1}
	\frac{1}{\phntc\!(a+ib)}=\IntOI e^{-iby}e^{-ay}\Untc\!(dy)=\Fo{\Untc_a}\!(-ib).
	\end{equation}
	Note that \eqref{eq:red} implies that $\arg \phntc\!\lbrb{a+ib} \in\lbbrbb{0,\frac{\pi}{2}}$ at least for $b$ large enough. Moreover, since $\mubarnspace{0}=\infty$ we note that
	\begin{equation}\label{eq:Rec=inf}
	\limi{|b|}\Re\!\lbrb{\phntc\!(a+ib)}=\phntc\!\!\lbrb{\infty}=\infty,
	\end{equation}
	see  \eqref{eq:red}. Also, \eqref{eq:Rec=inf} together with \eqref{eq:residual} yields that \[\arg\lbrb{1+\frac{\phnstc\!(a+ib)}{\phntc\!(a+ib)}}\in\lbrb{-\frac{\pi}{2},\frac{\pi}{2}}\] at least for $b$ large enough. Henceforth, \eqref{eq:argphic} holds directly for such $b$. From \eqref{eq:LTU1}, \eqref{eq:Rec=inf}, $\phnstc\!\!\lbrb{\infty}>0$ and $\Re\!\lbrb{\phntc\!\!\lbrb{a-ib}}>0$, see \eqref{eq:rephi} of Proposition \ref{propAsymp1}, we get that, for all $b\in\R$,
	\begin{equation}\label{eq:ReU}
	\begin{split}
	\Re\!\lbrb{\phnstc\!\!\lbrb{\infty}\Fo{\Untc_a}\!(-ib)}&=\phnstc\!\!\lbrb{\infty}\Re\!\lbrb{\frac{1}{\phntc\!\!\lbrb{\ab}}}\\
	&=\phnstc\!\!\lbrb{\infty}\frac{\Re\!\lbrb{\phntc\!\!\lbrb{a-ib}}}{\abs{\phntc\!\!\lbrb{\ab}}^2}>0.
	\end{split}
	\end{equation}
	Therefore,  we conclude that, for all $b$ large enough,
	\begin{equation}\label{eq:arg}
	\begin{split}
	&\arg\lbrb{1+\phnstc\!(\infty)\Fo{\Untc_a}\!(-ib)-\Fo{\Untc_a*\tau_a^{(c)}}\!(-ib)}\\
	&=\arg\lbrb{1+\phnstc\!(\infty)\Fo{\Untc_a}\!(-ib)}+ \arg\lbrb{1-\frac{\Fo{\Untc_a*\tau_a^{(c)}}(-ib)}{1+\phnstc\!(\infty)\Fo{\Untc_a}\!(-ib)}}
	\end{split}
	\end{equation}
	because from \eqref{eq:Rec=inf} and \eqref{eq:ReU}
	\[\abs{1+\phnstc\!(\infty)\Fo{\Untc_a}(-ib)}\geq 1+\Re\!\lbrb{\phnstc\!\!\lbrb{\infty}\Fo{\Untc_a}\!(-ib)}>1\] 
	and
	\[\limi{b}\Fo{\Untc_a*\tau_a^{(c)}}\!(-ib)=\limi{b}\Fo{\Untc_a}\!(-ib)\Fo{\tau_a^{(c)}}\!(-ib)=0.\]
	From \eqref{eq:LTU1}, as $b\to\infty$, we get with the help of the second fact in \eqref{eq:red}, that is $\Im\!\lbrb{\phntc\!(a+ib)}\geq 0$, and \eqref{eq:Rec=inf} that
	\begin{equation}\label{eq:H0}
	\begin{split}
	H(b)&=\abs{\Im\!\lbrb{\Fo{\Untc_a}\!(-ib)}}=\frac{\abs{\Im\!\lbrb{\overline{\phntc(a+ib)}}}}{\abs{\phntc\!\!\lbrb{a+ib}}^2}\\
	&=\frac{\Im\!\lbrb{\phntc\!(a+ib)}}{\lbrb{\Re\!\lbrb{\phntc\!\!(a+ib)}}^2+\lbrb{\Im\!\lbrb{\phntc\!(a+ib)}}^2}=\sospace{1}.
	\end{split}
	\end{equation}
	From \eqref{eq:red}  and \eqref{eq:Rec=inf} we have again that \[\arg \phntc\lbrb{\ab}=\arctan\!\lbrb{\frac{\Im\lbrb{\phntc(a+ib)}}{\Re\lbrb{\phntc(a+ib)}}}\] and we aim to show, as $b\to\infty$, that
	\begin{equation}\label{eq:toProve}
	H(b)=\abs{\Im\!\lbrb{\Fo{\Untc_a}\!(-ib)}}=\sospace{\arg\lbrb{\phntc\!\!\lbrb{\ab}}}.
	\end{equation}
	To this end, fix $n\in\N$ and note from \eqref{eq:H0} that $\limi{b}nH(b)=0$. Therefore, from the second fact in \eqref{eq:red}, \eqref{eq:Rec=inf} and \eqref{eq:H0}, for all $b$ large enough,
	\[\tan(nH(b))\leq 2nH(b)\leq \frac{2n}{\Re\!\lbrb{\phntc\!\!\lbrb{a+ib}}}\frac{\Im\!\lbrb{\phntc\!(a+ib)}}{\Re\!\lbrb{\phntc\!\!\lbrb{a+ib}}}=\sospace{1}\frac{\Im\!\lbrb{\phntc\!(a+ib)}}{\Re\!\lbrb{\phntc\!\!\lbrb{a+ib}}}.\]
	Therefore, since from \eqref{eq:red}, $\Im\!\lbrb{\phntc\!(a+ib)}\geq0$, as $b\to\infty$, taking $arctan$ in the last inequality we deduct that
	\begin{equation*}
	\begin{split}
	n\limsupi{b}\frac{H(b)}{\arg\lbrb{\phntc\!\!\lbrb{\ab}}} &\leq \limsupi{b}\frac{\arctan\!\lbrb{\sospace{1}\frac{\Im\lbrb{\phntc\!(a+ib)}}{\Re\lbrb{\phntc\!\!\lbrb{a+ib}}}}}{\arg\lbrb{\phntc\!\!\lbrb{\ab}}}\\&
	\leq \limsupi{b}\frac{\arctan\!\lbrb{\frac{\Im\lbrb{\phntc\!(a+ib)}}{\Re\lbrb{\phntc\!\!\lbrb{a+ib}}}}}{\arg\lbrb{\phntc\!\!\lbrb{\ab}}}=1.
	\end{split}
	\end{equation*}
	Hence, since $n\in\N$ is arbitrary we conclude \eqref{eq:toProve}. However, from \eqref{eq:ReU}, that is the inequality $\Re\!\lbrb{\phnstc\!\!\lbrb{\infty}\Fo{\Untc_a}\!(-ib)}>0$, elementary geometry in the complex plane and \eqref{eq:H0}, as $b\to\infty$, we arrive at
	\begin{equation*}
	\begin{split}
	\abs{\arg\!\lbrb{1+\phnstc\!(\infty)\Fo{\Untc_a}\!(-ib)}}\!&=\!\abs{\arg\!\lbrb{1+\phnstc\!(\infty)\Re\!\lbrb{\Fo{\Untc_a}\!(-ib)}+i\phnstc\!(\infty)\Im\lbrb{\Fo{\Untc_a}\!(-ib)}}}\\
	&\leq \abs{\arg\!\lbrb{1+i\phnstc\!(\infty)\Im\!\lbrb{\Fo{\Untc_a}\!(-ib)}}}\\
	&=\abs{\arctan\!\lbrb{\phnstc\!(\infty)\Im\!\lbrb{\Fo{\Untc_a}\!(-ib)}}}\\
	&\lesssim \phnstc\!(\infty)\abs{\Im\!\lbrb{\Fo{\Untc_a}\!(-ib)}}\\ &=\sospace{\arg\lbrb{\phntc\!\!\lbrb{\ab}}},
	\end{split}	
	\end{equation*}
	where we used \eqref{eq:H0} for the last inequality and \eqref{eq:toProve} for the last equality.
	Therefore we deduce easily for the right-hand side of \eqref{eq:arg} that, as $b\to\infty$,
	\begin{equation}\label{eq:arg1}
	\begin{split}
	&\arg\lbrb{1+\phnstc\!(\infty)\Fo{\Untc_a}\!(-ib)-\Fo{\Untc_a*\tau_a^{(c)}}\!(-ib)}
	\\
	&\qquad=\arg\lbrb{1-\frac{\Fo{\Untc_a*\tau_a^{(c)}}(-ib)}{1+\phnstc\!(\infty)\Fo{\Untc_a}\!(-ib)}}+\sospace{\arg\lbrb{\phntc\!\!\lbrb{\ab}}}.
	\end{split}	
	\end{equation}
	To confirm \eqref{eq:argphicSim} we need to study the first term in the second line of \eqref{eq:arg1}. From Lemma \ref{lem:Urau} below we know that for any $a>0$, \[G^{(c)}_a(dy)=\Untc_a*\tau_a^{(c)}(dy)=g^{(c)}_a(y)dy\]
	with $g^{(c)}_a\in\Lspaces{1}{\Rp}\cap \Lspaces{\infty}{\Rp}$. Also, for fixed $c\in\lbrb{0,c_2}$  from \eqref{eq:LTU1} we get that
	\[\Untc_a\!\lbrb{\Rp}=\IntOI e^{-ay}\Untc\!(dy)=\frac{1}{\phntc\!\!\lbrb{a}}.\]
	However, since $\int_{0}^{1}\PPn(y)dy=\infty$ then we obtain from $\PPn^{(c)}(y)=\lbrb{\PPn(y)-\PPn(c)}\ind{y<c}$ that $\frac{1}{\dep}\int_{0}^{1}\PPn^{(c)}(y)dy=\infty$ and thus from \eqref{eq:phic} we conclude that $\limi{a}\phntc\!(a)=\infty$ and hence $\limi{a}\Untc_a\!\!\lbrb{\Rp}=0$. Thus, we choose $a_c>0$ such that \[\Untc_a\!\!\lbrb{\Rp}<\frac{1}{4\phnstc\!(\infty)}\] for all $a\geq a_c$ and work with arbitrary such $a$. This leads to \[\sup_{b\in\R}\abs{\phnstc(\infty)\Fo{\Untc_a}(-ib)}<\frac{1}{4}\] and then we can deduct that
	\[\frac{\Fo{G^{(c)}_a}(-ib)}{1+\phnstc\!(\infty)\Fo{\Untc_a}\!(-ib)}=\Fo{G^{(c)}_a}\!(-ib)\sum_{n=0}^{\infty}(-1)^n\lbrb{\phnstc\!(\infty)}^n\lbrb{\Fo{\Untc_a}\!(-ib)}^n.\]
	Since $G^{(c)}_a(dy)=g^{(c)}_a(y)dy$, formally, the right-hand side is the Fourier transform of a measure $\Xi_a^{(c)}$ supported on $\Rp$ with density
	\begin{equation}\label{eq:chia}
	\chi_a^{(c)}(y)=g^{(c)}_a\!(y)+\sum_{n=1}^{\infty}(-1)^n\lbrb{\phnstc\!(\infty)}^n\int_{0}^{y}g^{(c)}_a\!(y-v)\lbrb{\Untc_a}^{*n}(dv).
	\end{equation}
	However, it is immediate with the assumptions and observations above that
	\[||\chi_a^{(c)}||_{\infty}\leq||g^{(c)}_a||_{\infty}+ ||g^{(c)}_a||_{\infty}\sum_{n=1}^{\infty}\lbrb{\phnstc(\infty)\Untc_a\!\!\lbrb{\Rp}}^n<||g^{(c)}_a||_{\infty}\sum_{n=0}^{\infty}\frac{1}{4^n}<\infty,\]
	and,
	\begin{equation}\label{eq:L1chi}
	\begin{split}
	||\Xi_a^{(c)}||_{TV}&=||\chi^c_a||_1=\IntOI \abs{\chi^c_a(y)}dy\leq \IntOI |g^{(c)}_a(y)|dy\lbrb{1+ \sum_{n=1}^{\infty}\lbrb{\phnstc(\infty)\Untc_a\!\!\lbrb{\Rp}}^n}\\
	&\leq 2||g^{(c)}_a||_1= 2||G^{(c)}_a||_{TV}<\infty.
	\end{split}
	\end{equation}
	Therefore,  $\Xi_a^{(c)}$  is a well-defined finite measure with density $\chi_a^{(c)}\in\Lspaces{1}{\Rp}\cap\:\Lspaces{\infty}{\Rp}$ and from \eqref{eq:arg1} we get for all $a\geq a_c$ and any $b\in\R$ that
	\[\arg\lbrb{1-\frac{\Fo{\Untc_a*\tau_a^{(c)}}\!(-ib)}{1+\phnstc\!(\infty)\Fo{\Untc_a}(-ib)}}=\arg\lbrb{1-\Fo{\Xi_a^{(c)}}\!(-ib)}.\]
	Combining this, \eqref{eq:arg1} and \eqref{eq:argphic} we conclude \eqref{eq:argphicSim} for any  $a\geq a_c$ and as $b\to\infty$. The final claims are also immediate from the discussion above. We just note from \eqref{eq:L1chi} that $||\Xi_a^{(c)}||_{TV}\leq 2||G^{(c)}_a||_{TV}$ and Lemma \ref{lem:Urau} shows that $\limi{a}||\Xi_a^{(c)}||_{TV}=0.$
\end{proof}
In the next result we discuss the properties of the measure $\Untc_a*\tau_a^{(c)}$ used in the proof above.
\begin{lemma}\label{lem:Urau}
	Fix $a>0$. The measure $G^{(c)}_a(dy)=\Untc_a*\tau_a^{(c)}(dy)$ has the following bounded density  on $\intervalOI$ 
	\[g^{(c)}_{a}\!(y)=e^{-ay}\int_{0}^{y}\tau^{(c)}\!(y-v)\Untc\!(dv)=e^{-ay}g^{(c)}\!(y)\]
	and $g^{(c)}_{a}\in\Lspaces{1}{\Rp}$ with $\limi{a}||G^{(c)}_a||_{TV}=\limi{a}||g^{(c)}_{a}||_1=0$.
\end{lemma}
\begin{proof}
	The existence and the form of $g^{(c)}_{a}$ is immediate from the definition of convolution and the fact that $\tau_a^{(c)}\!(y)dy=e^{-ay}\tau^{(c)}\!(y)dy,\,y>0$. Recall from \eqref{def:tau} that
	\[\tau^{(c)}(y)dy=\lbrb{\tau^{(c)}_1(y)+\tau^{(c)}_2(y)+\ind{y>c}\uunspace{y}}dy.\]
	Then, $A_1=||\tau^{(c)}_2||_\infty+||\upsilon_{\mis}\ind{y>c}||_\infty<\infty$ follows from Lemma \ref{lem:aux} and \eqref{eq:mu_-1} since then 
	\[\sup_{y>c}\upsilon_{\mis}(y)=\sup_{y>c}\lbrb{\IntOI u_{\pls}(v)\Pnn(y+dv)}\leq ||u_{\pls}||_{\infty}\PPn(c)<\infty.\] Therefore,
	we have with some constant $A_3>0$
	\begin{equation}\label{eq:convoEst1}
	\begin{split}
	\sup_{x>0}e^{-ax}\int_{0}^{x}\abs{\tau^{(c)}_2\!(x-y)+\uunspace{x-y}\ind{x-y>c}}\Untc\!(dy)&\leq \sup_{x>0}A_1e^{-ax}\Untc\!(x)\\
	&\leq A_2\sup_{x>0}xe^{-ax}\leq A_3,
	\end{split}	
	\end{equation}
	where we have used the fact that $\Untc:\Rb^+\mapsto\Rb^+$ is subadditive, see \cite[p.74]{Bertoin-96}. Also note that since $\phntc\!(0)=\phn(0)+\ptt> 0$, see \eqref{eq:phic}, then
	\[\Untc\lbrb{\Rp}=\frac{1}{\phntc\!\!\lbrb{0}}<\infty.\]
	 It remains to study the portion coming from $\tau^{(c)}_1$, which according to \eqref{eq:tau1} in Lemma \ref{lem:aux} is supported on $\lbrb{0,\frac{c}{2}}$ and is bounded by the expression in \eqref{eq:boundTau}. Thus, recalling that $\PPn^{(c)}(y)=\lbrb{\PPn(y)-\PPn(c)}\ind{y<c}$, we get that
	\begin{equation}\label{eq:convoEst2}
	\begin{split} 	
	&\sup_{x>0}e^{-ax}\int_{0}^{x}\abs{\tau^{(c)}_1\!(x-y)}\Untc\!(dy)=\sup_{x>0}e^{-ax}\int_{\max\curly{0,\,x-\frac{c}{2}}}^{x}\abs{\tau^{(c)}_1\!(x-y)}\Untc\!(dy)\\
	&\,\,\leq C_1\sup_{x>0}e^{-ax}\int_{\max\curly{0,x-\frac{c}{2}}}^{x}\lbrb{\int_{0}^{\frac{c}{2}}\PPn^{(c)}\!\lbrb{\rho+x-y}-\PPn^{(c)}\!\lbrb{\frac{c}{2}+x-y}d\rho}\Untc\!(dy)\\
	&\,\,\leq C_1\sup_{x>0}e^{-ax}\int_{\max\curly{0,x-\frac{c}{2}}}^{x}\lbrb{\int_{0}^{\frac{c}{2}}\PPn^{(c)}\!\lbrb{\rho+x-y}d\rho}\Untc\!(dy)\\
	&\,\,\leq C_1\sup_{x>0}e^{-ax}\int_{0}^{x} \PPPn^{(c)}\!\lbrb{x-y}\Untc\!(dy)\\
	&\,\,\leq C_1\dep\sup_{x>0}e^{-ax}=C_1\dep,
	\end{split}
	\end{equation}
	where for the first inequality we have used the bound \eqref{eq:boundTau} and for the very last one we have employed that from \cite[Chapter III, Proposition 2]{Bertoin-96}
	\[\frac{1}{\dep}\int_{0}^{x} \PPPn^{(c)}\!\lbrb{x-y}\Untc\!(dy)=\Pbb{\textbf{e}_{\phntc\!(0)}>T^\sharp_{\lbrb{x,\infty}}}\leq 1,\]
	where $\textbf{e}_{\phntc\!(0)}$ is an exponential random variable with parameter $\phntc\!(0)=\phn(0)+\ptt> 0$, see \eqref{eq:phic},  $T^\sharp_{\lbrb{x,\infty}}$ is the first passage time above $x>0$ of the unkilled subordinator related to \[\lbrb{\phntc}^\sharp\!(z)=\phntc\!(z)-\phntc\!(0)\in\Bc\] and, from \eqref{eq:phic}, \[\frac{1}{\dep} \PPPn^{(c)}\lbrb{y}=\frac{1}{\dep}\int_{y}^{\infty}\PPn^{(c)}(v)dv\]
	is the tail of the \LL measure associated to $\lbrb{\phntc}^\sharp$.  Summing \eqref{eq:convoEst1} and \eqref{eq:convoEst2} yields that $||g^{(c)}_{a}||_\infty\leq A_3+C_1\dep<\infty$. Finally, $g^{(c)}_{a}\in \Lspaces{1}{\Rp}$ follows immediately from the estimates before the last estimates in \eqref{eq:convoEst1} and \eqref{eq:convoEst2}. Clearly, from them we also get  $\limi{a}||G^{(c)}_{a}||_{TV}=0$ which ends the proof.
\end{proof}

\begin{proof}[Proof of Proposition \ref{prop:condClass2}]
	When $\phn\in\BP$ it follows immediately from \eqref{eq:AphiAsymp2} of Proposition \ref{thm:imagineryStirling} that $\phn\in\Bthd$ and item \eqref{it:class2_A} is proved. Let us proceed  with item \eqref{it:class2_B}.
	First, from \eqref{eq:mu_-1}, note that, for any $y<1$,
	\begin{equation*}
	\lbrb{\inf_{0\leq v\leq 1}u_{\pls}(v)}\lbrb{\PPn(y)-\PPn(1)}\leq \uunspace{y}\leq ||u_{\pls}||_\infty\PPn(y).
	\end{equation*}
	Then $\inf_{0\leq v\leq 1}u_{\pls}(v)>0$ since $u_{\pls}(0)=\frac{1}{\dep}>0$, $u_{\pls}\in\Ctt\!\lbrb{{\lbbrb{0,\infty}}}$ and $u_{\pls}$ never touches $0$ whenever $\dep>0$, see Proposition \ref{propAsymp1} \eqref{it:bernstein_cmi}, and \eqref{eq:muinfPinf1} is established.
	Next $\php\in\BP$ and $\mubrp\!\lbrb{0}<\infty$ trigger the simultaneous validity of Lemma \ref{lem:aux}, Lemma \ref{lem:aux1} and Lemma \ref{lem:argphi} provided $\mubarnspace{0}=\infty$ or equivalently $\int_{0}^{1}\PPn(y)dy=\infty$, see \eqref{eq:muinfPinf1} above. Assume the latter and note that $\mubarnspace{0}=\infty$ is only needed for Lemma \ref{lem:argphi} so that $\phntc\!(\infty)=\infty$ is valid when $\dem=0$, as it is the case in item \eqref{it:class2_B} here. Then, we can always choose $c\in\Rp$ such that \eqref{eq:argphicSim} holds for any $a\geq a_c>0$, namely, for all $b$ large enough
	\[\arg \phn(a+ib)=\arg\lbrb{\phntc\!\!\lbrb{a+ib}}\lbrb{1+\sospace{1}}+\arg\lbrb{1-\Fo{\Xi_a^{(c)}}(-ib)}\]
	with $\Xi_a^{(c)}$ as in Lemma \ref{lem:argphi}.
	Thus, from the definition of $\Aph$, see \eqref{eq:Aphi}, we get that, as $b\to\infty$,
	\[\Aphn\!(a+ib)=A_{\phntc}\!(a+ib)\lbrb{1+\sospace{1}}+\int_{0}^{b}\arg\lbrb{1-\Fo{\Xi_a^{(c)}}(-iu)}du.\]
	However, since in \eqref{eq:Stirling} of Theorem \ref{thm:Stirling}, 
	\begin{itemize}
		\item $\Gph$ does not depend on $b$,
		\item the term containing $\Eph$ and $\Rph$ is uniformly bounded  for the whole class $\Bc$ on $\Cb_a$, see \eqref{eq:err},
		\item $\phntc\in\Bc$ and $\Aph\geq 0$ for any $\phi\in\Bc$, see \eqref{eq:A=Theta},
	\end{itemize} we conclude that, for every $a>a_c$ fixed, as $b\to\infty$,
	\begin{equation*}
	\begin{split}
	\abs{\Wpn\!(a+ib)}&\asymp \frac{1}{\sqrt{\abs{\phn\!\lbrb{a+ib}}}}e^{-\Aphn\!(a+ib)}\\
	&=\frac{1}{\sqrt{\abs{\phn\!\lbrb{a+ib}}}}e^{-A_{\phntc}\!(a+ib)\lbrb{1+\so{1}}-\int_{0}^{b}\arg\lbrb{1-\Fo{\Xi_a^{(c)}}(-iu)}du}.
	\end{split}	
	\end{equation*}
	Therefore, using \eqref{eq:Stirling} for $W_{\phntc}$ with the same remarks as above, we get that for every $a>a_c$ fixed and any $\eta\in\lbrb{0,\frac{1}{2}}$, as $b\to\infty$,
	\begin{equation}\label{eq:WsimWc}
	\begin{split}
	&\abs{\Wpn\!(a+ib)}\lesssim \frac{\abs{\phntc\!\!\lbrb{a+ib}}^{\frac{1}{2}-\eta}}{\sqrt{\abs{\phn\!\lbrb{a+ib}}}}e^{-\int_{0}^{b}\arg\lbrb{1-\Fo{\Xi_a^{(c)}}\!(-iu)}du}\abs{{W_{\phntc}}\!(a+ib)}^{1-2\eta}\\
	&\abs{\Wpn\!(a+ib)}\gtrsim\frac{\abs{\phntc\!\!\lbrb{a+ib}}^{\frac{1}{2}+\eta}}{\sqrt{\abs{\phn\!\!\lbrb{a+ib}}}}e^{-\int_{0}^{b}\arg\lbrb{1-\Fo{\Xi_a^{(c)}}(-iu)}du}\abs{{W_{\phntc}}(a+ib)}^{1+2\eta}.
	\end{split}
	\end{equation}
	However, the \LL measure of $\phntc$ is $\mu^{(c)}_{\mis}(dy)=\ind{y<c}\dr^{-1}_+\PPn^{(c)}(y)dy$, see \eqref{eq:phic}, and then since $\PPn^{(c)}(y)$ is non-increasing on $\intervalOI$ we deduce  via integration by parts of \eqref{eq:phic} that
	\[\Psi^{(c)}(z)=z\phi_{\mis}^{(c)}\!(z)=\lbrb{\phn(0)+\ptt}z+\frac{1}{\dep}\int_{-c}^{0}\left(e^{zr} -1
	-zr \right)\Pi^{(c)}\!(-dr)\in\overNc.\]
	Since from Lemma \ref{lem:aux1} we have that $\phn(0)+\ptt>0$ we conclude that $\Psi^{(c)}\in\Nc$. Then, however, \cite[Theorem 5.1(5.3)]{Patie-Savov-16} shows that \[\phi_{\mis}^{(c)}\in\PcBinf \iff \PPn^{(c)}\!\lbrb{0}=\infty\] (the latter being equivalent to $\Ntt_{\rk}=\infty$ in the notation of \cite{Patie-Savov-16} and $W_{\phntc}=\Mcc_{V_\psi}$ therein). Moreover,  since $\PPn^{(c)}(y)=\lbrb{\PPn(y)-\PPn(c)}\ind{y\leq c}$, see Lemma \ref{lem:aux}, we obtain that \[\phntc\in\PcBinf \iff \PPn\!\lbrb{0}=\infty.\]  It remains henceforth to understand the terms to the right-hand  side of \eqref{eq:WsimWc} and show that they cannot disrupt the faster than any polynomial decay brought in by $\abs{{W_{\phi_{\mis}^{(c)}}}(a+ib)}^{1\pm2\eta}$. With the notation and the claim of Lemma \ref{lem:argphi} we have that \[\limi{a}||\Xi_a^{(c)}||_{TV}=0\] and thus there exists $a_0>a_c$ such that for all $a>a_0$, $||\Xi_a^{(c)}||_{TV}<1$. Therefore, from \eqref{eq:logFT1} of Proposition \ref{prop:harmConvo} we get that for all $u\in\Rb,$ \[\log_0\lbrb{1-\Fo{\Xi^{(c)}_a}\lbrb{-iu}}=-\Fo{\Logp{\Xi^{(c)}_a}}(-iu).\] Moreover, $||\Xi^{(c)}_a||_{TV}<1$ implies that the first expression in \eqref{eq:convo} holds. Henceforth,
	\begin{equation}\label{eq:argExpansion}
	\begin{split}
	\arg\lbrb{1-\Fo{\Xi^{(c)}_a}\!\!\lbrb{-iu}}&=\Im\!\lbrb{\log_0\!\lbrb{1-\Fo{\Xi^{(c)}_a}\!\!\lbrb{-iu}}}\\
	&=-\Im\!\lbrb{\Fo{\Xi^{(c)}_a}\!\!\lbrb{-iu}}-\sum_{n=2}^{\infty}\frac{\Im\!\lbrb{\Fo{\Xi^{(c)}_a}^n\!\!\lbrb{-iu}}}{n}\\
	&=-\Im\!\lbrb{\Fo{\Xi^{(c)}_a}\!\!\lbrb{-iu}}-\sum_{n=2}^{\infty}\frac{\Im\!\lbrb{\Fo{\lbrb{\Xi^{(c)}_a}^{\!*n}}\!\!\lbrb{-iu}}}{n}.
	\end{split}
	\end{equation}
	We work with the term in the exponent of \eqref{eq:WsimWc}, that is $\int_{0}^b \arg\lbrb{1-\Fo{\Xi^{(c)}_a}(-iu)}du$ and $b>0$. We start with the infinite sum in \eqref{eq:argExpansion} clarifying that, for each $n\geq 1$,
	\begin{equation}\label{eq:repr}
	\Im\!\lbrb{\Fo{\lbrb{\Xi^{(c)}_a}^{\!*n}}\!\lbrb{-iu}}=-\IntOI \sin\lbrb{uy}\lbrb{\chi^{(c)}_{a}}^{\!*n}\!\!(y)dy,
	\end{equation}
	where from Lemma \ref{lem:argphi}  $\chi^{(c)}_a(y)dy=e^{-ay}\chi^{(c)}(y)dy$ is the density of $\Xi^{(c)}_a(dy)$ and for any $a\geq a_c$, $||\chi^{(c)}_{a}||_{\infty}<\infty$. By definition $a>a_0>a_c$. Next, note that, for each $b>0$, from \eqref{eq:convoEst} in Proposition \ref{prop:convoEst}, we have that
	\begin{equation}\label{eq:absInteg}
	\begin{split}
	\int_{0}^{b}\IntOI\abs{\sin(uy)\sum_{n=2}^{\infty}\frac{\lbrb{\chi_a^{(c)}}^{*n}\!\!(y)}n}dydu&\leq b\IntOI \sum_{n=2}^{\infty}||\chi^{(c)}_{a_0}||_{\infty}^n\frac{y^{n-1}e^{-\lbrb{a-a_0}y}}{n!}dy\\
	&=b\sum_{n=2}^{\infty}\frac{||\chi^{(c)}_{a_0}||_{\infty}^n}{n\lbrb{a-a_0}^n}<\infty\iff a-a_0>||\chi^{(c)}_{a_0}||_{\infty}.
	\end{split}	
	\end{equation}
	So since \eqref{eq:WsimWc} is valid for any $a>a_0>a_c$ we, from now on, fix $a>2||\chi^{(c)}_{a_0}||_{\infty}+a_0$. Then \eqref{eq:absInteg} allows via Fubini's theorem and integration by parts in \eqref{eq:secTerm} below to conclude, using \eqref{eq:argExpansion}, \eqref{eq:repr} and again \eqref{eq:convoEst}  that, for any $b>0$,
	\begin{equation}\label{eq:secTerm}
	\begin{split}
	&\abs{\int_{0}^{b}\lbrb{\arg\lbrb{1-\Fo{\Xi_a^{(c)}}\!\!\lbrb{-iu}}+\Im\!\lbrb{\Fo{\Xi_a^{(c)}}\!\!\lbrb{-iu}}}du}=\abs{\int_{0}^{b}\sum_{n=2}^{\infty}\frac{\Im\!\lbrb{\Fo{\lbrb{\Xi_a^{(c)}}^{\!*n}}\!\lbrb{-iu}}}{n}du}\\
	=&\abs{\int_{0}^{b}\sum_{n=2}^{\infty}\frac{\IntOI \sin\lbrb{uy}\lbrb{\chi_a^{(c)}}^{*n}\!\!\!(y)dy}{n}du}=\abs{\IntOI \frac{1-\cos(by)}{y}\sum_{n=2}^{\infty}\frac{\lbrb{\chi_a^{(c)}}^{\!*n}\!\!\!(y)}ndy}\\
	\leq&\,2\IntOI \sum_{n=2}^{\infty}||\chi^{(c)}_{a_0}||_{\infty}^n\frac{y^{n-2}e^{-\lbrb{a-a_0}y}}{n!}dy=2\sum_{n=2}^{\infty}\frac{||\chi^{(c)}_{a_0}||_{\infty}^n}{n(n-1)\lbrb{a-a_0}^{n-1}}<\infty.
	\end{split}
	\end{equation}
	Since the right-hand  side of \eqref{eq:secTerm} is independent of $b>0$ we deduct that for $a>2||\chi^{(c)}_{a_0}||_{\infty}+a_0$, \eqref{eq:WsimWc} is simplified to
	\begin{equation}\label{eq:WsimWc1}
	\begin{split}
	&\abs{\Wpn\!(a+ib)}\lesssim  \frac{\abs{\phntc\!\!\lbrb{a+ib}}^{\frac{1}{2}-\eta}}{\sqrt{\abs{\phn\!\!\lbrb{a+ib}}}}e^{\int_{0}^{b}\Im\lbrb{\Fo{\Xi^{(c)}_a}\!(-iu)}du}\abs{{W_{\phntc}}\!(a+ib)}^{1-2\eta}\\
	&\abs{\Wpn\!(a+ib)}\gtrsim\frac{\abs{\phntc\!\!\lbrb{a+ib}}^{\frac{1}{2}+\eta}}{\sqrt{\abs{\phn\!\!\lbrb{a+ib}}}}e^{\int_{0}^{b}\Im\lbrb{\Fo{\Xi^{(c)}_a}\!(-iu)}du}\abs{{W_{\phntc}}\!(a+ib)}^{1+2\eta}.
	\end{split}
	\end{equation}
	Next, from \eqref{eq:repr}
	\begin{align}\label{eq:est}
	\abs{\int_{0}^{b}\Im\!\lbrb{\Fo{\Xi^{(c)}_a}\!(-iu)}du}&=\abs{\IntOI \frac{1-\cos(by)}{y}\chi^{(c)}_a(y)dy}\\
	\nonumber	&\leq  \abs{\int_{0}^{1} \frac{1-\cos(by)}{y}\chi^{(c)}_a(y)dy}
	+\abs{\int_{1}^{\infty} \frac{1-\cos(by)}{y}\chi^{(c)}_a(y)dy}.
	\end{align}
	From the Riemann-Lebesgue lemma applied to the absolutely integrable function $\chi^{(c)}_a(y)y^{-1}\ind{y>1}$ we get that
	\[\limi{b}\abs{\int_{1}^{\infty} \frac{1-\cos(by)}{y}\chi^{(c)}_a(y)dy}=\abs{\int_{1}^{\infty} \frac{\chi^{(c)}_a(y)}{y}dy}=:\,D_a.\]
	Therefore, using the fact that $||\chi^{(c)}_a||_\infty<\infty,$ see Lemma \ref{lem:argphi}, we conclude, for all $b>1$ big enough, that
	\begin{align}\label{eq:finalEst}
	\abs{\int_{0}^{b}\Im\!\lbrb{\Fo{\Xi^{(c)}_a}\!(-iu)}du}&\leq \int_{0}^{b} \frac{1-\cos y }{y}\abs{\chi^{(c)}_a\!\!\lbrb{\frac{y}{b}}}dy+2D_a\\
	\nonumber	&\leq ||\chi^{(c)}_a||_\infty\lbrb{\int_{0}^{1} \frac{1-\cos y }{y}dy+\int_{1}^{b} \frac{1-\cos y }{y}dy}+2D_a\\
	\nonumber&\leq ||\chi^{(c)}_a||_\infty\int_{0}^{1} \frac{1-\cos y }{y}dy+||\chi^{(c)}_a||_\infty\ln b +\tilde{D}_a,
	\end{align}
	where \[\tilde{D}_a=2D_a+||\chi^{(c)}_a||_\infty\sup_{b>1}\abs{\int_{1}^b\frac{\cos y }{y}dy}<\infty.\]
	This allows us to conclude in \eqref{eq:WsimWc1}, as $b\to\infty$,
	\begin{equation}\label{eq:WsimWc2}
	\begin{split}
	&\abs{\Wpn\!(a+ib)}\lesssim b^{||\chi^{(c)}_a||_\infty+\frac{1}{2}-\eta}\abs{{W_{\phntc}}\!(a+ib)}^{1-2\eta}\\
	&\abs{\Wpn\!(a+ib)}\gtrsim b^{-||\chi^{(c)}_a||_\infty-\frac{1}{2}}\,\,\abs{{W_{\phntc}}\!(a+ib)}^{1+2\eta},
	\end{split}
	\end{equation}
	where we have also used the standard relation $\abs{\phi(a+ib)}\simi\dr b+\sospace{b}$, see Proposition \ref{propAsymp1}\eqref{it:asyphid}, and for any $a>0$ and any $\phi\in\Bc$, $\abs{\phi(\ab)}\geq\Re\!\lbrb{\phi(\ab)}\geq\phi(a)>0$, see \eqref{eq:rephi}, to estimate
	\begin{equation*}
	\begin{split}
	\frac{\abs{\phntc\!\!\lbrb{a+ib}}^{\frac{1}{2}-\eta}}{\sqrt{\abs{\phn\!\!\lbrb{a+ib}}}}&\lesssim b^{\frac{1}{2}-\eta}\quad \text{ and }\quad
	\frac{\abs{\phntc\!\!\lbrb{a+ib}}^{\frac{1}{2}+\eta}}{\sqrt{\abs{\phn\!\!\lbrb{a+ib}}}}\gtrsim b^{-\frac{1}{2}}.
	\end{split}
	\end{equation*}
	
	Hence, as mentioned below \eqref{eq:WsimWc} from \cite[Theorem 5.1 (5.3)]{Patie-Savov-16} we have that \[\phntc\in\PcBinf \iff \PPn^{(c)}\!\!\lbrb{0}=\infty\] and since $\PPn^{(c)}\!(y)=\lbrb{\PPn(y)-\PPn(c)}\ind{y\leq c}$, see Lemma \ref{lem:aux}, we conclude that
	\[\phntc\in\PcBinf\iff \PPn(0)=\infty.\] This together with \eqref{eq:WsimWc2} and Lemma \ref{lem:Lindelof} shows that \[\int_{0}^{1}\PPn(y)dy=\infty\implies\PPn(0)=\infty\implies\phn\in\PcBinf.\] Let next $\int_{0}^{1}\PPn(y)dy<\infty$, but $\PPn(0)=\infty$. Unfortunately, we cannot easily use similar comparison as above despite that $\phntc\in\PcBinf \iff \PPn^{(c)}\!\!\lbrb{0}=\infty$. In fact Lemma \ref{lem:argphi} fails to give a good and quick approximation of $\arg \phn$ with $\arg\phntc$. We choose a different route. From the last claim of Proposition \ref{prop:VigonDens} since $\PPn(0)=\infty$ we get that
	\begin{equation}\label{eq:liminfU}
	\uunspace{0^+}=\limo{y}\uunspace{y}=\infty.
	\end{equation}
	Since  we aim to show that $\phn\in\Bi$ from Lemma \ref{lem:Lindelof}  we can work again with a single $a$ which we will choose later. From $\abs{\Wpn\!\!\lbrb{\ab}}=\abs{\Wpn\!\!\lbrb{a-ib}}$ we can focus on $b>0$ only as well.
	From the alternative expression for $\Aphn$, see \eqref{eq:Tphi} in the claim of Theorem \ref{thm:genFuncs}\eqref{it:Aphi}, we get that
	\begin{equation}\label{eq:Theta}
	\begin{split}
	\Aphn\!\!\lbrb{\ab}&=\int_{a}^{\infty}\ln\lbrb{\abs{\frac{\phn(u+ib)}{\phn(u)}}}du\\
	&\geq\int_{a}^{\infty}\ln\lbrb{\abs{1+\frac{\Re\!\lbrb{\phn(u+ib)-\phn(u)}}{\phn(u)}}}du.
	\end{split}
	\end{equation}
	Next, we note that since $\dem=0$,
	\[\frac{\Re\!\lbrb{\phn(u+ib)-\phn(u)}}{\phn(u)}=\frac{\IntOI\lbrb{1-\cos(by)}e^{-uy}\uunspace{y}dy}{\phn(u)}\geq 0.\]
	Moreover, since $\mubarnspace{0}=\IntOI \uunspace{y}dy<\infty$,
	\[\limi{u}\sup_{b\in \R}\IntOI\lbrb{1-\cos(by)}e^{-uy}\uunspace{y}dy\leq \limi{u}2\IntOI e^{-uy}\uunspace{y}dy=0\]
	and $\limi{u}\phn(u)=\phn(\infty)<\infty$, see \eqref{eq:phi}. We choose $\ak>0$ large enough so that, for all $u\geq \ak$, \[\phn(u)>\frac{\phn(\infty)}2\quad \text{ and }\quad \sup_{b\in \R}\IntOI\lbrb{1-\cos(by)}e^{-uy}\uunspace{y}dy\leq \frac{\phn(\infty)}{4}\]
	so that
	\[\frac{\Re\!\lbrb{\phn(u+ib)-\phn(u)}}{\phn(u)}\leq\frac{1}{2}.\]
	Therefore from \eqref{eq:Theta} and $\ln\lbrb{1+x}\geq Cx$, for all $x<\frac{1}{2}$, with some $C>0$, we deduce that for any $\varepsilon>0$ and $b>\frac{1}{\varepsilon}$
	\begin{equation*}
	\begin{split}
	\Aphn\!\!\lbrb{\ak+ib}&\geq C\int_{\ak}^{\infty}\frac{\IntOI\lbrb{1-\cos(by)}e^{-uy}\uunspace{y}dy}{\phn(u)}du\\
	&\geq \frac{C}{\phn\lbrb{\infty}}\int_{\ak}^{\infty}\IntOI \lbrb{1-\cos(by)}e^{-uy}\uunspace{y}dy du\\
	&=\frac{C}{\phn\lbrb{\infty}}\IntOI \lbrb{1-\cos(by)}e^{-\ak y}\uunspace{y}\frac{dy}y\\
	&\geq \frac{C}{\phn\lbrb{\infty}}\int_{\frac{1}{b}}^{\varepsilon}\lbrb{1-\cos (by) }e^{-\ak y}\uunspace{y}\frac{dy}{y}\\
	&\geq \frac{C}{\phn\lbrb{\infty}}e^{-\ak \varepsilon}\int_{1}^{b\varepsilon}\lbrb{1-\cos y }\uunspace{\frac{y}{b}}\frac{dy}{y}\\
	&\geq \frac{C}{\phn\lbrb{\infty}}e^{-\ak \varepsilon}\lbrb{\inf_{v\in\lbrb{0,\varepsilon}}\uunspace{v}}\int_{1}^{b\varepsilon}\lbrb{1-\cos y }\frac{dy}{y}.
	\end{split}
	\end{equation*}
	However, since $\int_{1}^{\infty}\frac{\cos y }{y}dy<\infty$ we conclude that for any $\varepsilon>0$, as $b\to\infty$,
	\begin{equation}\label{eq:lower}
	\limi{b}\frac{\Aphn\!\!\lbrb{\ak+ib}}{\ln b }\geq \frac{C}{4\phn\!\!\lbrb{\infty}}e^{-\ak \varepsilon}\inf_{v\in\lbrb{0,\varepsilon}}\uunspace{v}.
	\end{equation}
	Now \eqref{eq:liminfU}, that is $\limo{\varepsilon}\inf_{v\in\lbrb{0,\varepsilon}}\uunspace{v}=\infty$, and \eqref{eq:lower} prove the claim $\phn\in\Bi$, since for the fixed $\ak$ and as $b\to\infty$
	\[\abs{\Wpn\!(\ak+ib)}\asymp \frac{1}{\sqrt{\abs{\phn(\ak+ib)}}}e^{-\Aphn\!\lbrb{\ak+ib}}\lesssim\frac{1}{\sqrt{\abs{\phn(\ak)}}} e^{-\frac{C}{4\phn\lbrb{\infty}}e^{-\ak \varepsilon}\lbrb{\inf_{v\in\lbrb{0,\varepsilon}}\uun{v}}\ln b},\]
	see \eqref{eq:Stirling} and $\abs{\phn(\ak+ib)}\geq \phn(\ak)>0$, see \eqref{eq:rephi}.
	We conclude item \eqref{it:class2_B}. We proceed with the proof of item \eqref{it:class2_C}. Assume then that $\PPn(0)<\infty$. In this case we study directly $\Aphn$. Since \eqref{eq:mu_-1} holds in any situation when $\php\in\BP$ then $\mu_{\mis}(dy)=\uunspace{y}dy,\, y\in\intervalOI$, and
	\[||\upsilon_{\mis}||_\infty=\sup_{y\geq 0}\uunspace{y}\leq ||u_{\pls}||_\infty\sup_{y>0}\PPn(y)= ||u_{\pls}||_\infty\PPn(0)=A<\infty.\]
	Note that $\phn(\infty)=\phn(0)+\mubarnspace{0}$ and put $\upsilon^\star_{a}(y)=\frac{e^{-ay}}{\phn(\infty)}\uunspace{y},\,y>0$. Then, clearly from the first expression in \eqref{eq:phi}, for $z=a+ib\in\CbOI,\,a>0$,
	\begin{equation}\label{eq:phicfin}
	\phn(z)=\phn(0)+\mubrn(0)-\IntOI e^{-iby}e^{-ay}\upsilon_{\mis}(y)dy=\phn(\infty)\lbrb{1-\Fo{\upsilon^\star_{a}}(-ib)}.
	\end{equation}
	From $||\upsilon_{\mis}||_\infty\leq A$ then for all $a$ big enough we have that $||\upsilon^\star_{a}||_{1}<1$. Fix such $a$. Then for all $b\in\Rb$ we deduce from \eqref{eq:logFT1} of Proposition \ref{prop:harmConvo} and \eqref{eq:formalF} that
	\begin{align}\label{eq:arg3}
	\nonumber\arg\phn\!\lbrb{a+ib}&=\,\,\,\,\Im\!\lbrb{\log_0\!\lbrb{1-\Fo{\upsilon^\star_{a}}(-ib)}}\\
	&=-\Im\!\lbrb{\Fo{\upsilon^\star_{a}}(-ib)}-\sum_{n=2}^{\infty}\frac{\Im\!\lbrb{\Fo{\upsilon^\star_{a}}(-ib)}^n}n=-\Im\!\lbrb{\Fo{\upsilon^\star_{a}}(-ib)}+g_a(b).
	\end{align}
	Since $||\upsilon^\star||_\infty\leq \frac{A}{\phn(\infty)}<\infty$ with $\upsilon^\star=\upsilon^\star_{0}$ we can show, repeating without modification \eqref{eq:absInteg} and \eqref{eq:secTerm} above and in the process estimating the convolutions $\lbrb{\upsilon^\star_{a}}^{*n}\!,\,n\geq 2$, using \eqref{eq:convoEst} in  Proposition \ref{prop:convoEst} with $a'=0$, that \[\sup_{b>0}\abs{\int_{0}^{b}g_a(v)dv}<\infty.\]  Thus the latter does not contribute more than a constant to $\Aphn$ at least for this fixed $a$ big enough. Without loss of generality work with $b>0$. Then, from the definition of $\Aphn$, see \eqref{eq:Aphi}, \eqref{eq:arg3} and the preceding discussion, we get that
	\begin{align*}  	\Aphn\!\lbrb{a+ib}&=\int_{0}^{b}\arg\phn\!\lbrb{a+iu}du=\int_{0}^{b}g_a(u)du-\int_{0}^{b}\Im\!\lbrb{\Fo{\upsilon^\star_{a}}(-iu)}du\\
	&=\int_{0}^{b}g_a(u)du+\IntOI \frac{1-\cos(by)}{y}\upsilon^\star_{a}(y)dy,
	\end{align*}
	where $\sup_{b>0}\abs{\int_{0}^{b}g_a(u)du}<\infty$.
	Estimating precisely as in \eqref{eq:est} and \eqref{eq:finalEst}, since $y\mapsto \upsilon^\star_{a}(y)y^{-1}\ind{y>1}\in\Lspaces{1}{\Rp}$ and $||\upsilon^{\star}||_\infty\leq \frac{A}{\phn(\infty)}<\infty$, one gets that for some positive constant $C'_a$
	\begin{equation}\label{eq:Aphn1}
	\abs{\Aphn\!\!\lbrb{a+ib}-\int_{1}^{b}\frac{1-\cos y }{y}\upsilon^{\star}_{a}\!\lbrb{\frac{y}{b}}dy}\leq C'_a.
	\end{equation}
	We investigate the contribution of the integral as $b\to\infty$.  Next, note that
	\[\upsilon^{\star}_{a}\!\lbrb{0}=\limo{y}\frac{\IntOI u_{\pls}(v)\Pnn(y+dv)}{\phn(\infty)}=\frac{\IntOI u_{\pls}(v)\Pnn(dv)}{\phn(0)+\mubarn{0}}<\infty,\]
	which follows Proposition \ref{prop:VigonDens} because $\PPn(0)<\infty$. Thus, $\upsilon^{\star}_{a}$ is right-continuous at zero. Set $\overline{\upsilon}_{a}(y)=\upsilon^{\star}_{a}(y)-\upsilon^{\star}_{a}\lbrb{0}$, for $y\in\intervalOI$. Then, clearly,
	\begin{equation}\label{eq:newM}
	\begin{split}
	\int_{1}^{b}\frac{1-\cos y }{y}\upsilon^{\star}_{a}\!\lbrb{\frac{y}{b}}dy&=\int_{1}^{b}\frac{1-\cos y }{y}\overline{\upsilon}_{a}\!\lbrb{\frac{y}{b}}dy+\upsilon^{\star}_{a}\!\lbrb{0}\ln(b)-\upsilon^{\star}_{a}\!\lbrb{0}\int_{1}^{b}\frac{\cos(y)}{y}.
	\end{split}
	\end{equation}
	Fix $\rho\in\lbrb{0,1}$. Then,
	\begin{equation}\label{eq:bound1}
	\sup_{b>1}\int_{b\rho}^{b}\frac{1-\cos y }{y}\overline{\upsilon}_{a}\!\lbrb{\frac{y}{b}}dy\leq 4||\upsilon^{\star}_{a}||_\infty\abs{\ln\rho}.
	\end{equation}
	However, since $\upsilon^{\star}_{a}$ is right-continuous at zero we are able to immediately conclude that
	\[\limo{\rho}\sup_{y\leq b\rho}\abs{\overline{\upsilon}_{a}\!\lbrb{\frac{y}{b}}}=\limo{\rho}\overline{\overline{\upsilon}}\lbrb{\rho}=0,\]
	where $\overline{\overline{\upsilon}}\lbrb{\rho}=\sup_{v\leq\rho}\abs{\overline{\upsilon}_{a}\!\lbrb{v}}$.
	Therefore
	\begin{equation}\label{eq:bound3}
	\abs{\int_{1}^{b\rho}\frac{1-\cos y }{y}\overline{\upsilon}_{a}\!\lbrb{\frac{y}{b}}dy}\leq 2\lbrb{\sup_{y\leq\rho}\abs{\overline{\upsilon}_{a}\!\lbrb{y}}} \ln b=2\overline{\overline{\upsilon}}\!\lbrb{\rho}\ln b .
	\end{equation}
	Since $\sup_{b>1}\upsilon^{\star}_{a}\lbrb{0}\int_{1}^{b}\frac{\cos(y)}{y}dy<\infty$ we then combine  \eqref{eq:newM}, \eqref{eq:bound1} and \eqref{eq:bound3} in \eqref{eq:Aphn1} to get for any $\rho\in\lbrb{0,1}$ and  $C_{a,\rho}:=4||\upsilon^{\star}_{a}||_\infty\abs{\ln\rho}>0$ that
	\begin{equation}\label{eq:Aphn2}
	\abs{\Aphn\!\lbrb{a+ib}-\upsilon^{\star}_{a}\!\lbrb{0}\ln b }\leq C_{a,\rho}+2\overline{\overline{\upsilon}}\!\lbrb{\rho}\ln b .
	\end{equation}
	Thus, for all $a$ big enough and all $\rho\in\lbrb{0,1}$ we have from \eqref{eq:Stirling} of Theorem \ref{thm:Stirling} that
	\begin{equation}\label{eq:fin}
	\abs{\Wpn\!\!\lbrb{a+ib}}	\asymp\frac{1}{\sqrt{|\phn(a+ib)|}}e^{-\upsilon^{\star}_{a}\lbrb{0}\ln b -2\overline{\overline{\upsilon}}\lbrb{\rho}\ln b -C_{a,\rho}}.
	\end{equation}
	Since $\limo{\rho}\overline{\overline{\upsilon}}\lbrb{\rho}=0$ and \[\upsilon^{\star}_{a}\!\lbrb{0}=\upsilon^{\star}\!\lbrb{0}=\frac{\uunspace{0^+}}{\phn(\infty)}=\frac{\IntOI u_{\pls}(y)\Pi_{\mis}(dy)}{\phn(0)+\mubrn(0)},\]
	relation \eqref{eq:fin} settles the proof of item \eqref{it:class2_C} at least for all $a$ big enough.
	However, since $\mu_{\mis}$ is absolutely continuous, for any $a>0$ fixed, we have from Proposition \ref{propAsymp1}\eqref{it:finPhi} that $\limi{|b|}\phn(a+ib)=\phn\lbrb{\infty}$. From $\Wpn\!(1+\ab)=\phn\!\lbrb{\ab}\Wpn\!(\ab)$, see \eqref{eq:Wp}, we then get that \eqref{eq:fin} holds for any $a>0$ up to a multiplicative constant. This concludes the proof of item \eqref{it:class2_C} and therefore of Proposition \ref{prop:condClass2}.
	
	\subsection{Proof of Theorem \ref{thm:Aph}\eqref{eq:subexpp}}\label{page:subexpp}
	The proof is based on the observations of Remark \ref{rem:condClass2} which  collects the main ingredients under which the lengthy proof of Proposition \ref{prop:condClass2} is valid. We will sketch the idea here as well. The requirements $\mu(dy)=\upsilon(y)dy=\lbrb{\upsilon_1(y)+\upsilon_2(y)}dy$, $\upsilon_1\geq 0$ and non-increasing on $\Rp$, and $\upsilon_1,\upsilon_2\in \Lspace{1}{\Rp}$, and $\abs{\upsilon_2(y)}\leq \lbrb{\int_{y}^{\infty}\upsilon_1(r)dr}\vee C$, for some $C>0$, ensure that we can write
	\[\phi=\phi^c+\tau\,\]
	with $\phi^c\in\Bc$ such that $\mu^{c}(dy)=\upsilon^c(y)dy$, where $\upsilon^c=\upsilon_1\ind{\lbrb{0,c}}$  is non-increasing on $\Rp$. Then, it is known that $W_{\phi^c}$ is the Mellin transform of a self-decomposable random variable and its decay along complex lines is evaluated in \cite[Theorem 5.1, Section 5]{Patie-Savov-16}. The final part is to show that the speed of decay of $W_{\phi^c}$ and $\Wp$ along complex lines coincide and are equal to the quantities in item \eqref{eq:subexpp}.  All these details are made rigorous in the proof of Proposition \ref{prop:condClass2}.
\end{proof}

  \section{Proofs for Exponential functionals of L\'evy processes}\label{sec:EFLP}

  \subsection{Regularity, analyticity and representations  of the density: Proof of Theorem \ref{cor:smoothness}\eqref{it:MTfact}}\label{subsec:Reg}
  Recall that for any $\Psi\in\overNc$ and  we have that $\MPs$ satisfies \eqref{eq:fe} that is
  \begin{equation}\label{eq:fe1}
  \MPs(z+1)=\frac{-z}{\Psi(-z)}\MPs(z)
  \end{equation}
  at least for $z\in i\R\setminus \lbrb{\Zc_0\!\lbrb{\Psi}\cup\curly{0}}$, see Theorem \ref{thm:FormMellin}. Recall that the quantity $\dphn$ is defined in \eqref{eq:dphi}, and assume that $\dphn<0$. Then $\MPs$ satisfies \eqref{eq:fe} at least on $z\in\Cb_{\lbrb{0,-\dphn}}$. Therefore, for any $\Psi\in\Nc\subset\overNc$ such that $\dphn<0$, one has, see Remark \ref{rem:Maulik}, that $\MPsi$ solves \eqref{eq:fe1} at least for $z\in\Cb_{\lbrb{0,-\dphn}}$, which in the case  $\Psi\in\Nc\setminus\Nc_\dagger$ is thanks to \cite{Maulik-Zwart-06} and for $\Psi\in\Nc$ thanks to \cite{Arista-Rivero-16} (in fact the extension from $\Nc\setminus\Nc_\dagger$ to $\Nc$ is a simple analytical exercise). Since by definition $\Psi\in\Nc\iff \phn(0)>0$, see \eqref{eq:Nc}, we proceed to show that $\MPsi(z)=\phn(0)\MPs(z)$ or that the identities in \eqref{eq:MIPsi} hold. First, thanks to \cite[Proposition 6.8]{Patie-Savov-16} we have that \[\Ebb{I^{z-1}_{\php}}=\frac{\Gamma(z)}{\Wpp\!(z)},\,z\in\Cb_{\intervalOI}.\]
  Furthermore, it is immediate to verify  that $\phn(0)\Wpn\!\lbrb{1-z}\!,\,z\in\Cb_{\lbrb{-\infty,1-\dphn}}$, is the Mellin transform of the random variable $X_{\phn}$
  defined via the identity
  \begin{equation}\label{eq:auxRv}
  \Ebb{g\lbrb{X_{\phn}}}=\phn(0)\Ebb{\frac{1}{Y_{\phn}}g\lbrb{\frac{1}{Y_{\phn}}}}\!,
  \end{equation}
  where from Definition \ref{def:WB}, $Y_{\phn}$ is the random variable associated to  $\Wpn\in\Wc_{\Bc}$  and $g\in\Ctt_b\!\lbrb{\Rp}$. Therefore
  \begin{equation}\label{eq:MPsiCirc}
  \begin{split}		  \phn(0)\MPs(z)&=\Ebb{I_{\php}^{z-1}}\Ebb{X_{\phn}^{z-1}}\\
  &=\phn(0)\Wpn\!\lbrb{1-z}\frac{\Gamma(z)}{\Wpp(z)},\quad z\in\Cb_{\lbrb{0,1-\dphn}},
  \end{split}
  \end{equation}
  is the Mellin transform of $I_{\php}\times X_{\phn}$  and it therefore solves \eqref{eq:fe1} on $\Cb_{\lbrb{0,-\dphn}}$ with the condition $\phn(0)\MPs(1)=1$. Therefore, both $\phn(0)\MPs$ and $\MPsi$ solve \eqref{eq:fe1} on $\Cb_{\lbrb{0,-\dphn}}$ with $\phn(0)\MPs(1)=\MPsi(1)=1$ and are holomorphic on $\Cb_{\lbrb{0,1-\dphn}}$. However, note that $\phn(0)\MPs$ is zero-free on $\lbrb{0,1-\dphn}$ since $\Gamma$ is zero-free and, according to Theorem \ref{thm:Wp}, $\Wpn$ is zero-free on $\lbrb{\dphn,\infty}$. Thus, we conclude that  
  \[f(z)=\frac{\MPsi(z)}{\phn(0)\MPs(z)}\]
  with $f$ some entire holomorphic periodic function of period one, that is $f(z+1)=f(z),\,z\in\Cb_{\lbrbb{0,1}}$, and $f(1)=1$.
  Next, considering $z=a+ib\in\Cb_{\lbrb{0,1-\dphn}},$ $a$ fixed and $|b|\to\infty$, we get that
  \begin{equation*}
  \begin{split}
  \abs{f(z)}&=\abs{\frac{\MPsi(z)}{\phn(0)\MPs(z)}}\leq \frac{\Ebb{\IPsi^{a-1}}}{\abs{\phn(0)\MPs(z)}}\\
  &= \frac{\Ebb{\IPsi^{a-1}}}{\phn(0)}\frac{\abs{\Wpp\!(z)}}{\abs{\Gamma(z)\Wpn\!(1-z)}}
  =\bospace{1}\frac{\abs{\Wpp\!(z)}}{\abs{\Gamma(z)\Wpn\!(1-z)}}.
  \end{split}
  \end{equation*}
  Since $a>0$ is fixed, we apply \eqref{eq:Stirling} to $\abs{\Wpp\!\lbrb{z}}$ and \eqref{eq:Stirling1} and \eqref{eq:Stirling} to $\abs{\Wpn\lbrb{1-z}}$ to obtain the inequality
  \begin{equation}\label{eq:periodic1}
  \abs{f(z)}\leq \bospace{1} \frac{e^{-\Aphp\!(a+ib)+\Aphn\!\lbrb{1-a+n-ib}}}{\sqrt{\abs{\php(z)}}\abs{\Gamma(z)}}\times\prod_{j=0}^{n-1}\abs{\phn\!\lbrb{1-z+j}}\times \sqrt{\abs{\phn\!\lbrb{1-z+n}}},
  \end{equation}
  where in \eqref{eq:Stirling1} we have taken  $n=\left(\lfloor a-1\rfloor+1\right)\mathbb{I}_{\{1-a\leq 0\}}$ and $z=\ab$. Also, we have regarded any term in \eqref{eq:Stirling} and \eqref{eq:Stirling1} depending on $a$ solely as a constant included in $\bospace{1}$. From  Proposition \ref{propAsymp1}\eqref{it:asyphid}  we have, for $a>0$  fixed,  that $\abs{\phi(a+i|b|)}\simi |b|\lbrb{\dr+\sospace{1}}$. We recall, from \eqref{eq:rephi}, that, for any $a>0$,
  \[\abs{\phi\!\lbrb{a+ib}}\geq \Re(\phi(a+ib))\geq \phi(0)+\dr a+\IntOI \lbrb{1-e^{-ay}}\mu(dy)=\phi(a)>0.\]
  Applying these observations to $\php,\phn,$ in \eqref{eq:periodic1} and invoking \eqref{eq:A=Theta}, that is $\Aph\geq 0$ and $\Aph\!\lbrb{\ab}\leq\frac{\pi}{2}|b|$, we get, as $|b|\to\infty$,
  \begin{equation}\label{eq:periodic2}
  \abs{f(z)}\leq |b|^{n+\frac{1}{2}} \frac{e^{|b|\frac{\pi}{2}}}{\abs{\Gamma(z)}}\bospace{1}.
  \end{equation}
  However, from the well-known Stirling asymptotic for the gamma functions, see \eqref{eq:GammaStirling}, \eqref{eq:periodic2} is further simplified, as $|b|\to\infty$, to
  \begin{equation*}
  \abs{f(z)}\leq |b|^{n-a+1}e^{|b|\pi}\bospace{1}=\sospace{e^{2\pi|b|}}.
  \end{equation*}
  However, the fact that $f$ is entire periodic with period $1$ and $\abs{f(z)}=\so{e^{2\pi|b|}}\!,\,|b|\to\infty,$ for $z\in\Cb_{\lbrb{0,1-\dphn}}$ implies by a celebrated criterion for the maximal growth of periodic entire functions, see \cite[p.96, (36)]{Markushevich-66}, that $f(z)=f(1)=1$. Hence, $\MPsi(z)=\phn(0)\MPs(z),$ which concludes the proof of Theorem \ref{lemma:FormMellin} verifying \eqref{eq:MIPsi}, whenever $\dphn<0$. Recall, from \eqref{eq:lk0}, that $\Psi\in\overNc$ takes the form
  \begin{eqnarray} \label{eq:lk1}
  \Psi(z) =  \frac{\sigma^2}{2} z^2 + \gamma z +\IntII \left(e^{zr} -1
  -zr \ind{|r|<1}\right)\Pi(dr)+\Psi(0).
  \end{eqnarray}
  Next, assume that $\dphn=0$, see \eqref{eq:dphi}, and that either $-\Psi(0)=q>0$ or $\Psi'\!\lbrb{0^+}=\Ebb{\xi_1}\in\lbrbb{0,\infty}$ with $q=0$ hold in the definition of $\Psi\in\Nc$, i.e. \eqref{eq:Nc}. For any $\eta>0$ modify the \LL measure $\Pi$ in \eqref{eq:lk1} as follows
  \begin{equation}\label{eq:Pigam}
  \Pi^{\seta}(dr)=\ind{r>0}\Pi(dr)+\lbrb{\ind{r\leq-1}e^{\eta r}+\ind{r\in\lbrb{-1,0}}}\Pi(dr)= e^{\eta r \ind{r\leq-1}}\Pi(dr).
  \end{equation}
  Plainly  $\Pi^{\seta}$ is a \LL measure and we denote by $\Psi^{\seta} \in \overNc$ (resp.~$\xi^{\seta}$) the \LL\!-Khintchine exponent (resp.~the \LL process associated to  $\Psi^{\seta}$) with parameters $\sigma^2, \gamma$ and $\Psi^{\seta}\!(0)=\Psi(0)$ taken from $\Psi$ and \LL measure $\Pi^{\seta}$. Then set \[\Psi^{\seta}\!(z)=-\php^{\seta}\!(-z)\phn^{\seta}\!(z),\]
   see \eqref{eq:WH1}. However, \eqref{eq:Pigam} and \eqref{eq:lk1} imply that $\Psi^{\seta}$ (resp.~$\phn^{\seta}$) extends holomorphically at least to $\Cb_{\lbrb{-\eta,0}}$ (resp.~$\Cb_{\lbrb{-\eta,\infty}}$), which immediately triggers that $\overline{\mathfrak{a}}_{\phn^{\seta}}<0$ if $-\Psi(0)=-\Psi^{\seta}\!(0)>0$ in \eqref{eq:lk1},  see \eqref{eq:dphi}. However, when $\Psi'\!\lbrb{0^+}=\Ebb{\xi_1}\in\lbrbb{0,\infty}$ and $\Psi(0)=0$ are valid, then we get that \[(\Psi^{\seta})'\!(0^+)=\Ebb{\xi^{\seta}_1}\geq\Psi'(0^+)=\Ebb{\xi_1}>0\] since relation \eqref{eq:Pigam} shows that $\xi^{\seta}$ is derived from $\xi$ via a removal of negative jumps only. Henceforth, a.s.~$\limi{t}\xi^{\seta}_t=\infty$ which shows that the downgoing ladder height process associated to $\phn^{\seta}$ is killed, that is $\phn^{\seta}\!(0)>0$. This combined with $\phn^{\seta}\in\Ac_{\lbrb{-\eta,\infty}}$ gives that  $\overline{\mathfrak{a}}_{\phn^{\seta}}<0$, see \eqref{eq:dphi}. Therefore, we have that \eqref{eq:MIPsi} is valid for $I_{\Psi^{\seta}}$ and the probabilistic interpretation of $\Mcc_{I_{\Psi^{\seta}}}$ above gives that
  \begin{equation}\label{eq:WHgamma}
  I_{\Psi^{\seta}}\stackrel{d}{=}X_{\phn^{\seta}}\times I_{\php^{\seta}}.
  \end{equation}
  However, since \eqref{eq:Pigam} corresponds to the thinning of the negative jumps of $\xi$ we conclude that $I_{\Psi^{\seta}}\leq I_\Psi$ a.s. and clearly a.s.~\[\limo{\eta}I_{\Psi^{\seta}}=I_\Psi.\] Moreover, the truncation functions $h_{\eta}(r)=e^{\eta r \ind{r\leq-1}}$ in \eqref{eq:Pigam} satisfies the conditions of \cite[Lemma 4.9]{Pardo2012} and therefore we have that $\limo{\eta}\phi^{\seta}_\pms(u)=\phi_\pms(u),\,u\geq 0$. Therefore, with the help of Lemma \ref{lem:continuityW} we deduce that \[\limo{\eta}W_{\phi^{\seta}_\pms}(z)=W_{\phi_\pms}\!(z),\,z\in\CbOI,\] 
  and establish \eqref{eq:MIPsi} via
  \begin{equation*}	
  \begin{split}   \MPsi(z)&=\limo{\eta}\M_{I_{\Psi^{\seta}}}\!(z)\\
  &=\limo{\eta}\phn^{\seta}\!(0)\frac{\Gamma(z)}{W_{\php^{\seta}}\!(z)}W_{\phn^{\seta}}\!\lbrb{1-z}\\
  &=\phn(0)\frac{\Gamma(z)}{\Wpp\!(z)}\Wpn\!\lbrb{1-z},\,z\in\Cb_{\lbrb{0,1}}.
  \end{split}
  \end{equation*}
  It remains to consider the case $\Psi\in\Nc$, $\Psi(0)=0$, \[\Ebb{\xi_1\ind{\xi_1>0}}=\Ebb{-\xi_1\ind{\xi_1<0}}=\infty\]
  and $\limi{t}\xi_t=\infty$ a.s.~hold. We do so by killing the \LLP $\xi$ at rate $q>0$. Therefore, with the obvious notation, $\Psi^{\dagger q}(z)=-\php^{\dagger q}(-z)\phn^{\dagger q}(z)$ and $\limo{q}\phi^{\dagger q}_\pms(z)=\phi_\pms(z),\,z\in\CbOI$, since $\phi^{\dagger q}_\pms$ are the  exponents of the  bivariate ladder height processes $\lbrb{\zeta^\pms,H^\pms}$ as introduced in Section \ref{subsec:LP} and Proposition \ref{prop:convPhi} holds. Also  a.s.~$\limo{q} I_{\Psi^{\dagger q}}=\IPsi$. However, since \eqref{eq:MIPsi} holds whenever $-\Psi(0)=q>0$ we conclude \eqref{eq:MIPsi} in this case by again virtue of Lemma \ref{lem:continuityW} since the latter transfers $\limo{q}\phi^{\dagger q}_\pms(z)=\phi_\pms(z),\,z\in\CbOI$ to \[\limo{q}W_{\phpm^{\dagger q}}(z)=W_{\phpm}(z),\,z\in\CbOI.\] \qed

  \subsection{Proof  of Theorem \ref{cor:smoothness}\eqref{it:supp}}
  We use the identity \eqref{eq:GammaType2}, which has been implicitly established in Section \ref{subsec:Reg}, that is
  \begin{equation}\label{eq:GammaType2_1}
  \IPsi\stackrel{d}{=} \Ir_{\php}\times X_{\phn},
  \end{equation}
  where $\Ir_{\php}$ is the exponential functional of the possibly killed negative subordinator associated to $\php\in\Bc$.
  It is well known from \cite[Lemma 2.1]{Haas-Rivero-13} that $\supp\, \Ir_{\php}=\lbbrbb{0,\frac{1}{\dep}}$ unless $\php(z)=\dep z,\dep>0,$ in which case $\supp\, \Ir_{\php}=\curly{\frac{1}{\dep}}$. When $\dep=0$ then $\supp \, \IPsi=\supp \, \Ir_{\php}= \lbbrbb{0,\infty}$ in any case. Assume that $\dep>0$ and note from \eqref{eq:realStirling}, \eqref{eq:AsympHphi} and \eqref{eq:Gph} that
  \[\ln\Ebb{Y^n_{\phn}}=\ln\Wpn(n+1) \simi n\ln\phn(n+1),\]
  which clearly shows that $\supp\,Y_{\phn}\subseteq\lbbrbb{0,\phn\!\lbrb{\infty}}$ and $\supp\,Y_{\phn}\not\subseteq\lbbrbb{0,y}$, for any $0<y<\phn\!\lbrb{\infty}$, where $Y_{\phn}$ is the random variable associated to $\Wpn$, see Definition \ref{def:WB}.  From  Theorem \ref{thm:HY}\eqref{it:hy} we have that $\ln Y_{\phn}$ is infinitely divisible with \LL measure $\kappa(dx)=\int_{0}^{x}U_{\mis}(dx-y)y\mu_{\mis}(dy)$. However, when $\dep>0$ from \eqref{eq:mu_-2} we have
  \[\mu_{\mis}(dy)=\upsilon_{\mis}(y)dy=\lbrb{\IntOI u_{\pls}(v)\Pnn(y+dv)}dy\]
  and as $u_{\pls}>0$ on $\Rp$, as mentioned in the proof of Lemma \ref{lem:aux1}, then $\upsilon_{\mis}>\delta>0$ at least in a neighbourhood of zero, say $\lbbrb{0,\epsilon}$. Also, in this case $\upsilon_{\mis}(y)=\IntOI u_{\pls}(v)\Pnn(y+dv)$ is at least a c\`adl\`ag function, see \cite[Chapter 5, Theorem 16]{Doney-07-book} or in more generality the differentiated version of Proposition \ref{prop:VigonDens} below. Thus, for any $c>0$,
  \begin{equation*}
  \begin{split}
  \int_{0}^{c}\kappa(dx) &=\int_{0}^{c}\int_{y}^{c}(x-y)\upsilon_{\mis}(x-y)dxU_{\mis}(dy)\\
  &\geq \delta\int_{0}^{c}\int_{y}^{c\wedge \lbrb{y+\epsilon}}(x-y)dxU_{\mis}(dy)\\
  &\geq\delta\int_{0}^{\frac{c}{2}}\frac{(c-y)^2\wedge \epsilon^2}{2}U_{\mis}(dy)\\
  & \geq \frac{\delta \lbrb{\frac{c^2}{4}\wedge \epsilon^2}}{2} U_{\mis}\!\lbrb{\frac{c}{2}}>0 
  \end{split}
  \end{equation*}
  and a criterion in \cite{Tucker-75} shows that $\supp\,\ln Y_{\phn}=\lbbrbb{-\infty,\ln \phn(\infty)}$. Finally, from \eqref{eq:auxRv} we deduct that $\supp\,X_{\phi_-}=\lbbrbb{\frac{1}{\phn\lbrb{\infty}},\infty}$. From \eqref{eq:GammaType2_1} this concludes the proof taking into account that $\supp\, \Ir_{\php}=\lbbrbb{0,\frac{1}{\dep}}$ or $\supp\, \Ir_{\php}=\curly{\frac{1}{\dep}}$, and $\supp\,X_{\phi_-}=\curly{\frac{1}{\phn\lbrb{\infty}}}$, when $\phn(z)=\phn(\infty)\in\lbrb{0,\infty}$.  \qed
  \subsection{Proof of items \eqref{it:smooth} and \eqref{it:holomorphic} of Theorem \ref{cor:smoothness}}
  The smoothness and the analyticity in each of the cases follow by a utilization of the dominated convergence theorem in the simple Mellin inversion, which yields
  \begin{equation}\label{eq:MIT1}
  \pPsi^{(n)}\!(x)=\frac{(-1)^{n}x^{-n}}{2\pi }\IntII x^{-a-ib}\frac{\Gamma\!\lbrb{n+\ab}}{\Gamma\!\lbrb{\ab}}\MPsi\!\lbrb{a+ib}db,
  \end{equation}
  and whose application is justified by the implied rate of polynomial decay of $\abs{\MPsi}$, that is $\NPs\in\lbrbb{0,\infty}$, and respectively the rate of exponential decay of $\abs{\MPsi}$, that is  $\Theta\in\lbbrbb{0,\frac{\pi}{2}}$. Thus, \eqref{eq:MIT1} above, which is the expression \eqref{eq:MITd} in Theorem \ref{cor:smoothness}, follows in the ordinary sense. To prove its validity in the $\rm{L}^2$-sense it is sufficient, according to \cite{Titchmarsh-58}, to have that the mapping $b \mapsto\MPsi\!\lbrb{a+ib}\in \Lspace{2}{\R}$, which is the case whenever $\NPs>1/2$. \qed
  \subsection{Proof of Theorem \ref{cor:smoothness}\eqref{it:smallExp}} 	      	
  We start with an auxiliary result. It shows that the decay of $\abs{\MPs}$ can be extended to the left provided $\aphp<0$.
  \begin{proposition}\label{prop:extendDecay}
  	Let $\Psi\in\overNc$. Then, $\abs{\MPs(z)}$ has a power-law decay with exponent $\NPs\geq 0$, see \eqref{eq:NPs}, which is preserved along $\Cb_a$ for any $a\in\lbrbb{\aphp,0}$.
  \end{proposition}
  \begin{proof}
  	Let $\aphp<0$ and take any $a\in\lbrb{\max\curly{\aphp,-1},0}$. Then, since $\ab\neq 0$, we can  extend  meromorphically the function $\MPs$ via \eqref{eq:fe} and \eqref{eq:WH1} to derive that
  	\begin{equation}\label{eq:extendLeft}
  	\MPs\!\lbrb{\ab}=\frac{\Psi\!\lbrb{-a-ib}}{-a-ib}\MPs\!\lbrb{a+1+ib}=\frac{\php\!\lbrb{\ab}\phn\!(-a-ib)}{\ab}\MPs\!\lbrb{a+1+ib}.
  	\end{equation}
  	Then, Proposition \ref{propAsymp1}\eqref{it:asyphid}  gives that \[\abs{\Psi\!\lbrb{-a-ib}}=\bospace{b^2}\] and the result when $\Psi\in\Npi$ follows.  If $\Psi\in\NpNP,\,\NPs<\infty,$ then according to Theorem \ref{thm:asympMPsi} we have that $\php\in\BP,\,\phn\in\BP^c ,\,  \PP(0)<\infty$ and $\uunspace{y}= \mu_{\mis}\!(dy)/dy$. Therefore, again from Proposition \ref{propAsymp1}\eqref{it:asyphid}, it follows that $\limi{|b|}\abs{\frac{\php\lbrb{\ab}}{\ab}}=\dep>0$ and from Proposition \ref{propAsymp1}\eqref{it:finPhi}, $\limi{|b|}\phn\!\lbrb{-a-ib}=\phn\!\lbrb{\infty}$. With these observations,  the rate of decay $\NPs$ is preserved through \eqref{eq:extendLeft}. We recur this argument for any $a\in\lbrb{\aphp,-1}$, if $\aphp<-1$. For the case $a=0$ taking $b\neq 0$ and then using the recurrence equation \eqref{eq:Wp} applied to \eqref{eqM:MIPsi} we observe that
  	\[\MPs\!\lbrb{ib}=\frac{\php\!\lbrb{ib}}{ib}\frac{\Gamma(1+ib)}{\Wpp\!\lbrb{1+ib}}\Wpn\!\lbrb{1-ib}.\]
  	However, if $\Psi\in\Npi$ then  $\phn\in\PcBinf$ and/or  the decay of $\abs{\frac{\Gamma(1+ib)}{\Wpp\lbrb{1+ib}}}$ is faster than any polynomial. Then Proposition \ref{propAsymp1}\eqref{it:asyphid} shows the same rapid decay for $\abs{\MPs\lbrb{ib}}$. If $\Psi\in\NpNP,\,\NPs<\infty,$ then as above $\php\in\BP$ and thus $\limi{|b|}\abs{\frac{\php\lbrb{ib}}{ib}}=\dep>0$ and we conclude the proof of the statement.
  \end{proof}	   	
  We are ready to start the proof of Theorem \ref{cor:smoothness}\eqref{it:smallExp}.  A standard relation  of Mellin transforms gives that the restriction  on $\Cb_{\lbrb{-1,0}}$  of
  \begin{equation}\label{eq:MTint1}
  \MPIsi\!(z):=-\frac{1}{z}\MPsi(z+1)=-\frac{\phn(0)}{z}\MPs(z+1)\in\Ac_{\lbrb{-1,0}}\cap\Mtt_{\lbrb{\aphp-1,-\aphn}},
  \end{equation}
  is the Mellin transform in the distributional sense of $\PIoPs{x}=\Pbbs{\IPsi\leq x}$, where  we use the expression \eqref{eq:MIPsi} for  the form and the analytical properties of $\MPsi$. Next note that in this assertion we only consider $\Psi\in\Nc_\dagger,$ that is $-\Psi(0)>0$ and this implies that $\NPs>0$. From Theorem \ref{thm:FormMellin} and \eqref{eq:MTint1} we get that if $\tphp=-\infty$ or $-\tphp\notin\N$ then $\MPIsi$ has simple poles at all points in the set $\curly{-1,\cdots,-\lceil-1-\aphp\rceil+1}$ and otherwise if $-\tphp\in\N$ simple poles at all points in the set $\curly{-1,\cdots,\tphp}$.  The decay of $\abs{\MPIsi(z)}$ along $\Cb_a,\,a\in\lbrb{-1,0}$ is $\NPs+1$ since the decay of $\abs{\MPsi(z)}$ along $\Cb_a,\,a\in\lbrb{0,1}$, is of polynomial order with exponent $\NPs>0$, see Theorem \ref{thm:asympMPsi}\eqref{it:decayA}. However, thanks to Proposition \ref{prop:extendDecay} the decay of $\abs{\MPIsi\!(z)}$ is of polynomial order with exponent $\NPs+1>1$ along $\Cb_a,\,a\in\lbrb{\aphp-1,0}$. Therefore, \eqref{eq:MTint1} is the Mellin transform of $F_{\Psi}$ in the ordinary sense.
  Moreover, with \[\Ntt_{+}=\abs{\tphp}\mathbb{I}_{\{\abs{\tphp} \in \N\}}+\left(\lceil |\aphp|+1\rceil\right)\mathbb{I}_{\{\abs{\tphp}\notin \N\}},\] $\N\ni M<\Ntt_+$ and $a\in\lbrb{\lbrb{-M-1}\vee \lbrb{\aphp-1},-M}$  so that $a+M\in\lbrb{-1,0}$ we apply the Cauchy theorem to the Mellin inversion of $\PIoPs{x}$ to get that
  \begin{equation}\label{eq:F}
  \begin{split}
  \PIoPs{x}&=-\frac{\phn(0)}{2\pi i}\int_{a+M-i\infty}^{a+M+i\infty}x^{-z}\frac{\MPs(z+1)}zdz\\
  &=-\Psi(0)\sum_{k=1}^{M}\frac{\prod_{j=1}^{k-1}\Psi\lbrb{j}}{k!}x^{k}-\frac{\phn(0)}{2\pi i}\int_{a-i\infty}^{a+i\infty}x^{-z}\frac{\MPs(z+1)}zdz\\
  &=-\Psi(0)\sum_{k=1}^{M}\frac{\prod_{j=1}^{k-1}\Psi\lbrb{j}}{k!}x^{k}-\frac{\phn(0)}{2\pi i}\int_{a-i\infty}^{a+i\infty}x^{-z}\frac{\Gamma\lbrb{z}}{\Wpp\!\!\lbrb{z+1}}\Wpn\!\!\lbrb{-z}dz,
  \end{split}
  \end{equation}
  since the residues at each of those poles at $-k$ are of values \[\frac{\phn(0)}{k}\php(0)\frac{\prod_{j=1}^{k-1} \Psi(j)}{(k-1)!}=-\Psi(0)\frac{\prod_{j=1}^{k-1} \Psi(j)}{k!}.\]
  They  have been computed in turn as the residues of $\MPs$, see Theorem \ref{thm:FormMellin}, and the contribution of the other terms of \eqref{eq:MTint1}, that is $\lbrb{-\phn(0)/z}|_{z=-k}$. We recall that by convention $\sum_{j=1}^{0}=0$ and $\prod_{j=1}^{0}=1$. Thus, we prove \eqref{eq:smallExpansion} for $n=0$. The derivative of order $n$ is easily established via differentiating \eqref{eq:F} as long as $0\leq n<\NPs$.
  
  \begin{proof}[Proof of Corollary \ref{cor:smallExp}]
  	We sketch the proof of Corollary \ref{cor:smallExp}. If $\abs{\aphp}=\infty$ and $-\tphp\notin\N$ then it is immediate that  the infinite  sum  \eqref{eq:smallExpansion}, that is \[\PIp(x)\stackrel{0}{\approx}-\Psi(0)\sum_{k=1}^{\infty}\frac{\prod_{j=1}^{k-1}\Psi\lbrb{j}}{k!}x^{k}=:\PIp^*(x),\]
  	is an asymptotic expansion. From the identity $\Psi(j)=-\php(-j)\phn(j)$ and \eqref{eq:phi} we see that
  	\[\php(-j)=\php(0)-\dep j-\IntOI \lbrb{e^{jy}-1}\mu_{\pls}(dy)\]
  	and clearly $\abs{\php(-j)}$ has an exponential growth in $j$ if $\mu_{\pls}$ is not identically zero. In the latter case the asymptotic series cannot be a convergent series for any $x>0$. If $\php(z)=\php(0)+\dep z$ then $\Psi(j)\simi\dep j\phn(j)$ and hence
  	\[\frac{\prod_{j=1}^{k-1}\abs{\Psi\lbrb{j}}}{k!}\simi \frac{\dep^{k-1}}k \prod_{j=1}^{k-1}\phn(j).\]
  	Therefore, $\PIp^*$ is absolutely convergent if $x<\frac{1}{\dep \phn(\infty)}$ and divergent if  $x>\frac{1}{\dep \phn(\infty)}$. Finally, if $\php\equiv \php(\infty)\in\lbrb{0,\infty}$ then \[\frac{\prod_{j=1}^{k-1}\abs{\Psi\!\lbrb{j}}}{k!}\stackrel{\infty}{=} \php(\infty)^{k-1}\frac{\prod_{j=1}^{k-1}\phn\!\lbrb{j}}{k!}\] and since from Proposition \ref{propAsymp1}\eqref{it:asyphid}, $\phn(j)\simi j\lbrb{\dem+\sospace{1}}$, we deduct that $\PIp^*$ is absolutely convergent for $x<\frac{1}{\php(\infty)\dem}$ and divergent for $x>\frac{1}{\php(\infty)\dem}$. \qed
  \end{proof}
  
  \subsection{Proof of Theorem \ref{thm:largeAsymp}}\label{subsec:Cramer}
  Recall the definitions of the \textit{lattice class} and the \textit{weak non-lattice class}, see around \eqref{eq:weaknonlattice}. We start with a result which discusses when the decay of $\abs{\MPs(z)}$ can be extended to the line $\Cb_{1-\tphn}$.
  \begin{proposition}\label{prop:extendDecay1}
  	Let $\Psi\in\Nc$. If $\Psi\in\PcNinf\cap \Nc_\Wc$ then the rapid decay of $\abs{\MPsi(z)}$ along $\Cb_{1-\tphn}$ is preserved. Otherwise, if $\Psi\in\Nc_{\Pc}\lbrb{\NPs}$, $\NPs<\infty$, then the same power-law decay is valid for $\abs{\MPsi}$ along $\Cb_{1-\tphn}$.
  \end{proposition}
  \begin{proof}
  	Let $-\tphn=-\dphn\in\lbrb{0,\infty}$, that is $\phn(\tphn)=0$. Using \eqref{eq:WH1} we write for $b\neq 0$
  	\begin{equation}\label{eq:extendRight}
  	\MPs\!\lbrb{1-\tphn+ib}=\frac{\tphn-ib}{\Psi\!\lbrb{\tphn-ib}}\MPs\!\lbrb{-\tphn+ib}.
  	\end{equation}
  	Assume first that $\Psi\in\PcNinf$. Then we choose $k_0\in\N$, whose existence is guaranteed since $\Psi\in\Nc_\Wc$, such that \[\liminfi{|b|}|b|^{k_0}\abs{\Psi\!\lbrb{\tphn+ib}}>0.\] Premultiplying \eqref{eq:extendRight} with $|b|^{-k_0}$ and taking absolute values we conclude that the functions $\abs{\MPs\!\lbrb{1-\tphn+ib}}$ inherits the rapid decay of $\abs{\MPs\!\lbrb{-\tphn+ib}}$. Recall from Theorem \ref{thm:asympMPsi}\eqref{it:decayA} that $\Psi\in\Nc_{\Pc}\lbrb{\NPs}$ with \[\NPs<\infty\iff \php\in\BP,\,\phn\in\BP^c,\,  \PP(0)<\infty\] and from \eqref{eq:mu_-} that the \LL measure behind $\phn$ is absolutely continuous. Then, we conclude from Proposition \ref{propAsymp1}\eqref{it:asyphid} that $\limi{|b|}\frac{\abs{\php\lbrb{-\tphp+ib}}}{\abs{\tphp+ib}}=\dep>0$ and from an easy extension of Proposition \ref{propAsymp1}\eqref{it:finPhi} that $\limi{|b|}\phn\!\lbrb{\tphn+ib}=\phn\!\lbrb{\infty}>0$. Therefore, we conclude that in \eqref{eq:extendRight}, \[\limi{|b|}\frac{\abs{\tphn-ib}}{\abs{\Psi\!\lbrb{\tphn-ib}}}=\frac{1}{\dep\phn\!\lbrb{\infty}}>0.\] This shows that the speed of decay of $\abs{\MPs\lbrb{-\tphn+ib}}$ and therefore, via \eqref{eq:MIPsi} the speed of decay of $\abs{\MPsi(z)}$ are preserved.
  \end{proof}
  We start with the proof of \eqref{eq:tailCramer} of item \eqref{it:Cramer1}. Note that an easy computation related to Mellin transforms shows that  the restriction of
  \begin{equation}\label{eq:MTint}
  \MPoIsi(z)=\frac{1}{z}\MPsi(z+1)\in\Ac_{\lbrb{-1,0}}\cap\Mtt_{\lbrb{\aphp-1,-\aphn}},
  \end{equation}	
  to $\Cb_{\lbrb{0,-\dphn}}$, if $-\dphn>0$, is the Mellin transform of $\PIPs{x}=\Pbb{\IPsi\geq x}$ in the distributional sense.  Next, note from Theorem \ref{thm:asympMPsi}\eqref{it:decayA}
  that since $\NPs>0$ along $\Cb_a$, $a\in\lbrb{0,-\dphn}$, as long as $\Psi(z)\not\equiv \dep z$, then $|\MPoIsi(z)|$ decays either faster than any power-law, if $\Psi\in\PcNinf$, or by polynomial order with exponent $\NPs+1>1$. Therefore, $\MPoIsi$ is the Mellin transform of $\overline{F}_{\Psi}$ in  the ordinary sense. When $\Psi\in\Nc_\Zc$,  \eqref{eq:tailCramer} is known modulo to an unknown constant, see \cite[Lemma 4]{Rivero-05}. The value of this constant, that is $\frac{\phi_{\mis}(0)\Gamma\lbrb{-\tphn}\Wpn\!\lbrb{1+\tphn}}{\phi'_{\mis}\lbrb{\tphn^+}\Wpp\!\lbrb{1-\tphn}}$, can be immediately computed as in \eqref{eq:pstarstarstar} below.
  We proceed to establish \eqref{eq:densityCramer}. For this purpose we assume that either $\Psi\in\PcNinf\cap\Nc_\Wc$ or $\Psi\in\Nn, \: \NPs<\infty$. In any case, whenever $\NPs>1$,  the Mellin inversion theorem applies and yields that, for any $z\in\Cb_{a},\,a\in\lbrb{0,1-\dphn}$,
  \begin{equation}\label{eq:MIT}
  \pPsi(x)=\frac{1}{2\pi }\IntII x^{-a-ib}\MPsi\!\lbrb{a+ib}db.
  \end{equation}
  However, the assumptions $\dphn=\tphn\in\lbrb{-\infty,0}$, $\Psi\lbrb{\tphn}=-\php\!\lbrb{-\tphn}\phn\!\lbrb{\tphn}=0$ and $\abs{\Psi'\!\lbrb{\tphn^+}}<\infty$ of item \eqref{it:Cramer}, together with \eqref{eq:WH1}, lead to
  \[\Psi'\!\lbrb{\tphn^+}=\phn'\!\lbrb{\tphn^+}\php\!\lbrb{-\tphn}.\]
  Hence $\abs{\Psi'\!\lbrb{\tphn^+}}<\infty$ implies that $\abs{\phn'\!\lbrb{\tphn^+}}<\infty$. Differentiating once more we arrive at 
  \[\Psi''\!\lbrb{\tphn^+}=\phn''\!\lbrb{\tphn^+}\php\!\lbrb{-\tphn}-\phn'\!\lbrb{\tphn^+}\php'\!\lbrb{-\tphn}.\]
  Since \[\abs{\phn'\!\lbrb{\tphn^+}\php'\!\lbrb{-\tphn}}<\infty\]
  then $\abs{\Psi''\!\lbrb{\tphn^+}}<\infty$ implies that  $\abs{\phn''\!\lbrb{\tphn^+}}<\infty$.
  This observation, the form of $\MPs$, see \eqref{eq:MPsi1}, and the fact that $\Psi\in\Nc_\Wc\subset\Nc_\Zc$ permit us to write
  \begin{equation}\label{eq:representation}
  \begin{split}
  \MPsi(z)&=\phn(0)\frac{\Gamma(z)}{\Wpp\!(z)}\Wpn(1-z)\\
  &=\phn(0)\frac{\Gamma(z)}{\Wpp\!(z)}P(1-z)+\frac{\phn(0)\Wpn\!\lbrb{1+\tphn}}{\phn'\!\lbrb{\tphn^+}}\frac{1}{1-\tphn-z}\frac{\Gamma(z)}{\Wpp\!(z)},
  \end{split}
  \end{equation}
  with $P$  from Theorem \ref{thm:Wp}\eqref{it:E} having the form and the property that
  \[P(z):=\Wpn(z)-\frac{\Wpn\!\lbrb{1+\tphn}}{\phn'\!\lbrb{\tphn^+}\lbrb{z-\tphn}}\in\Ac_{\lbbrb{\tphn,\infty}}.\]
  Therefore, \eqref{eq:representation} shows that $\MPsi\in\Ac_{\lbrb{0,1-\tphn}}$ extends continuously to $\Cb_{1-\tphn}\setminus\lbcurlyrbcurly{1-\tphn}$. Next, we show that the contour in \eqref{eq:MIT} can be partly moved to the line $\Cb_{1-\tphn}$ at least for $|b|>c>0$, for any $c>0$. For this purpose we observe from Proposition \ref{prop:extendDecay1} that whenever $\Psi\in\Nc_\Wc\cap\PcNinf$ (resp.~$\Psi\in\Nn\cap\Nc,\,\NPs<\infty$ ) the decay of $\abs{\MPsi\!(z)}$ extends with the same speed to the complex line $\Cb_{1-\tphn}$.  Then, for any $ c>0$, $a\in\lbrb{0,1-\tphn}$ and  $x>0$, thanks to the Cauchy integral theorem valid because  $\NPs>1$,
  \begin{equation}\label{eq:boundaryContour}
  \begin{split}
  \pPsi(x)&=\frac{1}{2\pi i}\int_{z=a+ib} x^{-z}\MPsi\!\lbrb{z}dz\\ &=x^{\tphn-1}\frac{1}{2\pi}\int_{|b|>c}x^{-ib}\MPsi\!\!\lbrb{1-\tphn+ib}db  	+\frac{1}{2\pi i}\int_{B^{-}(1-\tphn,c)}x^{-z}\MPsi\!(z)dz\\
  &=\gPsi(x,c)+\hPsi(x,c),
  \end{split}
  \end{equation}
  where \[B^{-}(1-\tphn,c)=\lbcurlyrbcurly{z\in\Cb:\,|z-1+\tphn|=c \textrm{ and } \Re(z-1+\tphn)\leq 0}\] that is a semi-circle centered at $1-\tphn$. We note that the Riemann-Lebesgue lemma applied to the absolutely integrable function $\MPsi\!\lbrb{1-\tphn+ib}\ind{|b|>c}$ yields that
  \begin{equation}\label{eq:past}
  \limi{x}x^{1-\tphn}\gPsi(x,c)=\limi{x}\frac{1}{2\pi}\IntII e^{-ib\ln x}\MPsi\!\lbrb{1-\tphn+ib}\ind{|b|>c}db=0.
  \end{equation}
  Using \eqref{eq:representation} we write that
  \begin{equation}\label{eq:pstar}
  \begin{split}
  \hPsi(x,c)&=\frac{\phn(0)}{2\pi i}\int_{B^{-}(1-\tphn,c)}x^{-z}\frac{\Gamma(z)}{\Wpp(z)}P(1-z)dz\\
  &+\frac{1}{2\pi i}\frac{\phn(0)\Wpn\!\lbrb{1+\tphn}}{\phn'\!\lbrb{\tphn^+}}\int_{B^{-}(1-\tphn,c)}x^{-z}\frac{1}{1-\tphn-z}\frac{\Gamma(z)}{\Wpp\!(z)}dz\\
  &=\hPsih(x,c)+\hPsihh(x,c).
  \end{split}
  \end{equation}
  Since \[\frac{\Gamma(z)}{\Wpp\!(z)}P(1-z)\in\Ac_{\lbrbb{0,1-\tphn}}\] we can use the Cauchy integral theorem to shift the contour of integration to the set $[1-\tphn-ic,1-\tphn+ic]$ to show that, for every $c>0$ fixed,
  \begin{equation}\label{eq:pstarstar}
  \begin{split}
  \abs{\hPsih(x,c)}&=\abs{\frac{\phn(0)}{2\pi }\int_{-c}^{c}x^{-1+\tphn-ib}\frac{\Gamma(1-\tphn+ib)}{\Wpp\!(1-\tphn+ib)}P(\tphn-ib)db}\\
  &\leq x^{\tphn-1}\frac{c\phn(0)}{\pi}\sup_{z=1-\tphn+ib;\,b\in\lbrb{-c,c}}\abs{\frac{\Gamma(z)}{\Wpp\!(z)}P(1-z)}.
  \end{split}
  \end{equation}
  Therefore
  \begin{equation}\label{eq:pstarstar1}
  \limo{c}\limi{x}x^{1-\tphn}\abs{\hPsih(x,c)}\leq \limo{c}\frac{c\phn(0)}{\pi}\sup_{z=1-\tphn+ib;\,b\in\lbrb{-c,c}}\abs{\frac{\Gamma(z)}{\Wpp\!(z)}P(1-z)}=0.
  \end{equation}
  Next, we consider $\hPsihh(x,c)$. Since \[\frac{\Gamma(z)}{\Wpp\!(z)}\in\Ac_{\lbrb{0,\infty}},\] choosing $c$ small enough we have by the Cauchy's residual theorem that
  \begin{equation*}
  \begin{split}
  \hPsihh(x,c)&=\frac{\phn(0)\Wpn\!\lbrb{1+\tphn}\Gamma\!\lbrb{1-\tphn}}{\phn'\!\lbrb{\tphn^+}\Wpp\!\lbrb{1-\tphn}}x^{\tphn-1}\\
  &+\frac{1}{2\pi i}\frac{\phn(0)\Wpn\!\lbrb{1+\tphn}}{\phn'\!\lbrb{\tphn^+}}\int_{B^{+}(1-\tphn,c)}x^{-z}\frac{1}{1-\tphn-z}\frac{\Gamma(z)}{\Wpp\!(z)}dz,
  \end{split}
  \end{equation*}
  where \[B^{+}(1-\tphn,c)=\lbcurlyrbcurly{z\in\Cb:\,|z-1+\tphn|=c \textrm{ and } \Re(z-1+\tphn)\geq 0}.\] However, on $B^{+}(1-\tphn,c)$, we have that $z=1-\tphn+ce^{i\theta},\,\theta\in\lbrb{-\frac{\pi}{2},\frac{\pi}{2}}$, and thus for any such value of $\theta$
  \[\limi{x}\abs{x^{1-\tphn}x^{-z}}=\limi{x}x^{-c\cos\theta}=0.\]
  Therefore, since
  \[\sup_{z\in B^{+}(1-\tphn,c)}\abs{\frac{1}{1-\tphn-z}\frac{\Gamma(z)}{\Wpp\!(z)}}\leq \frac{1}{c}\sup_{z\in B^{+}(1-\tphn,c)}\abs{\frac{\Gamma(z)}{\Wpp\!(z)}}<\infty,\]
  we can apply the dominated convergence theorem to the integral term in the representation of $\hPsihh$ above
  to conclude that for all $c>0$ small enough
  \begin{equation}\label{eq:pstarstarstar}
  \limi{x}x^{1-\tphn}\hPsihh(x,c)=\frac{\phn(0)\Wpn\!\lbrb{1+\tphn}\Gamma\!\lbrb{1-\tphn}}{\phn'\!\lbrb{\tphn^+}\Wpp\!\lbrb{1-\tphn}}.
  \end{equation}
  Combining  \eqref{eq:pstarstarstar} and \eqref{eq:pstarstar1}, we get from \eqref{eq:pstar} that
  \begin{equation*}
  \begin{split}
  \limo{c}\limi{x}x^{1-\tphn}\hPsi(x,c)&=\limo{c}\limi{x}x^{1-\tphn}\hPsihh(x,c)=\frac{\phn(0)\Wpn\!\lbrb{1+\tphn}\Gamma\!\lbrb{1-\tphn}}{\phn'\!\lbrb{\tphn^+}\Wpp\!\lbrb{1-\tphn}}.
  \end{split}
  \end{equation*}
  However, since \eqref{eq:past} holds too, we get from \eqref{eq:boundaryContour} that
  \[\limi{x}x^{1-\tphn}\pPsi(x)=\limo{c}\limi{x}\lbrb{\gPsi(x,c)+\hPsi(x,c)}=\frac{\phn(0)\Wpn\!\lbrb{1+\tphn}\Gamma\!\lbrb{1-\tphn}}{\phn'\!\lbrb{\tphn^+}\Wpp\!\lbrb{1-\tphn}}.\]
  This completes the proof of \eqref{eq:densityCramer} of item \eqref{it:Cramer1} for $n=0$. Its claim for any $n\in\N,\,n\leq \lceil\NPs\rceil-2$, follows by the same techniques as above after differentiating \eqref{eq:boundaryContour} and observing that we have thanks to Proposition \ref{prop:extendDecay} that for every $n\in\N,\,n\leq \lceil\NPs\rceil-2,$
  \[\limi{b}|b|^n\abs{\MPsi\!\lbrb{1-\tphn+ib}}=0\]
  and $|b|^n\abs{\MPsi\!\lbrb{1-\tphn+ib}}$ is integrable.
  
  Let us proceed with the proof of item \eqref{it:dPsi}. Clearly, it is equivalent to showing that if $|\dphn|<\infty$ (resp.~$|\dphn|=\infty$, that is $-\Psi(-z)=\php(z)\in\Bc$), then for any $\underline{d}<|\dphn|<\overline{d}$ (resp.~$\underline{d}<\infty$), we have  that
  \begin{align}
  &\limsupi{x}\: x^{\underline{d}}\:\overline{F}_{\Psi}(x)=0,\label{eq:largedPs1}\\
  &\liminfi{x} \: x^{\overline{d}} \: \overline{F}_{\Psi}(x) \,=\infty.\label{eq:largedPs}
  \end{align}	
  We immediately rule out the case $\Psi(z)\equiv \dep z$ since in this case $\IPsi=\frac{1}{\dep}$ a.s., see Theorem \ref{cor:smoothness}\eqref{it:supp}, and hence  \eqref{eq:largedPs1} and \eqref{eq:logLargedPs} hold true. We consider the remaining cases. Clearly, from \eqref{eq:MTint} and \eqref{eq:domainAnalMPs} of Theorem \ref{thm:FormMellin}, one deducts that
  \[	\MPoIsi(z)=\frac{1}{z}\MPsi(z+1)\in\Ac_{\lbrb{0,\dphn}}\]
  provided $-\dphn>0$.
  We also note that $\MPoIsi$ has a power-law decay with exponent $\NPs+1>1$ along $\Cb_a,\,a\in\lbrb{0,-\dphn},$ if $-\dphn>0$.  Consider first \eqref{eq:largedPs1}, that is $\limsupi{x}\: x^{\underline{d}}\:\overline{F}_{\Psi}(x)=0$ for $\underline{d}<\abs{\dphn}$. Let $-\dphn>0$, as otherwise there is nothing to prove in \eqref{eq:largedPs1}, and choose $\underline{d}\in\lbrb{0,-\dphn}$.
  By Mellin inversion as in \eqref{eq:boundaryContour} and with $a=\underline{d}+\varepsilon<-\dphn,\,\varepsilon>0,$  we get that on $\Rp$			
  \[\PIPs{x}=x^{-\underline{d}-\varepsilon}\frac{1}{2\pi}\abs{\IntII x^{-ib} \MPoIsi\!\!\lbrb{\underline{d}+\varepsilon+ib}db}\leq C_{\underline{d}}x^{-\underline{d}-\varepsilon}\]
  and \eqref{eq:largedPs1} follows. The latter suffices for the claim \eqref{eq:logLargedPs} when $-\Psi(-z)=\php(z)\in\Bc$, since then \[\MPoIsi(z)=\frac{1}{z}\frac{\Gamma(z)}{\Wpp(z)}\in\Ac_{\lbrb{0,\infty}},\] $\abs{\dphn}=\infty$ and we can choose $\underline{d}$ as big as we wish.
  It remains therefore to prove \eqref{eq:largedPs} assuming that $-\Psi(-z)\not\equiv\php(z)$. If $\Pi(dy)\equiv 0\,dy$ on $\lbrb{-\infty,0}$ then necessarily $\phn(z)=\phn(0)+\dem z,\,\dem>0,$. In this case   the even stronger item \eqref{it:Cramer} is applicable since it can be immediately shown that its conditions are satisfied that is
  \begin{itemize}
  	\item since $\tphn=-\frac{\phn(0)}{\dem}$ and $\limi{|b|}\abs{\phn\!\lbrb{\tphn+ib}}=\infty$ then  \eqref{eq:weaknonlattice} holds and hence $\Psi\in\Nc_\Wc$;
  	\item  and $\abs{\phn'(\tphn^+)}=\dem<\infty, \abs{\phn''(\tphn^+)}=0$ ensure $\abs{\Psi''\!\lbrb{\tphn^+}}<\infty$.	 	
  \end{itemize}  
  We can assume from now that $\Pi(dy)\not\equiv 0\,dy$ on $\lbrb{-\infty,0}$ and $\Psi\in\Nc_\Zc$.  If the conditions of item \eqref{it:Cramer} are present then they immediately imply \eqref{eq:largedPs} as then a stronger version of the asymptotic holds true. If the conditions of item \eqref{it:Cramer} are violated we proceed by approximation to furnish them. To this end, we define, recalling the expression of $\Psi$ in \eqref{eq:lk1}, for each $\eta>0$,
  \begin{equation}\label{eq:PsiDelta}
  \Psime\!(z)=\frac{\sigma^2}{2} z^2 + \gamma z +\int_{\R} \left(e^{zr} -1
  -zr \ind{|r|<1}\right)\Pi^{(\eta)}(dr) +\Psi(0)
  \end{equation}
  where we set $\Pi^{\leta}(dr)=e^{-\eta r^2\ind{r<-1}}\Pi(dr)$.
  Since $\Pi^{\leta}$ is equivalent to $\Pi$ it is clear that for each $\eta>0$, \[\Psi\in\Nc_\Zc\implies\Psime\in\Nc_\Zc\] and set from \eqref{eq:WH1}, $\Psime\!(z)=-\php^{\leta}(-z)\phn^{\leta}(z)$. Let $\xi^{\leta}$ be the \LLP underlying $\Psime$. Note that the transformation $\Psi\mapsto\Psime$ leaves the killing rate $q:=-\Psi(0)$ invariant and has the sole effect of truncating at the level of paths some of the negative jumps smaller than $-1$ of the underlying \LLP $\xi$. Therefore, pathwise
  \begin{equation}\label{eq:I>Ieta}
  \IPsi=\int_{0}^{\textbf{e}_q} e^{-\xi_s}ds\geq\int_{0}^{\textbf{e}_q} e^{-\xi^{\leta}_s}ds=I_{\Psime}.
  \end{equation}
  Next, it is clear from \eqref{eq:PsiDelta} that $\Psime\in\Ac_{\lbrb{-\infty,0}}$ and then it is immediate that $\lbrb{\Psime}''>0$ on $\Rn$,  and thus $\Psime$ is convex on $\Rn$. Moreover, clearly
  \[\lim_{u\to-\infty}\frac{1}{|u|}\IntIO\lbrb{e^{ur}-1-ru\ind{|r|<1}}\Pi^{\leta}\!(dr)=\infty\]
  since $\Pi$ is not identically the zero measure on $\Rn$. Therefore, \[\lim_{u\to-\infty}\Psime\!(u)=\limi{u}\phn^{\leta}\!(-u)=\infty.\] Thus,  we conclude immediately, writing here $\thPsml:=\mathfrak{u}_{\phn^{\leta}}$, that $-\thPsml\in\lbrb{0,\infty}$ if $\Psi(0)=-q<0$. Also, if $q=0$, then, regardless of whether $\Ebb{\xi_1}\in\lbrbb{0,\infty}$ or $\Ebb{\xi_1\ind{\xi_1>0}}=\Ebb{-\xi_1\ind{\xi_1<0}}=\infty$ with a.s.~$\limi{t}\xi_t=\infty$ hold for $\xi$, we necessarily have $\Ebb{\xi^{\leta}_1}\in\lbrbb{0,\infty}$, see \eqref{eq:PsiDelta}, and thus $\phn^{\leta}\!(0)>0$ leads to $-\thPsml\in\lbrb{0,\infty}$.
  However, from \eqref{eq:I>Ieta} we get that
  \[\Pbb{\IPsi>x}\geq \Pbb{I_{\Psime}>x}.\]
  Hence, from \eqref{eq:tailCramer} of item \eqref{it:Cramer} we obtain, for any $\epsilon>0$ and fixed $\eta>0$, that, with some $C>0$,  		\begin{equation}\label{eq:725}
  \limsupi{x}x^{-\thPsml+\epsilon}\Pbb{\IPsi>x}\geq\limsupi{x}x^{-\thPsml+\epsilon}\Pbb{I_{\Psime}>x}=C\limsupi{x}x^{\epsilon}=\infty.
  \end{equation}
  Therefore \eqref{eq:largedPs} holds with $\overline{d}=-\thPsml+\epsilon$. Thus, it remains to show that \[\limo{\eta}\thPsml=\dphn\in\lbrbb{-\infty,0},\] where $\dphn>-\infty$ is evident since $z\mapsto -\Psi(-z)\notin\Bc$. It is clear from the identity
  \begin{equation*}
  \begin{split}
  \Psime\!\lbrb{u}&=\frac{\sigma^2}{2}u^2+c u+ \int_\R\lbrb{e^{ur}-1-ur\ind{|r|<1}}\lbrb{\Pi(dr)\ind{r>0}+e^{-\eta r^2\ind{r<-1}}\Pi(dr)\ind{r<0}},	
  \end{split}
  \end{equation*}
  that for any $u\leq 0$, $\Psime\!(u)$ is increasing in $\eta$ with  \[\limo{\eta}\Psime\!(u)=\Psi(u).\]  Fix $\eta_0>0$ small enough. Recall that, for any $\eta>0$, $\thPsml:=\mathfrak{u}_{\phn^{\leta}}$. Then, clearly, \[\Psi\!\lbrb{\tphn^{ (\eta_0)}}=\limo{\eta}\Psime\!\!\lbrb{\tphn^{(\eta_0)}}\geq\Psi^{(\eta_0)}\!\!\lbrb{\tphn^{(\eta_0)}}=0.\] 
  If $\Psi\!\lbrb{\tphn^{(\eta_0)}}=\infty$ then from \eqref{eq:WH1} we get that $\phn\!\lbrb{\tphn^{(\eta_0)}}=-\infty$ and thus $\tphn^{(\eta_0)}\leq \dphn$, see the definition \eqref{eq:dphi}. Similarly, if \[\Psi\!\lbrb{\tphn^{(\eta_0)}}=-\phn\!\lbrb{\tphn^{(\eta_0)}}\php\!\lbrb{-\tphn^{(\eta_0)}}\in\lbbrb{0,\infty}\] then $\phn\!\lbrb{\tphn^{(\eta_0)}}\leq 0$ and again $\tphn^{(\eta_0)}\leq \dphn$. Therefore,
  \[\underline{\mathfrak{u}}=\limsupo{\eta}\thPsml\leq \dphn.\]
  Assume that $\underline{\mathfrak{u}}<\dphn$ and choose $u\in\lbrb{\underline{\mathfrak{u}},\dphn}$. Then for any $\eta>0$ small enough $\thPsml<u$ and hence $\phn^{\leta}(u)>0$. Therefore, $\Psime\!(u)=-\php^{\leta}(-u)\phn^{\leta}(u)<0$ and thus \[\Psi(u)=\limo{\eta}\Psime(u)\leq 0.\] However, for $u<\dphn$, \[\Psi(u)=-\php(-u)\phn(u)\in\lbrbb{0,\infty},\] which is a contradiction and triggers the validity of $\underline{\mathfrak{u}}=\dphn$. Moreover, the monotonicity of $\Psime$ when $\eta\downarrow 0$ shows that in fact \[\limo{\eta}\thPsml=\underline{\mathfrak{u}}=\dphn\] and from \eqref{eq:725} valid for any $\eta,\epsilon>0$ we conclude the statement \eqref{eq:largedPs} when $\Psi\in\Nc_\Zc$. Hence, \eqref{eq:logLargedPs} holds true for $\Psi\in\Nc_\Zc$. Next, if $\Psi\notin\Nc_\Zc$ then as in Theorem \ref{thm:Wp}\eqref{it:D} one can check that
  \[\Pi(dr)=\sum_{n=-\infty}^\infty c_n \delta_{hn}(dr),\,\sum_{n=-\infty}^{\infty}c_n<\infty,\,c_n\geq 0,\,h>0, \text{ and $\sigma^2=\gamma=0$ in \eqref{eq:lk1}.}\]
  The underlying \LL process, $\xi$, is a possibly killed compound  Poisson process living on the lattice $\curly{hn}_{n=-\infty}^\infty$. We proceed by approximation. Set \[\Psi^{(\dr)}(z)=\Psi(z)+\dr z,\,\dr>0,\] with underlying \LL process $\xi^{(\dr)}$. Clearly, $\xi_t^{(\dr)}= \xi_t+\dr t,\,t\geq0,$ and hence $\IPsi\geq I_{\Psi^{(\dr)}}$ and, for any $x>0$, \[\Pbb{\IPsi> x}\geq\Pbb{I_{\Psi^{(\dr)}}>x}.\] Set \[\Psi^{(\dr)}\!(z)=-\php^{(\dr)}\!(-z)\phn^{(\dr)}\!(z),\,\,\, \dphn^{(\dr)}=\mathfrak{\overline{a}}_{\phn^{(\dr)}}, \text{ see \eqref{eq:dphi1}}.\] However, $\Psi^{(\dr)}\in\Nc_\Zc$ and therefore from \eqref{eq:largedPs}, for any $\overline{d}>\abs{\dphn^{(\dr)}}$, we deduct that
  \[\liminfi{x}x^{\overline{d}}\Pbb{\IPsi> x}\geq\liminfi{x}x^{\overline{d}}\Pbb{I_{\Psi^{(\dr)}}> x}=\infty.\]
  To establish \eqref{eq:largedPs} for $\Psi$ it remains to confirm that $\limo{
  	\dr}\dphn^{(\dr)}=\dphn$. Note that adding $\dr z$ to $\Psi(z)$ does not alter its range of analyticity and hence with the obvious notation $\aphn=\mathfrak{a}_{\phn^{(\dr)}}$. Set \[T=\inf\curly{t>0:\,\xi_t<0}\in\lbrbb{0,\infty} \text{ and } T^{(\dr)}=\inf\curly{t>0:\,\xi^{(\dr)}_t<0}\in\lbrbb{0,\infty}.\]
  Clearly, from $\xi_t^{(\dr)}= \xi_t+\dr t$ we get that  a.s.
  \[\limo{\dr}\lbrb{T^{(\dr)},\xi_{T^{(\dr)}}}= \lbrb{T,\xi_T}.\]
   However, $\lbrb{T,\xi_T}$ and $\lbrb{T^{(\dr)},\xi_{T^{(\dr)}}}$ define the distribution of the bivariate descending ladder time and height processes of $\Psi$ and $\Psi^{(\dr)}$, see Section \ref{subsec:LP}. Therefore, since from Lemma \ref{lem:WH} we can choose $\phn,\phn^{(\dr)}$ to represent the  descending ladder height process, that is $\phn=\kappa_{\mis},\,\phn^{(\dr)}=\kappa^{(\dr)}_{\mis}$ in the notation therein, we conclude that $\limo{\dr}\phn^{(\dr)}(z)=\phn(z)$ for any $z\in\Cb_{\lbrb{\aphn,\infty}}$ and hence $\limo{\dr}\mathfrak{u}_{\phn^{(\dr)}}=\tphn$. Thus $\limo{
  	\dr}\dphn^{(\dr)}=\dphn$. This, concludes Theorem \ref{thm:largeAsymp}. \qed
  
  \subsection{Proof of Theorem \ref{thm:smallTime}}\label{subsec:smallTime}
  Recall from \eqref{eq:MTint1} that $\MPIsi(z)$ is the Mellin transform of $\PIoPs{x}$ at least on $\Cb_{\lbrb{-1,0}}$.  We record and re-express it with the help of \eqref{eqM:MIPsi} and \eqref{eq:Wp} as
  \begin{equation}\label{eq:M*}
  \begin{split}
  \MPIsi(z)&=-\frac{1}{z}\MPsi(z+1)=-\frac{\phn(0)}{z}\MPs(z+1)\\
  &=-\frac{\phn(0)}{z}\frac{\Gamma(z+1)}{\Wpp\!(z+1)}\Wpn\!(-z)
  ,\,z\in\Cb_{\lbrb{-1,0}}.
  \end{split}
  \end{equation}
  From Theorem \ref{lemma:FormMellin}  we deduct that $\MPsi\in\Ac_{\lbrb{0,\,1-\aphn}}$ and that it extends continuously to $\Cb_{0}\setminus\curly{0}$. Moreover, Proposition \ref{prop:extendDecay} applied to  $\MPsi(z)$ for $z\in i\R$ shows that the decay of $\abs{\MPsi(z)}$  is preserved along $\Cb_{0}$. Therefore, either the rate of decay of $\abs{\MPIsi(z)}$ is faster than any power along $\Cb_a,\,a\in\lbbrb{-1,0},$ or it decays with power-like exponent of value $\NPs+1>1$.
  Via a Mellin inversion, choosing a contour, based on the line $\Cb_{-1}$ and a semi-circle, as in the proof of Theorem \ref{thm:largeAsymp}, see \eqref{eq:boundaryContour} and \eqref{eq:past}, we get that, for any $c\in\lbrb{0,\frac12}$, as $x\to0$,
  \begin{equation}\label{eq:F_Psi}
  \PIoPs{x}=\frac{1}{2\pi i}\int_{B^{+}(-1,c)}x^{-z}\MPIsi\!(z)dz+\sospace{x},
  \end{equation}
  where only the contour is changed to $B^{+}(-1,c)=\lbcurlyrbcurly{z\in\Cb:\,|z+1|=c \textrm{ and } \,\Re(z+1)\geq0}$. Apply \eqref{eq:Wp} to write from \eqref{eq:M*}
  \begin{equation}\label{eq:M*2}
  \MPIsi(z)=-\frac{\phn(0)}{z}\frac{\php\!\lbrb{z+1}}{z+1}\frac{\Gamma(z+2)}{\Wpp\!(z+2)}\Wpn\!(-z)=\frac{\php\lbrb{z+1}}{z+1}\MPIssi(z).
  \end{equation}
  Clearly, $\MPIssi\in\Ac_{\lbrb{-2,0}}$.  Recall that $\php^\sharp(z)=\php(z)-\php(0)\in\Bc$.
  Then, we have, noting $\MPIssi(-1)=\phn(0)$, that
  \begin{equation}\label{eq:M^*decomp}
  \begin{split}
  \MPIsi(z)&=\frac{\php(0)\phn(0)}{z+1}+\frac{\php^\sharp(z+1)}{z+1}\MPIssi(-1)+\php(z+1)\frac{\MPIssi(z)-\MPIssi(-1)}{z+1}\\
  \nonumber&=H_{\Psi,1}(z)+H_{\Psi,2}(z)+H_{\Psi,3}(z).
  \end{split}
  \end{equation}
  Plugging this in \eqref{eq:F_Psi} we get and set
  \begin{equation}\label{eq:FDecom}
  \PIoPs{x}=\frac{1}{2\pi i}\int_{B^{+}(-1,c)}x^{-z}\sum_{j=1}^3H_{\Psi,j}(z)dz+\sospace{x}=\sum_{j=1}^3F^{(c)}_j\!(x)+\sospace{x}.
  \end{equation}
  Since \[z\mapsto \frac{H_{\Psi,3}(z)}{\php\lbrb{z+1}}=\frac{\MPIssi(z)-\MPIssi(-1)}{z+1}\in\Ac_{\lbrb{-2,0}}\] and $\php(0)\in\lbbrb{0,\infty}$ then precisely as in \eqref{eq:pstarstar} and \eqref{eq:pstarstar1}, that is by shifting the contour from $B^{+}(-1,c)$ to $[-1-ic,-1+ic],$ we get that
  \begin{equation}\label{eq:F33}
  \limo{c}\limo{x}x^{-1}F^{(c)}_3\!(x)=0.
  \end{equation}
  Next, from \eqref{eq:phi} and the identity $\MPIssi(-1)=\phn(0)$ we get that \[H_{\Psi,2}(z)=\frac{\php^\sharp(z+1)}{z+1}\MPIssi(-1)=\phn(0)\int_{0}^{\infty}e^{-(z+1)y}\mubarp{y}dy+\phn(0)\dep.\]
  Clearly, if \[\lbrb{\php^\sharp}'\!\!(0)=\php'(0)=\IntOI\mubarp{y}dy<\infty,\] then the same arguments and shift of contour used above to prove \eqref{eq:F33} yield that
  \begin{equation}\label{eq:F32}
  \limo{c}\limo{x}x^{-1}F^{(c)}_2(x)=	\limo{c}\limo{x}\frac{x^{-1}}{2\pi i}\int_{B^{+}(-1,c)}x^{-z}H_{\Psi,2}(z)dz=0.
  \end{equation}
  Also, similarly, we deduce that in any case the term $\phn(0)\dep$ does not contribute to \eqref{eq:F32} and therefore assume that $\dep=0$ in the sequel.
  However, if \[\lbrb{\php^\sharp}'\!\!\lbrb{0^+}=\php'\!\lbrb{0^+}=\infty\] we could not apply this argument. We then split \[\mubarp{y}=\mubarp{y}\ind{y>1}+\mubarp{y}\ind{y\leq 1}\] and write accordingly
  \begin{equation*}
  \begin{split}
  H_{\Psi,2}(z)&=\phn(0)\int_{1}^{\infty}e^{-(z+1)y}\mubarp{y}dy+\phn(0)\int_{0}^{1}e^{-(z+1)y}\mubarp{y}dy\\
  &=H_{\Psi,2,1}(z)+H_{\Psi,2,2}(z).
  \end{split}
  \end{equation*}
  However, $\abs{H_{\Psi,2,2}(-1)}=\phn(0)\int_{0}^{1}\mubarp{y}dy<\infty$ and the same arguments show that the portion of $H_{\Psi,2,2}$ in $F^{(c)}_2$ is negligible in the sense of \eqref{eq:F32}. Then, we need discuss solely the contribution of $H_{\Psi,2,1}$ to $F^{(c)}_2$ that is
  \begin{equation*}
  \begin{split}
  F^{(c)}_{2,1}(x)&=\frac{\phn(0)}{2\pi i}\int_{B^{+}(-1,c)}x^{-z}\int_{1}^{\infty}e^{-(z+1)y}\mubarp{y}dydz\\
  &=\frac{\phn(0)}{2\pi i}\int_{1}^{\infty}\int_{B^{+}(-1,c)}x^{-z}e^{-(z+1)y}dz\mubarp{y}dy.
  \end{split}
  \end{equation*}
  The interchange is possible since evidently on $B^{+}(-1,c)$ we have that \[\sup_{z\in B^{+}(-1,c)}\abs{H_{\Psi,2,1}(z)}<\infty.\]
  The latter in turn follows from
  \[\sup_{z\in B^{+}(-1,c)}\abs{H_{\Psi,2,1}(z)}=\sup_{z\in B^{+}(-1,c)}\phn(0)\frac{\abs{\php^\sharp(z+1)}}{\abs{z+1}}=\sup_{z\in B^{+}(-1,c)}\phn(0)\frac{\abs{\php^\sharp(z+1)}}{c}<\infty.\]
  However, for any $x,y>0$, $z\mapsto x^{-z}e^{-(z+1)y}$ is an entire function and an application of the Cauchy theorem to the closed contour $B^{+}(-1,c)\cup \curly{-1+i\beta,\,\beta\in\lbbrbb{-c,c}}$ implies that
  \[\int_{B^{+}(-1,c)}x^{-z}e^{-(z+1)y}dz=ix\int_{-c}^{c}e^{-i\beta\lbrb{\ln x  +y}}d\beta=2x\frac{\sin\lbrb{c\lbrb{\ln x  +y}}}{\ln x  +y}i.\]
  Therefore, integrating by parts after representing $\mubarpspace{y}=\int_{y}^{\infty}\mu_{\pls}(dv)$, we get that
  \begin{equation*}
  \begin{split}
  F^{(c)}_{2,1}(x)&=x\frac{\phn(0)}{\pi }\limi{A}\int_{1}^{A}\frac{\sin\lbrb{c\lbrb{\ln x  +y}}}{\ln x  +y}\mubarp{y}dy\\
  &=x\frac{\phn(0)}{\pi }\limi{A}\int_{1}^{A}\int_{c\lbrb{1+\ln x  }}^{c\lbrb{v+\ln x  }}\frac{\sin\lbrb{y}}{y}dy\mu_{\pls}(dv)\\
  &+\,x\frac{\phn(0)}{\pi }\limi{A}\mubarpspace{A}\int_{c\lbrb{\ln x  +1}}^{c\lbrb{A+\ln x  }}\frac{\sin\lbrb{y}}{y}dy\\
  &=x\frac{\phn(0)}{\pi }\int_{1}^{\infty}\int_{c\lbrb{1+\ln x  }}^{c\lbrb{v+\ln x  }}\frac{\sin\lbrb{y}}{y}dy\mu_{\pls}(dv),
  \end{split}	
  \end{equation*}
  where we have used that the second limit in the second relation to the right-hand side vanishes because
  \[\sup_{a<b;a,b\in\Rb}\abs{\int_{a}^{b}\frac{\sin\lbrb{y}}{y}dy}<\infty,\]
  the mass of $\mu_+(dv)$ on $\lbrb{1,\infty}$ is simply $\mubarpspace{1}<\infty$ and $\limi{A}\mubarpspace{A}=0$. Also, this allows via the dominated convergence theorem to deduce that
  \begin{equation*}
  \limo{x}x^{-1}F^{(c)}_{2,1}(x)=\frac{\phi_-(0)}{\pi }\limo{x}\int_{1}^{\infty}\int_{c\lbrb{1+\ln x}}^{c\lbrb{v+\ln x}}\frac{\sin\lbrb{y}}{y}dy\mu_{\pls}(dv)=0.
  \end{equation*}
  Thus with the reasoning above we verify that \eqref{eq:F32} holds, that is $\limo{c}\limo{x}x^{-1}F^{(c)}_2(x)=0$. The latter together with \eqref{eq:F33}  applied in \eqref{eq:FDecom} allows us to understand the asymptotic of $x^{-1}F_{\Psi}(x)$, as $x\to 0$, in terms of $F^{(c)}_1$. However, from its very definition, \eqref{eq:M^*decomp} and $\Psi(0)=-\phn(0)\php(0)$, see \eqref{eq:WH1},
  \begin{equation*}
  F^{(c)}_1(x)=-\frac{\Psi(0)}{2\pi i}\int_{B^{+}(-1,c)}\frac{x^{-z}}{z+1}dz=-\Psi(0)x+\frac{\Psi(0)}{2\pi i}\int_{B^{-}\!(-1,c)}\frac{x^{-z}}{z+1}dz,
  \end{equation*}
  where we recall that
  \[B^{-}\!(-1,c)=\lbcurlyrbcurly{z\in\Cb:\,|z+1|=c \textrm{ and } \Re(z+1)\leq0}.\]
  Clearly, if $z\in B^{-}(-1,c)\setminus\curly{-1\pm ic}$ then $\limo{x}\abs{x^{-1}x^{-z}}=0$. Therefore, for any $c\in\lbrb{0,\frac12}$,
  \begin{equation}\label{eq:Fasym}
  \limo{x}x^{-1}	\PIoPs{x}=\limo{x}x^{-1}F^{(c)}_1(x)=-\Psi(0)\in\lbbrb{0,\infty}.
  \end{equation}
  Thus \eqref{eq:smallTime} holds true. When $\Psi\in\Nn$, for some $\NPs>1$, all arguments above applied to $\MPIsi$ can be  carried over directly  to $\MPsi$ but at $z=0$. When $f_\Psi$ is continuous at zero then the result is immediate from \eqref{eq:smallTime}. Thus, we obtain \eqref{eq:smallTime1}. \qed

  \subsection{Proof of  Theorem \ref{thm:momentsIt}\ref{it:negative}}\label{subsec:momentsIt}
  Let $\Psi\in\overNc\setminus\Nc_\dagger$ that is $\Psi(0)=0$, see \eqref{eq:Ncdag}. Recall that $\IPsi(t)=\int_{0}^{t}e^{-\xi_s}ds$. Then, if $\aphp<0$, we get that, for any $a\in\lbrb{0,-\aphp}$,
  \begin{equation}\label{eq:finiteMoment}
  \Ebb{\IPsi^{-a}\!(t)}\leq t^{-a}\Ebb{e^{a\sup_{v\leq t}\xi_v}}<\infty,
  \end{equation}
  where the finiteness of the exponential moments of $\sup_{v\leq t}\xi_v$ of order less than $-\aphp$ follows from the definition of $\aphp$, see \eqref{eq:aphi}, that is $\php\in \Ac_{\lbrb{-a,\infty}}$ for $a\in\lbrb{0,-\aphp}$, see e.g. \cite[Chapter VI]{Bertoin-96} or \cite[Chapter 4]{Doney-07-book}. This of course settles \eqref{eq:momentsIt}, that is \[a\in\lbrb{0,1-\aphp}\implies \Ebb{\IPsi^{-a}(t)}<\infty,\] when $\aphp=-\infty$.
  Next, we rewrite \eqref{eq:lk1} as follows
  \begin{equation}\label{eq:decomp}
  \begin{split}
  \Psi(z)&=\Psi_1(z)+\Psi_2(z)\\
  &=\lbrb{\frac{\sigma^2}{2} z^2 + \gamma z +\int_{-\infty}^{1} \left(e^{zr} -1
  	-zr \ind{|r|<1}\right)\Pi(dr)}+\int_{1}^{\infty}\left(e^{zr} -1
  \right)\Pi(dr),
  \end{split}
  \end{equation}
  where $\Psi_1,\Psi_2\in\overNc\setminus \Nc_\dagger$ and from \eqref{eq:decomp} $\lbrb{\xi_s}_{s\geq0}=\lbrb{\xi^{(1)}_s+\xi^{(2)}_s}_{s\geq0}$ where $\xi^{(1)},\xi^{(2)}$ are independent \LLPs with \LLK exponents $\Psi_1,\Psi_2$ respectively.
  Set as usual $\Psi_1\!(z)=-\php^{(1)}\!(-z)\phn^{(1)}\!(z)$ and note from \eqref{eq:decomp} that $\Psi_1\in\Ac_{\lbrb{0,\infty}}$ and hence $\php^{(1)}\in\Ac_{\lbrb{-\infty,0}}$ or equivalently $\mathfrak{a}_{\php^{(1)}}=-\infty$, see \eqref{eq:aphi}. Also note that $\mathfrak{a}_{\php^{(2)}}=\aphp$. Similarly to \eqref{eq:finiteMoment} we get that, for any $a\in\lbrb{0,1-\aphp}$,
  \begin{equation}\label{eq:finiteMoment1}
  \Ebb{\IPsi^{-a}(t)}\leq \Ebb{e^{a\sup_{v\leq t}\xi^{(1)}_v}}\Ebb{I^{-a}_{\Psi_2}(t)}.
  \end{equation}
  However, since $\mathfrak{a}_{\php^{(1)}}=-\infty$ we conclude from \eqref{eq:finiteMoment} that in fact $\Ebb{e^{a\sup_{v\leq t}\xi^{(1)}_v}}<\infty$, for any $a>0$.
  If $\Psi_2\equiv 0$ there is nothing to prove. So let \[h=\Pi\!\lbrb{\curly{1,\infty}}>0\] and write $\xi^{(2)}_s=\sum_{j=1}^{N_s}X_j$, where $\lbrb{N_s}_{s\geq 0}$ is a Poisson counting process  with parameter $h>0$ and $\lbrb{X_j}_{j\geq 1}$ are independent and identically distributed  random variables with law \[\Pbb{X_1\in dx}=\ind{x>1} \frac{\Pi(dx)}h.\]
  It is a well-known fact that $\lambda<-\aphp \Longrightarrow \Ebb{e^{\lambda X_1}}<\infty$. 
  Set, for $n \in \N$, \[S_n=\sum_{j=1}^{n}X_j \text{ with } S_0=0.\] Then
  \[I_{\Psi_2}(t)=t\ind{N_t=0}+\ind{N_t>0}\lbrb{\sum_{j=1}^{N_t}\textbf{e}_{j} e^{-S_{j-1}}+\lbrb{t-\sum_{j=1}^{N_t}\textbf{e}_{j}}e^{-S_{N_t}}}\]
  with $\lbrb{\textbf{e}_j}_{j\geq 1}$ a sequence of independent and identically distributed random variables with exponential law of parameter  $h$. Clearly then
  \begin{equation}\label{eq:repIt}
  \begin{split}
  \Ebb{I_{\Psi_2}(t)^{-a}}&=t^{-a}\Pbb{N_t=0}+\sum_{n=1}^{\infty}\Ebb{\lbrb{\sum_{j=1}^{n}\textbf{e}_{j} e^{-S_{j-1}}+\lbrb{t-\sum_{j=1}^{n}\textbf{e}_{j} }e^{-S_n}}^{-a}\!\!;N_t=n}\\
  &=t^{-a}\Pbb{N_t=0}+\sum_{n=1}^{\infty}A_n(t).\\
  \end{split}
  \end{equation}
  Note that
  \[\curly{N_t=n}=\curly{\sum_{j=1}^{n}\textbf{e}_{j}\leq \frac{t}{2}; \sum_{j=1}^{n+1}\textbf{e}_{j}>t}\bigcup\curly{\sum_{j=1}^{n}\textbf{e}_{j}\in \lbrb{\frac{t}{2},t}; \sum_{j=1}^{n+1}\textbf{e}_{j}>t}.\]
  We observe, {since $N_s$ follows a Poisson distribution with parameter $hs$, that, for any $t>0$ and $n-1 \in \N$,
  	\begin{equation}\label{eq:Nt}
  	\sup_{s\in\lbrbb{0,t}}\Pbb{N_s\geq n-1}\leq e^{ht}\frac{t^{n-1}h^{n-1}}{(n-1)!}.
  	\end{equation}
  	We split the quantity $A_n(t)=A_n^{(1)}(t)+A_n^{(2)}(t)$ by considering the two possible mutually exclusive cases for the event $\curly{N_t=n}$ discussed above. In the first scenario we have the following sequence of relations,  where we use the bound \eqref{eq:Nt} in the third line below,
  	\begin{equation}\label{eq:An1}
  	\begin{split}
  	A_n^{(1)}(t)&\,\,=\Ebb{\lbrb{\sum_{j=1}^{n}\textbf{e}_{j} e^{-S_{j-1}}+\lbrb{t-\sum_{j=1}^{n}\textbf{e}_{j} }e^{-S_n}}^{-a};\sum_{j=1}^{n}\textbf{e}_{j}\leq \frac{t}{2}; \sum_{j=1}^{n+1}\textbf{e}_{j}>t}\\
  	&\,\,\leq \Ebb{\lbrb{\textbf{e}_{1}+\frac{t}{2}e^{-S_n}}^{-a};\,\textbf{e}_{1}\leq \frac{t}{2};\sum_{j=2}^{n}\textbf{e}_{j}<t}\\
  	&\,\,\leq\Ebb{\lbrb{\textbf{e}_{1}+\frac{t}{2}e^{-S_n}}^{-a};\,\textbf{e}_{1}\leq \frac{t}{2}}\Pbb{N_t\geq n-1}\\
  	&\,\,\leq  e^{ht}\frac{h^nt^{n-1}}{(n-1)!} \Ebb{\int_{0}^{\frac{t}{2}}\lbrb{x+\frac{t}{2}e^{-S_n}}^{-a}e^{-hx}dx}\\
  	&\,\,\leq e^{ht}\frac{h^nt^{n-1}}{(n-1)!}\Ebb{\int_{0}^{\frac{t}{2}}\lbrb{x+\frac{t}{2}e^{-S_n}}^{-a}dx}\\
  	&\,\,\leq e^{ht}\frac{h^nt^{n-1}}{(n-1)!}\lbrb{\frac{2^{a-1}}{(a-1)t^{a-1}}\Ebb{e^{(a-1)S_n}} \ind{a>1}+\lbrb{\Ebb{S_n}+\ln(4)}\ind{a=1}}\\
  	&\quad\,\,+e^{ht}\frac{h^nt^{n-1}}{(n-1)!}\frac{t^{1-a}}{1-a}\ind{a\in\lbrb{0,1}}\\
  	&\,\,=e^{ht}\frac{h^nt^{n-1}}{(n-1)!}\lbrb{\frac{(2/t)^{a-1}}{\lbrb{a-1}}\lbrb{\Ebb{e^{(a-1)X_1}}}^n\ind{a>1}+\lbrb{n\Ebb{X_1}+\ln(4)}\ind{a=1}}\\
  	&\quad\,\,+e^{ht}\frac{h^nt^{n-1}}{(n-1)!}\frac{t^{1-a}}{1-a}\ind{a\in\lbrb{0,1}}
  	\end{split}
  	\end{equation}
  	where for the terms containing $\ind{a=1}$ and $\ind{a\in\lbrb{0,1}}$ in the derivation of the last inequality we have used that $S_n\geq n>0$ a.s.~since a.s.~$X_1\geq1$. Summing over $n$ we get that
  	\begin{equation}\label{eq:A1}
  	\begin{split}
  	\sum_{n=1}^{\infty}A_n^{(1)}(t)&\leq  t^{1-a}\frac{2^{a-1}h\Ebb{e^{(a-1)X_1}}}{a-1}e^{ht\lbrb{\Ebb{e^{(a-1)X_1}}+1}}\ind{a>1}+t^{1-a}\frac{h}{1-a}e^{2ht}\ind{a\in\lbrb{0,1}}\\
  	&+\lbrb{\Ebb{X_1}\lbrb{1+2ht}e^{2ht}+he^{2ht}\ln(4)}\ind{a=1}.
  	\end{split}
  	\end{equation}
  	In the second scenario we observe that the following inclusion holds
  	\begin{equation}\label{eq:inclusion}
  	\curly{\sum_{j=1}^{n}\textbf{e}_{j}\in \lbrb{\frac{t}{2},t}; \sum_{j=1}^{n+1}\textbf{e}_{j}>t}\subseteq\bigcup_{j=1}^{n}\lbrb{\curly{\textbf{e}_{j}\geq \frac{t}{2n}}\bigcap\curly{\sum_{1\leq i\leq n;i\neq j}\textbf{e}_{i}<t}}.
  	\end{equation}
  	Clearly then for any $j$, the events  $\curly{\textbf{e}_{j}\geq \frac{t}{2n}}$ and $\curly{\sum_{1\leq i\leq n;i\neq j}\textbf{e}_{i}<t}$ are independent and moreover
  	\[\Pbb{\sum_{1\leq i\leq n;i\neq j}\textbf{e}_{i}<t}=\Pbb{\sum_{1\leq i\leq n-1}\textbf{e}_{i}<t}= \Pbb{N_t\geq n-1}.\]
  	We are therefore able to estimate $A_n^{(2)}(t)$ using the relation between events in \eqref{eq:inclusion} in the following manner
  	\begin{equation*}
  	\begin{split}
  	A_n^{(2)}(t)&=\Ebb{\lbrb{\sum_{j=1}^{n}\textbf{e}_{j} e^{-S_{j-1}}+\lbrb{t-\sum_{j=1}^{n}\textbf{e}_{j} }e^{-S_n}}^{-a};\sum_{j=1}^{n}\textbf{e}_{j}\in\lbrb{\frac{t}{2},t}; \sum_{j=1}^{n+1}\textbf{e}_{j}>t}\\
  	&\leq \Ebb{\textbf{e}_{1}^{-a};t>\textbf{e}_{1}\geq \frac{t}{2n}; \sum_{j=2}^{n}\textbf{e}_{j}<t}\\
  	&+\,\,\ind{n>1}\sum_{j=2}^n\Ebb{\lbrb{\textbf{e}_{1}+\frac{t}{2n}e^{-S_{j-1}}}^{-a};\textbf{e}_{j}\geq \frac{t}{2n}; \sum_{i=2,i\neq j}^{n}\textbf{e}_{i}<t;\textbf{e}_{1}<t}\\
  	&\leq h\lbrb{\abs{\frac{1}{1-a}}\lbrb{2n}^at^{1-a}\ind{a\neq 1}+\ln(2n)\ind{a=1}}\Pbb{N_t\geq n-1}\\
  	&+\,\,\ind{n>1}\lbrb{\sum_{j=2}^{n} h\Ebb{\int_{0}^{t}\frac{1}{\lbrb{x+\frac{t}{2n}e^{-S_{j-1}}}^{a}}dx}}\Pbb{N_t\geq n-2},
  	\end{split}
  	\end{equation*}
  	where the terms in last estimate are derived in the first case by disintegration of $\textbf{e}_{1}$ on $\curly{t>\textbf{e}_{1}\geq \frac{t}{2n}}$ and the independence of $\sum_{j=2}^{n}\textbf{e}_{j}$ from $\textbf{e}_{1}$ and in the second case by disintegration of $\textbf{e}_{1}$ on $\lbrb{0,t}$, removing the set $\curly{\textbf{e}_{j}\geq \frac{t}{2n}}$  and invoking independence again for the remaining term  $\sum_{i=2,i\neq j}^{n}\textbf{e}_{i}$.
  	However, performing the integration and estimating precisely as in \eqref{eq:An1} we get with the help of \eqref{eq:Nt} applied only to the term  $\Pbb{N_t\geq n-2}$ that
  	\begin{equation}\label{eq:An2}
  	\begin{split}
  	A_n^{(2)}(t)
  	&\leq h\lbrb{\abs{\frac{1}{1-a}}\lbrb{2n}^at^{1-a}\ind{a\neq 1}+\ln(2n)\ind{a=1}}\Pbb{N_t\geq n-1}\\
  	&+\ind{n>1}\ind{a>1} e^{ht}\frac{t^{n-2}h^{n-1}}{(n-2)!}\sum_{j=2}^{n}\frac{\lbrb{2n}^{a-1}}{\lbrb{a-1}t^{a-1}}\lbrb{\Ebb{e^{\lbrb{a-1}X_1}}}^{j-1}\\
  	&+\ind{n>1}\ind{a=1} e^{ht}\frac{t^{n-2}h^{n-1}}{(n-2)!}\sum_{j=2}^{n}\lbrb{j\Ebb{X_1}+\ln\lbrb{4n}}\\
  	&+\ind{n>1}\ind{a\in\lbrb{0,1}} e^{ht}\frac{t^{n-2}h^{n-1}}{(n-2)!}\sum_{j=2}^{n}\frac{t^{1-a}}{1-a}\ind{a\in\lbrb{0,1}}\\
  	&\leq h\lbrb{\abs{\frac{1}{1-a}}\lbrb{2n}^at^{1-a}\ind{a\neq 1}+\ln(2n)\ind{a=1}}\Pbb{N_t\geq n-1}\\
  	&+\ind{n>1}\Big (e^{ht}\frac{t^{n-1-a}h^{n-1}}{(n-2)!}\frac{\lbrb{4n}^a }{a-1}\lbrb{\Ebb{e^{(a-1)X_1}}}^{n-1}\ind{a>1}\\
  	&+e^{ht}\frac{t^{n-2}h^{n-1}}{(n-2)!}\lbrb{n^2\Ebb{X_1}+n\ln\lbrb{4n}}\ind{a=1}+e^{ht}\frac{n}{1-a}\frac{t^{n-1-a}h^{n-1}}{(n-2)!}\ind{a\in\lbrb{0,1}}\Big).
  	\end{split}
  	\end{equation}
  	Summing over $n$ we arrive with the help of \eqref{eq:Nt} at
  	\begin{equation}\label{eq:A2}
  	\begin{split}
  	\sum_{n=1}^{\infty}A_n^{(2)}(t)&\leq 8^at^{1-a}\frac{he^{ht}\Ebb{e^{\lbrb{a-1}X_1}}}{a-1}\sum_{n=0}^{\infty}\frac{t^nh^n(n+2)^a\lbrb{\lbrb{\Ebb{e^{\lbrb{a-1}X_1}}+1}}^n}{n!}\ind{a>1}\\
  	&+e^{ht}h\sum_{n=0}^\infty \frac{t^nh^n}{n!}\lbrb{(n+2)^2\Ebb{X_1}+(n+3)\ln(4n)}\ind{a=1}\\
  	&+4t^{1-a}\frac{e^{ht}(h+1)}{1-a}\sum_{n=0}^\infty n\frac{t^nh^n}{n!}\ind{a\in\lbrb{0,1}}.
  	\end{split}
  	\end{equation}
  	Therefore from $A_n=A_n^{(1)}(t)+A_n^{(2)}(t)$, \eqref{eq:A1} and \eqref{eq:A2} applied in \eqref{eq:repIt} we easily get that
  	\[\Ebb{I_{\Psi_2}(t)^{-a}}=t^{-a}\Pbb{N_t=0}+\sum_{n=1}^\infty\lbrb{A_n^{(1)}(t)+A_n^{(2)}(t)}<\infty,\]
  	whenever either  $a\in\lbrb{0,1}$ or $a=1$ and $\Ebb{X_1}<\infty$ or $a>1$ and  $\Ebb{e^{(a-1)X_1}}<\infty$. However, it is a very well-known fact that
  	\[\Ebb{X_1}<\infty \iff \Ebb{\max\curly{\xi_1,0}}<\infty\iff \abs{\Psi'(0^+)}<\infty\]
  	and if $\aphp<0$ then
  	\[\Ebb{e^{(a-1)X_1}}<\infty\impliedby a\in\lbrb{0,1-\aphp}\]
  	and \[\Ebb{e^{(\aphp-1)X_1}}<\infty\iff \abs{\Psi(-\aphp)}<\infty\iff\abs{\php(\aphp)}<\infty.\] Hence via \eqref{eq:finiteMoment1} and the fact that $\Ebb{e^{a\sup_{v\leq t}\xi^{(1)}_v}}<\infty$, for any $a>0$, the relation \eqref{eq:momentsIt} and the backward directions of \eqref{eq:momentsIt1} and \eqref{eq:momentsIt2} follow. Let us provide a lower bound for $\Ebb{\IPsi^{-a}(t)}$ whenever $a\geq 1$. We again utilize the processes $\xi^{(1)},\xi^{(2)}$ in the manner
  	\begin{equation*}
  	\begin{split}
  	\Ebb{\IPsi(t)^{-a}}&\geq \Ebb{\IPsi(t)^{-a};\sup_{v\leq t}|\xi^{(1)}_v|\leq 1;N_t=1}\\
  	&\geq \Ebb{e^{-at\sup_{v\leq t}|\xi^{(1)}_v|}I^{-a}_{\Psi_2}(t);\sup_{v\leq t}|\xi^{(1)}_v|\leq 1;N_t=1}\\
  	&\geq e^{-at}\Pbb{\sup_{v\leq t}|\xi^{(1)}_v|\leq 1} \Ebb{I^{-a}_{\Psi_2}(t); N_t=1}\\
  	&\geq e^{-at}\Pbb{\sup_{v\leq t}|\xi^{(1)}_v|\leq 1}\Pbb{e_2>t}he^{-ht}\Ebb{\int_{0}^{t}\lbrb{x+te^{-X_1}}^{-a}dx}\\
  	&= e^{-at}\Pbb{\sup_{v\leq t}|\xi^{(1)}_v|\leq 1}\Pbb{e_2>t}he^{-ht}\\
  	&\times\lbrb{\lbrb{\frac{t^{1-a}}{a-1}\Ebb{e^{(a-1)X_1}}-1}\ind{a>1}+\Ebb{X_1}\ind{a=1}}.
  	\end{split}
  	\end{equation*}
  	The very last term is infinity if and only if $a-1>\aphp$ or  $a=1$ and $\Ebb{X_1}=\infty$ or  $a-1=\aphp$  and  $\Ebb{e^{\aphp X_1}}=\infty$. This shows the forward directions of  \eqref{eq:momentsIt1} and \eqref{eq:momentsIt2} and relation  \eqref{eq:momentsIt1} and \eqref{eq:momentsIt3}. It remains to show \eqref{eq:limitIt}. From  \eqref{eq:A1} and \eqref{eq:A2} we get that, for any $a>0$ and as $t\to 0$,
  	\[\sum_{n=1}^\infty A_n(t)\leq \sum_{n=1}^{\infty}A_n^{(1)}(t)+\sum_{n=1}^{\infty}A_n^{(2)}(t)=\bospace{t^{1-a}}.\]
  	Therefore from \eqref{eq:repIt} we get that
  	\[\limo{t}t^a\Ebb{\IPsi^{-a}(t)}=\limo{t}\lbrb{\Pbb{N_t=0}+t^a\sum_{n=1}^\infty A_n(t)}=\limo{t}\Pbb{N_t=0}=1.\]
  	This establishes the validity of \eqref{eq:limitIt} and concludes the proof of the whole theorem.
  	\subsection{Proof of  Theorem \ref{thm:momentsIt}\eqref{it:positive}}\label{subsec:Moments}  
  	Let $\Psi\in\overNc\setminus\Nc_\dagger$ that is $\Psi(0)=0$. Let $a>0$ and note that
  	\[\IPsi^a(t)\leq t^ae^{a\sup_{s\leq t}-\xi_s}=t^ae^{-a\inf_{s\leq t}\xi_s}.\]  
  	The finiteness of the negative exponential moments of $\inf_{s\leq t}\xi_s$ of order greater than $\aphn$ follows from the definition of $\aphn$, see \eqref{eq:aphi}, that is $\phn\in \Ac_{\lbrb{-a,\infty}}$ for $a\in\lbrb{0,-\aphn}$, see e.g. \cite[Chapter VI]{Bertoin-96} or \cite[Chapter 4]{Doney-07-book}. Thus \eqref{eq:pmomentsIt} is established as well as \eqref{eq:pmomentsIt1} when $\abs{\Psi\lbrb{\aphn}}<\infty$. To prove the opposite direction of \eqref{eq:pmomentsIt1} and \eqref{eq:pmomentsIt3} we find a suitable lower bound. As in the previous proof we decompose the \LLP $\xi$ underlying $\Psi$ as $\xi=\xi^1+\xi^2$ where $\xi^1$ collects all positive jumps and all jumps of value between $\lbrb{-1,0}$ and $\xi^2$ is a compound Poisson process of jumps less than $-1$. Then clearly
  	\[\IPsi^a(t)\geq e^{-a\inf_{s\leq t}\xi^1_s}\lbrb{\int_{0}^{t}e^{-\xi^2_s}ds}^a\]
  	and because $\xi^1$ has all negative exponential moments then, for any $a>0$,
  	\[J(a):=\Ebb{e^{-a\inf_{s\leq t}\xi^1_s}}<\infty.\]
  	Let $T_1=\inf_{s\geq 0}\curly{\abs{\xi^2_s-\xi^2_{s-}}>0}$ and $T_2$ stand for the time of the second jump of $\xi^2$. Also let $X_1=\xi^2_{T_1}$. Then
  	\[\Ebb{\IPsi^a(t)}\geq J(a)\Ebb{\lbrb{\int_{0}^{t}e^{-\xi^2_s}ds}^a;T_1\leq \frac{t}{2}; T_2>t}\geq J(a)\frac{t^a}{2^a}\Pbb{T_1\leq \frac{t}{2}; T_2>t}\Ebb{e^{-aX_1}}.\]
  	The left hand is infinite provided $a<\aphn$ or $a=\aphn$ and $\abs{\Psi\lbrb{\aphn}}=\infty$ because in this case $\Ebb{e^{-a\xi^2_1}}=\infty$ and hence  $\Ebb{e^{-aX_1}}=\infty$.
  	\subsection{Proof of  Theorem \ref{thm:Nc}}\label{subsec:Li}
  	Recall that
  	\[\Psi\in\overNc\setminus\Nc=\curly{\Psi\in\overNc:\,\phn(0)=0}\]
  	and as  above we write, for any $q\geq 0$, \[\Psi^{\dagger q}\!(z)=\Psi(z)-q=-\php^{\dagger q}(-z)\phn^{\dagger q}(z)\in\Nc.\]
  	We also repeat that $\IPsi(t)=\int_{0}^{t}e^{-\xi_s}ds,\,t\geq0,$ and $\IPsi(t)$ is non-decreasing in $t$. The relation \eqref{eq:It} then follows from the immediate bound
  	\begin{equation*}
  	\begin{split}
  	\Pbb{\IPsq{\dagger q}\leq x}=q\IntOI e^{-qt}\Pbb{\IPsi(t)\leq x}dt\geq \lbrb{1-e^{-1}}\Pbb{\IPsi\lbrb{1/q}\leq x},
  	\end{split}
  	\end{equation*}
  	combined with the representation \eqref{eq:smallExpansion} with $n=0$ for $\Pbb{\IPsiq{q}\leq x}$, the identity
  	\[\kappa_{\mis}(q,0)\php^{\dagger q}(0)=\phn^{\dagger q}(0)\php^{\dagger q}(0)=-\Psi(0)=q,\] which is valid since  \eqref{eq:WH} and Lemma \ref{lem:WH} hold. Thus, Theorem \ref{thm:Nc}\eqref{it:NC1} is settled and we proceed with Theorem \ref{thm:Nc}\eqref{it:NC2}.  We consider two separate cases.
  	
  	\textbf{Case $a\in\lbrb{0,1}$.} Denote by $\MPIsia{a}$, $a\in\lbrb{0,1}$, the Mellin transform of $V^{q}_a$ which denotes the cumulative distribution function of the measure $y^{-a}\Pbb{I_{\Psi^{\dagger q}}\in dy}\ind{y>0}$. Following \eqref{eq:M*} and noting from Lemma \ref{lem:WH} that  $\kappa_{\mis}(q,0)=\phn^{\dagger q}\!(0)$ we conclude that
  	\begin{equation}\label{eq:M*a}
  	\MPIsia{a}(z)=-\frac{1}{z}\MPsiq\!(z+1-a)=-\frac{\kappa_{\mis}(q,0)}{z}\MPsq{q}\!\lbrb{z+1-a}
  	,\,z\in\Cb_{\lbrb{a-1,0}},
  	\end{equation}
  	and at least $\MPIsia{a}\in\Ac_{\lbrb{a-1,0}}$ since $\MPsq{q}\in\Ac_{\lbrb{0,1}}$,  see Theorem \ref{lemma:FormMellin}\eqref{eq:domainAnalMPs}.
  	Similarly, as for any $\Psi\in\overNc$  such that $\Psi(z) \neq \dep z$,  which is the case under the conditions of this item, we have that $\NPs>0$ and we deduct that  $\limi{|b|}|b|^{\beta}\abs{\MPIsia{a}\!\lbrb{c+ib}}=0$ for some $\beta\in\lbrb{1,1+\NPs}$ and any $c\in\lbrb{a-1,0}$. Therefore, by Mellin inversion, for any $x>0$,
  	\begin{equation}\label{eq:MILi}
  	\begin{split}
  	V^{q}_a(x)&=\int_{0}^{x}y^{-a}\Pbb{I_{\Psi^{\dagger q}}\in\,dy}=\frac{x^{-c}}{2\pi}\IntII x^{-ib}\MPIsia{a}(c+ib)db\\
  	&=-\kappa_{\mis}(q,0)\frac{x^{-c}}{2\pi}\IntII x^{-ib}\frac{\MPsqq{^{\dagger q}}\!\!\lbrb{c+1-a+ib}}{c+ib}db.
  	\end{split}	
  	\end{equation}
  	However, since also
  	\begin{equation}\label{eq:Var}
  	V_a^q(x)=q\IntOI e^{-qt}\int_{0}^{x}y^{-a}\Pbb{I_{\Psi}(t)\in dy}dt,
  	\end{equation}
  	we have that
  	\begin{equation*}
  	\begin{split}
  	&\lim\limits_{q \to 0} \frac{q}{\kappa_{\mis}(q,0)}\IntOI e^{-qt}\int_{0}^{x}y^{-a}\Pbb{I_{\Psi}(t)\in dy}dt\\
  	&=\limo{q}\lbrb{ -\frac{x^{-c}}{2\pi}\IntII x^{-ib}\frac{\MPsqq{^{\dagger q}}\!\lbrb{c+1-a+ib}}{c+ib}db}\\
  	&=\limo{q}\lbrb{ -\frac{x^{-c}}{2\pi}\IntII x^{-ib}\frac{\lbrb{\MPsqq{^{\dagger q}}\!\lbrb{c+1-a+ib}-\MPs\!\lbrb{c+1-a+ib}}}{c+ib}db}\\
  	&\quad-\frac{x^{-c}}{2\pi}\IntII x^{-ib}\frac{\MPs\!\lbrb{c+1-a+ib}}{c+ib}db.    
  	\end{split}
  	\end{equation*}
  	From \eqref{eq:MrConv}, \eqref{eq:MrBoundedness} and \eqref{eq:uniformDecay} of Lemma \ref{cor:uniformDecay} with $\beta\in\lbrb{0,\NPs}$ we conclude that the dominated convergence theorem applies and yields that
  	\begin{equation}\label{eq:VarConv}
  	\begin{split}
  	&\lim\limits_{q \to 0} \frac{q}{\kappa_{\mis}(q,0)}\IntOI e^{-qt}\int_{0}^{x}y^{-a}\Pbb{I_{\Psi}(t)\in dy}dt= -\frac{x^{-c}}{2\pi}\IntII x^{-ib}\frac{\MPs\!\lbrb{c+1-a+ib}}{c+ib}db.
  	\end{split}
  	\end{equation}
  	The latter means that the renormalized Laplace transform of $\int_{0}^{x}y^{-a}\Pbb{I_{\Psi}(t)\in dy}$ converges.
  	Let $-\liminfi{t}\xi_t=\limsupi{t}\xi_t=\infty$ a.s.~or alternatively $\php\!(0)=\phn\!(0)=\kappa_{\mis}(0,0)=0$. Assume also that $\limi{t}\Pbb{\xi_t<0}=\rho\in\lbbrb{0,1}$. Then from the discussion succeeding \cite[Chapter 7, (7.2.3)]{Doney-07-book} (a chapter dedicated to equivalent statements related to the Spitzer's condition) we have that  $\kappa_{\mis}\!\lbrb{\cdot,0}\in RV_\rho$.
  	Since  \[\frac{\kappa_{\mis}(q,0)}q=\frac{\phn^{\dagger q}\!(0)}q=\frac{1}{\php^{\dagger q}\!(0)},\] then $\php^{\dagger \cdot}\!(0) \in RV_{1-\rho}$ with $1-\rho>0$, and since
  	\begin{equation}\label{eq:Var1}
  	t \mapsto		\int_{0}^{x}y^{-a}\Pbb{\IPsi(t)\in dy}=\Ebb{\IPsi^{-a}(t)\ind{\IPsi(t)\leq x}}
  	\end{equation}
  	is non-increasing  for any fixed $x>0$ and $a\in\lbrb{0,1}$, from a classical Tauberian theorem and the monotone density theorem, see \cite[Section 0.7]{Bertoin-96}, we conclude from \eqref{eq:MILi} and \eqref{eq:VarConv} that, for any $x>0$ and any $c\in\lbrb{a-1,0}$,
  	\begin{equation}\label{eq:MILit}
  	\limi{t}t\php^{\dagger \frac{1}{t}}(0)\int_{0}^{x}y^{-a}\Pbb{I_{\Psi}(t)\in dy}=-\frac{x^{-c}}{2\pi\Gamma\lbrb{1-\rho}}\IntII x^{-ib}\frac{\MPs\!\lbrb{c+1-a+ib}}{c+ib}db.
  	\end{equation}
  	With $t\php^{\dagger \frac{1}{t}}(0)=\frac{1}{\kappa_{\mis}(\frac{1}{t},0)}$ we deduce that
  	\[\frac{y^{-a}\Pbb{I_{\Psi}(t)\in dy}}{\kappa_{\mis}(\frac{1}{t},0)} \text{ converges vaguely to $\vartheta_a$},\]
  	whose distribution function is simply the integral to the right-hand side of \eqref{eq:MILit}. To show that it converges weakly, and thus prove \eqref{eq:weak} and \eqref{eq:V_a},  we need only show that
  	\[\limi{t}\frac{\Ebb{\IPsi^{-a}(t)}}{\kappa_{\mis}(\frac{1}{t},0)}=\limi{t}\frac{\IntOI y^{-a}\Pbb{I_{\Psi}(t)\in dy} }{\kappa_{\mis}(\frac{1}{t},0)}<\infty.\]
  	However, this is immediate from the fact that $\MPsiq(z),\,q>0,$ is always well defined on $\Cb_{\lbrb{0,1}}$, see \eqref{eq:MIPsi} of Theorem \ref{lemma:FormMellin}, $\limo{q}\MPsqq{\dagger q}(1-a)=\MPs(1-a),\,a\in\lbrb{0,1}$, justified in  \eqref{eq:MrConv} below, and utilizing again a Tauberian theorem and the monotone density theorem in
  	\begin{equation}\label{eq:MIPsir}
  	\MPsiq(1-a)=q\IntOI e^{-qt}\Ebb{\IPsi^{-a}(t)}dt=\kappa_{\mis}(q,0)\MPsqq{\dagger q}(1-a).
  	\end{equation}
  	Relation \eqref{eq:RVIPsi}, for any $a\in\lbrb{0,1}$,  also follows from \eqref{eq:MIPsir}.
  	
  	\textbf{Case $a\in\lbrb{0,1-\aphp}$.} This case is more convoluted. The main problem is that we study the moments of $\IPsi(t)$ at infinity via the Laplace transform $I_{\Psi^{\dagger q}}$ which depends on $\IPsi(t),t\downarrow 0,$ as well. The latter is different from the large time behaviour. Our approach is based on separation of the Laplace transform, see \eqref{eq:Var2}, and the subsequent understanding of both parts using the analytical information for the Mellin transform $\MPsiq$, see \eqref{eq:MPr}.
  	
  	Assume that $\aphp<0$ and fix any $q>0$. For any $a\in\lbrb{0,1-\aphp}$ set
  	\begin{equation}\label{eq:H}
  	H(q,a)=\int_{1}^{\infty}e^{-qt}\Ebb{\IPsi^{-a}(t)}dt.
  	\end{equation}
  	From Theorem \ref{thm:momentsIt} we have that $\Ebb{\IPsi^{-a}(t)}<\infty$ for all $t\geq 0$ and any $a\in\lbrb{0,1-\aphp}$. Therefore from $\Ebb{\IPsi^{-a}(1)}\geq \Ebb{\IPsi^{-a}(t)}$, for $t\geq 1$, we conclude that
  	\[H(q,a)\leq \Ebb{\IPsi^{-a}(1)}\int_{1}^{\infty}e^{-qt}dt<\infty\]
  	and thus $H(q,-z)$ can be extended analytically so that $H(q,\cdot)\in\Ac_{\lbrbb{\aphp-1,0}}$. From \eqref{eq:Var1} we immediately see that, for any $x>0$,
  	\[\infty>H(q,a)\geq \int_{1}^{\infty} e^{-qt}\int_{0}^{x}y^{-a}\Pbb{I_{\Psi}(t)\in dy}dt:=W_x(q,-a)\]
  	and hence with $a\mapsto z$ we conclude that  $W_x(q,\cdot)\in \Ac_{\lbrbb{\aphp-1,0}}$. We re-express \eqref{eq:Var} as
  	\begin{equation}\label{eq:Var2}
  	\begin{split}
  	\frac{1}{q}V^{q}_a(x)&=\int_{0}^{1}e^{-qt}\int_{0}^{x}y^{-a}\Pbb{I_{\Psi}(t)\in dy}dt+\int_{1}^{\infty}e^{-qt}\int_{0}^{x}y^{-a}\Pbb{I_{\Psi}(t)\in dy}dt\\
  	&=\int_{0}^{1}e^{-qt}\int_{0}^{x}y^{-a}\Pbb{I_{\Psi}(t)\in dy}dt+ W_x(q,-a)
  	\end{split}
  	\end{equation}
  	and the finiteness of $	\frac{1}{q}V^{q}_a(x)$ can solely fail thanks to the first term in the last identity. We proceed to understand it.  Next, since the analyticity of $\Psi^{\dagger q}$ obviously coincides with that of $\Psi$ we conclude that $\aphp=a_{\php^{\dagger q}}$ and $\aphn=a_{\phn^{\dagger q}}$ for any $q\geq 0$. Moreover, from $\php(0)=0$ then $\php<0$ on $\lbrb{\aphp,0}$ and $\tphp=0$, and since $\limo{q}\php^{\dagger q}(a)=\php(a)$ for any $a>\aphp$, see \eqref{eq:convPhi}, we easily deduct that
  	\[\limo{q}\mathfrak{u}_{\php^{\dagger q}}=\tphp=0.\]
  	Then there exists $q_0>0$ such that, for all $q\leq q_0$,  $\mathfrak{u}_{\php^{\dagger q}}\in\lbrb{\max\curly{-1,\aphp},0}$.  Therefore since, for $q\leq q_0$, $-\mathfrak{u}_{\php^{\dagger q}}$ is not an integer, from Theorem \ref{thm:FormMellin} we conclude that for all $q\leq q_0$, $\MPsiq\in \Mtt_{\lbrb{\aphp,1-\aphn}}$ and $\MPsiq$ has simple poles with residues $\php^{\dagger q}\!(0)\frac{\prod_{k=1}^{n} \Psi^{\dagger q}\!(k)}{n!}$, with $\prod_{k=1}^{0}=1$, at all non-positive integers $-n$ such that $-n>\aphp$. Henceforth, for any $-n_0>\aphp,\,n_0\in\Nb$, we have that
  	\begin{equation}\label{eq:MPr}
  	\MPsqq{\dagger q}(z)=\php^{\dagger q}\!(0)\sum_{n=0}^{n_0}\frac{\prod_{k=1}^{n} \Psi^{\dagger q}(k)}{n!}\frac{1}{z+n}+\MPsqq{\dagger q}^{\lbrb{n_0}}\!(z)
  	\end{equation}
  	with $\MPsqq{\dagger q}^{\lbrb{n_0}}\in\Ac_{\lbrb{\max\curly{-n_0-1,\aphp},1-\aphn}}$. Also since for any $\Psi\in\overNc$ that does not correspond to a pure drift, conservative \LLP, we have that $\NPs>0$, see Theorem \ref{thm:asympMPsi}, we conclude from \eqref{eq:MPr} that at least for any $\beta\in\lbrb{0,\min\curly{N_{\Psi^{\dagger q}},1}}$ and any $c\in\lbrb{\max\curly{-n_0-1,\aphp},1-\aphn}$
  	\begin{equation}\label{eq:MPr_decay}
  	\limi{|b|}|b|^\beta\abs{\MPsqq{\dagger q}^{\lbrb{n_0}}\!(c+ib)}=\limi{|b|}|b|^\beta\abs{\MPsqq{\dagger q}\!(c+ib)-\php^{\dagger q}\!(0)\sum_{n=0}^{n_0}\frac{\prod_{k=1}^{n} \Psi^{\dagger q}\!(k)}{n!}\frac{1}{c+ib+n}}=0.
  	\end{equation}
  	For $a\in\lbrb{0,1}$, $c\in\lbrb{a-1,0}$ and $-n_0>\aphp, n_0\in\Nb,$ \eqref{eq:MILi} together with \eqref{eq:Var2}, \eqref{eq:MPr}, \eqref{eq:MPr_decay} and $\php^{\dagger q}(0)\kappa_{\mis}(q,0)=\php^{\dagger q}(0)\phn^{\dagger q}(0)=q$, allow us to re-express \eqref{eq:MILi} as follows
  	\begin{equation}\label{eq:MILi1}
  	\begin{split} \frac{1}{q}V^{q}_a(x)&=\int_{0}^{1}e^{-qt}\int_{0}^{x}y^{-a}\Pbb{I_{\Psi}(t)\in dy}dt+ W_x(q,-a)\\
  	&=-\frac{\kappa_{\mis}(q,0)}q\frac{1}{2\pi i}\int_{z\in\Cb_c} x^{-z}\frac{\MPsqq{\dagger q}\!\lbrb{z+1-a}}{z}dz\\
  	&=\sum_{n=0}^{n_0}\frac{\prod_{k=1}^{n} \Psi^{\dagger q}(k)}{n!}\frac{1}{1-a+n}\lbrb{x^{1-a+n}\ind{ x\leq 1}+\ind{x>1}}\\
  	&\,\,-\frac{\kappa_{\mis}(q,0)}{q}\frac{1}{2\pi i}\int_{z\in\Cb_c} x^{-z}\frac{\MPsqq{\dagger q}^{\lbrb{n_0}}\!\lbrb{z+1-a}}{z}dz,
  	\end{split}	
  	\end{equation}
  	where the first term in the very last identity stems from the fact that for $a\in\lbrb{0,1}$ the function
  	$-\frac{1}{z\lbrb{z+1-a+n}}$ for $\Re(z)\in\lbrb{a-n-1,0}$ is the Mellin transform of \[\frac{1}{1-a+n}\lbrb{x^{1-a+n}\ind{ x\leq 1}+\ind{x>1}}\] and the integral in \eqref{eq:MILi1} is its Mellin inversion. Also as $\MPsqq{\dagger q}^{\lbrb{n_0}}\in\Ac_{\lbrb{\max\curly{-n_0-1,\aphp},1-\aphn}}$, the fact that we can choose $c<0$ as close to zero as we wish and \eqref{eq:MPr_decay} together with Fubini's and Morera's theorem allow us to conclude that
  	\begin{equation}\label{eq:hol}
  	\int_{z\in\Cb_c} x^{-z}\frac{\MPsqq{\dagger q}^{\lbrb{n_0}}\!\!\lbrb{z+1+\cdot}}{z}dz\in \Ac_{\lbrbb{\max\curly{-n_0-2,\aphp-1},0}}.
  	\end{equation}
  	However,  since  $W_x(q,\cdot)\in \Ac_{\lbrbb{\aphp-1,0}}$ as noted beneath \eqref{eq:Var2},   and \eqref{eq:hol} hold true,  we deduct upon substitution $-a\mapsto \zeta$ in \eqref{eq:MILi1} and equating the second and forth terms in \eqref{eq:MILi1} that as a function of $\zeta$
  	\begin{equation}\label{eq:MILi2}
  	\begin{split}
  	W_x&(q,\zeta)+\frac{\kappa_{\mis}(q,0)}{q}\frac{1}{2\pi i}\int_{z\in\Cb_c} x^{-z}\frac{\MPsqq{\dagger q}^{\lbrb{n_0}}\!\!\lbrb{z+1+\zeta}}{z}dz\\
  	&=\sum_{n=0}^{n_0}\frac{\prod_{k=1}^{n} \Psi^{\dagger q}(k)}{n!}\frac{\lbrb{x^{1+\zeta+n}\ind{ x\leq 1}+\ind{x>1}}}{1+\zeta+n}-\int_{0}^{1}e^{-qt}\int_{0}^{x}y^{\zeta}\Pbb{I_{\Psi}(t)\in dy}dt\\
  	:&=\sum_{n=0}^{n_0}\frac{\prod_{k=1}^{n} \Psi^{\dagger q}(k)}{n!}\frac{\lbrb{x^{1+\zeta+n}\ind{ x\leq 1}+\ind{x>1}}}{1+\zeta+n}-G_x(q,\zeta)\\
  	:&=\tilde{G}_x(q, n_0,\zeta)\in\Ac_{\lbrbb{\max\curly{-n_0-2,\aphp-1},0}},
  	\end{split}
  	\end{equation}
  	where \[G_x(q,\zeta):=\int_{0}^{1}e^{-qt}\int_{0}^{x}y^{\zeta}\Pbb{I_{\Psi}(t)\in dy}dt.\] Therefore, $\tilde{G}_x$ has no poles on the region $\lbrbb{\max\curly{-n_0-2,\aphp-1}\!,0}$ and the poles of  $G_x$ are cancelled by those of the sum above. We stress that we have started the considerations with $a\in\lbrb{0,1}$, see \eqref{eq:MILi1}, but the resulting identity \eqref{eq:MILi2} extends to the analytic function on $\Cb_{\lbrbb{\max\curly{-n_0-2,\aphp-1},0}}$.
  	Note that, for any $k\in\Nb$, by means of  the Taylor formula for $e^{-x}$, one gets
  	\begin{equation}\label{eq:Brep}
  	\begin{split}
  	G_x(q,\zeta)&=\sum_{j=0}^{k}(-1)^j\frac{q^{j}}{j!}\int_{0}^{1}t^j\int_{0}^{x}y^{\zeta}\Pbb{I_{\Psi}(t)\in dy}dt+\int_{0}^{1}f_{k+1}(q,t)\int_{0}^{x}y^{\zeta}\Pbb{I_{\Psi}(t)\in dy}dt\\
  	:&=\sum_{j=0}^{k}(-1)^j\frac{q^{j}}{j!}H_j(x,\zeta)+\tilde{H}_{k+1}(x,q,\zeta).
  	\end{split}
  	\end{equation}
  	Since $\abs{\int_{0}^{x}y^{\zeta}\Pbb{I_{\Psi}(t)\in dy}}\leq \Ebb{\IPsi(t)^{\Re(\zeta)}}$,  one gets from Theorem \ref{thm:momentsIt}\eqref{eq:limitIt} that \[\limo{t}t^{-\Re(\zeta)}\Ebb{\IPsi(t)^{\Re(\zeta)}}=1,\] for $\Re(\zeta)\in\lbrb{\aphp-1,0}$. Hence, for  $\Re(\zeta)\in \lbrb{\max\curly{-1-j,-1+\aphp},0}$
  	\[\abs{\int_{0}^{1}t^{j+\Re(\zeta)}t^{-\Re(\zeta)}\int_{0}^{x}y^{\zeta}\Pbb{I_{\Psi}(t)\in dy}dt}<\infty.\]	
  	Also $\limsupo{t}t^{-k-1}f_{k+1}(q,t)<\infty$  implies that for $\Re(\zeta)\in \lbrb{\max\curly{-2-k,-1+\aphp},0}$
  	\[\abs{\int_{0}^{1}t^{\Re(\zeta)}f_{k+1}(q,t)t^{-\Re(\zeta)}\int_{0}^{x}y^{\zeta}\Pbb{I_{\Psi}(t)\in dy}dt}<\infty.\]
  	Therefore, with these inequalities and Morera's theorem we can deduce that
  	\begin{equation}\label{eq:analit}
  	H_{j}\lbrb{x,\cdot}\in \Ac_{\lbrb{\max\curly{-1-j,-1+\aphp},0}} \text{ and } \tilde{H}_{k+1}\lbrb{x,q,\cdot}\in\Ac_{\lbrb{\max\curly{-2-k,-1+\aphp},0}}.
  	\end{equation}
  	Set $n'$ the largest integer smaller than $1-\aphp$. Then, from \eqref{eq:MILi2} applied with $n_0=n'$ we conclude that $G_x(q,\cdot)\in \Mtt_{\lbrb{\aphp-1,0}}$ with simple poles only at the locations $\curly{-n'-1,\cdots,-1}$. From \eqref{eq:Brep} applied with $k=n'$ we can deduct that $H_j\lbrb{x,\cdot},\,0\leq j\leq n',$ is meromorphic on $\lbrb{\aphp-1,0}$, that is $H_j\lbrb{x,\cdot}\in\Mtt_{\lbrb{\aphp-1,0}}$ with simple poles at the points $\curly{-n'-1,\cdots,-j-1}$. Let us sketch how this is inferred. From \eqref{eq:Brep} and \eqref{eq:analit} we check that $H_0(x,\cdot)\in\Mtt_{\lbrb{-2,0}}$ as all other quantities involved there are meromorphic on the strip $\Re(\zeta)\in\lbrb{-2,0}$. Also using the same arguments we deduct that
  	\[H_0(x,\cdot)-qH_1(x,\cdot)\in\Mtt_{\lbrb{-3,0}}.\]
  	Since the latter is valid for any $q\leq q_0$  by choosing different $q',q''\in \Rp$ with $\max\curly{q',q''}\leq q_0$ by subtraction we end up with
  	\[\lbrb{q'-q''}H_1(x,\cdot)\in\Mtt_{\lbrb{-3,0}}.\]
  	Feeding back this information we arrive at $H_0(x,\cdot)\in\Mtt_{\lbrb{-3,0}}$. This procedure can be continued.
  	Set, for $k\leq n'$,
  	\[H_k(x,\zeta)= \sum_{j=k+1}^{n'+1}\frac{b(j,k,x)}{\zeta+j}+\hat{H}_k(x,\zeta)\]
  	with $\hat{H}_{k}(x,\cdot)\in\Ac_{\lbrb{\aphp-1,0}}$. Then
  	\[G_x(q,\zeta)=\sum_{j=1}^{n'+1}\frac{P_{j}(x,q)}{\zeta+j}+\sum_{j=0}^{n'}(-1)^j\frac{q^j}{j!}\hat{H}_j\lbrb{x,\zeta}+\tilde{H}_{n'+1}(x,q,\zeta)\]
  	with $P_{j}(x,q),j=1,\cdots,n'+1,$ polynomials in $q$ and $ \tilde{H}_{n'+1}(x,q,\cdot)\in\Ac_{\lbrb{\aphp-1,0}}$. However, since $\Psi^{\dagger q}(z)=\Psi(z)-q$ we deduce that $\prod_{k=1}^{n} \Psi^{\dagger q}(k)$ are polynomials in $q$ and since from the definition of $n'$, \eqref{eq:MILi2} defines analytic function, that is $\tilde{G}_x(q, n',\cdot)\in\Ac_{\lbrbb{\aphp-1,0}}$, we conclude that
  	\begin{equation*}
  	\begin{split}
  	\tilde{G}_x(q, n',\zeta)&=\sum_{j=0}^{n'}\frac{\prod_{k=1}^{n} \Psi^{\dagger q}(k)}{j!}\frac{1}{\zeta+j+1}\lbrb{x^{\zeta+j+1}\ind{ x\leq 1}+\ind{x>1}}-G_x(q,\zeta)\\
  	&=\sum_{j=0}^{n'}\frac{\prod_{k=1}^{j} \Psi^{\dagger q}(k)}{j!}\frac{1}{\zeta+j+1}\lbrb{x^{1+\zeta+n}\ind{ x\leq 1}+\ind{x>1}}-\sum_{j=0}^{n'}\frac{P_{j+1}(x,q)}{\zeta+j+1}\\
  	&-\sum_{j=0}^{n'}(-1)^j\frac{q^j}{j!}\hat{H}_j\lbrb{x,\zeta}-\tilde{H}_{n'+1}(x,q,\zeta)\\
  	&=-\sum_{j=0}^{n'}(-1)^j\frac{q^j}{j!}\hat{H}_j\lbrb{x,\zeta}-\tilde{H}_{n'+1}(x,q,\zeta)
  	\end{split}
  	\end{equation*}
  	as the poles evidently have to cancel to ensure that $\tilde{G}_x(q, n',\cdot)\in\Ac_{\lbrbb{\aphp-1,0}}$. Clearly, then we are able to deduct that the following limit holds
  	\begin{equation*}
  	\begin{split}
  	\limo{q}\tilde{G}_x(q, n',\zeta)&=\limo{q}\lbrb{-\sum_{j=0}^{n'}(-1)^j\frac{q^j}{j!}\hat{H}_j\lbrb{x,\zeta}-\tilde{H}_{n'+1}(x,q,\zeta)}\\
  	&=-\hat{H}_0(x,\zeta)
  	\end{split}
  	\end{equation*}
  	since from \eqref{eq:Brep} we have the limit $\limo{q}f_{n'+1}(q,t)=0$ and the asymptotic relation $f_{n'+1}(q,t)\asymp q^{n'+1}t^{n'+1}$ which yield the validity of
  	\[\limo{q}\int_{0}^{1}f_{n'+1}(q,t)\int_{0}^{x}y^{\zeta}\Pbb{I_{\Psi}(t)\in dy}dt=\limo{q}\tilde{H}_{n'+1}(x,q,\zeta)=0.\]
  	Henceforth,	 from \eqref{eq:MILi2} we conclude that, for any $-a\in\lbrb{\aphp-1,0}$, we have the following limit in $q$
  	\begin{equation}\label{eq:Wx}
  	\limo{q}\lbrb{W_x(q,-a)+\frac{\kappa_{\mis}(q,0)}{q}\frac{1}{2\pi i}\int_{z\in\Cb_c} x^{-z}\frac{\MPsqq{\dagger q}^{\lbrb{n'}}\!\lbrb{z+1-a}}{z}dz}=-\hat{H}_0(x,-a).
  	\end{equation}
  	Since it is true that $\php^{\dagger q}(0)\kappa_{\mis}(q,0)=\php^{\dagger q}(0)\phn^{\dagger q}(0)=q$ then from \eqref{eq:MPr} we have the ensuing set of relations valid when $q\to 0$,
  	\begin{equation*}
  	\begin{split}
  	\frac{\kappa_{\mis}(q,0)}{q}\abs{\MPsqq{\dagger q}(z)-\MPsqq{\dagger q}^{\lbrb{n'}}(z)}&\leq\php^{\dagger q}(0)\frac{\kappa_{\mis}(q,0)}{q}\sum_{n=0}^{n'}\abs{\frac{\prod_{k=1}^{n} \Psi^{\dagger q}(k)}{n!}\frac{1}{z+n}}\\
  	&=\sum_{n=0}^{n'}\abs{\frac{\prod_{k=1}^{n} \Psi(k)-q}{n!}\frac{1}{z+n}}\\
  	&=\sum_{n=0}^{n'}\abs{\lbrb{\frac{\prod_{k=1}^{n} \Psi(k)}{n!}+\sospace{1}}\frac{1}{z+n}}.
  	\end{split}
  	\end{equation*}
  	Therefore, by our freedom, for fixed $a\in\lbrb{0,1-\aphp}$, to choose $|c+1-a|$ to be non-integer and $c\in\lbrb{a-1+\aphp,0}$ we get that
  	\begin{equation}\label{eq:somebound}
  	\begin{split}
  	&\frac{\kappa_{\mis}(q,0)}{q}\abs{\int_{z\in\Cb_c} x^{-z}\frac{\MPsqq{\dagger q}\!\lbrb{z+1-a}}{z}dz-\int_{z\in\Cb_c} x^{-z}\frac{\MPsqq{\dagger q}^{\lbrb{n'}}\!\lbrb{z+1-a}}{z}dz}\\
  	&\leq x^{-c}\IntII \sum_{n=0}^{n'}\abs{\lbrb{\frac{\prod_{k=1}^{n} \Psi(k)}{n!}+\sospace{1}}\frac{1}{1-a+c+ib+n}}\frac{db}{\abs{c+ib}}=\bospace{1},
  	\end{split}
  	\end{equation}
  	where in the last step we have invoked the dominated convergence theorem.
  	From these observations  we deduct from \eqref{eq:Wx} and \eqref{eq:somebound} that as $q$ approaches zero
  	\begin{equation*}
  	\begin{split}
  	W_x(q,-a)&=\int_{1}^{\infty}e^{-qt}\int_{0}^{x}y^{-a}\Pbb{\IPsi(t)\in dx}dt\\
  	&=-\frac{\kappa_{\mis}(q,0)}{q}\frac{1}{2\pi i}\int_{z\in\Cb_c} x^{-z}\frac{\MPsqq{\dagger q}^{\lbrb{n'}}\!\lbrb{z+1-a}}{z}dz-\hat{H}_0(x,-a)+\sospace{1}\\
  	&=\frac{\kappa_{\mis}(q,0)}{q}\frac{1}{2\pi i}\lbrb{\int_{z\in\Cb_c}x^{-z}\lbrb{ \frac{\MPsqq{\dagger q}\!\lbrb{z+1-a}}{z}-\frac{\MPsqq{\dagger q}^{\lbrb{n'}}\!\!\lbrb{z+1-a}}{z}}dz}
  	\\
  	&-\hat{H}_0(x,-a)+\sospace{1}-\frac{\kappa_{\mis}(q,0)}{q}\frac{1}{2\pi i}\int_{z\in\Cb_c} x^{-z}\frac{\MPsqq{\dagger q}\!\lbrb{z+1-a}}{z}dz\\
  	&\simo -\frac{\kappa_{\mis}(q,0)}{q}\frac{1}{2\pi i}\int_{z\in\Cb_c} x^{-z}\frac{\MPsqq{\dagger q}\!\lbrb{z+1-a}}{z}dz.
  	\end{split}
  	\end{equation*}
  	Finally, using this asymptotic relation and employing  \eqref{eq:MrConv},\eqref{eq:MrBoundedness} and \eqref{eq:uniformDecay} of Lemma \ref{cor:uniformDecay} we further arrive at the following asymptotic behaviour
  	\begin{equation*}
  	\begin{split}
  	W_x(q,-a)&\simo-\frac{\kappa_{\mis}(q,0)}{q}\frac{1}{2\pi i}\int_{z\in\Cb_c} x^{-z}\lbrb{\frac{\MPsqq{\dagger q}\!\lbrb{z+1-a}}{z}-\frac{\MPs\!\lbrb{z+1-a}}{z}}dz\\
  	&-\frac{\kappa_{\mis}(q,0)}{q}\frac{1}{2\pi i}\int_{z\in\Cb_c} x^{-z}\frac{\MPs\!\lbrb{z+1-a}}{z}dz\\
  	&\simo -\frac{\kappa_{\mis}(q,0)}{q}\frac{1}{2\pi i}\int_{z\in\Cb_c} x^{-z}\frac{\MPs\!\lbrb{z+1-a}}{z}dz.
  	\end{split}
  	\end{equation*}
  	However, from \eqref{eq:MILi1} we then get that
  	\begin{equation*}
  	\begin{split}
  	W_x(q,-a)&=\int_{1}^{\infty}e^{-qt}\int_{0}^{x}y^{-a}\Pbb{I_{\Psi}(t)\in dy}dt\\
  	&\simo -\frac{\kappa_{\mis}(q,0)}{q}\frac{1}{2\pi i}\int_{z\in\Cb_c} x^{-z}\frac{\MPs\!\lbrb{z+1-a}}{z}dz
  	\end{split}
  	\end{equation*}
  	or the behaviour of the Laplace transform of the function $\ind{t\geq 1}\int_{0}^{x}y^{-a}\Pbb{I_{\Psi}(t)\in dy}$.
  	We then conclude \eqref{eq:weak} and\eqref{eq:V_a} as in the case $a\in\lbrb{0,1}$, see the arguments around \eqref{eq:Var1} and \eqref{eq:MILit}. Finally, we show, for $a\in\lbrb{0,1-\aphp}$, that
  	\[\limi{t}\frac{\Ebb{\IPsi^{-a}(t)}}{\kappa_{\mis}\lbrb{\frac{1}{t},0}}=\vartheta_a\lbrb{\Rp}<\infty\]
  	by decomposing \eqref{eq:MIPsir} precisely as $V^{q}_a(x)$ in \eqref{eq:MILi1} and proceeding as there. This also verifies the expression for $\vartheta_a\lbrb{\Rp}$ in \eqref{eq:RVIPsi}.
  	
  	The proof above relies on the ensuing claims.
  	\begin{lemma}\label{cor:uniformDecay}
  		Let $\Psi\in\overNc$ and recall that for any $q\geq 0$, $\Psi^{\dagger q}(z)=\Psi(z)-q$. Fix $a\in\lbrb{\aphp,1}$ such that $-a\notin\N$. Then for any $z\in\Cb_a$ we have that
  		\begin{equation}\label{eq:MrConv}
  		\limo{q}\MPsqq{\dagger q}(z)=\MPs(z).
  		\end{equation}
  		Moreover, for any $\bk>0$ and $\rk<\infty$,
  		\begin{equation}\label{eq:MrBoundedness}
  		\sup_{0\leq q\leq\rk}\sup_{|b|\leq \bk}\abs{\MPsqq{\dagger q}(a+ib)}<\infty.
  		\end{equation}
  		Finally, for any $0<\beta<\NPs$, we have that
  		\begin{equation}\label{eq:uniformDecay}
  		\limsupi{|b|}|b|^\beta	\sup_{0\leq q\leq\rk}\abs{\MPsqq{\dagger q}(\ab)-\MPs\lbrb{\ab}}=0.
  		\end{equation}
  	\end{lemma}
  	\begin{proof}
  		Let $q\geq 0$. Set \[\Psi^{\dagger q}(z)=\Psi(z)-q=-\php^{\dagger q}(-z)\phn^{\dagger q}(z).\] Since the analyticity of $\Psi^{\dagger q}$ obviously coincides with that of $\Psi$ we conclude that $\aphp=a_{\php^{\dagger q}}$ and $\aphn=a_{\phn^{\dagger q}}$ for any $q\geq 0$.
  		We start with some preparatory work by noting that for $a\in\lbrb{\aphp,1}$ and any $q\geq 0$,
  		\begin{equation}\label{eq:WpBound}
  		\sup_{b\in\R}\abs{W_{\phn^{\dagger q}}\!\lbrb{1-a-ib}}\leq W_{\phn^{\dagger q}}\!\lbrb{1-a},
  		\end{equation}
  		since, from Definition \ref{def:WB}, $\lbrb{W_{\phn^{\dagger q}}(n+1)}_{n\geq 0}$ is the moment sequence of the random variable $Y_{\phn^{\dagger q}}$. Also, for any non-integer $a\in\lbrb{\aphp,1}$, $z=\ab\in\Cb_a$ and any $q\geq 0$ we get from \eqref{eq:Stirling1} that
  		\begin{equation}\label{eq:GW}
  		\frac{\Gamma(\ab)}{W_{\php^{\dagger q}}(\ab)}=\lbrb{\prod_{j=0}^{l-1}\frac{\php^{\dagger q}\lbrb{a+j+ib}}{a+j+ib}}\frac{\Gamma(a+l+ib)}{W_{\php^{\dagger q}}\!\lbrb{a+l+ib}},
  		\end{equation}
  		where $l=0$ if $a>0$ and $l+a>0,l\in\N$ otherwise. The convention $\prod_{0}^{-1}=1$ is also in force. This leads to
  		\begin{equation}\label{eq:WpBound1}
  		\sup_{b\in\R}\abs{\frac{\Gamma\!\lbrb{\ab}}{W_{\php^{\dagger q}}\!\lbrb{a+ib}}}\leq \sup_{b\in\R}\lbrb{\prod_{j=0}^{l-1}\frac{\abs{\php^{\dagger q}\!\lbrb{a+j+ib}}}{\abs{a+j+ib}}}\frac{\Gamma(a+l)}{W_{\php^{\dagger q}}\!\lbrb{a+l}}
  		\end{equation}
  		since for $a>0$, $\frac{\Gamma\lbrb{a+ib}}{\Wp\lbrb{\ab}}$ is the Mellin transform of $I_\phi$, see the proof of Theorem \ref{thm:factorization} in section \ref{subsec:Reg}. Next, observe from \eqref{eq:Stirling} that for any $0\leq q \leq \rk$, $z=\ab$ and fixed $a>0$,
  		\begin{equation}\label{eq:supW}
  		\begin{split}
  		\sup_{0\leq q\leq \rk}\abs{W_{\phi_{\pms}^{\dagger q}}(z)}&=\sup_{0\leq q \leq \rk}\lbrb{\frac{\sqrt{\phi_\pms^{\dagger q}\!(1)}}{\sqrt{\phi_\pms^{\dagger q}\!(a)\phi_\pms^{\dagger q}\!(1+a)|\phi_\pms^{\dagger q}\!(z)|}}e^{G_{\phi_\pms^{\dagger q}}\!(a)-A_{\phi_\pms^{\dagger q}}\!(z)}e^{-E_{\phi_\pms^{\dagger q}}\!(z)-R_{\phi_\pms^{\dagger q}}\!(a)}}.
  		\end{split}
  		\end{equation}
  		First, the error term, namely the last  product term above, is uniformly bounded over the whole class of Bernstein functions, see \eqref{eq:uniBoundRE1}. Second, from Proposition \ref{prop:convPhi}, we have that
  		\[\limo{q}\phi^{\dagger q}_\pms\!(1)=\phi_\pms\!(1),\,\,\limo{q}\phi^{\dagger q}_\pms(a)=\phi_\pms(a)\text{ and } \limo{q}\phi^{\dagger q}_\pms(1+a)=\phi_\pms(1+a).\]
  		Similarly, from  \eqref{eq:Gphi},  $\limo{q}G_{\phi_\pms^{\dagger q}}(a)=G_{\phi_\pms}(a)$.
  		Therefore, \eqref{eq:supW} is simplified to
  		\begin{equation}\label{eq:supW1}
  		\begin{split}
  		\sup_{0\leq q\leq \rk}\abs{W_{\phi_\pms^{\dagger q}}(z)}&\asymp\sup_{0\leq q\leq \rk}\frac{1}{\sqrt{|\phi_\pms^{\dagger q}\!(z)|}}e^{-A_{\phi_\pms^{\dagger q}}\!(z)}.
  		\end{split}
  		\end{equation}
  		However, according to Lemma \ref{lem:Aphi}, $q \mapsto A_{\phi_\pms^{\dagger q}}(z)$ are non-increasing on $\R^+$. Henceforth, \eqref{eq:supW1} yields that, for any $z=\ab,\,a>0$,
  		\begin{equation}\label{eq:uniformBounds1}
  		C'_a \inf_{0\leq q \leq\rk}\lbrb{\frac{\sqrt{\abs{\phi_\pms\lbrb{z}}}}{\sqrt{\abs{\phi^{\dagger q}_\pms(z)}}}}\abs{W_{\phi_\pms}\!(z)} \leq \sup_{0\leq q\leq\rk}\abs{W_{\phi_\pms^{\dagger q}}(z)}\leq C_a \sup_{0\leq q\leq\rk}\lbrb{\frac{\sqrt{\abs{\phi^{\dagger \rk}_\pms\lbrb{z}}}}{\sqrt{\abs{\phi_\pms^{\dagger q}(z)}}}}\abs{W_{\phi^{\dagger \rk}_\pms}(z)},
  		\end{equation}
  		where $C_a,\,C_a'$ are two absolute constants.	Next, from Proposition \ref{prop:convPhi}\eqref{eq:convPhi} we have, for any $z\in\CbOI$, that $\limo{q}\phi^{\dagger q}_\pms(z)=\phi_\pms(z)$  and hence from Lemma \ref{lem:continuityW} one obtains, for any $z\in\CbOI$, that the next limit holds
  		\[\limo{q}W_{\phi^{\dagger q}_\pms}(z)=W_{\phi_\pms}(z).\]
  		Also from Proposition \ref{prop:convPhi}\eqref{eq:convPhi} and \eqref{eq:GW} we get that, for any non-integer $a\in\lbrb{\aphp,1}$ and fixed $z\in\Cb_a$,
  		\begin{equation*}
  		\begin{split}
  		\limo{q}\frac{\Gamma\lbrb{z}}{W_{\php^{\dagger q}}(z)}&=\limo{q}\lbrb{\prod_{j=0}^{l-1}\frac{\php^{\dagger q}\lbrb{z+j}}{\abs{z+j}}\frac{\Gamma(z+l)}{W_{\php^{\dagger q}}\!\lbrb{z+l}}} \\
  		&=\prod_{j=0}^{l-1}\frac{\php\!\lbrb{z+j}}{\abs{z+j}}\frac{\Gamma(z+l)}{W_{\php}\!\lbrb{z+l}}\\
  		&=\frac{\Gamma\lbrb{z}}{W_{\php}\!(z)}.
  		\end{split}
  		\end{equation*}
  		Recalling, from \eqref{eqM:MIPsi}, that $
  		\MPs(z)=\frac{\Gamma(z)}{W_{\php}\!(z)}W_{\phn}\!\lbrb{1-z}$, one gets from the last two limits that, for any $ z\in \Cb_a$,
  		\begin{equation*}
  		\begin{split}
  		\limo{q}\MPsqq{\dagger q}(z)&=\limo{q}\frac{\Gamma(z)}{W_{\php^{\dagger q}}(z)}W_{\phn^{\dagger q}}\!\lbrb{1-z}\\
  		&=\frac{\Gamma(z)}{W_{\php}\!(z)}W_{\phn}\!\lbrb{1-z}
  		=\MPs(z),
  		\end{split}
  		\end{equation*}
  		and \eqref{eq:MrConv} follows. Next, \eqref{eq:GW} and \eqref{eq:uniformBounds1} give with the help of \eqref{eqM:MIPsi} that, for any $\rk\in\Rp$ and for any non-integer $a\in\lbrb{\aphp,1}$,
  		\begin{equation}\label{eq:estimateSup}
  		\begin{split}
  		\sup_{0\leq q \leq\rk}\abs{\MPsqq{\dagger q}\!\lbrb{\ab}}&=\sup_{0\leq q \leq\rk}\prod_{j=0}^{l-1}\frac{\abs{\php^{\dagger q}\!\lbrb{a+j+ib}}}{\abs{a+j+ib}}\frac{\abs{\Gamma\!\lbrb{a+l+ib}}}{\abs{W_{\php^{\dagger q}}\!\lbrb{a+l+ib}}}\abs{W_{\phn^{\dagger q}}(1-a-ib)}\\
  		&\leq  C_a\Bigg[\sup_{0\leq q \leq\rk}\frac{\sqrt{\abs{\php^{\dagger q}\!\lbrb{a+l+ib}}}}{\sqrt{\abs{\php\!\lbrb{a+l+ib}}}}\sup_{0\leq q\leq \rk}\frac{\sqrt{\abs{\phn^{\dagger \rk}\!\lbrb{1-a-ib}}}}{\sqrt{\abs{\phn^{\dagger q}\!(1-a-ib)}}}\\
  		&\,\times  \sup_{0\leq q \leq\rk}\prod_{j=0}^{l-1}\frac{\abs{\php^{\dagger q}\!\lbrb{a+j+ib}}}{\abs{a+j+ib}}\\
  		&\,\times \frac{\abs{\Gamma\!\lbrb{a+l+ib}}}{\abs{W_{\php}\!\!\lbrb{a+l+ib}}}\abs{W_{\phn^{\dagger \rk}}(1-a-ib)}\Bigg]\\
  		&=C_a\lbrb{J_1(b)\times J_2(b)\times J_3(b)}.
  		\end{split}
  		\end{equation}
  		However, \eqref{eq:convPhi} in Proposition \ref{prop:convPhi} triggers that
  		\begin{equation*}
  		\begin{split}
  		\sup_{|b|\leq\bk}J_1(b)&=\sup_{|b|\leq\bk}\lbrb{\sup_{0\leq q \leq\rk}\frac{\sqrt{\abs{\php^{\dagger q}\!\lbrb{a+l+ib}}}}{\sqrt{\abs{\php\!\lbrb{a+l+ib}}}}\sup_{0\leq q \leq\rk}\frac{\sqrt{\abs{\phn^{\dagger \rk}\!\lbrb{1-a-ib}}}}{\sqrt{\abs{\phn^{\dagger q}\!(1-a-ib)}}}}\\
  		&\leq\sup_{|b|\leq\bk}\lbrb{\sup_{0\leq q \leq\rk}\frac{\sqrt{\abs{\php^{\dagger q}\!\lbrb{a+l+ib}}}}{\sqrt{\php\!\lbrb{a+l}}}\sup_{0\leq q \leq\rk}\frac{\sqrt{\abs{\phn^{\dagger \rk}\!\lbrb{1-a-ib}}}}{\sqrt{\phn^{\dagger q}(1-a)}}}<\infty
  		\end{split}
  		\end{equation*}
  		since $1-a>0,\,a+l>0$ and \eqref{eq:rephi} holds, that is $\Re\lbrb{\phi\lbrb{\ab}}\geq \phi(a)>0$. The same is valid for $\sup_{|b|\leq\bk}J_3(b)$ (resp.~$\sup_{|b|\leq\bk}J_2(b)$) thanks to \eqref{eq:WpBound}, \eqref{eq:WpBound1} and $\limo{q}W_{\phi^{\dagger q}_\pms}(z)=W_{\phi_\pms}(z)$ (resp.~\eqref{eq:convPhi} and $a$ not an integer). Henceforth, \eqref{eq:MrBoundedness} follows. It remains to show \eqref{eq:uniformDecay}.
  		Let $\NPs=\infty$ first. We note that for any $q>0$, \[\Psi\in\Ni\iff \Psi^{\dagger q}\in\Ni,\text{ or, equivalently, } \quad\NPs=\infty\iff\Ntt_{\Psi^{\dagger q}}=\infty.\]
  		This is due to the fact that the decay of $\abs{\MPs}$ along complex lines is only determined by the \LL triplet $\lbrb{\gamma,\sigma^2,\Pi}$, see \eqref{eq:lk1}, which is unaffected in this case. Henceforth, we are ready to consider the terms in \eqref{eq:estimateSup} and observe that, for any $\beta>0$,
  		\begin{equation}\label{eq:J3}
  		\limsupi{|b|}|b|^\beta J_3(b)=\limsupi{|b|}|b|^\beta \frac{\abs{\Gamma\!\lbrb{a+l+ib}}}{\abs{W_{\php}\!\lbrb{a+l+ib}}}\abs{W_{\phn^{\dagger \rk}}(1-a-ib)}=0,
  		\end{equation}
  		because if \[\limsupi{|b|}|b|^{\beta}\frac{\abs{\Gamma\!\lbrb{a+l+ib}}}{\abs{W_{\php}\!\lbrb{a+l+ib}}}>0\]
  		then from Theorem \ref{thm:asympMPsi}\eqref{eq:subexp} we have that  $\php\in\BP,\,\mubrp\!\lbrb{0}<\infty$. However, then Proposition \ref{prop:convPhi} implies that $\php^\rk\in\BP,\,\mubrp^\rk\!\lbrb{0}<\infty$. Henceforth, necessarily, for any $\beta>0$,
  		\[\limi{|b|}|b|^{\beta}\abs{W_{\phn^{\dagger \rk}}(1-a-ib)}=0,\]
  		as otherwise we would have $\Ntt_{\Psi^{\dagger \rk}}<\infty$ and $\NPs=\infty$ which we have seen to be impossible.
  		Next from \eqref{eq:convPhi1} and Proposition \ref{propAsymp1}\eqref{it:asyphid} we deduct that
  		\begin{equation}\label{eq:J2}
  		\begin{split}
  		\limsupi{|b|}J_2(b)&=\limsupi{|b|}\sup_{0\leq q \leq\rk}\lbrb{\prod_{j=0}^{l-1}\frac{\abs{\php^{\dagger q}\!\lbrb{a+j+ib}}}{\abs{a+j+ib}}}\\
  		&\leq\limsupi{|b|}\lbrb{\prod_{j=0}^{l-1}\frac{\sup\limits_{0\leq q \leq\rk}\lbrb{\abs{\php^{\dagger q}\!\lbrb{a+j+ib}-\php\!\lbrb{a+j+ib}}}+\abs{\php\!\lbrb{a+j+ib}}}{\abs{b}}}\\
  		& <\infty.
  		\end{split}
  		\end{equation}
  		Finally, from \eqref{eq:rephi}, \eqref{eq:convPhi}, \eqref{eq:convPhi1} and Proposition \ref{propAsymp1}\eqref{it:asyphid} we conclude that
  		\begin{equation}\label{eq:J1}
  		\begin{split}
  		\limsupi{|b|}\frac{J_1(b)}{|b|^2}&=\limsupi{|b|}\frac{1}{|b^2|}\sup_{0\leq q \leq\rk}\frac{\sqrt{\abs{\php^{\dagger q}\!\lbrb{a+l+ib}}}}{\sqrt{\abs{\php\!\lbrb{a+l+ib}}}}\sup_{0\leq r<\rk}\frac{\sqrt{\abs{\phn^{\dagger \rk}\!\lbrb{1-a-ib}}}}{\sqrt{\abs{\phn^{\dagger q}(1-a-ib)}}}\\
  		&\leq C_a\sup_{0\leq q \leq\rk}\frac{1}{\sqrt{\abs{\php\!\lbrb{a+l}}}}\frac{1}{\sqrt{\phn^{\dagger q}\!\lbrb{1-a}}}<\infty,
  		\end{split}
  		\end{equation}
  		where we also recall that $1-a\in\lbrb{0,1-\aphp}$.
  		Collecting the estimates \eqref{eq:J3}, \eqref{eq:J2} and \eqref{eq:J1} we prove \eqref{eq:uniformDecay} when $\NPs=\infty$ since from \eqref{eq:estimateSup}
  		\begin{equation*}
  		\begin{split}
  		\limi{|b|}|b|^\beta\sup_{0\leq q \leq\rk}\abs{\Mcc_{\Psi^{\dagger q}}\!\lbrb{\ab}}\leq C_a\limi{|b|}\frac{J_1(b)}{|b|^2}J_2(b)|b|^{\beta+2}J_3(b)=0.
  		\end{split}
  		\end{equation*}
  		Assume next that $\NPs<\infty$ which triggers from Theorem \ref{thm:asympMPsi}\eqref{eq:subexp} and Proposition \ref{prop:convPhi} that $\php,\,\php^{\dagger q}\in\BP$ with $\dep=\dep^q,\forall\,q>0,\,$ $\phn,\phn^{\dagger q}\in\BP^c$ and $\PP(0)<\infty$. The latter implies that $\bar{\mu}^{\dagger q}_{\pms}(0)<\infty,\,q\geq 0$.  	  	
  		Thus, from \eqref{eq:murtomu} and \eqref{eq:convPhi} of Proposition \ref{prop:convPhi}, we  conclude from \eqref{eq:NPs} that
  		\begin{equation}\label{eq:convergenceNPs}
  		\begin{split}
  		\limo{q} \mathtt{N}_{\Psi^{\dagger q}}&=\limo{q}\lbrb{\frac{\upsilon^{\dagger q}_{\mis}(0^+)}{\phn^{\dagger q}(0)+\bar{\mu}^{\dagger q}_{\mis}(0)}}+\limo{q}\lbrb{\frac{\php^{\dagger q}(0)+\bar{\mu}^{\dagger q}_{\pls}(0)}{\dep}}\\
  		&=  \lbrb{\frac{\upsilon_{\mis}(0^+)}{\phn(0)+\bar{\mu}_{\mis}(0)}}+\frac{\php(0)+\bar{\mu}_{\pls}(0)}{\dep}\\
  		&=\NPs,
  		\end{split}
  		\end{equation}
  		wherein it has not only been checked yet that \[\limo{q}\upsilon^{\dagger q}_{\mis}(0^+)=\upsilon_{\mis}(0^+).\] 
  		However,  whenever $\Ntt_{\Psi^{\dagger q}}<\infty$, from \eqref{eq:mu_-2} we have that \[\upsilon_{\mis}^{\dagger q}(0^+)=\IntOI u^{\dagger q}_{\pls}(y)\Pnn(dy),\] since $\Pi^{\dagger q}=\Pi$. Also, in this case since $\dep=\dep^{\dagger q}$ from \eqref{eq:u} we have that, for any $y\geq 0$,
  		\begin{equation*}
  		u^{\dagger q}_{\pls}(y)=\frac{1}{\dep}+\sum_{j=1}^{\infty}\frac{\lbrb{-1}^j}{\dep^{j+1}}\lbrb{\mathbf{1}*(\php^{\dagger q}(0)+\bar{\mu}_{\pls}^{\dagger q})^{*j}}(y)=\frac{1}{\dep}+\tilde{u}_{\pls}^{\dagger q}(y).
  		\end{equation*}
  		The infinite sum above is locally uniformly convergent, see the proof of \cite[Proposition 1]{Doering-Savov-11}, and therefore we can show using that $\limo{q}\php^{\dagger q}(0)=\php(0)$, $\limo{q}\bar{\mu}_{\pls}^{\dagger q}=\bar{\mu}_{\pls}$ and \eqref{eq:murtomu} and \eqref{eq:convPhi} of Proposition \ref{prop:convPhi}, that for any $y>0$,
  		\[\limo{q}u^{\dagger q}_{\pls}(y)=u_{\pls}(y)\] and hence $\limo{q}\upsilon_{\mis}^{\dagger q}(0^+)=\upsilon_{\mis}(0^+)$ follows. Thus, \eqref{eq:convergenceNPs} holds true. Note that since $\dep>0$ the \LLP underlying $\Psi^{\dagger q}$ is not a compound  Poisson process and hence from Lemma \ref{lem:WH} we have that that \[\mu^{\dagger q}_\pms(dy)=\IntOI e^{-rt+\Psi(0)t}\mu^\sharp_\pms(dt,dy),\] where $\mu^\sharp_{\pm}$ stands for the \LL measure of the ladder height processes of the conservative \LLP underlying $\Psi^\sharp(z)=\Psi(z)-\Psi(0)=(\Psi^{\dagger q})^\sharp(z)$. Therefore, in the sense of measures on $\intervalOI$, $\mu^{\dagger q}_\pms(dy)\leq \mu^\sharp_\pms(dy)$, for all $q\geq0$. Since $\bar{\mu}^{\dagger q}_{\mis}(0)<\infty,\,q\geq 0,$ and $\phn,\phn^{\dagger q}\in\BP^c$ we conclude from Proposition \ref{prop:convPhi}\eqref{eq:convPhi} that for any $a>\aphp$ and $\rk>0$
  		\begin{equation}\label{eq:supphn}
  		\sup_{b\in\R}\sup_{0\leq q \leq \rk} \abs{\phn^{\dagger q}(a+ib)}\leq \sup_{0\leq q \leq\rk}\lbrb{\phn^{\dagger q}(0)}+\IntOI \lbrb{e^{ay}+1}\mu^\sharp_{\mis}(dy)<\infty.
  		\end{equation}	
  		Also from $\dep^{\dagger q}=\dep>0$ and $\bar{\mu}^{\dagger q}_{\pls}(0)<\infty,\,q\geq 0,$ we obtain, for fixed $a>\aphp,\,-a\notin\N$, that
  		\begin{equation}\label{eq:supphp}
  		\sup_{b\in\R}\sup_{0\leq q \leq \rk} \frac{\abs{\php^{\dagger q}(a+ib)}}{\abs{a+ib}}\leq \sup_{0\leq q \leq\rk}\frac{\php^{\dagger q}(0)}{|a|}+\dep+\frac{1}{|a|}\IntOI \lbrb{e^{ay}+1}\mu^\sharp_{\pls}(dy)<\infty.
  		\end{equation}
  		Therefore from \eqref{eq:supphp}
  		\begin{equation}\label{eq:J2_1}
  		\limsupi{|b|}J_2(b)=\limsupi{|b|}\sup_{0\leq q \leq\rk}\lbrb{\prod_{j=0}^{l-1}\frac{\abs{\php^{\dagger q}\lbrb{a+j+ib}}}{\abs{a+j+ib}}}<\infty
  		\end{equation}
  		and from \eqref{eq:supphn}, \eqref{eq:supphp}, \eqref{eq:rephi} and the fact that
  		\[\inf_{b\in\R}\frac{\sqrt{\abs{\php\lbrb{a+l+ib}}}}{\sqrt{\abs{a+ib}}}>0\]
  		we arrive at
  		\begin{equation}\label{eq:J1_1}
  		\begin{split}
  		\limsupi{|b|}J_1(b)&=\limsupi{|b|}\sup_{0\leq q \leq\rk}\frac{\sqrt{\abs{\php^{\dagger q}\!\lbrb{a+l+ib}}}}{\sqrt{\abs{\php\!\lbrb{a+l+ib}}}}\sup_{0\leq q\leq\rk}\frac{\sqrt{\abs{\phn^{\dagger \rk}\!\lbrb{1-a-ib}}}}{\sqrt{\abs{\phn^{\dagger q}\!(1-a-ib)}}}\\
  		&\leq C_a\sup_{0\leq q \leq\rk}\frac{\frac{\sqrt{\abs{\php^{\dagger q}\lbrb{a+l+ib}}}}{\sqrt{\abs{a+l+ib}}}}{\frac{\sqrt{\abs{\php\lbrb{a+l+ib}}}}{\sqrt{\abs{a+l+ib}}}}\frac{1}{\sqrt{\phn^{\dagger q}\!\lbrb{1-a}}}<\infty.
  		\end{split}
  		\end{equation}
  		Relations \eqref{eq:J2_1} and \eqref{eq:J1_1} allow the usage of \eqref{eq:estimateSup} to the effect that
  		\begin{equation}\label{eq:estimateSup1}
  		\begin{split}
  		&\sup_{0\leq q \leq\rk}\sup_{|b|\leq\bk}\abs{\Mcc_{\Psi^{\dagger q}}\lbrb{\ab}}
  		\leq C_a\frac{\abs{\Gamma\lbrb{a+l+ib}}}{\abs{W_{\php}\!\lbrb{a+l+ib}}}\abs{W_{\phn^{\dagger \rk}}(1-a-ib)}.
  		\end{split}
  		\end{equation}
  		Then \eqref{eq:uniformDecay} follows from \eqref{eq:estimateSup1} and \eqref{eq:convergenceNPs} as $\rk$ can be chosen as small as we wish so that each of the two summands computing $\mathtt{N}_{\Psi^{\dagger \rk}}$  is as close as we need to each of the summands evaluating $\NPs$. This means that the exponent of the power decay of $\abs{W_{\phn^{\dagger \rk}}}$ can be as close to the exponent of $\abs{W_{\phn}}$ and \eqref{eq:uniformDecay} is indeed verified. This concludes the proof of this lemma.
  	\end{proof}
  
  \section{Factorization of laws and intertwining between self-similar semigroups}\label{sec:addons}
  \subsection{Proof of Theorem \ref{thm:factorization}}
  As in the case $\dphn<0$, see section \ref{subsec:Reg}, we recognize   $\frac{\Gamma(z)}{\Wpp\!(z)}$ as the Mellin transform of the random variable $I_{\php}$ and $ \phn(0)\Wpn\!\lbrb{1-z}$ as the Mellin transform of $X_{\phn}$ as defined in \eqref{eq:auxRv}. This leads to the first factorization \eqref{eq:GammaType2} of Theorem \ref{thm:factorization}.
  Next, we proceed with the proof of the second identity in law  of Theorem \ref{thm:factorization}.
  The product representations of the functions $\Wpp,\Wpn,\Gamma$, see \eqref{eq:BernWeier} and the recurrent equation \eqref{eq:Wp}, allow us to obtain that for any Bernstein function $\Wp$ and $z\in\CbOI$
  \begin{equation*}
  \Wp(z+1)=\phi(z)\Wp(z)=e^{-\gamma_{\phi} z}\prod_{k=1}^{\infty}\frac{\phi(k)}{\phi(k+z)}e^{\frac{\phi'(k)}{\phi(k)}z},
  \end{equation*}
  which by obvious analytic extension is valid even on $\Cb_{\lbrb{-1,\infty}}$.
  Next, we simply substitute the latter and invoke the  product representation \eqref{eq:BernWeier} in \eqref{eq:MIPsi}, for $z\in\Cb_{\lbrb{-1,0}}$, to get that
  \begin{equation}\label{eq:unifiedProduct}
  \begin{split}	\MPsi(z+1)&=\Ebb{I^{z}_{\Psi}}=\phn(0)\frac{\Gamma(z+1)}{\Wpp(z+1)}\Wpn(-z)\\
  &= e^{z\lbrb{\gamma_{\php}+\gamma_{\phn}-\gamma}}\frac{\phn(0)}{\phn(-z)}\prod_{k=1}^{\infty} \frac{\phn(k)}{\phn(k-z)}\frac{k\php\!\lbrb{k+z}}{\php(k)(k+z)}C^z_\Psi(k)\\
  &=\prod_{k=0}^{\infty}\frac{\phn(k)}{\phn(k-z)}\frac{(k+1)\php\!\lbrb{k+1+z}}{\php(k+1)(k+1+z)}C^z_\Psi(k),
  \end{split}
  \end{equation}
  where, for  $k\geq1$,
  \[C_\Psi(k)=e^{\lbrb{\frac{1}{k}-\frac{\php'(k)}{\php(k)}-\frac{\phn'(k)}{\phn(k)}}}\text{ and }C_\Psi(0)=e^{\lbrb{\gamma_{\php}+\gamma_{\phn}-\gamma}}.\]
  Performing a change of variables in  Proposition \ref{propAsymp1}\eqref{it:bernstein_cmi} (resp.~in the expression \eqref{eq:phi}), we get, recalling that $\Upsilon_{\mis}(dy)=U_{\mis}\lbrb{\ln(dy)},\,y>1,$ is the image of $U_{\mis}$ via the mapping $y\mapsto \ln y$, that the following identities hold true
  \begin{equation*}
  \begin{split}
  &\int_1^{\infty} y^{z}\Upsilon_{\mis}(dy)=\IntOI e^{zy}U_{\mis}(dy)=\frac{1}{\phn(-z)}\\
  &\int_0^{1} y^{z}\left(\bar{\mu}_{\pls}(-\ln y)dy+\php(0)dy+\dep \delta_{1}(dy)\right)=\frac{\php\!\lbrb{1+z}}{(1+z)}.
  \end{split}
  \end{equation*}
  Note that the last identities prove \eqref{eq:rvs}, that is
  \begin{equation*}
  \begin{split}
  \Pbb{X_{\Psi} \in dx}&=\frac{1}{\php(1)}\left(\bar{\mu}_{\pls}(-\ln x)dx+\php(0)dx+\dep \delta_{1}(dx)\right),\,x\in\lbrb{0,1},\\
  \Pbb{Y_{\Psi} \in dx}&=\phn(0)\Upsilon_{\mis}(dx),\,x>1,
  \end{split}
  \end{equation*}
  are probability measures on $\Rp$.
  Next, it is trivial that, for any $z \in i\R$,
  \[\int_0^{\infty}x^z\P(X_{\Psi}\times Y_{\Psi} \in dx)= \frac{\phn(0)}{\php(1)}\frac{\php\!\lbrb{1+z}}{\phi_{\mis}(-z) (1+z)}.\]
  Then it is clear that for $k=0,\ldots$, we have that for any bounded measurable function $f$
  \begin{equation*}
  \begin{split}
  \Ebb{f(\mathfrak{B}_{k} X_{\Psi})}\Ebb{f(\mathfrak{B}_{-k} Y_{\Psi})}&=\frac{(k+1)\php(1)}{\php(k+1)}\frac{\phn(k)}{\phn(0)}\Ebb{X_{\Psi}^{k}f(X_{\Psi})}\Ebb{Y_{\Psi}^{-k}f(Y_{\Psi})},
  \end{split}
  \end{equation*}
  where we recall that
  $ \Ebb{f(\mathfrak{B}_k X)} = \frac{\Ebb{X^kf(X)}}{\Eb[X^k]}$
  and evidently
  \[\Ebb{X^{k}_\Psi}=\frac{\php(k+1)}{\lbrb{k+1}\php(1)}\quad \text{ and }\quad \Ebb{Y^{-k}_{\Psi}}=\frac{\phn(0)}{\phn(k)}.\]
  Setting $f(x)=x^z$ for some $z\in i\R$ we deduce that for $k=0,1,\cdots$
  \begin{equation*}
  \begin{split}
  \Ebb{(\mathfrak{B}_k X_{\Psi})^z}\Ebb{(\mathfrak{B}_{-k} Y_{\Psi})^z}&=\frac{\phn(k)}{\phn(k-z)}\frac{(k+1)}{\php(k+1)}\frac{\php\!\lbrb{k+1+z}}{\lbrb{k+1+z}}.
  \end{split}
  \end{equation*}
  Therefore the second identity in law of \eqref{eq:InfW} follows immediately by inspecting the terms in the last expression of \eqref{eq:unifiedProduct}. \qed
  \subsection{Proof of Theorem \ref{thm:PSSMP}} \label{sec:prof_int}
  Let $\Psi\in\Nc_{\unrhd}$. If $\Psi'\lbrb{0}\in\lbrb{0,\infty}$, that is, the underlying L\'evy process drifts to infinity, \eqref{eqn:GammaType} and hence \eqref{eq:feVPsi} can be verified directly from   \eqref{eq:MIPsi} and \cite{Bertoin-Yor-02-b} wherein it is shown that \[\Ebb{f\!\lbrb{\VPsi}}=\frac{1}{\Ebb{\IPsi^{-1}}}\Ebb{\frac{1}{\IPsi}f\!\lbrb{\frac{1}{\IPsi}}}\] for any $f\in\cco\!\lbrb{\lbbrb{0,\infty}}$. Indeed from the latter we easily get that 
  \[\Ebb{\IPsi^{-1}}=\phn(0)\php'\!\lbrb{0^+},\] where we recall that $\Psi'\lbrb{0}\in\lbrb{0,\infty}$ triggers $\phn(0)>0$ since the \LLP goes to infinity. Then a substitution yields the result. If $\Psi'(0)=0$ that is the underlying process oscillates we proceed by approximation. Set $\Psi_{\rk}(z)=\Psi(z)+\rk z$ and note that $\Psi'_\rk(0)=\rk>0$. Then \eqref{eqn:GammaType} and hence \eqref{eq:feVPsi} are valid for \[\Psi_\rk(z)=-\php^{(\rk)}(-z)\phn^{(\rk)}(z).\]  From the Fristedt's formula, see \eqref{eq:Fristed} below, Lemma \ref{lem:WH} and fact that the underlying process is conservative, that is $\Psi(0)=0$, we get that
  \[\phi^{(\rk)}_{\pls}\!(z)=h^{(\rk)}(0)e^{\IntOI\int_{\lbbrb{0,\infty}} \lbrb{e^{-t}-e^{-zx}}\frac{\Pbb{\xi_t+\rk t\in dx}}{t}dt},\,\,z\in\Cb_{\lbbrb{0,\infty}},\]
  where $h^{(\rk)}(0)=1$ since the \LLP $\xi^{(\rk)}$ corresponding to $\Psi_\rk$ is not a compound  Poisson process, see \eqref{eq:h}.  Then as \[\limo{\rk}\Pbb{\xi_t+\rk t\in \pm dx}=\Pbb{\xi_t\in \pm dx}\]  weakly on $\lbbrb{0,\infty}$ we conclude that $\limo{\rk}\php^{(\rk)}(z)=\php(z),\,z\in\Cb_{\lbbrb{0,\infty}}$. This together with the obvious $\limo{\rk}\Psi_\rk(z)=\Psi(z)$ gives that  $\limo{\rk}\phn^{(\rk)}(z)=\phn(z),\,z\in\Cb_{\lbbrb{0,\infty}}$. Thus, from Lemma \ref{lem:continuityW} we get that \[\limo{\rk}W_{\phi_{\pms}^{(\rk)}}(z)=W_{\phi_{\pms}}(z)\] on $\CbOI$. Therefore, \eqref{eqn:GammaType}, that is
  \begin{equation*}
  \begin{split}
  \limo{\rk}\mathcal{M}_{V_{\Psi_\rk}}(z)&=  \limo{\rk}\frac{1}{\lbrb{\php^{\rk}(0^+)}'}\frac{\Gamma(1-z)}{W_{\phi^{(\rk)}_{\pls}}(1-z)}W_{\phi^{(\rk)}_{\mis}}(z)\\
  &=\frac{1}{\php'(0^+)}\frac{\Gamma(1-z)}{W_{\php}(1-z)}W_{\phn}(z),\quad  z\in \C_{\lbrb{\dphn,1}},	    
  \end{split}
  \end{equation*}
  holds provided that $\limo{\rk}\lbrb{\php^{(\rk)}(0^+)}'=\php'(0^+)$.   However, from $\limo{\rk}\php^{(\rk)}(z)=\php(z)$ we deduct from the second expression in \eqref{eq:phi} with $\php(0)=\php^{(\rk)}\!(0)=0$ that on $\CbOI$
  \[\limo{\rk}\lbrb{\dep^{(\rk)}+\IntOI e^{-zy}\mubrp^{(\rk)}\lbrb{y}dy}=\dep+\IntOI e^{-zy}\mubrp\lbrb{y}dy.\]
  Since by assumption $\php'(0^+)<\infty$ and hence $\lbrb{\php^{(\rk)}}'\!\!\!(0^+)<\infty$. Then, in an obvious manner from \eqref{eq:phi'} we can get that
  \[\limo{\rk}\lbrb{\php^{(\rk)}}'\!(0^+)=\limo{\rk}\lbrb{\dep^{(\rk)}+\IntOI \mubrp^{(\rk)}\!\lbrb{y}dy}=\dep+\IntOI\mubrp\!\lbrb{y}dy=\php'(0^+).\]
  Thus, item \eqref{it:entrance} is settled. All the claims of item \eqref{it:intertwining} follow from the following sequence of arguments. First under the condition that $\Pi(dy)=\pi_{\pls}(y)dy,\,y>0$, $\pi_{\pls}$ non-increasing on $\Rp$ from \cite{Pardo2012} we get that \eqref{eq:InfW} is refined to $\IPsi\stackrel{d}{=} \Ir_{\php}\times \Ir_{\psi}$, where $\psi(z)=z\phn(z)\in\Nc_{\unrhd}$. Secondly, this factorization is transferred to $\VPsi\stackrel{d}{=} V_{\php}\times V_{\psi}$ via \eqref{eqn:GammaType}. Finally  the arguments in the proof of \cite[Theorem 7.1]{Patie-Savov-16} depend on the latter factorization of the entrance laws and the zero-free property of $\M_{\VPsi}(z)$ for $z\in\Cb_{\lbrb{0,1}}$ which via \eqref{eqn:GammaType} is a consequence of Theorem \ref{thm:Wp} which yields that $\Wp(z)$ is zero free on $\CbOI$ for any $\phi\in\Bc$.

  \newpage
  
  \begin{appendix}\label{sec:appendix}
  	\section{Some fluctuation details on \LL processes and their exponential functional}\label{subsec:LP}
  	Recall that a \LLP $\xi=(\xi_t)_{t\geq0}$ is a real-valued stochastic process which possesses stationary and independent increments with a.s.~right-continuous paths. We allow killing of the \LLP by means of the following procedure. Consider the conservative that is unkilled version of $\xi$. Then, writing $-\Psi(0)=q\geq 0$ pick an  exponential variable $\textbf{e}_q$, of parameter $q$, independent of $\xi$, and set $\xi_t=\infty$ for any $t\geq {\textbf{e}_q}$. Note that $\textbf{e}_0=\infty$ a.s. and in this case the \LLP coincides with its conservative version.  The law of a possibly killed \LLP $\xi$ is  characterized via its characteristic exponent, i.e.~$\log\Eb\left[e^{z\xi_t}\right]=\Psi(z)t$, where $\Psi : i\R \rightarrow \mathbb{C}$  admits the following L\'evy-Khintchine representation
  	\begin{eqnarray} \label{eq:lk}
  	\Psi(z) =  \frac{\sigma^2}{2} z^2 + \gamma z +\IntII \left(e^{zr} -1
  	-zr \ind{|r|<1}\right)\Pi(dr)+\Psi(0),
  	\end{eqnarray}
  	where $\Psi(0)\leq0$ is the killing rate, $\sigma^2 \geq 0$, $\gamma\in \R$, and, the L\'evy measure
  	$\Pi$ satisfies the integrability condition $\IntII(1
  	\wedge r^2 )\:\Pi(dr) <+  \infty$.  With each \LL process, say $\xi$, there are the bivariate ascending and descending ladder time and height processes  $\lbrb{\zeta^\pms,H^\pms}=\lbrb{\zeta^\pms_t,H^\pms_t}_{t\geq0}$ associated to $\xi$ via $\lbrb{H^\pms_t}_{t\geq 0}=\lbrb{\xi_{\zeta^\pms_t}}_{t\geq 0}$ and we refer to \cite[Chapter VI]{Bertoin-96} for more information on these processes. Let us consider the bivariate ladder height processes related to the conservative version of $\xi$, that is $\xi^\sharp$. Then since these processes are bivariate subordinators we denote by $\kappa_{\pms}$ their Laplace exponents. The Fristedt's formula, see \cite[Chapter VI, Corollary 10]{Bertoin-96}, then evaluates those on $z\in\Cb_{\lbbrb{0,\infty}},q\geq 0$, as
  	\begin{equation}\label{eq:Fristed}
  	\kappa_{\pms}\lbrb{q,z}=e^{\IntOI\int_{\lbbrb{0,\infty}}\lbrb{e^{-t}-e^{-zx-qt}}\Pbb{\pm\xi^\sharp_t\in dx}\frac{dt}{t}},
  	\end{equation}
  	where $\xi^\sharp$ is the conservative \LLP constructed from $\xi$ by letting it  evolve on an infinite time horizon. Set
  	\begin{equation}\label{eq:h}
  	h(q)=e^{-\IntOI\lbrb{e^{-t}-e^{-qt}}\Pbb{\xi^\sharp_t=0}\frac{dt}{t}}
  	\end{equation}
  	and note that $h:\lbbrb{0,\infty}\mapsto\Rp$ is a decreasing, positive function.
  	Then, the analytical form of the Wiener-Hopf factorization of $\Psi\in\overNc$ is given by the expressions
  	\begin{eqnarray}\label{eq:WH}
  	\Psi(z)=-\php(-z)\phn(z)=-h(q)\kappa_{\pls}\!\lbrb{q,-z}\kappa_{\mis}\!\lbrb{q,z},\: z\in i\R,
  	\end{eqnarray}
  	where  $\phi_\pms  \in \Be$ with $\phi_\pms(0)\geq 0$ and the characteristics of $\phi_\pms$, that is $\lbrb{\phi_\pms(0),\dr_\pms,\mu_\pms}$, depend on $q=-\Psi(0)\geq0$. Then we have the result.
  	\begin{lemma}\label{lem:WH}
  		For any $\Psi\in\overNc$ it is possible to choose $\php(z)=h(q)\kappa_{\pls}(q,z)$ and $\phn(z)=\kappa_{\mis}(q,z)$. The function $h:\lbbrb{0,\infty}\mapsto\Rp$ is not identical to $1$ if and only if $\PP(0)<\infty,\,\sigma^2=\gamma=0$, see \eqref{eq:lk}, that is $\xi$ is a compound  Poisson process. Then, on $\Rp$, \[\mu_{\mis}(dy)=\int_{0}^{\infty}e^{-qy_1}\mu^\sharp_{\mis}(dy_1,dy) \text{ and } \mu_{\pls}(dy)=h(q)\int_{0}^{\infty}e^{-qy_1}\mu^\sharp_{\pls}(dy_1,dy),\] where $\mu^\sharp_{\pms}(dy_1,dy)$ are
  		the \LL measures of the bivariate ascending and descending ladder height and time processes $\lbrb{\zeta^\pms,H^\pms}=\lbrb{\zeta^\pms_t,H^\pms_t}_{t\geq0}$ associated to the conservative \LLP $\xi^\sharp$.
  	\end{lemma}
  	\begin{proof}
  		The proof is straightforward from \cite[p.27]{Doney-07-book} and the fact that for fixed $q\geq0$, $\kappa_{\pms}\in\Bc$.
  	\end{proof}
  	We refer to the excellent monographs \cite{Bertoin-96} and \cite{Sato-99} for background on the probabilistic and path-wise properties of general \LL processes and their associated \LLK exponent $\Psi\in\overNc$. Also they contain the bulk of the fluctuation theory of \LL processes.
  	
  	We proceed with providing an alternative expression for $\Aph\lbrb{z}=\int_{0}^{b}\arg \phi\lbrb{a+iu}du$, $z=a+ib$, see \eqref{eq:Aphi} when $\phi$ is a Wiener-Hopf factor of some $\Psi\in\overNc$.
  	\begin{lemma}\label{lem:Aphi}
  		Let $\Psi\in\overNc$ and take $\phpm$ from \eqref{eq:WH}. Then, for any $z=\ab\in\CbOI$, we have that the following identity holds
  		\begin{equation}\label{eq:AphiAlt}
  		A_{\phpm}(z)=\IntOI \int_{\lbbrb{0,\infty}}\frac{1-\cos(bx)}{x}e^{-ax}e^{\Psi(0)t}\Pbbs{\pm\xi^{\sharp}_t\in dx}\frac{dt}{t}.
  		\end{equation}
  		Also if, for any  $q\geq 0$, $\Psi^{\dagger q}(\cdot)=\Psi(\cdot)-q=-\php^{\dagger q}(-\cdot)\phn^{\dagger q}(\cdot)$, then for any  fixed  $z=\ab\in\CbOI$ the functions $q\mapsto A_{\phpm^{\dagger q}}(z)$ are non-increasing.
  	\end{lemma}
  	\begin{proof}
  		It suffices to consider $z=\ab$ with $b>0$ only. Clearly, from \eqref{eq:Fristed} and Lemma \ref{lem:WH}, we have that, for any $u\geq0$,
  		\begin{equation*}
  		\begin{split}
  		\arg \phpm(a+iu)&=\Im\!\lbrb{\IntOI\int_{\lbbrb{0,\infty}}\lbrb{e^{-t}-e^{-(a+iu)x+\Psi(0)t}}\Pbbs{\pm\xi^\sharp_t\in dx}\frac{dt}{t}}\\
  		&=\IntOI\int_{\lbbrb{0,\infty}}\sin\lbrb{ux}e^{-ax}e^{\Psi(0)t}\Pbbs{\pm\xi^\sharp_t\in dx}\frac{dt}{t}.
  		\end{split}
  		\end{equation*}
  		A simple integration then leads to
  		\begin{equation*}
  		\begin{split}
  		\Aphm\!\!\lbrb{z}&=\int_{0}^{b}\arg \phpm\!\lbrb{a+iu}du\\
  		&=\IntOI\int_{\lbbrb{0,\infty}}\frac{1-\cos\lbrb{bx}}{x}e^{-ax}e^{\Psi(0)t}\Pbbs{\pm\xi^\sharp_t\in dx}\frac{dt}{t}
  		\end{split}
  		\end{equation*}
  		or \eqref{eq:AphiAlt} is recovered. If $\Psi^{\dagger q}(0)=\Psi(0)-q$ then the monotonicity in $q$ is clear from \eqref{eq:AphiAlt} since all terms are non-negative, do not depend on $q$ and $e^{\Psi^{\dagger q}(0)t}=e^{\Psi(0)t-qt}$ decreases in the variable $q$.
  	\end{proof}
  	
  	\section{A simple extension of the \'{e}quation amicale invers\'{e}e}\label{subsec:Vigon}
  	Let $\Psi\in\overNc$ and recall its Wiener-Hopf factorization  $\Psi(z)=-\php(-z)\phn(z),\,z\in i\R$. When in addition $\Psi(0)=0$, then the Vigon's  \textit{\'{e}quation amicale invers\'{e}e}, see \cite[5.3.4]{Doney-07-book}, states that
  	\begin{equation}\label{eq:Vigon}
  	\mubarnspace{y}=h(q)\IntOI \PPn(y+v) U_{\pls}(dv),\,y>0,
  	\end{equation}
  	where $\mu_{\mis}$ is the \LL measure of $\phi_-$ or of the descending ladder height process and $U_{\pls}$ is the potential measure associated to $\phi_+$ or of  the ascending ladder height process, see Section \ref{subsec:LP} and relation \eqref{eq:LTU}. We now extend \eqref{eq:Vigon} to all $\Psi\in\overNc$.\newpage
  	\begin{proposition}\label{prop:Vigon}
  		Let $\Psi\in\overNc$. Then \eqref{eq:Vigon} holds.
  	\end{proposition}
  	\begin{proof}
  		Recall that $\Psi^\sharp(z)=\Psi(z)-\Psi(0)\in\overNc$ and $\Psi^\sharp$ corresponds to a conservative \LL process.  Set $-\Psi(0)=q$. From Lemma \ref{lem:WH} we have that $\mu_{\mis}(dy)=\IntOI e^{-qt}\mu^\sharp_{\mis}\lbrb{dt,dy},\,y>0$. However, from \cite[Corollary 6, Chapter 5]{Doney-07-book} we have that
  		\[\mu_{\mis}^\sharp(dt,dy)=\IntOI U^\sharp_{\pls}(dt,dv)\Pnn(v+dy),\]
  		where $U_{\mis}^\sharp$ is the bivariate potential measure associated to $\lbrb{\lbrb{\zeta^{\mis}}^\sharp,\lbrb{H^{\mis}}^\sharp}$, see \cite[Chapter 5]{Doney-07-book} for more details. Therefore,
  		\[\mubarn{y}=\IntOI\IntOI e^{-qt} U^\sharp_{\pls}(dt,dv)\PPn(v+y).\]
  		Assume first that the underlying \LLP is not a compound  Poisson process. Then from \cite[p.~50]{Doney-07-book} and Lemma \ref{lem:WH}, we have,  for any $\eta>0$,
  		\[\frac{1}{\php(\eta)}=\frac{1}{\kappa_{\pls}(q,\eta)}=\IntOI e^{-\eta v}\IntOI e^{-qt}U^{\sharp}_{\pls}(dt,dv)\]
  		and from Proposition \ref{propAsymp1}\eqref{it:bernstein_cmi} we conclude that $U_{\pls}(dv)=\IntOI e^{-qt}U^{\sharp}_{\pls}(dt,dv)$ since
  		\[\frac{1}{\php(\eta)}=\IntOI e^{-\eta v}U_{\pls}(dv).\]
  		Thus by plugging the espression for $U_{\pls}$ in the relation for $\mubarnspace{y}$ above, \eqref{eq:Vigon} is established for any $\Psi\in\overNc$ such that the underlying \LLP is not a compound  Poisson process. In the case of compound  Poisson process the claim of \eqref{eq:Vigon} follows easily by noting that $\php(\eta)=h(q)\kappa_{\pls}(q,\eta)$ and hence the relation
  		\[U_{\pls}(dv)=\frac{1}{h(q)}\IntOI e^{-qt}U^{\sharp}_{\pls}(dt,dv).\]
  		This concludes the proof of the statement.
  		
  	\end{proof}
  	Next we investigate when from \eqref{eq:Vigon} it can be deduced that $\mu_{\mis}$ has a density. When $\Psi(0)=0$ the ensuing result has been established by Vigon, see \cite[5.3.4]{Doney-07-book}.
  	\begin{proposition}\label{prop:VigonDens}
  		Let $\Psi\in\overNc$ and $\dep>0$. Then $U_{\pls}(dx)=u_{\pls}(x)dx,x>0,$ and
  		\begin{equation}\label{eq:mu_-2}
  		\begin{split}
  		\uun{y}=\IntOI u_{\pls}(v)\Pnn(y+dv)=\int_{y}^{\infty}u_{\pls}(v-y)\Pnn(dv).
  		\end{split}
  		\end{equation}
  		is a right-continuous version of the density of $\mu_{\mis}$. If $\PPn(0)<\infty$ then $\upsilon_{\mis}(0^+)=\int_{0}^{\infty}u_{\pls}(v)\Pnn(dv)\in\lbbrb{0,\infty}$ and otherwise $\upsilon_{\mis}(0^+)=\infty$.
  	\end{proposition}
  	\begin{proof}
  		When $\dep>0$ we know from Proposition \ref{propAsymp1} \ref{it:bernstein_cmi} that the potential density exists and  $u_{\pls}\in\Ctt\!\lbrb{\lbbrb{0,\infty}}$. Next, for any $x>0$ we integrate the second relation in \eqref{eq:mu_-2} on $\lbrb{x,\infty}$ to get
  		\begin{equation*}
  		\begin{split}
  		&\int_{x}^{\infty}\int_{y}^{\infty}u_{\pls}(v-y)\Pnn(dv)dy=\int_{x}^{\infty}\int_{x}^{v}u_{\pls}(v-y)	dy\Pnn(dv)\\
  		&=\int_{x}^{\infty}U_{\pls}(v-x)\Pnn(dv)=\int_{0}^{\infty}\PPn\lbrb{v+x}u_{\pls}(v)dv\stackrel{\ref{eq:mu_-1}}{=}\mubarnspace{x}.
  		\end{split}
  		\end{equation*}
  		Thus relation \eqref{eq:mu_-2} is established since $\dep>0$ leads to  $h(q)=1$ in \eqref{eq:mu_-1}. When $\PPn(0)<\infty$ then since $u_{\pls}\in\Ctt\!\lbrb{\lbbrb{0,\infty}}$ we can take right limit in \eqref{eq:mu_-2} and thus $\upsilon_{\mis}(0^+)=\int_{0}^{\infty}u_{\pls}(v)\Pnn(dv)\in\lbbrb{0,\infty}$. Let $\PPn(0)<\infty$ and choose $\epsilon>0$ small enough such that $u_{\pls}\geq \frac{1}{2\dep}$ on $\lbrb{0,\epsilon}$. Then from \eqref{eq:mu_-2}
  		\begin{equation*}
  		\begin{split}
  		 \upsilon_{\mis}(0^+)&=\limo{y}\int_{y}^{\infty}u_{\pls}(v-y)\Pnn(dv)\geq \frac{1}{2\dep}\limo{y}\int_{y}^{\epsilon}\Pnn(dv)\\
  		&= \frac{1}{2\dep}\limo{y}\lbrb{\PPn\lbrb{y}-\PPn\lbrb{\epsilon}}=\infty.
  		\end{split}
  		\end{equation*} 
  		This concludes the proof of the proposition.
  	\end{proof}

  	\section{Some remarks on killed \LLPs}\label{subsec:killedLP}\label{sec:LP}
  	The next claim is also a general fact that seems not to have been recorded in the literature at least in such a condensed form.
  	\begin{proposition}\label{prop:convPhi}
  		Let $\Psi\in\overNc$ and for any $q>0$, \[\Psi^{\dagger q}(z)=\Psi(z)-q=\php^{\dagger q}(-z)\phn^{\dagger q}(z),\,z\in i\R,\] with the notation $\lbrb{\phi^{\dagger q}_{\pms}(0),\dr_\pms^{\dagger q},\mu_{\pms}^{\dagger q}}$ for the triplets defining the Bernstein functions $\phi_\pms^{\dagger q}$. Then, for any $q>0$, $\dr^{\dagger q}_\pms=\dr_\pms$ and \[\bar{\mu}^{\dagger q}_{\pls}(0)=\infty\iff \mubarpspace{0}=\infty \text{ and } \bar{\mu}^{\dagger q}_{\mis}(0)=\infty\iff \mubarnspace{0}=\infty.\]
  		Moreover, we have that vaguely  on $\intervalOI$ in general and weakly  on $\intervalOI$ when some of the measures $\mu_{\pms}$ is finite, the convergence 
  		\begin{equation}\label{eq:murtomu}
  		\limo{q}\mu^{\dagger q}_{\pms}(dy)=\mu_{\pms}(dy)
  		\end{equation}
  		holds. Therefore for any $a>\aphp$ and $[b_1,b_2]\subset \R$ with $-\infty<b_1<0<b_2<\infty$
  		\begin{equation}\label{eq:convPhi}
  		\limsupo{\rk}\sup_{b\in [b_1,b_2]}\sup_{0\le q  \leq\rk}\abs{\php^{\dagger q}(a+ib)-\php(\ab)}=0
  		\end{equation}
  		and
  		\begin{equation}\label{eq:convPhi1}
  		\limsupo{\rk}\sup_{b\in \R\setminus[b_1,b_2]}\sup_{0\le q  \leq\rk}\frac{\abs{\php^{\dagger q}(a+ib)-\php(\ab)}}{|b|}=0.
  		\end{equation}
  		Relations \eqref{eq:convPhi} and \eqref{eq:convPhi1} also hold with $\phn,\phn^{\dagger q}$  for any fixed $a>\aphn$.
  	\end{proposition}
  	\begin{proof}
  		The \LLP $\xi^{\dagger q}$ underlying $\Psi^{\dagger q}$ is killed at rate $-\Psi(0)+q$ but otherwise possesses the same \LL triplet $\lbrb{\gamma,\sigma^2,\Pi}$ as $\xi$. Therefore, for any $q>0$, $\dr_\pms^{\dagger q}=\dr_\pms$,
  		\[\bar{\mu}^{\dagger q}_{\pls}(0)=\infty\iff \mubarp{0}=\infty \text{ and } \bar{\mu}^{\dagger q}_{\mis}(0)=\infty\iff \mubarn{0}=\infty\]
  		since those are local properties unaffected by the additional killing rate. Moreover, even $\mathfrak{a}_{\phi_\pms^{\dagger q}}=\mathfrak{a}_{\phi_\pms}$, see \eqref{eq:aphi1}, since the analyticity of $\Psi$ and hence of $\phi_\pms$ is unaltered. Next, the weak convergence $\limo{q}\mu^{\dagger q}_{\pms}(dy)=\mu_{\pms}(dy)$ in \eqref{eq:murtomu} follows immediately from Lemma \ref{lem:WH} as it represents $\mu^{\dagger q}_{\pms}$ in terms of the \LL measure of the ladder height processes of the conservative process underlying $\Psi^\sharp$, that is the relations
  		\begin{equation*}
  		\begin{split}
  		\mu^{\dagger q}_{\mis}(dy)&=\int_{0}^{\infty}e^{-\lbrb{q-\Psi(0)}y_1}\mu^\sharp_{\mis}(dy_1,dy)\\
  		\mu^{\dagger q}_{\pls}(dy)&=h(q-\Psi(0))\int_{0}^{\infty}e^{-\lbrb{q-\Psi(0)}y_1}\mu^\sharp_{\pls}(dy_1,dy),
  		\end{split}
  		\end{equation*}
  		and by a simple application of the monotone convergence theorem $\limo{q}h\lbrb{q-\Psi(0)}=h\lbrb{-\Psi(0)}$, see \eqref{eq:h}. It remains to prove \eqref{eq:convPhi} and \eqref{eq:convPhi1}. Fix $a>\aphp$ and $[b_1,b_2]$ as in the statement.
  		Then from the second expression of \eqref{eq:phi} and the fact that $\dr_\pms^{\dagger q}=\dr_\pms$ we arrive at
  		\begin{equation}\label{eq:phiDCT}
  		\begin{split}
  		&\sup_{b\in[b_1,b_2]}\sup_{0\le q  \leq\rk}\abs{\php^{\dagger q}(a+ib)-\php(\ab)}\leq
  		\sup_{0\leq q \leq\rk}\abs{\php^{\dagger q}(0)-\php(0)}\\ & +2\lbrb{\max\curly{|b_1|,b_2}+|a|}\IntOI e^{-ay}\sup_{0\leq q \leq\rk}\abs{\mubarp{y}-\bar{\mu}^{\dagger q}_{\pls}(y)}dy.
  		\end{split}
  		\end{equation}
  		Clearly, from the  Fristedt's formula, see \eqref{eq:Fristed}, Lemma \ref{lem:WH}, and the monotone convergence theorem when $\Psi(0)=0$ or the dominated convergence theorem when $\Psi(0)<0$ we get that
  		\begin{equation*}
  		\begin{split}
  		\limo{q}\php^{\dagger q}(0)&=\limo{q}h\lbrb{q-\Psi(0)}\kappa_{\pls}(q-\Psi(0),0)\\&=\limo{q}h\lbrb{q-\Psi(0)}e^{\IntOI\int_{\lbbrb{0,\infty}} \lbrb{e^{-t}-e^{-\lbrb{q-\Psi(0)}t}}\frac{\Pbb{\xi^\sharp_t\in dx}}{t}dt}\\
  		&=h\!\lbrb{-\Psi(0)}e^{\IntOI\int_{\lbbrb{0,\infty}} \lbrb{e^{-t}-e^{\Psi(0)t}}\frac{\Pbb{\xi^\sharp_t\in dx}}{t}dt}\\
  		&=h\!\lbrb{-\Psi(0)}\kappa_{\pls}\lbrb{-\Psi(0),0}=\php(0),
  		\end{split}
  		\end{equation*}
  		where $\xi^\sharp$ is the conservative \LLP underlying $\Psi^\sharp(z)=\Psi(z)-\Psi(0)$.
  		Next, from Lemma \ref{lem:WH} it follows that for any $y>0$ and any $\rk>0$
  		\begin{equation}\label{eq:muDCT}
  		\sup_{0\leq q \leq \rk}e^{-ay}\bar{\mu}^{\dagger q}_{\pls}(y)\leq h\!\lbrb{-\Psi(0)}e^{-ay}\bar{\mu}^\sharp_{\pls}(y)
  		\end{equation}
  		with the latter being integrable on $\intervalOI$ since $a>\aphp$. Moreover, again from Lemma \ref{lem:WH} we get that, for any $y>0$,
  		\begin{equation*}
  		\begin{split}
  		\sup_{0\leq q \leq \rk}\abs{\mubarpspace{y}-\bar{\mu}^{\dagger q}_{\pls}(y)}&=\sup_{0\leq q \leq \rk}\abs{\IntOI\lbrb{1-e^{-q t}}e^{\Psi(0)t}\mu_{\pls}^\sharp\!\lbrb{dt,\lbrb{y,\infty}}}\\
  		&\leq \rk\int_{0}^{1} t\mu^\sharp_{\pls}\!\lbrb{dt,\lbrb{y,\infty}}+\int_{1}^{\infty}\lbrb{1-e^{-\rk t}}\mu_{\pls}^\sharp\!\lbrb{dt,\lbrb{y,\infty}}
  		\end{split}
  		\end{equation*}
  		provided $\xi$ underlying $\Psi$ is not a compound  Poisson process and
  		\begin{equation*}
  		\begin{split}
  		&\sup_{0\leq q \leq \rk}\abs{\mubarpspace{y}-\bar{\mu}^{\dagger q}_{\pls}(y)}\\
  		&\leq h\!\lbrb{-\Psi(0)}\lbrb{\rk\int_{0}^{1} te^{\Psi(0)t}\mu^\sharp_{\pls}\!\lbrb{dt,\lbrb{y,\infty}}+\int_{1}^{\infty}\lbrb{1-e^{-\rk t}}e^{\Psi(0)t}\mu_{\pls}^\sharp\!\lbrb{dt,\lbrb{y,\infty}}}\\
  		&+\lbrb{h\!\lbrb{\rk-\Psi(0)}-h(-\Psi(0))}\IntOI e^{\Psi(0)t}\mu^{\sharp}_{\pls}\!\!\lbrb{dt,\Rp}
  		\end{split}
  		\end{equation*}
  		otherwise.
  		Evidently, in both cases, the right-hand  side goes to zero, for any $y>0$,  as $\rk\to 0$,  and this together with \eqref{eq:muDCT} and the dominated convergence theorem show from \eqref{eq:phiDCT} that \eqref{eq:convPhi} holds true. In fact \eqref{eq:convPhi1} follows in the same manner from  \eqref{eq:phiDCT} by first dividing by $2\max\curly{|b|+|a|}$ for $b\in\R\setminus[b_1,b_2]$ and then taking supremum in $b$.
  	\end{proof}
  	\section{Tables with frequently used symbols}\label{sec:symbols}
  	The first table describes different subclasses of the negative definite functions (NDFs) usually denoted by $\Psi$ with underlying \LLP $\xi$ and Wiener-Hopf factors $\phpm$.
  	
  	\begin{tabular}{|c|c|c|}
  		\hline
  		\rule[-1ex]{0pt}{2.5ex}  Notation & Meaning  &  Page of Appearance \\
  		\hline
  		\rule[-1ex]{0pt}{2.5ex} $\overNc$ & space of NDFs & \pageref{overNc} \\
  		\hline
  		\rule[-1ex]{0pt}{2.5ex}$\Nc$  & space of NDFs such that $\IPsi=\IntOI e^{-\xi_s}ds<\infty$   & \pageref{eq:Nc} \\
  		\hline
  		\rule[-1ex]{0pt}{2.5ex} $\Nc_\dagger$ & space of NDFs with $\Psi(0)<0$ & \pageref{eq:Ncdag} \\
  		\hline
  		\rule[-1ex]{0pt}{2.5ex}$\Npb$ &  NDFs with $\abs{\MPs}$ decaying polynomially of power $\beta$ & \pageref{eq:polyclass} \\
  		\hline
  		\rule[-1ex]{0pt}{2.5ex} $\Nth$ & NDFs with $\abs{\MPs}$ decaying exponentially of speed $\Theta$ & \pageref{eq:polyclass1} \\
  		\hline
  		\rule[-1ex]{0pt}{2.5ex}$\Nc_{\Zc}$  & NDFs with $\xi$ not supported by a lattice  &  \pageref{eq:nonlatticeclass}\\
  		\hline
  		\rule[-1ex]{0pt}{2.5ex} $\Nc_\Wc$ & NDFs in $\Nc_{\Zc}$ with a minute analytical requirement & \pageref{eq:weaknonlattice} \\
  		\hline
  		\rule[-1ex]{0pt}{2.5ex}  $\Nc_{\unrhd}$& NDFs with finite moment of ascending  height  of $\xi$ &\pageref{Nm}  \\
  		\hline
  		\rule[-1ex]{0pt}{2.5ex} $\Pc$ & space of positive definite functions  &\pageref{sym:P}  \\
  		\hline
  	\end{tabular}
  	
  	The next table collects some commonly used symbols that depend on different functions.
  	
  	\begin{tabular}{|c|c|c|}
  		\hline
  		\rule[-1ex]{0pt}{2.5ex} Notation & Meaning  & Page of Appearance \\
  		\hline
  		\rule[-1ex]{0pt}{2.5ex}  $\Zc_a\!\lbrb{f}$& the zeros of the function $f(-\cdot)$ on the  line $a+i\R\,\,$ & \pageref{sym:P} \\
  		\hline
  		\rule[-1ex]{0pt}{2.5ex}  $\IPsi$& the exponential functional of \LLPs & \pageref{eq:Ipsi} \\
  		\hline
  		\rule[-1ex]{0pt}{2.5ex}  $\MPsi$& the Mellin transform of $\IPsi$\quad\quad & \pageref{sym:MT} \\
  		\hline
  		\rule[-1ex]{0pt}{2.5ex}  $\aph,\tph,\dph$& analiticity and roots of the Bernstein functions $\phi$ & \pageref{eq:aphi} \\
  		\hline
  		\rule[-1ex]{0pt}{2.5ex}  $\gamma_\phi$& generalized Euler-Mascheroni constant for $\phi\in\Bc$ & \pageref{sym:gphi} \\
  		\hline
  	\end{tabular}
  	
  	The next table presents subclasses of Bernstein functions (BFs) and contains generic quantities pertaining to $\phi\in\Bc$.
  	
  	\begin{tabular}{|c|c|c|}
  		\hline
  		\rule[-1ex]{0pt}{2.5ex} Notation & Meaning  & Page of Appearance \\
  		\hline
  		\rule[-1ex]{0pt}{2.5ex}  $\Bc$& the space of BFs & \pageref{Bc} \\
  		\hline
  		\rule[-1ex]{0pt}{2.5ex} $\BP$ & subset of BFs with positive linear term, i.e. $\dr>0$  &  \pageref{eq:BP}\\
  		\hline
  		\rule[-1ex]{0pt}{2.5ex} $\BP^c$ &  the complement of $\BP$, that is  $\dr=0$& \pageref{BPc} \\
  		\hline
  		\rule[-1ex]{0pt}{2.5ex} $\Beb$ & BFs with $\abs{\Wp}$ decaying polynomially of power $\beta$ & \pageref{eq:polyclassB} \\
  		\hline
  		\rule[-1ex]{0pt}{2.5ex}  $\Bth$& BFs with $\abs{\Wp}$ decaying exponentially of speed $\theta$ &  \pageref{eq:polyclassB1}\\
  		\hline
  		\rule[-1ex]{0pt}{2.5ex} $\Bc_{\alpha}$ & BFs with $\dr=0$ and  regular variation of $\mu$ &\pageref{eq:classesPhi1}  \\
  		\hline
  	\end{tabular}
  	
  	The next table gathers some functional spaces and domains.
  	
  	\begin{tabular}{|c|c|c|}
  		\hline
  		\rule[-1ex]{0pt}{2.5ex} Notation & Meaning & Page of Appearance  \\
  		\hline
  		\rule[-1ex]{0pt}{2.5ex} $\Cb_I$ & complex numbers $z\in\Cb$ such that $\Re(z)\in I\subseteq\R$ & \pageref{CbI} \\
  		\hline
  		\rule[-1ex]{0pt}{2.5ex} $\Cb_a$ & $z\in\Cb$ with $\Re(z)=a$ &  \pageref{Ca}\\
  		\hline
  		\rule[-1ex]{0pt}{2.5ex} $\Cb_{\lbrb{a,b}}$ &  $z\in\Cb$ with $a<\Re(z)<b$ & \pageref{Ca} \\
  		\hline
  		\rule[-1ex]{0pt}{2.5ex} $\Ac_{\lbrb{a,b}}$ & the space of holomorphic functions on $\Cb_{\lbrb{a,b}}$ &  \pageref{Ca}\\
  		\hline
  		\rule[-1ex]{0pt}{2.5ex} $\Ac_{\lbbrb{a,b}}$ & functions in $\Ac_{\lbrb{a,b}}$ with continuous extension to $\Cb_a$ & \pageref{Ca} \\
  		\hline
  		\rule[-1ex]{0pt}{2.5ex}$\Mtt_{\lbrb{a,b}}$  & the space of meromorphic functions on $\Cb_{\lbrb{a,b}}$  & \pageref{Ca} \\
  		\hline
  		\rule[-1ex]{0pt}{2.5ex} $\Ctt^k(\Kb)$ & $k$-times differentiable functions on $\Kb$ & \pageref{Ck} \\
  		\hline
  		\rule[-1ex]{0pt}{2.5ex} $\cco^k\lbrb{\Rp}$& $\Ctt^k(\Rp)$ functions with derivatives vanishing at $\infty$  &\pageref{Cok} \\
  		\hline
  		\rule[-1ex]{0pt}{2.5ex}  $\Ctt_b^k\lbrb{\Rp}$ &functions in $\Ctt^k(\Rp)$ with bounded derivatives & \pageref{Cob} \\
  		\hline
  	\end{tabular}
  \end{appendix} 




\textbf{Acknowledgement:} The authors are grateful to two anonymous referees for a thorough reading and insightful comments that lead to a substantial improvement of the quality and presentation of the manuscript. They also would like to thank Adam Barker for his careful reading of the manuscript and for pointing out numerous typos and errors.


\end{document}